\documentclass{amsart}

\usepackage{amsmath}
\usepackage{amssymb}
\usepackage{mathtools}
\usepackage{url}
\usepackage{hyperref}
\usepackage{autonum}
\usepackage{xspace}
\usepackage{xcolor}
\usepackage{hyperref}
\usepackage[msc-links]{amsrefs}

\usepackage{graphicx}
\usepackage{tikz}
\usepackage{pgfplots} 

\usepackage{caption}
\usepackage{subcaption}

\numberwithin{equation}{section}

\hypersetup{
		colorlinks,
		linkcolor={red!50!black},
		citecolor={blue!50!black},
		urlcolor={blue!80!black}
	}

\definecolor{color0}{rgb}{0.12156862745098,0.466666666666667,0.705882352941177}
\definecolor{color1}{rgb}{1,0.498039215686275,0.0549019607843137}
\definecolor{color2}{rgb}{0.172549019607843,0.627450980392157,0.172549019607843}
\definecolor{color3}{rgb}{0.83921568627451,0.152941176470588,0.156862745098039}
\definecolor{color4}{rgb}{0.580392156862745,0.403921568627451,0.741176470588235}
\definecolor{color5}{rgb}{0.549019607843137,0.337254901960784,0.294117647058824}
\definecolor{color6}{rgb}{0.890196078431372,0.466666666666667,0.76078431372549}

\usepackage{enumitem}
\setenumerate{label=\upshape(\alph*),itemsep=3pt,topsep=3pt}

\newtheorem{theorem}{Theorem}[section]
\newtheorem{lemma}[theorem]{Lemma}
\newtheorem{corollary}[theorem]{Corollary}
\newtheorem{proposition}[theorem]{Proposition}
\newtheorem{definition}[theorem]{Definition}
\newtheorem{remark}[theorem]{Remark}
\newtheorem{assumption}[theorem]{Assumption}
\newtheorem{conclusion}[theorem]{Conclusion}

\newcommand{\Frechet}{Fréchet\xspace}

 \newcommand{\quotes}[1]{``#1''}

\newcommand{\dint}[1]{\,\textup{d} #1}

\DeclareMathOperator{\Real}{Re}
\DeclareMathOperator{\Imag}{Im}

\DeclareMathOperator{\vol}{vol}
\newcommand{\ci}{\mathrm{i}}

\newcommand{\ext}{\mbox{\tiny ext}}

\newcommand{\deltah}{h} 
\newcommand{\deltaH}{H} 

\newcommand{\kommentare}[1]{}
\newcommand{\N}{\mathbb{N}}

\newcommand{\ip}[2]{ ( #1,#2 ) }
\newcommand{\dualp}[2]{ \langle #1,#2 \rangle}
\newcommand{\abil}[2]{b(#1,#2)}
\newcommand{\LOD}{{\mbox{\normalfont\tiny LOD}}}
\newcommand{\abilmag}[2]{a_{\MagPot}(#1,#2)}
\newcommand{\abilmagLOD}[2]{a_{\bfAapp}^{\LOD}(#1,#2)}

\newcommand{\abilthree}{r} 

\newcommand{\abilusimple}[2]{b(#1,#2)}

\newcommand{\ipsymLtwo}[2]{ m ( #1,#2 ) }

\newcommand{\bfH}{\mathbf{H}} 
\newcommand{\bfA}{\mathbf{A}} 
\newcommand{\bfAapp}{\mathbf{A}_{\star}}  
\newcommand{\kappainv}{\frac{1}{\kappa}}
\newcommand{\kappainvci}{\frac{\ci}{\kappa}}

\newcommand{\Hcurl}{\mathbf{H}(\curl)}
\newcommand{\HcurlOm}{\mathbf{H}(\curl,\Omega)}
\newcommand{\Hdiv}{\mathbf{H}(\div)}
\newcommand{\HdivOm}{\mathbf{H}(\div,\Omega)}

\newcommand{\fext}{f_{\ext}}
\newcommand{\uext}{u_{\ext}}
\newcommand{\MagPotext}{\MagPotgen_{\ext}}

\newcommand{\Omegaext}{\widehat{\Omega}}
\newcommand{\Omegaextk}[1]{\Omegaext_{#1}}

\newcommand{\HonekappaSpace}{H^1_\kappa}
\newcommand{\HtwokappaSpace}{H^2_\kappa}
\newcommand{\Honekappa}[1]{\norm{\HonekappaSpace}{#1}}

\newcommand{\HonenSpace}{\mathbf{H}^1_\mathrm{n}}
\newcommand{\HonenSpaceDiv}{\mathbf{H}^1_{\mathrm{n},\div}}

\newcommand{\Htwokappa}[1]{\norm{\HtwokappaSpace}{#1}}

\newcommand{\HonecombiSpace}{\HonekappaSpace \times \mathbf{H}^1}
\newcommand{\Honecombi}[1]{\norm{\HonecombiSpace}{#1}}

\newcommand{\h}{h}
\newcommand{\testfun}{\varphi}

\newcommand{\Ltwoproj}{\pi_{\deltaH}}
\newcommand{\LtwoprojFEM}{\pi_{\deltaH}^{\mbox{\tiny\normalfont FEM}}}
\newcommand{\LtwoprojLOD}{\pi_{\deltaH}^{\LOD}}

\newcommand{\boldL}{L}

\newcommand{\errh}{e_{\deltaH}}
\newcommand{\Errh}{\mathbf{E}_{\deltah}}

\newcommand{\Csol}{\rho_{\sol,\MagPot}(\kappa)} 
\newcommand{\Csolinv}{\rho_{\sol,\MagPot}(\kappa)^{-1}}

\newcommand{\Honeperp}{(\ci \sol)^{\perp}}

\newcommand{\energy}{E_{\mbox{\tiny GL}}} 
\newcommand{\energystab}{E} 

\newcommand{\duenergy}{\partial_{\sol} \energystab}
\newcommand{\dutwoenergy}{\partial^2_{\sol} \energystab}

\newcommand{\dAenergy}{\partial_{\MagPot} \energystab}
\newcommand{\dAtwoenergy}{\partial^2_{\MagPot} \energystab}

\newcommand{\dudAtwoenergy}{\partial_{\sol,\MagPot}  \energystab}
\newcommand{\dAdutwoenergy}{\partial_{\MagPot,\sol}  \energystab}

\renewcommand{\div}{\operatorname{div}}
\newcommand{\curl}{\operatorname{curl}}
\newcommand{\diag}{\operatorname{diag}}

\newcommand{\intervalop}[1]{(#1)}

\newcommand{\norm}[2]{\lvert| #2 \rvert|_{#1}}

\newcommand{\sol}{u}

\newcommand{\MagPot}{\bfA} 
\newcommand{\MagPotgen}{\mathbf{U}} 
\newcommand{\MagPotgentwo}{\mathbf{V}} 
\newcommand{\MagF}{\mathbf{H}}

\newcommand{\rhs}{f}
\newcommand{\rhsord}{f^1}
\newcommand{\rhsA}{\mathbf{F}^2}

\newcommand{\dx}{\,\,\mbox{d}x}

\newcommand{\R}{\mathbb{R}}
\newcommand{\C}{\mathbb{C}}
\newcommand{\VSh}{V_{H}}
\newcommand{\VShLOD}{V_{H}^{\LOD}}

\newcommand{\VSAhzero}[1]{\mathbf{V}_{\h,0}^{#1}}

\usepackage{geometry}
\begin{document}

\title[A multiscale approach to the stationary Ginzburg--Landau equations]
{A multiscale approach to the stationary Ginzburg--Landau equations of superconductivity}

\date{\today}

\author[C.~D{\"o}ding]{Christian D{\"o}ding}
\address{Institute for Numerical Simulation, University of Bonn, 53115 Bonn, Germany}
\email{doeding@ins.uni-bonn.de}

\author[B. D{\"o}rich]{Benjamin D{\"o}rich}
\address{Institute for Applied and Numerical Mathematics, 
	Karlsruhe Institute of Technology, 76149 Karlsruhe, Germany}
\curraddr{}
\email{benjamin.doerich@kit.edu}
\thanks{The first author is funded by the Deutsche Forschungsgemeinschaft (DFG, German Research Foundation) under Germany's Excellence Strategy – EXC-2047/1 – 390685813. The second author is funded by the Deutsche Forschungsgemeinschaft (DFG, German Research Foundation) -- Project-ID 258734477 -- SFB 1173. 
	The third authors also acknowledges the support by the Deutsche Forschungsgemeinschaft  through the grant HE 2464/7-1.
}

\author[P. Henning]{Patrick Henning}
\address{Department of Mathematics, Ruhr-University Bochum, 44801 Bochum, Germany}
\curraddr{}
\email{patrick.henning@rub.de}



\date{\today} 

\dedicatory{}

\begin{abstract}
In this work, we study the numerical approximation of minimizers of the Ginzburg--Landau free energy, a common model to describe the behavior of superconductors under magnetic fields. The unknowns are the order parameter, which characterizes the density of superconducting charge carriers, and the magnetic vector potential, which allows to deduce the magnetic field that penetrates the superconductor. Physically important and numerically challenging are especially settings which involve lattices of quantized vortices which can be formed in materials with a large Ginzburg--Landau parameter $\kappa$. In particular, $\kappa$ introduces a severe mesh resolution condition for numerical approximations. In order to reduce these computational restrictions, we investigate a particular discretization which is based on mixed meshes where we apply a Lagrange finite element approach for the vector potential and a localized orthogonal decomposition (LOD) approach for the order parameter. We justify the proposed method by a rigorous a-priori error analysis (in $L^2$ and $H^1$) in which we keep track of the influence of $\kappa$ in the main error contributions. This allows us to conclude $\kappa$-dependent resolution conditions for the various meshes and which only impose moderate practical constraints compared to a conventional finite element discretization. 
While our results only provide information on the approximability of the minimizers, we conclude by further proposing a minimization procedure to illustrate
our theoretical findings by numerical experiments.
\end{abstract}

\maketitle

\section{Introduction}
\label{sec:intro}

In most materials the flow of an electric current is countered with an electric resistance which leads to a loss in energy. Materials with no electrical resistance, usually referred to as superconductors, are rare in nature, but open up a large variety of possible applications.  
To consider a mathematical model for superconductivity, we let $\Omega \subset \R^3$ denote a cuboid which is occupied by the superconducting material. The superconductivity itself is described by a complex-valued wave function 
$\sol \colon \Omega \rightarrow \mathbb{C}$ which is called the order parameter. Though not a physical observable on its own, we can extract from $\sol$ the density of the superconducting electron pairs $|\sol|^2$. This density is real-valued and can be observed in physical experiments. In fact, the scaling in the corresponding models enforces $0 \leq |\sol|^2 \leq 1$, where $|\sol(x)|^2  = 0$ implies that the material is not superconducting (in normal state) in $x\in\Omega$ and $|\sol(x)|^2  = 1$ implies a perfect superconductor, locally in $x$. In between, the percentage of superconducting charge carriers might drop to a value between $0$ and $1$. In these mixed normal-superconducting states, both phases can coexist in a so-called Abrikosov vortex lattice \cite{Abr04} with $|\sol(x)|^2=0$ in the vortex centers. These kinds of configurations can only occur for so-called type-II superconductors when a sufficiently strong (but not too strong) external magnetic field $\mathbf{H}$ is applied. In fact, in mixed normal-superconducting states, the magnetic field partially penetrates the superconducting material. In this paper, we focus on exactly these kind of settings.

Considering such a situation, the relevant order parameter and the unknown internal magnetic field can be characterized as minimzers of the so-called Ginzburg--Landau (GL) free energy (cf. \cite[Sec.~3]{DuGP92}), which is given by
\begin{eqnarray} \label{eq:energy_functional}
	\energy(\sol,\MagPot) 
	&\coloneqq&
	\frac{1}{2} \int_\Omega 
	|\kappainvci \nabla \sol +   \MagPot \sol |^2 
	+ 
	\frac{1}{2} \bigl( 1- |\sol|^2 \bigr)^2
	+
	|\curl \MagPot - \MagF |^2
	\dx.
\end{eqnarray}
Here, $\MagF$ is a given external magnetic field and $\kappa \in \mathbb{R}^+$ is a material parameter, often called the Ginzburg--Landau parameter. The unknown $\MagPot$ denotes the magnetic vector potential, from which we can obtain the internal magnetic field given by $\curl \MagPot$. In fact, besides the density $|\sol|^2$, the magnetic field $\curl \MagPot$ is the second physical quantity of interest. Recalling that we are interested in vortex states, it is important to note that the size of the parameter $\kappa$ determines the structure of the vortex lattice \cite{SaS07,SaS12,Ser99,SeS10}. In particular, for small values of $\kappa$, no vortices will appear, whereas with increasing $\kappa$, the number of vortices grows  and they become more localized \cite{Aftalion99,SaS07}. Thus, the regime of large values of $\kappa$ is the physically most interesting regime accompanied by many challenges for its numerical approximation. The  phenomenon is also closely related to the appearance of vortices in superfluids \cite{Du2003}.

First results on the numerical approximation of minimizers to \eqref{eq:energy_functional} were obtained in the pioneering works by Du, Gunzburger and Peterson \cite{DuGP92,DuGP93} who derived $H^1$-error estimates in finite element (FE) spaces for both the order parameter $u$ and the magnetic vector potential $\MagPot$. Even though estimates of optimal convergence order could be provided in a (joint) mesh parameter $H$, the proof techniques could not take into account the precise role of $\kappa$ and how it affects potential constraints on the mesh size for $\sol$ and $\MagPot$ respectively. An error analysis for alternative discretizations based on a covolume method \cite{DuNicolaidesWu98} or a finite volume method \cite{QuJu05} can be also found in the literature. However, both works only establish convergence, but no rates in $H$ and $\kappa$.

First error estimates that are indeed explicit with respect to $\kappa$ and the mesh size $H$ were recently obtained in \cite{DoeH24} for a finite element discretization of \eqref{eq:energy_functional}, however, in a simplified setting where the vector potential $\MagPot$ was assumed to be given and the minimization of the energy only involved $\sol$. In this case, the last term in \eqref{eq:energy_functional} can be dropped. In the aforementioned work, it was found that the mesh size $H$ needs to fulfill a resolution condition of at least $H \lesssim \kappa^{-1}$ to obtain reliable approximations.
Even more, the error estimates indicated a pre-asymptotic convergence regime, subject to a second resolution condition that depends on the strength of local convexity of $\energy$ in a neighborhood of a minimizer, i.e., on the smallest eigenvalues of $\energy^{\prime\prime}(\sol)$. In fact, the numerical experiments indicated that this resolution condition is not an artifact of the analysis, but FE spaces with too coarse meshes are not able to capture the correct vortex patterns, which leads to significant practical constraints.

To overcome these constraints it was suggested in \cite{DoeH24} to use a discretization based on Localized Orthogonal Decomposition (LOD). This idea was later realized in \cite{BDH24} (still in the setting of given $\MagPot$). The LOD is a numerical homogenization technique designed by M{\aa}lqvist and Peterseim \cite{MaP14} to tackle elliptic multiscale problems. In the last decade it was generalized multiple times and applied to a large variety of different problems, where we exemplary refer to \cite{DoHaMa23,DHW24,GHV17,EHMP19,HaP23,HeP13,LaLjMa24,LjMaMa22,MaPer18,Mai21,MaVer22,Pet17,WuZh22} and the reference therein, as well as to the reviews given in \cite{AHP21acta} and \cite{MaP21}. Applying the LOD to approximate minimizers of the Ginzburg--Landau energy can be motivated with its fast convergence under comparably weak regularity assumptions. This is achieved by constructing an approximation space (the LOD space) that contains problem-specific information, in particular it is based on $\kappa$ and $\MagPot$. A comprehensive error analysis of the resulting method was given in \cite{BDH24}, revealing that the $\kappa$-dependent resolution conditions can be indeed relaxed with this strategy and correct vortex patters could be computed on rather coarse meshes. Furthermore, the locality for the LOD shape functions was quantified, where it was found that approximate shape functions with a diameter of order $\mathcal{O}(\log(H \kappa) H)$ are sufficient to preserve the overall approximation properties of the ideal LOD method. However, in the aforementioned work the LOD spaces were computed in an offline phase and the error analysis was only carried out for the simplified energy under the assumption that the vector potential $\MagPot$ is a-priori known and not part of the minimization process.

Turning to the full problem \eqref{eq:energy_functional}, one would naively try and use an LOD approach on both $\sol$ and $\MagPot$. However, this is computationally extremely expensive and, as one can see from our error analysis and numerical experiments, there is typically no need for a fine resolution of the vector potential $\MagPot$. This motivates one of the main questions of this work: What are suitable, possibly different discretizations of the pair $(\sol,\MagPot)$ such that we can achieve high convergence rates under possibly low regularity assumptions and weak resolution conditions and can we quantify the corresponding error if different ansatz spaces are used for $\sol$ and $\MagPot$?
The key in the error analysis of \cite{DoeH24} was a detailed a-priori analysis of the continuous and discrete minimizers in order to obtain bounds which are sharp in their $\kappa$-scaling. 
Following this approach, we generalize these techniques and derive bounds for continuous and discrete minimizers which are explicit in $\kappa$. Most notably, we show that the bounds on $\sol$ in Sobolev norms depend on $\kappa$ whereas the bounds on $\MagPot$ up to its second derivative are independent of $\kappa$, which resembles the simpler structure of $\MagPot$ even in large $\kappa$-regimes. Denoting by $H$ and $h$ the spatial mesh size of the discrete spaces for $\sol$ and $\MagPot$ respectively, our error bounds allow us to extract optimal coupling conditions between $H$ and $h$ depending on the polynomial degree of the FE space and the parameter $\kappa$.
The main challenge is the intrinsic coupling of $\sol$ and $\MagPot$ within the energy and the \Frechet derivatives. Here, again it turned out to be crucial to work with appropriately scaled norms for both quantities
in order to derive sharp estimates in $\kappa$.
Also in the error analysis we need to carefully balance the distribution
of regularity and integrability in our estimates 
such that no superfluous powers of $\kappa$ enter the final error bounds. 
Our experiments then indeed confirm that these bounds are optimal
with one minor exception. In our theory, the $H^3$-norm of $\MagPot$ is expected to grow linearly in $\kappa$ which is not visible in the experiments. To us it remains open whether this is an artifact of the analysis or other examples could support our theory. We note that we chose $\Omega$ as a cuboid to obtain all necessary regularity estimates in a rigorous way even without assuming a smooth boundary.
Furthermore, let us emphasize that our results are only concerned with the approximability of the exact minimizers by their discrete counterparts, but not with the analysis of iterative methods for finding such minimizers. 
To perform our numerical experiments, we propose a novel minimization algorithm, however proving its convergence is far beyond the scope of this work.

Finally, let us mention that there has also been a lot of work on the time-dependent Ginzburg--Landau equation which typically has a gradient flow structure and which is used to describe the dynamics within a superconductor. Corresponding numerical methods, convergence results and error estimates can be for example found in \cite{Bartels2005,Bar05,BarMO11,Chen97,CheD01,Du94b,Du94,Du97,DuGray96,DuanZhang22,GJX19,GaoS18,Li17,LiZ15,LiZ17,MaQ23} and the references therein. To the best of our knowledge, questions regarding vortex-resolution conditions depending on $\kappa$ and $\MagPot$ have not yet been studied in the time-dependent case. Due to the different nature of the time-dependent problem, we will not discuss the equation any further here.

\smallskip

The rest of the paper is organized as follows:
In Section~\ref{sec:ana_frame}, we introduce the analytical framework and present several results on a-priori bounds and the regularity of the continuous minimizers. Furthermore, we discuss the gauge conditions and study the resulting properties of the \Frechet derivatives. 
The core findings of our paper are stated 
in Section~\ref{sec:space_main}. Here we present the construction of LOD spaces in our problem setting together with our corresponding main results. The main results are proved step by step in the sections after.
First, in Section~\ref{sec:ana_prep}, we provide further analytical findings which are crucial for the later error estimates. An abstract error analysis is then established in Section~\ref{sec:abstract_error_analysis}.
Finally, the abstract results are applied in Section~\ref{sec:error_analysis} to the considered LOD discretization and we give the corresponding proofs to our main results.
Numerical experiments which illustrate our theoretical findings are shown in Section~\ref{sec:num_exp}.
The regularity theory and technical computations are postponed to the Appendices~\ref{sec:H2_reg_A} and \ref{sec:app_computations}.

\subsection*{Notation}
For a complex number $z \in \mathbb{C}$, we use $z^*$ for the complex conjugate of $z$. In the whole paper we further denote by $L^2(\Omega):=L^2(\Omega,\mathbb{C})$ the Hilbert space of $L^2$-integrable complex functions, but equipped with the {\it real} scalar product $( \sol , v)_{L^2} :=\Real \int_{\Omega} v \, w^* \,dx$ for $v,w\in L^2(\Omega)$. Hence, we interpret the space as a {\it real} Hilbert space. Analogously, we equip the space $H^1(\Omega):=H^1(\Omega,\mathbb{C})$,
which will be the solution space for the order parameter,
with the scalar product $( v , w)_{L^2} +( \nabla v , \nabla w)_{L^2}$. This interpretation is crucial so that the \Frechet derivatives of $E$ are meaningful and exist on $H^1(\Omega)$. For any space $X$, we denote its dual space by $X'$. Note that this implies, that the elements of the dual space of $H^1$ consist of real-linear functionals, which are not necessarily complex-linear. For example, if $F(v):=(f,v)_{L^2}$ for some $f \in L^2(\Omega)$, then it holds $F(\alpha \, v)=\alpha\, F(v)$ if $\alpha \in \mathbb{R}$, but in general {\it not} if $\alpha \in \mathbb{C}$. \\

Throughout the paper, we let $\varepsilon>0$ denote an arbitrarily small, but fixed, constant which is independent of $\kappa$ or mesh parameters.

For the {\it real-valued} vector potentials, we use boldface letters and denote
$\mathbf{L}^2(\Omega) := L^2(\Omega;\mathbb{R}^3)$
and
$\mathbf{H}^1(\Omega)\coloneqq H^1(\Omega;\mathbb{R}^3)$. Note that functions in $H^1(\Omega)$ are complex-valued, whereas functions in $\mathbf{H}^1(\Omega)$ are real-valued. Analogously, we transfer the notation to higher order Sobolev spaces, i.e., $H^k(\Omega)\coloneqq H^k(\Omega;\mathbb{C})$ and $\mathbf{H}^k(\Omega)\coloneqq H^k(\Omega;\mathbb{R}^3)$ for  $k \in \mathbb{N}_0$.
Further, we use the standard spaces for the weak rotation and divergence, i.e., $\Hcurl = \HcurlOm$ and
$\Hdiv = \HdivOm$, both for real-valued functions.

Throughout the paper $C$ denotes a generic constant which is independent of $\kappa$ and the spatial mesh parameters $H$ and $h$, but might depend on numerical constants as well as $\Omega$, $\MagF$ and $\varepsilon$. 
In particular, we write $\alpha \lesssim \beta$ if there is a constant $C$ independent of $\kappa$, $H$ and $h$ such that $\alpha \leq C \,\beta $.

\section{Analytical framework}
\label{sec:ana_frame}
	
In the following, recall that $\Omega \subset \R^3$ denotes the computational domain which we assume to be a rectangular cuboid. 
For the naturally appearing boundary conditions of minimizers as well as the gauging process, we additionally introduce the subspace of $\mathbf{H}^1(\Omega)$ of functions with vanishing normal trace as
\begin{eqnarray} \label{eq:def_H1n}
	\HonenSpace(\Omega) &\coloneqq& \{ \, \mathbf{B} \in \mathbf{H}^1(\Omega) \, \mid \, \mathbf{B}  \cdot \nu |_{\partial \Omega} = 0 \,\}.
\end{eqnarray}
Among all order parameters $\sol \in H^1(\Omega)$ and vector potentials $\MagPot \in \mathbf{H}^1(\Omega)$, we are interested in finding a pair that minimizes the Ginzburg--Landau free energy in \eqref{eq:energy_functional}
with a given external magnetic field
$\MagF \in \Hcurl$
and a material parameter $\kappa \in \mathbb{R}^+$. 
In this setting, we seek $(\sol ,\MagPot) \in H^1(\Omega) \times \mathbf{H}^1(\Omega)$ such that
\begin{eqnarray*}
\energy(\sol,\MagPot) &=& \inf_{(v,\mathbf{B})  \in H^1(\Omega) \times\MagF^1(\Omega)}  \energy(v,\mathbf{B}).
\end{eqnarray*}
It is well known that minimizers cannot be unique since the GL energy functional $\energy$ is invariant under certain gauge transformations \cite{DuGP92}. To be precise, for any (real-valued) $\phi \in H^2(\Omega;\mathbb{R})$ we define the corresponding gauge transformation $G_{\phi} : H^1(\Omega) \times \MagF^1(\Omega) \rightarrow H^1(\Omega) \times \MagF^1(\Omega)$ by
\begin{eqnarray}
\label{def-gauge-transform}
G_{\phi}(\sol,\MagPot) \coloneqq (	\sol \, e^{\ci\kappa \phi} , \MagPot + \nabla \phi ).
\end{eqnarray}
It is easily checked that $\energy$ is gauge invariant in the sense that
\begin{eqnarray*}
\energy(\sol,\MagPot) = \energy( \,G_{\phi}(\sol,\MagPot) \,) \qquad \mbox{for all } \enspace (\sol,\phi,\MagPot) \in H^1(\Omega) \times H^2(\Omega) \times \MagF^1(\Omega).
\end{eqnarray*}
Hence, if $(\sol,\MagPot)$ is a minimizer, then $G_{\phi}(\sol,\MagPot)$ is a minimizer, too. For smooth domains, it can be shown that any pair $(\sol,\MagPot) \in H^1(\Omega) \times \MagF^1(\Omega)$ is gauge equivalent to a pair $(v,\mathbf{B}) \in H^1(\Omega) \times \MagF^1(\Omega)$ where the corresponding vector potential $\mathbf{B}$ is divergence-free and has a vanishing normal trace, cf. \cite[Lemma 3.1]{DuGP92}. To be precise, let $\phi \in H^1(\Omega;\R)$ denote the zero-average solution to the Poisson problem $- \Delta \phi = - \div \MagPot$ with inhomogeneous Neumann boundary condition $\nabla \phi \cdot \nu \vert_{\partial \Omega} = \MagPot \cdot \nu \vert_{\partial \Omega}$, then it holds $\phi \in H^2(\Omega;\R)$ (this follows by combining the results of \cite[Theorem~3.2.1.3]{Grisvard85} and \cite[Lemma~3.7 and Theorem~3.9]{GirR86} and by decomposing $\phi$ accordingly into an affine contribution, a solution to a homogeneous Poisson problem with inhomogeneous Neumann boundary condition and a solution to an inhomogeneous Poisson problem with homogeneous Neumann boundary condition). 
With this, we have $G_{\phi}(\sol,\MagPot)  \in H^1(\Omega) \times \HonenSpaceDiv(\Omega)$ where
\begin{eqnarray*}
\HonenSpaceDiv(\Omega) \coloneqq \{ \, \mathbf{B} \in \mathbf{H}^1(\Omega)\, \mid \,  \div \mathbf{B}=0  \, \mbox{ and } \, \mathbf{B}  \cdot \nu |_{\partial \Omega} = 0 \, \}.
\end{eqnarray*}
As a direct conclusion, we can, without loss of generality, restrict the minimization of $\energy$ to functions in $H^1(\Omega) \times \HonenSpaceDiv(\Omega)$. This corresponds to the Coulomb gauge of the vector potentials. Furthermore, if we restrict the vector potentials to divergence-free functions then they are only gauge equivalent to itself. In particular, if $(\sol,\MagPot), (v,\mathbf{B})  \in H^1(\Omega) \times \HonenSpaceDiv(\Omega)$, then
\begin{eqnarray*}
G_{\phi}(\sol,\MagPot) = (v,\mathbf{B}) \qquad \Longleftrightarrow \qquad \MagPot=\mathbf{B} \,\,\mbox{ and }\,\, v= \sol \, e^{\ci\kappa \phi} \mbox{ for } \phi \in \mathbb{R}. 
\end{eqnarray*}
This equivalence is easily seen by the observation that if $ \div \MagPot=0 $ and $\div(  \MagPot + \nabla \phi  )= \div \mathbf{B} =0$, then $\Delta \phi =0$ with $\nabla \phi  \cdot \nu= \MagPot \cdot \nu =0$ on $\partial \Omega$, hence, $\phi$ must be a constant and the gauge transform reduces to $G_{\phi}(\sol,\MagPot) = ( \sol \, e^{\ci\kappa \phi} , \MagPot )$ on $H^1(\Omega) \times \HonenSpaceDiv(\Omega)$. Since minimization over divergence-free functions can be cumbersome in practice, the energy can be stabilized by a penalty term that depends on the divergence of a vector field. We therefore define
\begin{equation} \label{eq:energy_functional_stab}
\energystab(\sol,\MagPot) \, := \, \energy(\sol,\MagPot)
+ \frac{1}{2} \int_\Omega 
|\div \MagPot|^2
\dint{x}
\end{equation}
and observe that any minimizer of $\energystab$ must be also a minimizer of $\energy$ and, vice versa, any minimizer of $\energy$ is gauge-equivalent to a minimizer of $\energystab$ for which $\div \MagPot=0$ holds. In the following, we can therefore restrict our analysis to the minimization of the stabilized energy $\energystab$. The existence of minimizers was proved by Du et al. \cite{DuGP92} and we have the following result.
\begin{theorem}[{\cite[Thm.~3.8]{DuGP92}}] \label{thm:exsistence_of_miminizers_cont}
There exists at least one minimizer of the energy \eqref{eq:energy_functional_stab}, i.e., there is $(\sol,\MagPot) \in H^1(\Omega) \times \HonenSpace(\Omega) $ such that
\begin{align}
\label{minimization-problem}
(\sol,\MagPot) = \underset{(v,\mathbf{B}) \in H^1(\Omega) \times \HonenSpace(\Omega) }{\mbox{\normalfont arg\,min}} E(v,\mathbf{B}).
\end{align}
In particular, for any minimizer $(\sol,\MagPot)$ the vector potential $\MagPot$ satisfies $ \div \MagPot = 0 $,
and thus also minimizes \eqref{eq:energy_functional}. 

The result remains valid for external fields $\MagF \in \mathbf{L}^2(\Omega)$. 
\end{theorem}
As discussed above, the modification of the energy functional from $\energy$ to $\energystab$ restricts the gauge transforms to divergence-free vector fields, which in turn implies that the gauge transforms $G_{\phi}(\sol,\MagPot)$ in \eqref{def-gauge-transform} are only admissible if $\phi$ is a constant real number. Consequently, we call $(\sol,\MagPot) \in H^1(\Omega) \times \HonenSpace(\Omega)$ {\it gauge equivalent} to $(v,\mathbf{B}) \in H^1(\Omega) \times \HonenSpace(\Omega)$ for $\energystab$, if and only if $\MagPot=\mathbf{B}$ and 
\begin{eqnarray}
\label{gauge-phase-shift}
v = u \, e^{\ci \omega}\,\,\, \mbox{ for some } \omega \in [-\pi , \pi ).
\end{eqnarray}
Note that $\omega$ corresponds to $\kappa \phi$ in \eqref{def-gauge-transform}.

\subsection{\Frechet derivatives and stability bounds}
Crucial components of our error analysis are the derivatives of the energy and corresponding first- and second-order conditions for minimizers. In the following, we start with summarizing the arising \Frechet derivatives of $\energystab$, where we refer to \cite[Section~3.3]{DuGP92}.
\begin{lemma}\label{lem:Frechet_der_1}
Let $\energystab$ denote the energy functional given by \eqref{eq:energy_functional_stab}, then $\energystab$ is (infinitely) \Frechet differentiable where, for any $(\sol,\MagPot) \in H^1(\Omega) \times \HonenSpace(\Omega)$, the first partial derivatives 
\begin{eqnarray*}
\duenergy(\sol,\MagPot): H^1(\Omega) \rightarrow \R \quad \mbox{and}  \quad
\dAenergy(\sol,\MagPot) : \HonenSpace(\Omega) \rightarrow \R
\end{eqnarray*}
are respectively given by 
\begin{align}  	\label{eq:ginzburg-landau-eqn}
\duenergy(\sol,\MagPot) \varphi &= 
	\Real  \int_\Omega \bigl( \kappainvci  \nabla \sol +   \MagPot \sol \bigr)  \cdot \bigl(\kappainvci  \nabla \testfun +  \MagPot \testfun \bigr)^*  
	+
	 \bigl( |\sol|^2 -1 \bigr)  \sol \testfun^* 
	\dint{x}, 
\\
	\dAenergy(\sol,\MagPot) \mathbf{B} &= 	\int_\Omega 
	 |\sol|^2 \MagPot \cdot \mathbf{B}  
	+
	\kappainv  \Real \bigl( \ci  \sol^* \nabla \sol \cdot \mathbf{B}  \bigr) 
	+
	\curl \MagPot \cdot \curl \mathbf{B}
	+
	\div \MagPot \cdot \div \mathbf{B} 
	-
	\mathbf{H} \cdot \curl \mathbf{B}
		\dint{x} 
\end{align}
for $\testfun \in H^1(\Omega)$ and $\mathbf{B} \in \HonenSpace(\Omega)$.
\end{lemma}
The first order conditions for minimizers (cf. \cite{Casas-trroeltsch-2015}) imply $\energystab^{\prime}(\sol,\MagPot)=0$ if $(\sol,\MagPot)$ fulfils \eqref{minimization-problem}. By splitting $\energystab^{\prime}(\sol,\MagPot)$ into $\duenergy(\sol,\MagPot) $ and $\dAenergy(\sol,\MagPot)$ we obtain the Ginzburg--Landau equations. For readability, we highlight this observation in the following lemma.
\begin{lemma}[Ginzburg--Landau equations]
\label{lemma-ginzburg-landau-eqns}
Let $(\sol,\MagPot) \in H^1(\Omega) \times \HonenSpace(\Omega)$ be a minimizer of problem \eqref{minimization-problem}. Then, it holds $\duenergy(\sol,\MagPot) =0$ and $\dAenergy(\sol,\MagPot)=0$. By expressing these identities in variational form, we obtain that $(\sol,\MagPot) \in H^1(\Omega) \times \HonenSpace(\Omega)$ solves the Ginzburg--Landau equations
\begin{eqnarray*}
	\Real  \int_\Omega \bigl( \kappainvci  \nabla \sol +   \MagPot \sol \bigr)  \cdot \bigl(\kappainvci  \nabla \testfun +  \MagPot \testfun \bigr)^*  
	+
	 \bigl( |\sol|^2 -1 \bigr)  \sol \testfun^* 
	\dint{x} &=& 0, 
\\
		\int_\Omega 
	 |\sol|^2 \MagPot \cdot \mathbf{B}  
	+
	\kappainv  \Real \bigl( \ci  \sol^* \nabla \sol \cdot \mathbf{B}  \bigr) 
	+
	\curl \MagPot \cdot \curl \mathbf{B}
	+
	\div \MagPot \cdot \div \mathbf{B} 
	-
	\mathbf{H} \cdot \curl \mathbf{B}
		\dint{x} &=& 0,
\end{eqnarray*}
for all $\testfun \in H^1(\Omega)$ and $\mathbf{B} \in \HonenSpace(\Omega)$.
\end{lemma}
Next, we turn to the second partial \Frechet derivatives of $\energystab$ which we later require for second order minimality conditions. The second derivatives of $\energystab$ are given as follows.
\begin{lemma} \label{lem:Frechet_der_2}
Let $\energystab$ be given by \eqref{eq:energy_functional_stab}. For $(\sol,\MagPot) \in H^1(\Omega) \times \HonenSpace(\Omega)$ we denote the second order partial \Frechet derivatives by
\begin{eqnarray*}
\langle \dutwoenergy(\sol,\MagPot) \,\,\cdot\,\,,  \varphi \rangle &\coloneqq&  \frac{\partial}{\partial \sol} ( \duenergy(\sol,\MagPot) \varphi ) :  H^1(\Omega) \rightarrow \R,\\
\langle \dAdutwoenergy(\sol,\MagPot) \,\,\cdot\,\,,    \mathbf{B} \rangle &\coloneqq&  \frac{\partial}{\partial\sol} ( \dAenergy(\sol,\MagPot)  \mathbf{B} )  :  H^1(\Omega) \rightarrow \R, \\
\langle \dudAtwoenergy(\sol,\MagPot) \,\,\cdot\,\,,   \varphi \rangle &\coloneqq&  \frac{\partial}{\partial\MagPot} ( \duenergy(\sol,\MagPot) \varphi )  :  \HonenSpace(\Omega) \rightarrow \R, \\
\langle \dAtwoenergy(\sol,\MagPot) \,\,\cdot\,\,,   \mathbf{B} \rangle &\coloneqq&  \frac{\partial}{\partial\MagPot} ( \dAenergy(\sol,\MagPot) \mathbf{B} )  :  \HonenSpace(\Omega) \rightarrow \R,
\end{eqnarray*}
where $\testfun \in H^1(\Omega)$ and $\mathbf{B} \in \HonenSpace(\Omega)$. For $\psi \in H^1(\Omega)$ and $\mathbf{C} \in \HonenSpace(\Omega)$ the derivatives are given by
	\begin{align}  
		\dualp{\dutwoenergy(\sol,\MagPot) \psi}{\varphi}  
		&= 	\Real  \int_\Omega 
		\bigl(	\kappainvci  \nabla \testfun +   \MagPot \testfun \bigr)  \cdot \bigl( \kappainvci  \nabla \psi +   \MagPot \psi \bigr)^*  
		+
		\bigl( |\sol|^2 -1 \bigr)  \testfun \psi^* + \sol^2 \testfun^* \psi^* + |\sol|^2 \testfun \psi^* 
		\dint{x},\\
		\dualp{\dAtwoenergy(\sol,\MagPot) \mathbf{B}}{ \mathbf{C} } &= \int_\Omega 
		|\sol|^2 \mathbf{C} \cdot \mathbf{B}  
		+
		\curl \mathbf{C} \cdot \curl \mathbf{B} 
		+
		\div \mathbf{C} \cdot \div \mathbf{B} 
		\dint{x},\\ 
		\dualp{\dudAtwoenergy(\sol,\MagPot) \mathbf{B}}{\testfun} &= \int_\Omega 
		2     \Real (\sol \testfun^* ) \MagPot \cdot \mathbf{B}  
		+
		\kappainv  \Real \bigl( \ci  
		\sol^* \nabla \testfun
		+
		\ci \testfun^* \nabla \sol \bigr) 
		\cdot \mathbf{B} 
		\dint{x}
	\end{align}
and $\dualp{\dudAtwoenergy(\sol,\MagPot) \mathbf{B}}{\testfun}=\dualp{\dAdutwoenergy(\sol,\MagPot) \testfun}{\mathbf{B}}$.
\end{lemma}

The proof follows with straightforward calculations.

In order to quantify the $\kappa$-dependencies in our error estimates, we also require suitable stability estimates for the minimzers $(\sol,\MagPot)$ of the energy \eqref{eq:energy_functional_stab}. For this, we use the following $\kappa$-weighted norms throughout the paper:
\begin{subequations}\label{eq:def_norms}
\begin{eqnarray} 
	\Honekappa{\varphi}^2 &=& \kappa^{-2} \norm{L^2}{\nabla \varphi }^2 + \norm{L^2}{\varphi}^2,
	\hspace{30pt}
	 \Htwokappa{\varphi}^2\,\,\,=\,\,\, \kappa^{-2} \Honekappa{ \nabla \varphi }^2 +  \norm{L^2}{\varphi}^2 ,
	\\
	\norm{H^1}{ \mathbf{B} }^2 &=& \norm{ \boldL^2}{ \mathbf{B} }^2 + \norm{\boldL^2}{\nabla \mathbf{B} }^2,
	\hspace{43pt}
	 \norm{H^2}{ \mathbf{B} }^2 \,\,\,=\,\,\, \norm{\boldL^2}{D^2 \mathbf{B} }^2 + \norm{H^1}{\mathbf{B}}^2.
\end{eqnarray}
\end{subequations}
Here we formally define the norms that involve derivatives of the (vector-valued) functions $\mathbf{B}$ as $\norm{\boldL^2}{\nabla \mathbf{B} }^2=\sum\limits_{i,k=1}^3 \norm{L^2}{ \partial_{x_i} \mathbf{B}_j }^2$ and $ \norm{\boldL^2}{D^2 \mathbf{B} }^2=\sum\limits_{i,j,k=1}^3 \norm{L^2}{ \partial_{x_ix_j} \mathbf{B}_k }^2$.

\begin{lemma}[Stability bounds] \label{lem:sol_bounds_ex_1}
Let $(\sol,\MagPot) \in H^1(\Omega) \times \HonenSpace(\Omega)$ be a minimizer of problem \eqref{minimization-problem} in Theorem~\ref{thm:exsistence_of_miminizers_cont}.
 Then, the following stability bounds hold
	\begin{align}
		|\sol| \leq 1 \, \text{a.e.},\quad
		\norm{L^2}{\tfrac{1}{\kappa} \nabla \sol} \lesssim \norm{L^2}{\sol} + \norm{L^2}{\MagPot \sol}, \quad
		\Honekappa{\sol} \lesssim 1 +  \norm{\boldL^2}{\MagF}, \quad
		%
		%
	\norm{H^1}{\MagPot}
		 &\lesssim 1 + \norm{\boldL^2}{\MagF}.
	\end{align}
\end{lemma}

\begin{proof}
	The pointwise bound
	$\norm{L^\infty}{\sol} \leq 1$ 
	is shown in \cite[Prop.~3.11]{DuGP92}.
	Next, note that
	\begin{equation}
		\energystab(0,\mathbf{0}) = \tfrac{1}{2} \vol(\Omega) + \tfrac{1}{2} \norm{\boldL^2}{\MagF}^2,
	\end{equation}
	and thus for any minimizer it holds
	\begin{equation}
	\norm{\boldL^2}{\curl \MagPot - \MagF}^2 + \norm{\boldL^2}{\div \MagPot}^2 \lesssim 1 + \norm{\boldL^2}{\MagF}^2.
	\end{equation}
	At the same time, basic manipulations give us 
	\begin{equation}
	\norm{\boldL^2}{\curl \MagPot}^2 \lesssim \norm{\boldL^2}{\curl \MagPot - \MagF}^2 + \norm{\boldL^2}{ \MagF}^2.
	\end{equation}
	Combining the two estimates yields $\norm{\boldL^2}{\curl \MagPot}+\norm{L^2}{\div \MagPot} \lesssim 1 + \norm{\boldL^2}{\MagF}$.
	Using \cite[Lemma~3.6]{GirR86} gives the $L^2$ bound,
	and \cite[Theorem~3.9]{GirR86} the $H^1$ bound on $\MagPot$.
	
	Further, we obtain with $|\sol| \leq 1$ the estimate
	\begin{equation}
		\norm{L^2}{\tfrac{1}{\kappa} \nabla u} \leq \norm{L^2}{\tfrac{1}{\kappa} \nabla \sol + \ci   \MagPot \sol }  +  \norm{L^2}{\MagPot \sol}
		\lesssim 1 +  \norm{L^2}{\MagF} 
		+  \norm{L^2}{\MagPot}
		\lesssim 1 +  \norm{L^2}{\MagF}.
	\end{equation}
It remains to prove the second estimate, which we obtain with Lemma \ref{lemma-ginzburg-landau-eqns} as
\begin{eqnarray*}
\int_{\Omega} | \tfrac{\ci}{\kappa}  \nabla \sol +   \MagPot \sol |^2 \dx 
= \int_{\Omega} ( 1- |\sol|^2) \, |\sol|^2  \dx  \le \| \sol \|_{L^2}^2.
\end{eqnarray*}
The estimate now follows with $\norm{L^2}{\tfrac{1}{\kappa} \nabla u} \leq \norm{L^2}{\tfrac{1}{\kappa} \nabla \sol + \ci   \MagPot \sol }  +  \norm{L^2}{\MagPot \sol} \le \| \sol \|_{L^2} +  \norm{L^2}{\MagPot \sol}$.
\end{proof}

\subsection{Higher order regularity of the minimizers}
Next, we will investigate the higher order regularity of minimizers together with corresponding $\kappa$-explicit regularity bounds. 

Concerning the vector potential $\MagPot$, we can characterize it as the solution $\MagPotgen \in \HonenSpaceDiv(\Omega)$ to a problem of the following form
\begin{subequations} \label{eq:curlcurl_prob}
\begin{equation}\label{eq:curlcurl_var_form}
	\int_\Omega \curl \MagPotgen \cdot \curl \mathbf{B} \,
	\dint{x} 
	=
	\int_\Omega \mathbf{F} \cdot  \mathbf{B}  +  \MagF \cdot \curl \mathbf{B}  \dint{x}  ,
	\quad \text{for all } \mathbf{B} \in \HonenSpace(\Omega) \,,
\end{equation}
where $\mathbf{F} \in L^2(\Omega;\R^3)$ can be read of Lemma~\ref{lemma-ginzburg-landau-eqns}.
Problem \eqref{eq:curlcurl_var_form} corresponds to the weak form of
\begin{equation}\label{eq:curlcurl_strong_form}
\Delta \MagPotgen = \mathbf{F} + \curl \MagF,
\qquad
\MagPotgen \cdot \nu |_{\partial \Omega} = 0,
\quad
\curl \MagPotgen \times  \nu |_{\partial \Omega} = \MagF \times  \nu |_{\partial \Omega}.
\end{equation}
\end{subequations}
For simplicity, we restrict ourselves from now on to homogeneous boundary conditions for the external magnetic field, i.e., following \cite{DuGP92}, we assume $\MagF \in \Hcurl$ with the (well-defined) traces
\begin{subequations} \label{eq:prop_H_for_Regul}
\begin{align+} 
\MagF \times  \nu |_{\partial \Omega} &= 0 \,,
\label{eq:prop_H_for_Regul_a}
 \\
\curl \MagF \cdot  \nu |_{\partial \Omega} &= 0 \,.
\label{eq:prop_H_for_Regul_b}
\end{align+}
\end{subequations}
The regularity of solutions to problem \eqref{eq:curlcurl_var_form} is presented in the following theorem. The first part is a direct consequence of \cite[Lemma~3.7]{HocJS15}. The proof of the second part is given in Appendix~\ref{sec:H2_reg_A}.

\begin{theorem} \label{thm:reg_MagPot_H2_H3_gen}
Let $\MagPotgen \in \HonenSpaceDiv(\Omega)$ be the solution of
\eqref{eq:curlcurl_var_form}
with  $\mathbf{F} \in L^2(\Omega;\R^3)$.

(a) If $\MagF$ satisfies \eqref{eq:prop_H_for_Regul_a}, then $\MagPotgen \in \mathbf{H}^2(\Omega)$ with
\begin{equation}
	\norm{H^2}{\MagPotgen} \lesssim \norm{L^2}{ \mathbf{F} } + \norm{L^2}{\curl \MagF} .
\end{equation}

(b) If in addition $\curl \MagF \in \mathbf{H}^1(\Omega)$ satisfies 
 \eqref{eq:prop_H_for_Regul_b}
 and $\mathbf{F} \cdot  \nu |_{\partial \Omega} = 0$,
then  $\MagPotgen \in \mathbf{H}^3(\Omega)$ with
\begin{equation}
	\norm{H^3}{\MagPotgen} \lesssim \norm{H^1}{ \mathbf{F} } + \norm{H^1}{\curl \MagF} .
\end{equation}
The hidden constants only depend on the domain $\Omega$.
\end{theorem}

\begin{remark} \label{exa:H2_reg_A}
(a) If $\Omega$ is a general polyhedral domain instead of a cube, similar results are available in Section $4.4$ of \cite{CosD00} if the boundary conditions in \eqref{eq:curlcurl_strong_form} are replaced by $\MagPotgen \times \nu = \div \MagPotgen = 0$ on $\partial \Omega$. 

(b) Relaxing the condition \eqref{eq:prop_H_for_Regul}	is a delicate issue. To do so, we assume $\div \MagF = 0 $ and need to find a smooth vector potential $\MagPotgentwo$ such that 
\begin{equation}
	\curl \MagPotgentwo = \MagF,
	\quad 
	 \div \MagPotgentwo = 0, 
	\quad
	\MagPotgentwo \cdot   \nu |_{\partial \Omega} = 0 
\end{equation}
holds. Then, one can replace $\MagPot$ by $\MagPot - \MagPotgentwo$ and one is in the situation of \eqref{eq:curlcurl_prob} with $\MagF = 0$. However, 
the results
for convex polyhedral domains \cite[Thm.~3.5]{GirR86}
 only yield some $\MagPotgentwo \in \mathbf{H}^1(\Omega)$. To derive the same regularity as in Theorem~$\ref{thm:reg_MagPot_H2_H3_gen}$, we would require higher regularity of $\MagPotgentwo$ in $\mathbf{H}^2(\Omega)$ or $\mathbf{H}^3(\Omega)$, respectively, but this is beyond the scope of the present paper.
\end{remark}

For our proofs, we also require higher order regularity for the order parameter. In order to obtain it, we will make use of the following auxiliary result which we prove in Appendix~\ref{sec:H2_reg_A}.
\begin{lemma}
	\label{lemma-H3-regularity-u}
	Let $\Omega \subset \mathbb{R}^3$ denote a cuboid and let $f \in H^1(\Omega)$. If $u \in H^1(\Omega)$ solves the Neumann problem
	\begin{eqnarray*}
		- \Delta u = f \mbox{ in } \Omega\qquad \mbox{and} \qquad 
		\nabla u \cdot \nu \vert_{\partial \Omega} = 0,
	\end{eqnarray*} 
	then it holds $u \in H^3(\Omega)$ with $\| u\|_{H^3(\Omega)} \lesssim \| u\|_{L^2(\Omega)}  + \| f\|_{H^1(\Omega)}$. Furthermore, if $f \in L^p(\Omega)$ for some $1<p<\infty$, then there exists a constant $C_p>0$ (depending on $\Omega$ and $p$), such that
	\begin{eqnarray*}
		\| u \|_{W^{2,p}(\Omega)} \le C_p \, ( \| u\|_{L^p(\Omega)} + \| f\|_{L^p(\Omega)}  ).
	\end{eqnarray*}
	Here, $\| \cdot \|_{W^{2,p}}$ denotes the usual $W^{2,p}$-norm on $\Omega$ with $\| u \|_{W^{2,p}}\coloneqq \sum\limits_{|\boldsymbol{\alpha}|\le 2} \| D^{\boldsymbol{\alpha}} u \|_{L^p(\Omega)}$.
\end{lemma}

With Theorem~\ref{thm:reg_MagPot_H2_H3_gen} and  Lemma~\ref{lemma-H3-regularity-u}, we can conclude the $H^2$-bounds on the vector potential $\MagPot$ and the order parameter $\sol$.

\begin{corollary} \label{cor:sol_bounds_ex_A_H2}
Let $(\sol,\MagPot) \in H^1(\Omega) \times \HonenSpace(\Omega)$ be a minimizer of problem \eqref{minimization-problem}.

(a) 
Then $\MagPot \in \mathbf{H}^2(\Omega)$ with 
\begin{align}
	\norm{H^2}{\MagPot} &\lesssim 1 + \norm{L^2}{\MagF} + \norm{L^2}{\curl \MagF}
\end{align}
and constants independent of $\kappa$. Note that the estimate implies a $L^{\infty}$-bound for $\MagPot$. Furthermore, together with the second inequality in Lemma \ref{lem:sol_bounds_ex_1} this yields $\tfrac{\| \sol \|_{\HonekappaSpace}}{\| \sol \|_{L^2}} \lesssim 1 + \norm{L^2}{\MagF} + \norm{L^2}{\curl \MagF}$.

(b)
 Then
 $\sol \in H^2(\Omega)$ and $\nabla u \cdot \mathbf{n}\vert_{\partial \Omega} =0$ with
 \begin{align}  \label{eq:W1q_bound_sol}
 \Htwokappa{\sol} &\lesssim \bigl( 1 + \norm{L^2}{\MagF} + \norm{L^2}{\curl \MagF} \bigr)^2
 		\quad \text{ and } \quad
	 	\norm{L^4}{ \tfrac{1}{\kappa} \nabla \sol} \lesssim 1
 \end{align}
and constants independent of $\kappa$.
\end{corollary}

\begin{proof}
	(a) Since $\MagPot$ is a minimizer, it holds $\dAenergy(\sol,\MagPot) \mathbf{B} = 0$ for all $\mathbf{B} \in \HonenSpace(\Omega)$,
	and hence by Lemma~\ref{lem:Frechet_der_1}
	\begin{align}
		\int_\Omega 
		\curl \MagPot \cdot \curl \mathbf{B}
		+
		\div \MagPot \cdot \div \mathbf{B}
		\dint{x} 
		=
		- \int_\Omega 
		|\sol|^2 \MagPot \cdot \mathbf{B}
		+
		\kappainv  \Real \bigl( \ci  \sol^* \nabla \sol \cdot \mathbf{B} \bigr) 
		+
		\MagF \cdot \curl \mathbf{B}
		\dint{x} \,.
	\end{align}
	Since $\div \MagPot = 0$, Theorem~\ref{thm:reg_MagPot_H2_H3_gen} (a) together with the estimates in Lemma~\ref{lem:sol_bounds_ex_1} yield the desired
	regularity and a-priori estimate for $\MagPot$.
	
(b) We follow the lines of the proof in \cite[Theorem~2.2]{DoeH24}, and only have to establish the $H^2$-bound.
We exploit $\duenergy(\sol,\MagPot) \varphi = 0$ for all $\varphi \in H^1(\Omega)$ to obtain 
that the minimizer $\sol$ satisfies
\begin{equation}
	 ( \nabla \sol , \nabla \varphi)_{L^2} = ( \rhs , \varphi)_{L^2} ,
	\quad
	\text{with}
	\quad
		f = - \kappa^2  (|\sol|^2 -1 ) \sol - 2 \ci \kappa \MagPot \nabla \sol - \kappa^2 |\MagPot|^2 \sol \,.
\end{equation}
The Sobolev embedding $H^2(\Omega) \hookrightarrow L^\infty(\Omega)$ for $\MagPot$, the estimates in part (a), and Lemma~\ref{lem:sol_bounds_ex_1} imply
\begin{equation}
\norm{L^2}{\rhs} \lesssim \kappa^2 ( 1 + \norm{L^2}{\MagF} + \norm{L^2}{\curl \MagF} )^2,
\end{equation}
which yields $\kappa^{-2} |u|_{H^2} 
\lesssim 
\kappa^{-2}  (\norm{L^2} {f} + \norm{L^2}{\sol})
 \lesssim ( 1 + \norm{L^2}{\MagF} + \norm{L^2}{\curl \MagF} )^2$ by Lemma \ref{lemma-H3-regularity-u}.
Combining this with the already established $\HonekappaSpace$- and $L^2$-bound gives the claimed $\HtwokappaSpace$-bound.
\end{proof}
In order to prove optimal order convergence rates for our numerical approximations of $\sol$ and $\MagPot$, we need to establish $H^3$-regularity for both unknowns including corresponding regularity estimates that are explicit with respect to $\kappa$. This is done in the following lemma.
\begin{lemma} \label{cor:sol_bounds_ex_A_H3}
We consider a minimizer $(\sol,\MagPot) \in H^1(\Omega) \times \HonenSpace(\Omega)$ of problem \eqref{minimization-problem}
and assume that $\curl \MagF \in \mathbf{H}^1(\Omega)$.

	(a) It holds $\MagPot \in \mathbf{H}^{3}(\Omega)$ and the estimate  
	\begin{align}
		\norm{H^3}{\MagPot} &\lesssim
		\kappa 
	\end{align}
	(with a hidden constant independent of $\kappa$).
	
	(b)
	It holds
	$\sol \in H^{3}(\Omega)$ and $\sol \in W^{2,p}(\Omega)$ for any $1<p<\infty$ with
	\begin{eqnarray}  
	\label{Wkp-estimates-u}
	\norm{H^3}{\sol} &\lesssim&  \kappa^{3}, \qquad \norm{W^{2,p}}{\sol} \,\,\,\lesssim\,\,\, C_p\, \kappa^{2}
	\qquad \mbox{and} \qquad \norm{W^{1,p}}{\sol} \,\,\,\lesssim\,\,\, C_p\, \kappa,
	\end{eqnarray}
	where $C_p>0$ depends on $p$ and can be different for $W^{1,p}$ and $W^{2,p}$.
\end{lemma}

\begin{proof}
(a) We aim to employ Theorem~\ref{thm:reg_MagPot_H2_H3_gen} and proceed as in Corollary~\ref{cor:sol_bounds_ex_A_H2} by estimating
\begin{eqnarray*}
\lefteqn{ \norm{H^1}{|\sol|^2 \MagPot 
	+
	\kappainv  \Real \bigl( \ci  \sol^* \nabla \sol  \bigr) 
	+\curl
	\MagF } }
\\
&\lesssim& 
\norm{H^1}{\MagPot}
+
\norm{L^2}{ \nabla \sol}
\norm{L^\infty}{\MagPot}
+
\kappainv\norm{L^4}{ \nabla \sol}^2
+
\kappainv \norm{H^2}{ \sol    }
+\norm{H^1}{ \curl \MagF }
\,\, \lesssim \,\, \kappa \,.
\end{eqnarray*}
The boundary conditions for $\MagPot$ and $\sol$ yield that in the notation of the theorem $\mathbf{F} \cdot  \nu |_{\partial \Omega} = 0$ holds and with \eqref{eq:prop_H_for_Regul_b} the statement follows form Theorem~\ref{thm:reg_MagPot_H2_H3_gen}.

(b) Using Lemma \ref{lem:sol_bounds_ex_1}, 
we only need to bound $f$ from the proof of Corollary~\ref{cor:sol_bounds_ex_A_H2} in the $H^1$-norm. Here we obtain with the results from Lemma \ref{lem:sol_bounds_ex_1} and Corollary \ref{cor:sol_bounds_ex_A_H2} that
\begin{eqnarray*}
\lefteqn{\norm{H^1}{\rhs}
\,\,\,\lesssim\,\,\,  \norm{H^1}{ \kappa^2 (|\sol|^2 -1 ) \sol - 2 \ci \kappa \MagPot \nabla \sol - \kappa^2 |\MagPot|^2 \sol } }
\\
&\lesssim& \kappa^2 \norm{H^1}{ \sol}
+
\kappa \norm{L^4}{ \MagPot } \norm{L^4}{\nabla \sol} 
+
\kappa \norm{L^\infty}{\MagPot}  \norm{H^2}{ \sol} 
 +
\kappa^2 \norm{L^\infty}{\MagPot}^2  \norm{H^1}{ \sol }
+
\kappa^2 \norm{H^2}{ \MagPot}^2
\,\,\,\lesssim\,\,\,  \kappa^3.
\end{eqnarray*}
Next, we use the second part of Lemma \ref{lem:sol_bounds_ex_1} with
\begin{eqnarray*}
\norm{L^4}{\rhs}
\,\,\,\lesssim\,\,\,   
\kappa^2 
+
\kappa \norm{L^{\infty}}{ \MagPot } \norm{L^4}{\nabla \sol} 
+
\kappa^2 \norm{L^\infty}{\MagPot}^2  \norm{L^{\infty}}{ \sol }
\,\,\,\lesssim\,\,\,  \kappa^2
\end{eqnarray*}
to conclude that for any $1< p \le 4$ we have
\begin{eqnarray*}
\| \sol \|_{W^{2,p}} \lesssim \norm{L^4}{\rhs} + \norm{L^4}{\sol} \lesssim \kappa^2.
\end{eqnarray*}
Next, we consider $\| \sol \|_{W^{1,8}}$ for which we obtain
\begin{eqnarray*}
\| \nabla \sol \|_{L^8}^8 &=& \int_{\Omega} \nabla \sol \cdot \nabla \sol^* |\nabla \sol|^{6} \dx
\overset{\nabla \sol \cdot \nu\vert_{\partial \Omega}=0 }{=} - \int_{\Omega} \sol  \,\, \mbox{div}( \nabla \sol^* |\nabla \sol|^{6}) \dx \\
&\overset{\mbox{\tiny H\"older}}{\lesssim}& \| \sol \|_{L^{\infty}} \| \nabla \sol \|_{L^8}^6 \, \| \sol \|_{W^{2,4}} \,\,\, \lesssim \,\,\, \| \nabla \sol \|_{L^8}^6 \, \kappa^2.
\end{eqnarray*}
We conclude $\| \nabla \sol \|_{L^8}\lesssim \kappa$. With this regularity estimate at hand, we can return to $f$ to see that $\| f\|_{L^8} \lesssim \kappa^2$ and hence, with Lemma \ref{lem:sol_bounds_ex_1},  $\| \sol \|_{W^{2,8}}\lesssim \kappa^2$. It becomes apparent that the argument can be repeated recursively to obtain $\| \nabla \sol \|_{L^p} \lesssim \kappa$ and $\| \sol \|_{W^{2,p}} \lesssim \kappa^2$ for any $1<p<\infty$. For example, assume that $\| \nabla \sol \|_{L^q} \lesssim \kappa$ holds for some $q \ge 2$, then $\| f \|_{L^q} \lesssim \kappa^2$ and consequently $\| \sol \|_{W^{2,q}} \lesssim \kappa^2$. With this, we have in turn
\begin{eqnarray*}
\| \nabla \sol \|_{L^{2q}}^{2q} &=&
- \int_{\Omega} \sol  \,\, \mbox{div}( \nabla \sol^* |\nabla \sol|^{2q-2}) \dx 
\,\,\lesssim\,\, \| \sol \|_{L^{\infty}} \|\, |\nabla \sol|^{2q-2} \|_{L^{2q/(2q-2)}} \, \| \sol \|_{W^{2,q}} \\
&=&  \| \sol \|_{L^{\infty}} \| \nabla \sol \|_{L^{2q}}^{2q-2} \, \| \sol \|_{W^{2,q}} \,\, \lesssim \,\, \| \nabla \sol \|_{L^{2q}}^{2q-2} \, \kappa^2
\quad \Rightarrow \quad \| \nabla \sol \|_{L^{2q}} \lesssim \kappa.
\end{eqnarray*}
We can repeat with $2q$. Note however that the hidden constants in the above estimates can potentially explode for $p \rightarrow \infty$ and we cannot conclude that the estimates hold for $p=\infty$.
\end{proof}

\subsection{Kernel in the second \Frechet derivative $\energystab''$}

In this section, we want to specify second order conditions for our minimizers. 

Recalling the results of Lemma \ref{lem:Frechet_der_2}, the second \Frechet derivative of $\energystab$ in $(\sol,\MagPot) \in H^1(\Omega) \times \HonenSpace(\Omega)$ is given by
\begin{eqnarray}
\label{secE-calculated} 
\nonumber\lefteqn{ \dualp{\energystab''(\sol,\MagPot) (\varphi,\mathbf{B}) }{(\psi, \mathbf{C} )} } \\
\nonumber&=& 
\dualp{\dutwoenergy(\sol,\MagPot) \varphi}{\psi} 
+
\dualp{\dudAtwoenergy(\sol,\MagPot) \mathbf{C}}{\testfun} 
+
\dualp{\dudAtwoenergy(\sol,\MagPot) \mathbf{B}}{\psi} 
+
\dualp{\dAtwoenergy(\sol,\MagPot) \mathbf{B}}{\mathbf{C}}
\\
\nonumber&=&
\Real  \int_\Omega \bigl( \kappainvci  \nabla \psi +  \MagPot \psi \bigr)  \cdot \bigl( \kappainvci  \nabla \testfun +  \MagPot \testfun \bigr)^*  
+
\bigl( |\sol|^2 -1 \bigr)  \psi \testfun^* + \sol^2 \psi^* \testfun^* + |\sol|^2 \psi \testfun^* 
\dint{x} 
\\
\nonumber&\enspace&+
\int_\Omega 
2     \Real (\sol \testfun^* ) \MagPot \cdot \mathbf{C}  
+
\kappainv  \Real \bigl( \ci  
\sol^* \nabla \testfun
+
\ci \testfun^* \nabla \sol \bigr) 
\cdot \mathbf{C}  
\dint{x}
\\
\nonumber&\enspace&+
 \int_\Omega 
2     \Real (\sol \psi^* ) \MagPot \cdot \mathbf{B}  
+
\kappainv  \Real \bigl( \ci  
\sol^* \nabla \psi
+
\ci \psi^* \nabla \sol \bigr) 
\cdot \mathbf{B}  
\dint{x}
\\
&\enspace&+
\int_\Omega 
|\sol|^2 \mathbf{C} \cdot \mathbf{B}  
+
\curl \mathbf{C} \cdot \curl \mathbf{B} 
+
\div \mathbf{C} \cdot \div \mathbf{B} 
\dint{x} 
\end{eqnarray}
for $(\testfun,\mathbf{B}),(\psi,\mathbf{C}) \in H^1(\Omega) \times \HonenSpace(\Omega)$.

\begin{lemma}\label{lemma:degenerate-direction}
Let $(\sol,\MagPot) \in H^1(\Omega) \times \HonenSpace(\Omega)$ be a minimizer of \eqref{minimization-problem}. Then, it holds
\begin{align}
\dualp{\energystab''(\sol,\MagPot) (\ci \sol, \mathbf{0} ) }{(\psi,\mathbf{C})} &= 0
\end{align}
for all $(\psi,\mathbf{C}) \in H^1(\Omega) \times \HonenSpace(\Omega)$. Thus, $\energystab''(\sol,\MagPot)$ is singular and cannot be
coercive. 
\end{lemma}

\begin{proof}
Using \eqref{secE-calculated}, we have by linearity that the terms with $\mathbf{B}={\mathbf{0}}$ vanish and thus obtain
\begin{align}
	& \dualp{\energystab''(\sol,\MagPot) (\ci \sol ,\mathbf{0} ) }{(\psi, \mathbf{C})} \\
	&= 
	\dualp{\dutwoenergy(\sol,\MagPot) \ci \sol}{\psi} 
	+
	\dualp{\dudAtwoenergy(\sol,\MagPot) \mathbf{C} }{\ci \sol} 
	+
	\dualp{\dudAtwoenergy(\sol,\MagPot) \mathbf{0}}{\psi} 
	+
	\dualp{\dAtwoenergy(\sol,\MagPot) \mathbf{0}}{\mathbf{C}}
	\\
	&= 
	\dualp{\dutwoenergy(\sol,\MagPot) \ci \sol}{\psi} 
	+
	\dualp{\dudAtwoenergy(\sol,\MagPot) \mathbf{C}}{\ci \sol} .
\end{align}	
Further, we conclude from  Lemma \ref{lemma-ginzburg-landau-eqns} and the fact that $\ci \sol$ is still a minimizer that we have $\langle \duenergy(\ci \sol,\MagPot) , \psi \rangle = 0$. Using Lemmas \ref{lem:Frechet_der_1} and \ref{lem:Frechet_der_2}, this implies
\begin{eqnarray*}
\dualp{\dutwoenergy(\sol,\MagPot) \ci \sol }{\psi} = \langle \duenergy(\ci \sol,\MagPot) , \psi \rangle +
\Real  \int_\Omega  \sol^2 (\ci u)^* \psi^* + |\sol|^2 \ci \sol \psi^* \dint{x} = 0.
\end{eqnarray*}
Since we also have 
\begin{align}
\dualp{\dudAtwoenergy(\sol,\MagPot) \mathbf{C}}{\ci \sol } = \int_\Omega 
2     \Real ( -\ci |\sol|^2 ) \MagPot \cdot \mathbf{C}  
-
\kappainv  \Real \bigl( -  
\sol^* \nabla \sol
+
 \sol^* \nabla \sol \bigr) 
\cdot \mathbf{C}  
\dint{x} = 0,
\end{align}	
 the claim follows.
\end{proof}
Lemma \ref{lemma:degenerate-direction} can be interpreted through smooth curves $\gamma(t)$ in $H^1(\Omega) \times \HonenSpace(\Omega)$. If the curve is locally (in a neighborhood of $t=0$) of the form $\gamma(t) \coloneqq (\sol e^{\ci \omega t} , \MagPot )$ for fixed a minimizer $(\sol,\MagPot) \in H^1(\Omega) \times \HonenSpace(\Omega)$ and for some $\omega \in \mathbb{R} \setminus \{ 0 \}$, then, due to the gauge invariance of $\energystab$ under complex phase shifts of $\sol$ (cf. \eqref{gauge-phase-shift}), we have $E(\,\gamma(t)\,) \equiv \mbox{const}$\, in a neighborhood of $t=0$. Together with $\gamma^{\prime}(0) = \omega (\ci \sol , \mathbf{0})$ and Lemma \ref{lemma-ginzburg-landau-eqns}, we conclude 
$$
0 = \tfrac{\mbox{\tiny d}^2}{\mbox{\tiny d}^2 t}E(\,\gamma(t)\,)_{\vert t=0} = \langle E^{\prime\prime}(\gamma(0)) \, \gamma^{\prime}(0) ,  \gamma^{\prime}(0)  \rangle
+   E^{\prime}(\gamma(0)) \, \gamma^{\prime\prime}(0) = \omega^2 \langle E^{\prime\prime}(\sol,\MagPot)  \, (\ci \sol , \mathbf{0}) , (\ci \sol , \mathbf{0}) \rangle.
$$
In other words, $(\ci \sol , \mathbf{0})$ is an eigenfunction of $E^{\prime\prime}(\sol,\MagPot)$ with eigenvalue $0$, which immediately implies the statement of Lemma \ref{lemma:degenerate-direction}. Furthermore, if all other eigenvalues of $E^{\prime\prime}(\sol,\MagPot)$ are positive, then this implies that $(\ci \sol , \mathbf{0})$ is the only direction for a curve $\gamma(t)$ with $\gamma(0) = (\sol , \MagPot )$ such that $\tfrac{\mbox{\tiny d}^2}{\mbox{\tiny d}^2 t}E(\,\gamma(t)\,)_{\vert t=0} = 0$. If this is fulfilled, then $ (\sol , \MagPot )$ is an isolated minimizer of $\energystab$ up to the gauge transformations \eqref{gauge-phase-shift}, i.e., it is locally quasi-unique.

With these thoughts, we consider the orthogonal complement of $(\ci \sol , \mathbf{0})$ which is given by the space
\begin{equation} \label{eq:def_Vfullperp}
\Honeperp \times \HonenSpace(\Omega),
\quad
\mbox{with } \Honeperp \coloneqq \{  \varphi \in H^1(\Omega) \mid 
\Real  \int_\Omega \ci \sol \, \psi^* \dint{x} = 0 \}
\end{equation}
and define \quotes{local quasi-uniqueness} of minimizers by assuming that the spectrum of $E^{\prime\prime}(\sol,\MagPot)$ is positive on $\Honeperp \times \HonenSpace(\Omega)$. This is fixed in Definition \ref{definition-local-quasi-uniquness} below. Note that $E^{\prime\prime}(\sol,\MagPot)$ cannot have negative eigenvalues since this implies the existence of a direction in which the energy $\energystab$ is further reduced, which would contradict the assumption that $(\sol,\MagPot)$ is a minimizer of $\energystab$. The definition below summarizes the above discussion and follows \cite[Definition 2.4]{DoeH24}.

\begin{definition}[Local quasi-uniqueness]
\label{definition-local-quasi-uniquness}
Let
\begin{eqnarray*}
 ((\varphi,\mathbf{B}), (\psi,\mathbf{C} ))_{L^2 \times \boldsymbol{L^2}} \coloneqq \Real  \int_\Omega \varphi \, \psi^* \dint{x} +   \int_\Omega \mathbf{B}  \cdot  \mathbf{C} \dint{x}.
\end{eqnarray*} 
We call a minimizer $(\sol,\MagPot) \in H^1(\Omega) \times \HonenSpace(\Omega)$ of \eqref{minimization-problem} \emph{locally quasi-unique} if $E^{\prime\prime}(\sol,\MagPot)$ has positive spectrum on $\Honeperp \times \HonenSpace(\Omega)$, i.e., if $(\varphi_j,\mathbf{B}_j) \in H^1(\Omega) \times \HonenSpace(\Omega)$ is an eigenfunction with eigenvalue $\lambda_j \in \mathbb{R}$ such that
\begin{eqnarray*}
\langle E^{\prime\prime}(\sol,\MagPot) (\varphi_j,\mathbf{B}_j), (\psi,\mathbf{C} ) \rangle = \lambda_j \, ((\varphi_j,\mathbf{B}_j), (\psi,\mathbf{C} ))_{L^2 \times \boldsymbol{L^2}}
\end{eqnarray*} 
for all $(\psi,\mathbf{C} ) \in H^1(\Omega) \times \HonenSpace(\Omega)$, then $\lambda_j \ge 0$ for all $j \in \mathbb{N}$ and $\lambda_j=0$ if and only if $\varphi_j \in \mbox{\normalfont span}\{ \ci \sol\}$ and $\mathbf{B}_j=0$.
\end{definition}
For the final error estimates we assume that the minimizers are locally quasi-unique in the sense of the above definition. Whenever we need the assumption it will be explicitly mentioned in the corresponding result.

\begin{assumption} \label{ass:local_uniq}
The minimizers $(\sol,\MagPot)$ of the Ginzburg--Landau energy \eqref{eq:energy_functional_stab} are
locally quasi-unique in the sense of Definition \ref{definition-local-quasi-uniquness}.
\end{assumption}

For locally quasi-unique minimizers we have coercivity of $\energystab''(\sol,\MagPot)$ on $\Honeperp \times \HonenSpace(\Omega)$.

\begin{proposition} \label{prop:coecivity_Epp}
Let $(\sol,\MagPot)$ be a minimizer of \eqref{eq:energy_functional_stab} that is locally quasi-unique in the sense of Definition \ref{definition-local-quasi-uniquness}. 
Then the second \Frechet derivative $\energystab''(\sol,\MagPot)$ is coercive on 
$\Honeperp \times \HonenSpace(\Omega)$,
i.e., there exists a constant $\Csol>0$ such that
\begin{equation}
\dualp{\energystab''(\sol,\MagPot) (\varphi, \mathbf{B} )}{(\varphi, \mathbf{B} )}
\geq \Csol^{-1}
\Honecombi{(\varphi, \mathbf{B} )}^2,
\quad \text{for all } (\varphi, \mathbf{B} ) \in \Honeperp \times \HonenSpace(\Omega)\, 
\end{equation}
where $\Honecombi{(\varphi, \mathbf{B} )}^2\coloneqq \norm{\HonekappaSpace}{\varphi}^2 + \norm{ H^1}{ \mathbf{B}}^2$.
Furthermore, it holds
\begin{eqnarray*}
	|\dualp{\energystab''(\sol,\MagPot) (\varphi, \mathbf{B} ) }{(\psi,  \mathbf{C} )}| &\lesssim&
	\Honecombi{(\varphi, \mathbf{B})} \,\,
	\Honecombi{(\psi, \mathbf{C})}
\end{eqnarray*}
for all $(\varphi, \mathbf{B} ), (\psi,  \mathbf{C} ) \in H^1(\Omega) \times \HonenSpace(\Omega)$ and with a constant independent of $\kappa$.
\end{proposition}

\begin{remark}
To the best of our knowledge the precise dependence of $\Csol$ on $\kappa$ could so far not be resolved in the literature. However, there is quite some effort in the literature to estimate the smallest eigenvalue $\lambda(\kappa)$  of the magnetic Neumann Laplacian  
(later defined as $a_{\MagPot}$ in \eqref{eq:def_abilmag}), which basically corresponds to the very special case of $\sol = 0$. For example in 
\cite[Thm.~8.1.1 \& 9.1.1]{ForH10},
it is shown that in a bounded and smooth domain  asymptotically $\lambda(\kappa) \sim \kappa^{-1}$ holds if $\curl \MagPot$ does not vanish on $\Omega$. 
By the proof below this would lead to $\Csol \sim \kappa$.
In alignment with these theoretical results, 
our numerical experiments similarly indicate that $\Csol \sim \kappa^{\alpha}$ with $\alpha \ge 1$ on rectangular domains, cf. \cite{DoeH24,BDH24}.

\end{remark}

\begin{proof}[Proof of Proposition~\ref{prop:coecivity_Epp}]
The local quasi-uniqueness (Definition \ref{definition-local-quasi-uniquness})
guarantees 
the existence of the second-smallest eigenvalue $\lambda_2>0$ of $\energystab''(\sol,\MagPot)$ such that
\begin{eqnarray}
\label{estimate-eigenvalue}
\dualp{\energystab''(\sol,\MagPot) (\varphi, \mathbf{B} )}{(\varphi, \mathbf{B} )}
&\geq& \lambda_2
\norm{L^2 \times \mathbf{L}^2}{(\varphi, \mathbf{B}) }^2 
\end{eqnarray}
for all $(\varphi, \mathbf{B} ) \in \Honeperp \times \HonenSpace(\Omega)$. 
On the other hand, we can use \eqref{secE-calculated} together with the identity $\Real((\sol \testfun^*)^2  + |\sol|^2 |\varphi|^2)=2 (\Real(\sol \testfun^*))^2$ to obtain
	\begin{eqnarray*}
\lefteqn{ \dualp{\energystab''(\sol,\MagPot) (\varphi,\mathbf{B}) }{(\varphi, \mathbf{B} )} \,\,
=\,\, \int_\Omega | \tfrac{\ci}{\kappa}  \nabla \testfun +  \MagPot \testfun |^2 + \bigl( |\sol|^2 -1 \bigr)  |\testfun|^2 +2 \Real(\sol \testfun^*)^2 \dint{x}  } \\
\nonumber&\enspace&\hspace{-6pt}+
\int_\Omega 
4     \Real (\sol \testfun^* ) \MagPot \cdot \mathbf{B}  
+
\tfrac{2}{\kappa}  \Real \bigl( \ci  
\sol^* \nabla \testfun
+
\ci \testfun^* \nabla \sol \bigr) 
\cdot \mathbf{B}  
\dint{x} +
\int_\Omega 
|\sol|^2 |\mathbf{B}|^2 
+
|\curl \mathbf{B}|^2
+
|\div \mathbf{B}|^2
\dint{x} .
\end{eqnarray*}
Since $ | \tfrac{\ci}{\kappa}  \nabla \testfun +  \MagPot \testfun |^2 \ge \tfrac{1}{2 \kappa^2} |\nabla \testfun|^2 - |\MagPot|^2 |\testfun|^2 $ and $\|  \mathbf{B} \|_{H^1} \lesssim  \| \div \mathbf{B}\|_{L^2} + \| \curl \mathbf{B}\|_{L^2}$ (cf. \cite[Lem.~3.6]{GirR86} and \cite[Thm.~3.9]{GirR86}), we conclude together with the $L^{\infty}$-bounds for $\sol$ and $\MagPot$ from Lemma \ref{lem:sol_bounds_ex_1} and Corollary \ref{cor:sol_bounds_ex_A_H2} that
	\begin{eqnarray*}
\lefteqn{ \dualp{\energystab''(\sol,\MagPot) (\varphi,\mathbf{B}) }{(\varphi, \mathbf{B} )} } \\
&\ge&
 \int_\Omega  \tfrac{1}{2} | \tfrac{1}{\kappa} \nabla \testfun|^2  + \bigl( |\sol|^2 -1 - |\MagPot|^2 \bigr)  |\testfun|^2 - 4  |\sol| \, |\testfun| \, |\MagPot| \,|\mathbf{B}|   
 - \tfrac{2}{\kappa}  (|\sol| \, |\nabla \testfun| + |\testfun| \, |\nabla \sol| )  |\mathbf{B} |
 \dint{x}   \\
\nonumber&\enspace&\hspace{-6pt}+ 
\int_\Omega 
|\curl \mathbf{B}|^2
+
|\div \mathbf{B}|^2
\dint{x} \\
&\gtrsim& \Honekappa{\testfun}^2 + \| \mathbf{B} \|_{H^1}^2 - c_1 \| \testfun \|_{L^2}^2 - c_2 \| \mathbf{B} \|_{L^2}^2
\end{eqnarray*}
for constants $c_1,c_2 \ge 0$.
In the last step of the estimate, we also used the Young's inequality $\| \tfrac{2}{\kappa} |\sol| \, |\nabla \testfun| |\mathbf{B} |\,\|_{L^1} \le \| \tfrac{1}{4} |\tfrac{1}{\kappa} \nabla \testfun|^2 \|_{L^1} + 4  \,\| |\sol|^2 |\mathbf{B} |^2 \|_{L^1}$, as well as $\| \nabla \sol \|_{L^4}\lesssim \kappa$ which yields $\| \tfrac{2}{\kappa}  |\testfun| \, |\nabla \sol| \,  |\mathbf{B} | \,\|_{L^1} \le 2 \| \tfrac{1}{\kappa} \nabla \sol \|_{L^4} \| \testfun \|_{L^2} \| \mathbf{B} \|_{L^4} \lesssim \tfrac{1}{\varepsilon} \| \tfrac{1}{\kappa} \nabla \sol \|_{L^4}^2  \| \testfun \|_{L^2}^2 + \varepsilon  \| \mathbf{B} \|_{H^1}^2$. We conclude that the following G{\aa}rding inequality holds:
\begin{equation}
	\dualp{\energystab''(\sol,\MagPot) (\varphi, \mathbf{B} )}{(\varphi, \mathbf{B} )}
	\geq C_1
	\Honecombi{(\varphi, \mathbf{B})}^2
	-
	C_2
	\norm{L^2 \times \mathbf{L}^2}{(\varphi, \mathbf{B})}^2.
\end{equation}
Together with \eqref{estimate-eigenvalue} we obtain for $(\varphi, \mathbf{B} ) \in \Honeperp \times \HonenSpace(\Omega)$ 
the coercivity estimate
\begin{equation}
	(1+\tfrac{C_2}{\lambda_2}) \dualp{\energystab''(\sol,\MagPot) (\varphi, \mathbf{B} )}{(\varphi, \mathbf{B} )}
	\geq C_1
	\Honecombi{(\varphi, \mathbf{B})}^2,
\end{equation}
and $\Csol = (1+\tfrac{C_2}{\lambda_2}) \,C_1^{-1}$.
In addition, the continuity estimate for $|\dualp{\energystab''(\sol,\MagPot) (\varphi, \mathbf{B} ) }{(\psi,  \mathbf{C} )}|$ follows from Lemma~\ref{lem:energy2prim_perturbation_v2} (c) below.
\end{proof}

We have the following direct consequence of Proposition~\ref{prop:coecivity_Epp}.
\begin{lemma} \label{lem:inf_sup_cond_exact_bounds}
Let $(\sol,\MagPot)$ be a locally quasi-unique minimizer of \eqref{eq:energy_functional_stab} in the sense of Definition \ref{definition-local-quasi-uniquness}. 
Then, for all 
$\mathbf{\rhs} \in L^2(\Omega) \times \mathbf{L}^2(\Omega) \subset (\Honeperp \times \HonenSpace(\Omega))^{\ast}$, 
there exists a unique $(z, \mathbf{Z} ) \in \Honeperp \times \HonenSpace(\Omega)$ 
which solves
	\begin{equation} \label{eq:var_prob_energe_pp}
		\dualp{\energystab''(\sol,\MagPot) (z, \mathbf{Z} ) }{(\psi,\mathbf{C})} = 
		(\mathbf{\rhs} , (\psi,\mathbf{C}) )_{L^2 \times \mathbf{L}^2}, 
		\quad \text{ for all } (\psi,\mathbf{C}) \in \Honeperp \times \HonenSpace(\Omega). 
	\end{equation}
	The solution further satisfies
	\begin{equation}
		\Honecombi{(z, \mathbf{Z} )}\,\,\lesssim\,\, \Csol \, \norm{L^2 \times {\mathbf{L}}^2}{ \mathbf{\rhs} }.
	\end{equation}
\end{lemma}
Note that the $\kappa$--weighted topology in Lemma~\ref{lem:inf_sup_cond_exact_bounds} reflects the principal part 
$-\tfrac{1}{\kappa^{2}}\Delta z + z$ of the $z$--component of $E^{\prime\prime}(\sol)$ and is therefore the natural coercivity scale for stability estimates.

\section{LOD discretization and main results}
\label{sec:space_main}
In this section we introduce the LOD discretization for minimizers of the GL free energy by adapting the constructions proposed in \cite{BDH24} and \cite{DoeH24}. For that, let $\mathcal{T}_H$ and $\mathcal{T}_h$ be two shape-regular and quasi-uniform triangulations of $\Omega$ with mesh sizes $H$ and $h$ respectively. The mesh $\mathcal{T}_H$ will be used to approximate the order parameter $\sol$ and $\mathcal{T}_h$ to approximate the vector potential $\MagPot$. In the first step, let us define the $\mathcal{P}_1$-Lagrange finite element space on $\mathcal{T}_H$ by 
\begin{equation}
	\VSh = \{ \, \varphi_H \in C^0(\overline{\Omega};\mathbb{C}) 
	\mid
	\varphi_H |_{K}  \in \mathcal{P}_1(K) 
	\text{ for all } K \in \mathcal{T}_H \, \}  
\end{equation}
and the $\mathcal{P}_k$-Lagrange finite element space of degree $k=1,2$ on $\mathcal{T}_h$ by 
\begin{equation} \label{eq:FEM_Space_A}
	\VSAhzero{k} = \{ \, \mathbf{B}_h \in C^0(\overline{\Omega};\mathbb{R}^3) 
	\mid
	\mathbf{B}_h |_{K}  \in \mathcal{P}_k(K)^3 
	\text{ for all } K \in \mathcal{T}_h 
	\text{ and }
	\mathbf{B}_h \cdot \nu = 0 \text{ on } \partial \Omega \,\} .
\end{equation}
The usual approximation properties of these Lagrange FE spaces (cf. \cite{BrennerScott}) together with an Aubin--Nitsche argument yield the estimates 
\begin{subequations} \label{eq:approx_FEM_Lag}
\begin{eqnarray}
\inf_{\varphi_H \in \VSh } \left(  \| \varphi - \varphi_H \|_{L^2} + \kappa H \| \varphi - \varphi_H \|_{H^1_{\kappa}} \right)
&\lesssim& (\kappa H)^2 \, \| \varphi \|_{H^2_{\kappa}} 
 \qquad
 \mbox{and}\\
 \inf_{\mathbf{B}_h \in \VSAhzero{k} } \left( \| \mathbf{B} - \mathbf{B}_h \|_{\mathbf{L}^2} + h \| \mathbf{B} - \mathbf{B}_h \|_{\mathbf{H}^1} \right)
&\lesssim& h^{k+1} \| \mathbf{B} \|_{\mathbf{H}^{k+1}} 
\end{eqnarray}
\end{subequations}
for all $(\varphi,\mathbf{B}) \in H^2(\Omega) \times \mathbf{H}^{k+1}(\Omega)$ where $k=1,2$.

In order to improve the approximation properties of $\VSh$ w.r.t. $\sol$ we enrich the nodal basis functions using the so-called {\it magnetic Laplacian}, which is represented by the bilinear form
\begin{eqnarray}
\label{eq:def_abilmag}
a_{\bfAapp} ( \varphi , \psi) &\coloneqq&  \Real  \int_\Omega \bigl( \kappainvci  \nabla \testfun +  \bfAapp \testfun \bigr)  \cdot \bigl( \kappainvci  \nabla \psi +  \bfAapp \psi \bigr)^* \dint{x} 	
\end{eqnarray}
for $\testfun,\psi \in H^1(\Omega)$ 
and where 
$\bfAapp \in \HonenSpaceDiv(\Omega) \cap \mathbf{L}^{\infty}(\Omega)$ 
denotes a selected vector potential that is seen as an arbitrary approximation of the exact (unknown) potential $\MagPot$ of some arbitrary minimizer $(\sol,\MagPot)$ of \eqref{eq:energy_functional_stab}. A reasonable choice is for example obtained by selecting $\bfAapp$ such that $\div \bfAapp=0$ and $\curl \bfAapp = \mathbf{H}$ for the given external magnetic field $\mathbf{H}$. 
Loosely speaking, we want to  
define the LOD space as the image of the inverse magnetic Laplacian under $\VSh$. However, the bilinear form $\abilmagLOD{\cdot}{\cdot}$ is typically not coercive and possibly singular so that the inverse does not necessarily exist. However, as proved in \cite[Lemma 4.1]{BDH24}, we have coercivity on the kernel $W=\mbox{kern} \,\LtwoprojFEM \vert_{H^1(\Omega)}$ of the $L^2$-projection $\LtwoprojFEM : H^1(\Omega) \rightarrow \VSh$ which is given by 
\begin{eqnarray}
\label{definition-L2-projection}
(\LtwoprojFEM \varphi , \psi_H )_{L^2(\Omega)} &=& ( \varphi , \psi_H )_{L^2(\Omega)} \qquad \mbox{for all } \psi_H \in \VSh
\end{eqnarray}
with the standard approximation property
\begin{eqnarray}
\label{L2-proj-est}
\| \varphi -  \LtwoprojFEM \varphi \|_{L^2} &\lesssim&  H \| \nabla \varphi \|_{L^2} \qquad \mbox{for all } \varphi \in H^1(\Omega).
\end{eqnarray}
Note that due to our assumption $\mathcal{T}_H$ is quasi-uniform, the $L^2$-projection $\LtwoprojFEM$ is $H^1$-stable \cite{BanY14}, and consequently, its kernel is a closed subspace of $H^1(\Omega)$. The following lemma summarizes the statement.
\begin{lemma}
\label{lemma:coercivity-abil-on-W}
For any $\bfAapp \in \HonenSpaceDiv(\Omega) \cap \mathbf{L}^{\infty}(\Omega)$, there is a constant $C_{\mbox{\normalfont\tiny res}}>0$ that depends on $\Omega$, $\| \bfAapp \|_{L^{\infty}}$ and the shape-regularity and uniformity constants of $\mathcal{T}_{H}$ such that if $H \le C_{\mbox{\normalfont\tiny res}} \kappa^{-1}$, then it holds
\begin{eqnarray*}
 a_{\bfAapp}( w , w)  &\ge& \tfrac{1}{2} \| w \|_{\HonekappaSpace}^2 \qquad \mbox{for all } w \in W, 
\end{eqnarray*}
where, for  $\LtwoprojFEM : H^1(\Omega) \rightarrow \VSh$ given by \eqref{definition-L2-projection}, 
\begin{eqnarray*}
W &\coloneqq&\{ w \in H^1(\Omega) \, | \, \LtwoprojFEM w =0 \}.
\end{eqnarray*}
\end{lemma}
For the proof, we refer to \cite[Lemma 4.1]{BDH24}.

Exploiting the coercivity on the so-called {\it detail space} $W$, we can introduce the (well-defined) correction operator $\mathcal{C}: H^1(\Omega) \rightarrow W$ by
\begin{eqnarray}
\label{eqn:def-ideal-correctors}
 a_{\bfAapp} ( \mathcal{C} \varphi , w )  &=&  a_{\bfAapp} ( \varphi ,  w )  \qquad \mbox{for all } w\in W.
\end{eqnarray}
The operator allows us to correct the elements of $\VSh$ to obtain the LOD space as
\begin{eqnarray} \label{eq:FEM_Space_u} 
\VShLOD &\coloneqq& (1-\mathcal{C}) \VSh \,\,\, = \,\,\, \{  \varphi_{H} - \mathcal{C}\varphi_{H} \,|\,\varphi_{H} \in \VSh \}.
\end{eqnarray}
For practical aspects on the construction of $\VShLOD$ and additional errors arising from its discrete approximation we refer to \cite{BDH24} where this is described and analyzed in detail for the Ginzburg--Landau equation for given $\bfAapp$. Also note that the hidden real parts in \eqref{eqn:def-ideal-correctors} can be formally dropped. This is seen by testing with $\ci w$ in \eqref{eqn:def-ideal-correctors} to obtain that also the imagery parts of the integrals are necessarily the same.

The main goal of our paper is to study the approximation properties of the spaces $\VShLOD$ and $\VSAhzero{k}$ defined in \eqref{eq:FEM_Space_u}
and \eqref{eq:FEM_Space_A} with respect to minimizers of the Ginzburg--Landau energy. To be precise, we consider discrete minimizers $(\sol_{H}^{\LOD},\mathbf{A}_{h,k}^{\mbox{\tiny FEM}})  \in \VShLOD \times \VSAhzero{k}$ with
\begin{equation}  \label{eq:disc_approx_LOD}
	\energystab(\sol_{H}^{\LOD},\mathbf{A}_{h,k}^{\mbox{\tiny FEM}})  
	= 
	\underset{\in \VShLOD \times \VSAhzero{k}}{ \min_{  (\varphi_{H}^{\LOD} , \mathbf{B}_{h} ) } }
	\energystab( \varphi_{H}^{\LOD} , \mathbf{B}_{h} )
\end{equation}
and are concerned with quantifying their distance to some exact minimizer. In particular we want to show that, besides achieving super convergence with respect to the mesh size $H$, the $\kappa$-dependent smallness condition for $H$ in $\VShLOD$ is significantly relaxed compared to the analogous condition in the standard space $\VSh$. Furthermore, we show that for smooth external magnetic fields $\bfH$, this is already achieved if $\bfAapp$ is a crude generic approximation of $\bfA$ that can be a priori selected.

We present the corresponding error estimates for the approximations $\sol_{H}^{\LOD}$ and $\mathbf{A}_{h,k}^{\mbox{\tiny FEM}}$ in the next subsection.
We shall now present our main result, which includes error estimates in the $\HonekappaSpace \times \mathbf{H}^1$- and the $L^2 \times \mathbf{L}^2$-norm, as well as an estimate for the energy error.
In order to obtain optimal order error estimates, both in the LOD space and the quadratic elements for $\MagPot$ in $\VSAhzero{2}$, we require the following regularity assumption. 

\begin{assumption}[Regularity of external magnetic field and vector potential]
\label{assumption:reg-mag-vec-pot}
The external magnetic field is assumed to fulfil $\curl \MagF \in \mathbf{H}^1(\Omega)$. With this, Lemma \ref{cor:sol_bounds_ex_A_H3} guarantees, for any minimizing pair $(\sol,\MagPot) \in H^1(\Omega) \times \HonenSpace(\Omega)$ of \eqref{minimization-problem}, that $\MagPot \in \mathbf{H}^{3}(\Omega)$ with  $\norm{\mathbf{H}^2}{\MagPot} \lesssim 1$ and $\norm{\mathbf{H}^3}{\MagPot} \lesssim \kappa$. In the construction of  $\VShLOD$ we further assume that $\bfAapp \in \HonenSpaceDiv(\Omega)$ has a consistent regularity and stability, i.e., $\bfAapp \in \mathbf{H}^3(\Omega)$ with $\| \bfAapp \|_{\mathbf{H}^2(\Omega)} \lesssim 1$ and $\| \bfAapp \|_{\mathbf{H}^3(\Omega)} \lesssim \kappa$.
\end{assumption}
If the full regularity in Assumption \ref{assumption:reg-mag-vec-pot} is available, the next theorem shows that there is a large class of admissible ad-hoc choices for $\bfAapp$ such that optimal convergence in $\VShLOD$ is achieved (i.e. the same order as for the ideal choice $\bfAapp=\bfA$). However, if there is reduced regularity, the approximation properties of the LOD space can be reduced (at most by one order). Details are given in Lemma \ref{lemma:bestapprox-LOD} in Section \ref{sec:error_analysis} where the precise effect of $\bfAapp$ on the error estimates is traced.
\begin{theorem}[Error estimates for LOD approximations]
\label{thrm:main-results}
Let Assumptions \ref{ass:local_uniq} and \ref{assumption:reg-mag-vec-pot} hold. 
We consider an arbitrary discrete minimizer $(\sol_{H}^{\LOD},\mathbf{A}_{h,k}^{\mbox{\tiny FEM}})  \in \VShLOD \times \VSAhzero{k}$ of problem \eqref{eq:disc_approx_LOD} for either $k=1$ or $k=2$. 
If $\kappa H \lesssim 1$, then the error in energy is bounded by
\begin{eqnarray*}
0\,\,\,\, \le \,\,\,\, \energystab(\sol_{H}^{\LOD},\mathbf{A}_{h,k}^{\mbox{\tiny FEM}}) 
\,\,\,\,\, -\,\, 
 \underset{H^1(\Omega) \times \HonenSpace(\Omega)}{\underset{(v,\mathbf{B}) \in}{\mbox{\normalfont min}}} \hspace{-5pt}E(v,\mathbf{B})
		&\lesssim& \kappa^6 \deltaH^6  + \kappa^{2k-2} \deltah^{2k} 
\end{eqnarray*} 
Furthermore, there exists a minimizer $(\sol,\MagPot) \in  H^1(\Omega) \times \HonenSpace(\Omega)$ of \eqref{minimization-problem} with $\sol_{H}^{\LOD} \in \Honeperp$ such that if $(\deltaH,\deltah)$ is sufficiently small with at least 
\begin{eqnarray}
\label{resolution-condition}
( \kappa^2 H^2 + h) \,\, \kappa^{\varepsilon} \, \Csol \,\,\, \lesssim \,\,\, 1
\end{eqnarray}
then it holds 
\begin{eqnarray}
\label{H1-est-LOD}
 \| (\sol - \sol_{H}^{\LOD} , \MagPot - \mathbf{A}_{h,k}^{\mbox{\tiny FEM}}) \|_{\HonekappaSpace \times \mathbf{H}^1} 
&\lesssim& 
\kappa^3 \deltaH^3  + \kappa^{k-1} \deltah^k. 
\end{eqnarray}
and
\begin{eqnarray}
\nonumber \lefteqn{ \| (\sol - \sol_{H}^{\LOD} , \MagPot - \mathbf{A}_{h,k}^{\mbox{\tiny FEM}}) \|_{L^2 \times \mathbf{L}^2} \,\,\, \lesssim \,\,\, \kappa^4\,H^4 + \kappa^k H \deltah^k  }\\
\label{L2-est-LOD}
&\enspace&
\quad+ 
\, \kappa^{\varepsilon} \, \Csol \left(\kappa^3 \deltaH^3  + \kappa^{k-1} \deltah^k\right) (\kappa^2\,H^2 + h) 
\,\,+\,\, \Csol ( \kappa^8 H^6  + \kappa^2 \, h^{2k} ) 
\end{eqnarray}
All hidden constants in the above estimates are independent of $\kappa$ and $(\deltaH,\deltah)$.
\end{theorem}
The result is a summary of Conclusion \ref{conclusion:energy-error}, Proposition \ref{proposition-H1-est} and Proposition \ref{prop:L2-error-estimates} which we prove in Section \ref{sec:error_analysis}.

The error estimates in Theorem \ref{thrm:main-results} demonstrate convergence of order $\mathcal{O}(\kappa^3 \deltaH^3 + \kappa^{k-1} \deltah^k)$ for the $H^1_{\kappa}$-error. The necessary resolution $\deltaH$ for the order parameter $\sol$ is constrained by $\kappa$ with at least $H \lesssim \kappa^{-1-\varepsilon/2} \, \Csol^{-1/2}$. In fact, a careful inspection of the arguments in the proof of Proposition \ref{proposition-H1-est} shows that $ \| (\sol - \sol_{H}^{\LOD} , \MagPot - \mathbf{A}_{h,k}^{\mbox{\tiny FEM}}) \|_{\HonekappaSpace \times \mathbf{H}^1} $ behaves as the best-approximation error in this regime.
 The necessary resolution for $\MagPot$ in terms of $\deltah$ is only weakly constrained by $\kappa$. Though the resolution condition for attaining the best-approximation requires $h \lesssim \kappa^{-\varepsilon} \, \Csolinv$, the asymptotic convergence rate in $h$ is only mildly (if at all) affected by $\kappa$ (see Remark \ref{rem:A_in_H3} below).

As for the resolution condition \eqref{resolution-condition}, recall that we expect $\Csol$ to behave as $\kappa^{\alpha}$ for some positive $\alpha$, which would result in the constraints $\deltaH \lesssim \kappa^{-1-(\varepsilon + \alpha)/2}$ and $\deltah \lesssim  \kappa^{-(\varepsilon + \alpha)}$. However, the constraint \eqref{resolution-condition} can be dropped on the expense on the additional term 
\begin{eqnarray}
\label{dropped-term-H1-est}
(\kappa^3 \deltaH^3  + \kappa^{k-1} \deltah^k)\,( \kappa^2 H^2 + h) \,\, \kappa^{\varepsilon} \, \Csol
\end{eqnarray}
on the right hand side of \eqref{H1-est-LOD}. This shows that reasonable approximations can be already obtained on coarse meshes, e.g. requiring $H \lesssim \kappa^{-1 - \varepsilon/5} \Csol^{-1/5}$ instead of  $H \lesssim \kappa^{-1 - \varepsilon/2} \Csol^{-1/2}$ (which is needed for a quasi-best-approximation).
Hence, we expect that there is only short pre-asymptotic convergence regime caused by this additional resolution condition. 

For the error in energy, we observe that the convergence order of the $H^1$-error is squared. Furthermore, no additional resolution depending on $\Csol$ is needed but only the natural minimal resolution condition $H \lesssim \kappa^{-1}$.

Finally, note that the $L^2$-error estimate indicates a stronger resolution condition for the $L^2$-error w.r.t. $H$, as far as the optimal rate $(\kappa H)^4$ is concerned. Here we require  $H \lesssim \kappa^{-1 - \varepsilon} \Csol^{-1}$ such that the middle term on the right hand side of \eqref{L2-est-LOD} behaves like $(\kappa H)^4$. However, the term is identical to \eqref{dropped-term-H1-est}, which is exactly the dropped term in the $H^1$-error which originally lead to the resolution condition \eqref{resolution-condition}. Hence, we can still expect that $L^2$- and $H^1$-error become small at the same time, though the optimal rate for the $L^2$-error might not be visible instantly. In fact, our experiments as well as previous experiments \cite{DoeH24,BDH24} could not find any indications that there is a stronger influence of $\Csol$ on the $L^2$-error than on the $H^1$-error.

\begin{remark}[$\kappa$ constraint for $\deltah$] \label{rem:A_in_H3}
The error estimates in Theorem \ref{thrm:main-results} show, for $k=2$, a convergence rate of $\kappa \deltah^2$ for the $\HonekappaSpace$-error and a rate of  $\kappa \deltah^3$ for the $L^2$-error. The additional $\kappa$ entered through the regularity estimate $ \| \MagPot \|_{H^3} \lesssim \kappa$. Again, we could not find numerical evidence that this estimate is sharp, and we rather observe constants which indicate $\| \MagPot \|_{H^3} \lesssim 1$. If this is true, then we could remove the $\kappa$-dependence in front of $\deltah$ in all our error estimates. However, due to our computational limitations for studying very large $\kappa$-values, it is not yet possible to draw any definite conclusions from our numerical experiments.
\end{remark}
We conclude with a comparison to a standard finite element discretization with $\VSh$ instead of $\VShLOD$, i.e. both spaces of the same dimension but different approximation properties. The proof of the following result is analogous to the LOD case by exploiting the abstract convergence theory from Section \ref{sec:abstract_error_analysis}. Recall that for the case $k=2$ we again assume $\curl \MagF \in \mathbf{H}^1(\Omega)$.
\begin{theorem}[Error estimates for FEM approximations]
\label{thrm:main-results-FEM}
Let Assumption \ref{ass:local_uniq} hold and let $(\sol_{H},\mathbf{A}_{h,k})  \in \VSh \times \VSAhzero{k}$ fulfill for $k=1,2$:
\begin{eqnarray*}
	\energystab(\sol_{H},\mathbf{A}_{h,k})  
	= 
	\underset{\in \VSh \times \VSAhzero{k}}{ \min_{  (\varphi_{H} , \mathbf{B}_{h} ) } }
	\energystab( \varphi_{H} , \mathbf{B}_{h} ). 
\end{eqnarray*}
If $\kappa H \lesssim 1$, then the error in energy is bounded by
\begin{eqnarray*}
0\,\,\,\, \le \,\,\,\, \energystab( \sol_{H} ,  \mathbf{A}_{h,k})
\,\,\,\,\, -\,\, 
 \underset{H^1(\Omega) \times \HonenSpace(\Omega)}{\underset{(v,\mathbf{B}) \in}{\mbox{\normalfont min}}} \hspace{-5pt}E(v,\mathbf{B})
		&\lesssim& \kappa^2 \deltaH^2  + \kappa^{2k-2} \deltah^{2k}.
\end{eqnarray*} 
Furthermore, there exists a minimizer $(\sol,\MagPot) \in  H^1(\Omega) \times \HonenSpace(\Omega)$ of \eqref{minimization-problem} with $\sol_{H} \in \Honeperp$ such that if $(\deltaH,\deltah)$ is sufficiently small with at least 
\begin{eqnarray}
\label{resolution-condition-FEM}
( \kappa H + h) \,\, \kappa^{\varepsilon} \, \Csol \,\,\, \lesssim \,\,\, 1
\end{eqnarray}
then it holds 
\begin{eqnarray*}
 \| (\sol - \sol_{H} , \MagPot - \mathbf{A}_{h,k}) \|_{\HonekappaSpace \times \mathbf{H}^1} 
&\lesssim& 
\kappa \deltaH  + \kappa^{k-1} \deltah^k
\end{eqnarray*}
and
\begin{eqnarray*}
 \lefteqn{ \| (\sol - \sol_{H} , \MagPot - \mathbf{A}_{h,k}) \|_{L^2 \times \mathbf{L}^2}  }\\
&\lesssim&
\, \kappa^{\varepsilon} \, \Csol \left(\kappa^2 \deltaH^2 + \kappa^k \deltaH  \deltah^k + \kappa^{k-1} \deltah^{k+1} \right)  
\,\,+\,\, \Csol ( \kappa^3 H^2  + \kappa^2 \, h^{2k} ).
\end{eqnarray*}
\end{theorem}
Comparing the results of Theorem~\ref{thrm:main-results} and Theorem~\ref{thrm:main-results-FEM} we do not only observe that the LOD approximations converge much faster in $\HonekappaSpace$ (third order in $\deltaH$ vs. first order in $\deltaH$), but also that the necessary resolution condition for small errors is significantly reduced with $\deltaH \lesssim \kappa^{-1 - \varepsilon/5} \Csol^{-1/5} $ for LOD approximations vs. $\deltaH \lesssim \kappa^{-1 - \varepsilon/2} \Csol^{-1/2} $ for standard $\mathcal{P}_1$-Langrange FE approximations. The difference becomes even more pronounced by noting that in order to obtain a quasi-best approximation, the LOD approximation requires $\deltaH \lesssim \kappa^{-1 - \varepsilon/2} \Csol^{-1/2} $ whereas the standard $\mathcal{P}_1$-FEM requires the stronger condition $\deltaH \lesssim \kappa^{-1 - \varepsilon} \Csol^{-1} $. The resolution conditions in terms of $h$ for the vector potential are in essence the same for both types of approximations. 
Finally, we also note that the $L^2$-error estimate for the FEM approximation has an explicit scaling with $\Csol$ in the leading order term, which is absent in the $L^2$-error estimate for the LOD-approximation. 

\section{Analytical preparations}
\label{sec:ana_prep}

Before we can start with the error analysis, we require a few more analytical preparations regarding the regularity of certain auxiliary functions which later appear as solutions to a dual problem in an Aubin--Nitsche argument. Furthermore, we state an alternative representation of $\energystab''(\sol,\MagPot)$ that will be useful for the aforementioned duality arguments. In the following, we assume that $(\sol,\MagPot)$ always denotes a (fixed) minimizer of \eqref{eq:energy_functional_stab}.  To keep the notation short, we introduce 
another bilinear form that is
used for the rest of the paper.
 $\mathbf{B},\mathbf{C}  \in \HonenSpace(\Omega)$ we define
\begin{eqnarray} 
\abilusimple{ \mathbf{B} }{\mathbf{C}} &\coloneqq& \int_\Omega 
	\curl \mathbf{B} \cdot \curl  \mathbf{C}  	+	\div \mathbf{B} \, \div  \mathbf{C}  \dint{x} \,.
		\label{eq:def_abilusimple}
\end{eqnarray}
Furthermore, we recall from \eqref{eq:def_abilmag} the magnetic Laplacian $a_{\bfAapp} ( \varphi , \psi)$ for the vector potential $\bfAapp \in \HonenSpaceDiv(\Omega)$. On it is own, it is not necessarily an elliptic operator, however, if a sufficiently large $L^2$-contribution is added, e.g. $ ( \,( |\bfAapp|^2 + 1) \, \varphi , \psi )_{L^2}$, then it becomes coercive and bounded with respect to the $\HonekappaSpace$-norm with constants uniformly bounded in $\kappa$, see \cite[Lemma~2.1]{DoeH24}. On the contrary,
the bilinear form in \eqref{eq:def_abilusimple} is always coercive and bounded with respect to the $\HonenSpace$-norm, see \cite[Theorem~3.9]{GirR86}, where the respective norms were defined in \eqref{eq:def_norms}.
We start with a lemma that characterizes the regularity of solutions to problems that involve either the bilinear form $a_{\bfAapp}(\cdot,\cdot)$ or the bilinear form $\abilusimple{\cdot}{\cdot}$.

\begin{lemma} \label{lem:regularity_abilu_abilmag} 
(a)	Let 
         $\mathbf{f} \in \mathbf{L}^2(\Omega)$,
         then there exists a unique $\mathbf{B} \in \HonenSpace(\Omega) \cap  \mathbf{H}^2(\Omega)$ such that
	\begin{equation}
		\abilusimple{ \mathbf{B} }{ \mathbf{C} } = ( \mathbf{f} , \mathbf{C} )_{L^2} 
		\qquad 
		\mbox{for all } \mathbf{C} \in \HonenSpace(\Omega)
	\end{equation}
	and it holds
	\begin{equation}
		\norm{H^1}{ \mathbf{B} } \lesssim \norm{(\HonenSpace)^{\ast}}{ \mathbf{f} } \qquad \mbox{and} \qquad \norm{H^2}{\mathbf{B} } \lesssim \norm{L^2}{ \mathbf{f} } .
	\end{equation}
(b)	Let $\bfAapp \in \HonenSpaceDiv(\Omega) \cap \mathbf{L}^{\infty}(\Omega)$ be such that $\| \bfAapp \|_{\mathbf{L}^{\infty}} \lesssim 1$, then $a_{\bfAapp}(\cdot,\cdot)+( (1+|\bfAapp|^2) \cdot , \cdot )_{L^2}$ is coercive and continuous w.r.t. the $H^1_{\kappa}$-norm with constants independent of $\kappa$.
Furthermore, for each $f\in L^2(\Omega)$
there is a unique $z \in \Honeperp \cap H^2(\Omega)$ with
\begin{equation}
	a_{\bfAapp} \hspace{-2pt}( z , \psi) + ( (1\hspace{-1pt}+\hspace{-1pt}|\bfAapp|^2) \,z , \psi )_{L^2} =  ( f , \psi )_{L^2}  
	\qquad \mbox{for all } \psi \in \Honeperp
\end{equation}
and 
\begin{equation}
\Honekappa{z} \lesssim \norm{L^2}{f}
\qquad \mbox{and}
\qquad
\Htwokappa{z} \lesssim \norm{L^2}{f}.
\end{equation}
\end{lemma}

\begin{proof}
(a)	Since $\Omega$ is a cuboid, the results in \cite[Lemma~3.6, Theorem~3.9]{GirR86} are applicable and give the coercivity of $\abilusimple{\cdot}{\cdot}$ on $\HonenSpace(\Omega)$. This yields the unique solvability for any $\mathbf{f} \in \mathbf{L}^2(\Omega)$ and the corresponding stability estimate. The $H^2$-estimate follows from Theorem~\ref{thm:reg_MagPot_H2_H3_gen} for the case $\mathbf{H}=\mathbf{0}$.
	
	(b) The result follows by similar argument as in \cite[Lemma~2.8]{DoeH24}, where we make use of the estimate
		\begin{equation}
	\Big| \frac{( \varphi , \ci \sol )_{L^2} }{\norm{L^2}{\sol}}
	a_{\bfAapp} \hspace{-2pt}( z , \ci \sol ) \Big|
	\lesssim 
	\norm{L^2}{\varphi} \Honekappa{z} \Honekappa{\sol}
\lesssim
\norm{L^2}{\varphi} \norm{L^2}{f} ,
	\end{equation}
	to conclude the $H^2$-regularity.
\end{proof}

The next lemma gives an alternative representation of $\energystab''(\sol,\MagPot)$ for minimizers $(\sol,\MagPot)$.
\begin{lemma} \label{lem:energy2prim_perturbation_v2}
	Let $(\sol,\MagPot) \in H^1(\Omega) \times \HonenSpace(\Omega)$ be a minimizer of \eqref{minimization-problem}	with the regularity established in Corollary~\ref{cor:sol_bounds_ex_A_H2}. Then, it holds for all $(\varphi, \mathbf{B}), (\psi,\mathbf{C}) \in H^1(\Omega) \times \HonenSpace(\Omega)$
	\begin{eqnarray}
		\label{identity-sec-derivative-energy}
		\lefteqn{ \dualp{\energystab''(\sol,\MagPot) (\varphi,\mathbf{B}) }{(\psi,\mathbf{C})} } \\
		\nonumber&=& 
		\abilmag{\varphi}{\psi}
		\,\,+\,\,
		( (1+|\bfA|^2) \varphi , \psi )_{L^2}
		\,\,+\,\,
		\abilusimple{\mathbf{B}}{\mathbf{C}} 
		\,\,+\,\, 
		\abilthree \bigl( (\varphi,\mathbf{B}) , (\psi,\mathbf{C}) \bigr)
	\end{eqnarray}
	for some continuous remainder bilinear form $\abilthree(\cdot,\cdot)$ with the following properties:
	\begin{enumerate}[label=\normalfont(\roman*)]
	\item For $\psi = 0 :$
	\begin{align}
		|\abilthree \bigl( (\varphi,\mathbf{B}) , (0,\mathbf{C}) \bigr)| 
		\,\,\lesssim\,\,
		 (\kappa^{\varepsilon}  \Honekappa{\varphi} + \norm{L^2}{\mathbf{B}} ) \,\norm{L^2}{\mathbf{C}} 
		\,\,\lesssim\,\,
		\kappa^{\varepsilon}\,
		\norm{\HonekappaSpace \times \HonenSpace }{(\varphi,\mathbf{B})}
				 \, \norm{L^2}{\mathbf{C}}.
	\end{align}
	\item For $\mathbf{C}= \mathbf{0} :$
	\begin{align}
		|\abilthree \bigl( (\varphi,\mathbf{B}) , (\psi, \mathbf{0}) \bigr)| 
			&\,\,\lesssim\,\,
			\bigl( 
			\norm{L^2}{\varphi} 
			+
			\norm{H^1}{\mathbf{B}} 
			\bigr)
			\norm{L^2}{\psi}
					\lesssim
			\norm{\HonekappaSpace \times \HonenSpace }{ (\varphi,\mathbf{B}) }
			\, \norm{L^2}{\psi} .
		\end{align}
\item For arbitrary $(\psi,\mathbf{C}) \in H^1(\Omega) \times \HonenSpace(\Omega):$
\begin{align}
	|\abilthree \bigl( (\varphi,\mathbf{B}) , (\psi,\mathbf{C}) \bigr)| 
	&\lesssim 
	\bigl( 
	\norm{L^2}{\varphi} 
	+
	\norm{H^1}{\mathbf{B}} 
	\bigr)
	\bigl(
	\norm{L^2}{\psi}
	+
	\norm{H^1}{\mathbf{C}} 
	\bigr) .
\end{align}
\end{enumerate}
\end{lemma}

\begin{proof}
	First, note that by using \eqref{secE-calculated}, we can identify $\abilthree(\cdot,\cdot)$ as
	\begin{eqnarray*}
	\lefteqn{ \abilthree \bigl( (\varphi,\mathbf{B}) , (\psi,\mathbf{C}) \bigr)
		\,\,\,=\,\,\, -
		\Real \int_\Omega   \bigl( 1+|\MagPot|^2 \bigr)  \varphi \psi^* \dint{x}  
		+
		\Real \int_\Omega \bigl( |\sol|^2 -1 \bigr)  \testfun \psi^* + \sol^2 \testfun^* \psi^* + |\sol|^2 \testfun \psi^* 
		\dint{x} } 
		\\
		&\enspace& +
		\Real \int_\Omega 
		2      (\sol \testfun^* ) \MagPot \cdot \mathbf{C}  
		+
		\kappainv   \bigl( \ci  
		\sol^* \nabla \testfun
		+
		\ci \testfun^* \nabla \sol \bigr) 
		\cdot \mathbf{C}  
		\dint{x} \\
		&\enspace& +
		\Real \int_\Omega 
		2      (\sol \psi^* ) \MagPot \cdot \mathbf{B}  
		+
		\kappainv   \bigl( \ci  
		\sol^* \nabla \psi
		+
		\ci \psi^* \nabla \sol \bigr) 
		\cdot \mathbf{B}  
		\dint{x}
		\,\,\,+\,\,\,
		 \int_\Omega 
		|\sol|^2
		\mathbf{C} \cdot \mathbf{B}  
		\dint{x}.\hspace{80pt}
	\end{eqnarray*}
	We estimate the individual terms one after another.
	\begin{enumerate}
		\item 
		With the embedding $H^2(\Omega) \hookrightarrow L^{\infty}(\Omega)$ (in 3d) and the $H^{2}$-bound for $\MagPot$ we have
		\begin{align}
			|\Real \int_\Omega   \bigl( |\MagPot|^2  + 1 \bigr)  \varphi \psi^* \dint{x} |
			&\,\,\lesssim\,\, 
			   \norm{L^2}{\varphi } \norm{L^2}{\psi}.
		\end{align}

		\item For the second term, we have readily with $|\sol| \le 1$ that
		\begin{align}
			| \Real \int_\Omega  		\bigl( |\sol|^2 -1 \bigr)  \testfun \psi^* + \sol^2 \testfun^* \psi^* + |\sol|^2 \testfun \psi^* 
			\dint{x} |
			\,\,\lesssim\,\, \norm{L^2}{\testfun} \norm{L^2}{\psi}.
		\end{align}	
		
		\item	Using again the $L^{\infty}$-bounds for $\sol$ and $\MagPot$, we have
		\begin{align}
			|\Real \int_\Omega 
			2      (\sol \testfun^* ) \MagPot \cdot \mathbf{C}  
			\dint{x}|
			&\,\,\lesssim \,\,
			 \norm{L^2}{\varphi } \norm{L^2}{\mathbf{C}}.
		\end{align}
		
		\item The fourth term can be estimated in two different ways. On the one hand, we have
				\begin{align}
			&|\kappainv \Real \int_\Omega 
			\ci  
			\sol^* \nabla \testfun
			\cdot \mathbf{C}  
			\dint{x}|
			\,\,\lesssim\,\, \Honekappa{\varphi} \,\norm{L^2}{\mathbf{C}}.
		\end{align}
		On the other hand, we can apply integration by parts to obtain with Lemma \ref{lem:sol_bounds_ex_1}
			\begin{align}
			&|\kappainv \Real \int_\Omega 
			\ci  
			\sol^* \nabla \testfun
			\cdot \mathbf{C}  
			\dint{x}|
			\,\,=\,\,
			|\kappainv \Real \int_\Omega 
			\ci  
			\bigl(
			\nabla \sol^*  \testfun 	\cdot \mathbf{C} 
			+
			\sol^*  \testfun 	\div \mathbf{C}       
			\bigr) 
			\dint{x}| \,\,\lesssim\,\, \norm{L^2}{\testfun} \, \norm{H^1}{\mathbf{C}}.
		\end{align}
		
		\item	 The next term is also estimated in two ways. First, by using \eqref{eq:W1q_bound_sol} 
		we have
		\begin{align}
			|\kappainv \Real \int_\Omega 
			\ci \testfun^* \nabla \sol  
			\cdot \mathbf{C}  
			\dint{x}|
			&\,\,\,\lesssim\,\,\, \kappainv \norm{L^2}{\varphi} \, \norm{L^4}{\nabla \sol} \, \norm{L^4}{\mathbf{C}}	
	                \,\,\overset{\eqref{eq:W1q_bound_sol}}{\lesssim}\,\, \norm{L^2}{\varphi}  \, \norm{H^1}{\mathbf{C}}.
		\end{align}
         For the alternative estimate, we apply the H\"older inequality for a sufficiently small $\delta>0$ with the coefficients $p_1=2+\delta$, $p_2=\tfrac{4+2\delta}{\delta}$ and $p_3=2$ 
		to obtain with \eqref{Wkp-estimates-u}
		\begin{align}
		|\kappainv \Real \int_\Omega \ci \testfun^* \nabla \sol \cdot \mathbf{C} \dint{x}|
			&\,\,\,\lesssim\,\,\,
	\kappainv \norm{L^{2+\delta}}{\varphi} \norm{L^{(4+2\delta)/\delta}}{\nabla \sol}  \norm{L^2}{\mathbf{C}}
	\,\,\overset{\eqref{Wkp-estimates-u}}{\lesssim}\,\, C_{\delta} \,\norm{L^{2+\delta}}{\varphi} \,\norm{L^2}{\mathbf{C}},
		\end{align}
		where $C_{\delta}$ depends on $\delta$ through $\norm{L^{(4+2\delta)/\delta}}{\nabla \sol} \le C_{(4+2\delta)/\delta} \, \kappa$. To estimate $\norm{L^{2+\delta}}{\varphi}$, we use the Gagliardo--Nirenberg interpolation estimate \cite{BrezisMironescu18} which states in our case (by log-convexity of $L^p$-norms) that
		\begin{eqnarray*}
		\norm{L^{2+\delta}}{\varphi} \le \norm{L^{2}}{\varphi}^{1-\varepsilon_{\delta}} \norm{L^{6}}{\varphi}^{\varepsilon_{\delta}} 
		\qquad \mbox{for } \varepsilon_{\delta} \coloneqq  \tfrac{3}{2} \tfrac{\delta}{2+\delta} .
		\end{eqnarray*}
         Since $\norm{L^{2}}{\varphi}\le \Honekappa{\varphi}$ and $\norm{L^{6}}{\varphi} \lesssim \kappa \Honekappa{\varphi}$, we conclude by combining the previous estimates that
		\begin{align}
		|\kappainv \Real \int_\Omega \ci \testfun^* \nabla \sol \cdot \mathbf{C} \dint{x}|
			&\,\,\,\lesssim\,\,\,C_{\delta} \,\Honekappa{\varphi}^{1-\varepsilon_{\delta}} (\kappa \Honekappa{\varphi} )^{\varepsilon_{\delta}}  \,\norm{L^2}{\mathbf{C}} \,\,= \,\, C_{\delta} \, \kappa^{\varepsilon_{\delta}} \, \Honekappa{\varphi}  \,\norm{L^2}{\mathbf{C}}.
		\end{align}
		Note that $\varepsilon_{\delta} \rightarrow 0$ for $\delta \rightarrow 0$. We select $\delta$ such that $\varepsilon_{\delta}=\varepsilon$ for some fixed $\varepsilon$ according to our general notation.
		\item The sixth term is readily estimated with the $L^{\infty}$ bounds for $\sol$ and $\MagPot$ as
		\begin{align}
			| \Real \int_\Omega 
			2      (\sol \psi^* ) \MagPot \cdot \mathbf{B}  
			\dint{x} |
			&\,\,\lesssim\,\,
				\norm{L^2}{\psi} \norm{L^\infty}{\MagPot}  \norm{L^2}{\mathbf{B}}	
				\,\,\lesssim\,\, \norm{L^2}{\psi}  \norm{L^2}{\mathbf{B}} .
		\end{align}

		\item Analogous to (d) we obtain the two estimates
		\begin{align}
			|\kappainv	\Real \int_\Omega 
			\ci  
			\sol^* \nabla \psi
			\cdot \mathbf{B}  
			\dint{x}|
			\,\,\lesssim\,\, \Honekappa{\psi }   \norm{L^2}{\mathbf{B}}
			\quad \mbox{and}
			\quad
						|\kappainv	\Real \int_\Omega 
			\ci  
			\sol^* \nabla \psi
			\cdot \mathbf{B}  
			\dint{x}|
			\,\,\lesssim\,\, \norm{L^2}{\psi} \norm{H^1}{\mathbf{B}}.
		\end{align}
		
		\item We can proceed as for the first estimate in (e). Using \eqref{eq:W1q_bound_sol} 
		we have $\norm{L^4}{\nabla \sol}  \lesssim \kappa$ and hence
		\begin{align}
			|\kappainv	\Real \int_\Omega 
			\ci \psi^* \nabla \sol  
			\cdot \mathbf{B}  
			\dint{x}|
			&\,\,\lesssim\,\,
			\kappainv \norm{L^2}{\psi} \norm{L^4}{\nabla \sol} \norm{L^4}{\mathbf{B}}
			\,\,\lesssim \,\,
			\norm{L^2}{\psi}  \norm{H^1}{\mathbf{B}}.
		\end{align}
		\item Finally, we also have
		\begin{align}
			| \int_\Omega |\sol|^2 \mathbf{C} \cdot \mathbf{B}  \dint{x} |
			\,\,\leq\,\, 
			\norm{L^2}{\mathbf{B}} \norm{L^2}{\mathbf{C}}.
		\end{align}
	\end{enumerate}
	By combining the previous estimates we obtain two alternative estimates for  $\abilthree \bigl( (\varphi,\mathbf{B}) , (\psi,\mathbf{C}) \bigr)$:
	\begin{eqnarray*}
	 |\abilthree \bigl( (\varphi,\mathbf{B}) , (\psi,\mathbf{C}) \bigr)| 
	&\lesssim&  \norm{L^2}{\varphi } ( \norm{L^2}{\psi} +  \norm{H^1}{\mathbf{C}} ) + \norm{H^1}{\mathbf{B}} \norm{L^2}{\psi} + \norm{L^2}{\mathbf{B}} \norm{L^2}{\mathbf{C}}
	\end{eqnarray*}
	and
	\begin{eqnarray*}
	|\abilthree \bigl( (\varphi,\mathbf{B}) , (\psi,\mathbf{C}) \bigr)| 
	&\lesssim& ( \norm{L^2}{\varphi } + \norm{H^1}{\mathbf{B}}) \norm{L^2}{\psi} +  (
	\kappa^{\varepsilon}  \Honekappa{\varphi} + \norm{L^2}{\mathbf{B}} ) \,\norm{L^2}{\mathbf{C}}.
	\end{eqnarray*}
	Both estimates together prove (i)-(iii). 
\end{proof}
Later we will consider auxiliary problems based on the operator $\energystab''(\sol,\MagPot)$. The following proposition yields $H^2$-regularity estimates for the corresponding solutions.
\begin{proposition} \label{prop:H2_reg_E_primeprime}
	Let $(z,\mathbf{Z}) \in 
	\Honeperp \times \HonenSpace(\Omega)$ be the solution of \eqref{eq:var_prob_energe_pp} with $\mathbf{\rhs}=(\rhsord,\rhsA)\in L^2(\Omega) \times \mathbf{L}^2(\Omega)$. 
	Then $(z,\mathbf{Z}) \in H^2(\Omega) \times \mathbf{H}^2(\Omega)$ and
	\begin{eqnarray*}
		  \norm{\HtwokappaSpace}{z} &\lesssim&
		 \norm{L^2}{\rhsord}  +  \Csol \, \| \mathbf{\rhs} \|_{L^2 \times \mathbf{L}^2} ,
		\\
		\norm{H^2}{\mathbf{Z}} &\lesssim&
		\norm{L^2}{\rhsA} + 
		\kappa^{\varepsilon} \, \Csol \,
		\| \mathbf{\rhs} \|_{L^2 \times \mathbf{L}^2}.
	\end{eqnarray*}
\end{proposition}
\begin{proof}
	Using \eqref{identity-sec-derivative-energy}, we first note that we can rewrite problem \eqref{eq:var_prob_energe_pp} as
	\begin{equation}
		\abilmag{z}{\psi} + ((1+|\bfA|^2)z,\psi)_{L^2} + \abilusimple{\mathbf{Z}}{\mathbf{C}} \,\,=\,\,  (\mathbf{\rhs} ,  (\psi,\mathbf{C}))_{L^2 \times \mathbf{L^2}} - 	\abilthree \bigl( (z,\mathbf{Z}) , (\psi,\mathbf{C}) \bigr).
	\end{equation}
	With this, we divide the proof in two parts and study the regularity separately.
	
	(a)	Setting $\mathbf{C}=\mathbf{0}$, yields that $z \in \Honeperp$ solves for all $\psi \in \Honeperp$
	\begin{equation}
		\abilmag{z}{\psi} + ((1+|\bfA|^2)z,\psi)_{L^2} \,\,=\,\,  ( \mathbf{\rhs} , (\psi,\mathbf{0}) )_{L^2 \times \mathbf{L^2}} - 	\abilthree \bigl( (z,\mathbf{Z}) , (\psi,\mathbf{0}) \bigr).
	\end{equation}
	We estimate the right-hand side with Lemma~\ref{lem:inf_sup_cond_exact_bounds} and Lemma~\ref{lem:energy2prim_perturbation_v2} by
	\begin{eqnarray*}
		| ( \mathbf{\rhs} , (\psi, \mathbf{0} ) )_{L^2 \times \mathbf{L^2}} - 	\abilthree \bigl( (z,\mathbf{Z}) , (\psi,\mathbf{0}) \bigr)| 
		&\lesssim& \norm{L^2}{\rhsord} \norm{L^2}{\psi} + \norm{\HonekappaSpace \times \HonenSpace }{ (z,\mathbf{Z}) } \norm{L^2}{\psi} \\ 
		&\lesssim&
		\bigl( \norm{L^2}{\rhsord} + \Csol  
		\| \mathbf{\rhs} \|_{L^2 \times \mathbf{L}^2}
		\bigr) 
		\norm{L^2}{\psi},
	\end{eqnarray*}
	i.e., $( \mathbf{\rhs}  , (\cdot, \mathbf{0} ) )_{L^2 \times \mathbf{L^2}} - 	\abilthree \bigl( (z,\mathbf{Z}) , (\cdot,\mathbf{0}) \bigr) \in L^2(\Omega)^{\ast}$, and hence by Lemma~\ref{lem:regularity_abilu_abilmag} (b) we have $z\in H^2(\Omega)$ with
	\begin{eqnarray*}
		\Htwokappa{z} &\lesssim& \norm{L^2}{\rhsord} + \Csol \, 
		\| \mathbf{\rhs} \|_{L^2 \times \mathbf{L}^2}.
	\end{eqnarray*}

	(b) Setting $\psi = 0$, then $\mathbf{Z} \in \HonenSpace(\Omega)$ solves 
	\begin{equation}
		\abilusimple{\mathbf{Z}}{\mathbf{C}} \,\,=\,\,   (\mathbf{\rhs} , (0,\mathbf{C}) )_{L^2 \times \mathbf{L}^2} - 	\abilthree \bigl( (z,\mathbf{Z}) , (0,\mathbf{C}) \bigr) 
	\end{equation}
	for all $\mathbf{C} \in \HonenSpace(\Omega)$.
	We estimate the right-hand side with
	Lemma~\ref{lem:inf_sup_cond_exact_bounds} and Lemma~\ref{lem:energy2prim_perturbation_v2} (i) by
	\begin{align}
		|  (\mathbf{\rhs} , (0,\mathbf{C}) )_{L^2 \times \mathbf{L}^2} - 	\abilthree \bigl( (z,\mathbf{Z}) , (0,\mathbf{C}) \bigr) | 
		&\lesssim
		\bigl( \norm{L^2}{\rhsA} + 
		\kappa^{\varepsilon}\, \Csol 
		\| \mathbf{\rhs} \|_{L^2 \times \mathbf{L}^2}
		\bigr) 
		\norm{L^2}{\mathbf{C}}
	\end{align}
	and hence by Lemma~\ref{lem:regularity_abilu_abilmag} (a) we have $\mathbf{Z}\in\mathbf{H}^2(\Omega)$ and
	\begin{equation}
		\norm{H^2}{\mathbf{Z}} \lesssim \norm{L^2}{\rhsA} + 
		\kappa^{\varepsilon}\, \Csol \, 
		\| \mathbf{\rhs} \|_{L^2 \times \mathbf{L}^2}.
	\end{equation}
	This proves the proposition.
\end{proof}

\section{Abstract error analysis}
\label{sec:abstract_error_analysis}
As a basis for our error analysis in FEM and LOD spaces, we will start with some abstract convergence results in this section. For this, we consider an arbitrary family of (non-empty) finite-dimensional spaces 
\begin{eqnarray*}
X_{H(\delta)} \times \mathbf{X}_{h(\delta)} \subset H^1(\Omega) \times \HonenSpace(\Omega)
\end{eqnarray*}
which are parametrized by a small parameter $\delta>0$. The notation $H(\delta)$ and $h(\delta)$ is used to indicate that different mesh sizes could be used for the approximations of order parameter and magnetic potential. We further assume that functions in $H^1(\Omega) \times \HonenSpace(\Omega)$ can be approximated by arbitrary accuracy in the sense that for each $(\varphi,\mathbf{B})\in H^1(\Omega) \times \HonenSpace(\Omega)$ it holds
\begin{eqnarray*}
\lim_{\delta \rightarrow 0} \, \underset{\in X_{H(\delta)} \times \mathbf{X}_{h(\delta)}}{\inf_{(\varphi_{H(\delta)},\mathbf{B}_{h(\delta)}) }} \| (\varphi - \varphi_{H(\delta)},\mathbf{B} -  \mathbf{B}_{h(\delta)}) \|_{H^1_{\kappa} \times{\mathbf{H}^1}} = 0.
\end{eqnarray*}
This assumption is fulfilled for all reasonable families of standard approximation spaces, such as finite element spaces. For brevity, we skip from now on $\delta$ in the notation and just write $H=H(\delta)$ and $h=h(\delta)$, unless the role of $\delta$ is explicitly required. With this, we are looking for discrete minimizers of
\begin{equation}  \label{eq:energy_functional_approx}
	\energystab(\sol_{\deltaH} ,\MagPot_{\deltah})  
	= 
	\min_{  (\varphi_{\deltaH} , \mathbf{B}_{\deltah} ) \in X_{\deltaH} \times \mathbf{X}_{\deltah} }
	\energystab(\varphi_{\deltaH},\mathbf{B}_{\deltah}) 
\end{equation}
which thus satisfy
\begin{align}  
	\duenergy(\sol_{\deltaH},\MagPot_{\deltah}) \, \varphi_{\deltaH} &= 0, \qquad \text{ for all } \varphi_{\deltaH} \in X_{\deltaH},
	\\
	\dAenergy(\sol_{\deltaH},\MagPot_{\deltah}) \, \mathbf{B}_{\deltah} &= 	0 \qquad\hspace{5pt} \text{ for all } \mathbf{B}_{\deltah} \in \mathbf{X}_{\deltah} .
\end{align}

The following lemma provides uniform bounds on the discrete minimizers only using the minimizing properties. The bounds are in line with the respective bounds obtained for the continuous minimizers.

\begin{lemma} \label{lem:sol_bounds_approx}
	Let $(\sol_{\deltaH},\MagPot_{\deltah})$ be a minimizer of \eqref{eq:energy_functional_approx} in $X_{\deltaH} \times \mathbf{X}_{\deltah}$. Then it holds
	\begin{align}
		\energystab(\sol_{\deltaH},\MagPot_{\deltah}) &\lesssim 1,
		\qquad
		\Honekappa{\sol_{\deltaH}}
		+
		\norm{L^4}{\sol_{\deltaH}}  
		\lesssim 1,
		\qquad
		\norm{H^1}{\MagPot_{\deltah}} \lesssim 1 
	\end{align}
with hidden constants independent of $\kappa$ and $\delta$, but depending on the external field $\mathbf{H}$.
\end{lemma}

\begin{proof}
Since $(0,\mathbf{0}) \in X_{\deltaH} \times \mathbf{X}_{\deltah}$, we obtain the bound on the energy. Further, as in \cite{DoeH24}, we obtain $\norm{L^2}{\sol_{\deltaH}} \lesssim 1$, and conclude with this
\begin{align}
\norm{L^4}{\sol_{\deltaH}}^4 &\,\,\leq\,\, \int_\Omega 
\bigl( 1 - |\sol_{\deltaH}|^2\bigr)^2
\dint{x} 
+ 2 \norm{L^2}{\sol_{\deltaH}}^2 + \norm{L^2}{1 }^2 
\,\,\leq\,\, 2 \energystab(\sol_{\deltaH},\MagPot_{\deltah})   + 2 \norm{L^2}{\sol_{\deltaH}}^2 + \norm{L^2}{1 }^2  
\,\,\lesssim\,\, 1.
 \end{align}
The $H^1$-bound $\| \MagPot_{\deltah} \|_{H^1} \lesssim 1 + \| \mathbf{H} \|_{L^2} \lesssim 1$ follows as in Lemma~\ref{lem:sol_bounds_ex_1}. 
Next, we use the identity
\begin{equation}
2 \energystab(\sol_{\deltaH},\MagPot_{\deltah})  \,\,\geq\,\, 	\norm{L^2}{\tfrac{1}{\kappa}  \nabla \sol_{\deltaH} + \ci  \MagPot_{\deltah} \sol_{\deltaH} }^2 
\,\,=\,\, \norm{L^2}{ \tfrac{1}{\kappa}  \nabla \sol_{\deltaH} }^2 
+
\norm{L^2}{\MagPot_{\deltah} \sol_{\deltaH} }^2 
+2 \Real \ip{\tfrac{1}{\kappa}  \nabla \sol_{\deltaH}}{\ci  \MagPot_{\deltah} \sol_{\deltaH} }_{L^2},
\end{equation}
and bound the inner product via
\begin{align}
2 \Real \ip{\tfrac{1}{\kappa}  \nabla \sol_{\deltaH}}{\ci  \MagPot_{\deltah} \sol_{\deltaH} }_{L^2}
\leq
 2 \norm{L^2}{\tfrac{1}{\kappa}  \nabla \sol_{\deltaH} }
\norm{L^4}{\MagPot_{\deltah}} \norm{L^4}{\sol_{\deltaH}} \leq \tfrac{1}{2} \norm{L^2}{\tfrac{1}{\kappa}  \nabla \sol_{\deltaH} }^2 + 2 \norm{L^4}{\MagPot_{\deltah}}^2 \norm{L^4}{\sol_{\deltaH}}^2,
\end{align}
to conclude $\norm{L^2}{\tfrac{1}{\kappa}  \nabla \sol_{\deltaH} } \lesssim 1$.
\end{proof}

This allows us to conclude the following abstract convergence result, which is fully analogous to \cite[Prop.~5.1]{DoeH24}.
\begin{proposition} \label{prop:abstract_convergence}
	Denote by $(\sol_{H(\delta)},\MagPot_{h(\delta)})_{\delta>0}$ a family of minimizers of \eqref{eq:energy_functional_approx}. Then, there exists an exact minimizer $(\sol_0,\MagPot_0) \in  H^1(\Omega) \times \HonenSpace(\Omega) $ of problem \eqref{minimization-problem} such that there is a monotonically decreasing sequence 
	$(\delta_n)_{n\in\N}$
	with
	\begin{equation}
		\lim\limits_{n\to\infty}	\Honecombi{\sol_0 - \sol_{H(\delta_n)}, \MagPot_0- \MagPot_{h(\delta_n)}} = 0.
	\end{equation}
In particular, we can define the twisted approximations to $\sol_0$ by
\begin{equation}
\widetilde{\sol}_{H(\delta_n)} \coloneqq e^{\ci \omega_n}	\sol_{H(\delta_n)},
\quad
\text{ where } \, \omega_n \in [- \tfrac{\pi}{2} , \tfrac{\pi}{2}] 
\text{ is chosen such that } 
\, 
\ipsymLtwo{\widetilde{\sol}_{H(\delta_n)}}{\sol_0} = 0
\end{equation}
which also converge in $H^1$, i.e.,
\begin{equation}
	\lim\limits_{n\to\infty}	\Honecombi{ (\sol_0 - \widetilde{\sol}_{H(\delta_n)} , \MagPot_0- \MagPot_{h(\delta_n)}) } = 0.
\end{equation}
Conversely, for any $n$, the minimizer $\sol_{H(\delta_n)}$ is an approximation to $e^{\ci \omega_n} \sol_0$.
\end{proposition}

\begin{proof}
The proof is along the lines of
\cite{CheGZ10} and \cite[Prop.~5.1]{DoeH24}.
We employ the bounds in Lemma~\ref{lem:sol_bounds_approx} and the semi-lower continuity of $\energystab$, see e.g. \cite[Theorem~1.6]{Struwe08}. Since we only apply constant rotations $\omega_n$, the magnetic potential is unaffected, and the proof in \cite[Prop.~5.1]{DoeH24} is applicable.
\end{proof}
In the light of Proposition \ref{prop:abstract_convergence} we can assume without loss of generality that the considered discrete minimizers $( \sol_{\deltaH},  \MagPot_{\deltah})$ are such that $\sol_{\deltaH} \in X_{\deltaH} \cap \Honeperp$ for a suitable exact minimizer $(\sol , \MagPot )$ to which we compare it. In such a setting, we derive error bounds for $(\sol - \sol_{\deltaH}, \MagPot - \MagPot_{\deltah})$ depending on the best-approximation properties of the space $X_{\deltaH} \times \mathbf{X}_{\deltah}$. 

As a first step towards this result we need to verify that for any $v \in \Honeperp$ the $H^1_{\kappa}$-best-approximation in $X_{\deltaH} \cap \Honeperp$ behaves as the $H^1_{\kappa}$-best-approximation in $X_{\deltaH}$, i.e., the discrete space without the orthogonality constraint.
\begin{lemma}\label{lemma:ritz-projections:abstract-new}
Let $(\sol,\MagPot) \in  H^1(\Omega) \times \HonenSpace(\Omega) $ be a minimizer of \eqref{minimization-problem} and let $\Ltwoproj : H^1(\Omega) \rightarrow X_{\deltaH}$ denote the $L^2$-projection on $ X_{\deltaH}$. 
If $\Ltwoproj\sol  \not= 0$, then we have for any $\varphi \in \Honeperp$
\begin{eqnarray}
\label{est-Ritzprojsoldelta-new}
\inf_{\varphi_H \in X_{\deltaH} \cap \Honeperp} \| \varphi - \varphi_H \|_{\HonekappaSpace} &\lesssim& \| \varphi - \Ltwoproj \varphi \|_{\HonekappaSpace} + \tfrac{ \| \Ltwoproj \sol \|_{\HonekappaSpace} }{\| \sol \|_{L^2} - \| \sol - \Ltwoproj\sol \|_{L^2} }  \| \varphi - \Ltwoproj \varphi \|_{L^2}.
\end{eqnarray}
In particular, if $X_{\deltaH}$ is such that the $L^2$-projection is $H^1_{\kappa}$-stable, $\inf\limits_{w_H \in X_H} \| u - \pi_H u \|_{L^2} \lesssim \kappa H \| u \|_{H^1_{\kappa}}$, and $H \lesssim \kappa^{-1}$,
then we have
\begin{equation} \label{est-Ritzprojsoldelta-2-new}
 \inf_{\varphi_H \in X_{\deltaH} \cap \Honeperp} \| \varphi - \varphi_H \|_{\HonekappaSpace} \lesssim \inf_{\varphi_{\deltaH} \in X_{\deltaH} } \| \varphi - \varphi_{\deltaH} \|_{\HonekappaSpace}  .
\end{equation}
\end{lemma}
\begin{proof}
Let $\textup{P}_H^{\perp} : \Honeperp \rightarrow X_{\deltaH} \cap \Honeperp$ denote the orthogonal projection with respect to the $H^1_{\kappa}$-inner product, i.e. $(\varphi,\psi)_{H^1_{\kappa}} := (\varphi,\psi)_{L^2} + \tfrac{1}{\kappa^2}(\nabla \varphi, \nabla \psi)_{L^2}$. We have
\begin{eqnarray*}
\| \varphi- \textup{P}_H^{\perp}(\varphi) \|_{H^1_{\kappa}} &=& \inf_{\varphi_H \in X_{\deltaH} \cap \Honeperp}  \| \varphi-  \varphi_H \|_{H^1_{\kappa}}.
\end{eqnarray*}
In order to get a quasi-best approximation on the full space $X_{\deltaH}$ with estimates \eqref{est-Ritzprojsoldelta-new} and \eqref{est-Ritzprojsoldelta-2-new}, we exploit that $\varphi,\sol \in \Honeperp$ and proceed analogously as in \cite[Lemma 5.11]{DoeH24} to get 
\begin{eqnarray*}
\| \varphi - \textup{P}_H^{\perp} \varphi \|_{\HonekappaSpace} &\le& \| \varphi - \Ltwoproj \varphi + \tfrac{(\Ltwoproj \varphi - \varphi, \ci \sol )_{L^2} }{ ( \Ltwoproj(\ci \sol) , \ci \sol )_{L^2} } \Ltwoproj (\ci \sol) \|_{\HonekappaSpace}.
\end{eqnarray*}
Estimate \eqref{est-Ritzprojsoldelta-new} follows straightforwardly with $( \Ltwoproj(\ci \sol) , \ci \sol )_{L^2}=  \| \sol\|_{L^2}^2 + ( \Ltwoproj(\ci \sol) -  \ci \sol, \ci \sol )_{L^2}$ and by noting that $\Ltwoproj\sol \not =0$ implies $ \| \sol - \Ltwoproj\sol \|_{L^2} < \| \sol \|_{L^2}$. For $\Ltwoproj\sol  =0$, we trivially have the estimate $\| \varphi - \textup{P}_H^{\perp} \varphi \|_{\HonekappaSpace} \le \| \varphi - \Ltwoproj \varphi\|_{\HonekappaSpace}$. Next, note that the second estimate in Lemma \ref{lem:sol_bounds_ex_1} together with $\| \MagPot \|_{L^{\infty}} \lesssim \| \MagPot \|_{H^{2}} \lesssim 1$ from Corollary \ref{cor:sol_bounds_ex_A_H2} imply $\norm{\HonekappaSpace}{\sol} \lesssim \norm{L^2}{\sol}$. Hence, if $\Ltwoproj$ is $H^1_{\kappa}$-stable, i.e., $\| \Ltwoproj \varphi \|_{\HonekappaSpace} \lesssim \| \varphi \|_{\HonekappaSpace}$, then we get 
\begin{eqnarray*}
 \tfrac{ \| \Ltwoproj \sol \|_{\HonekappaSpace} }{\| \sol \|_{L^2} - \| \sol - \Ltwoproj\sol \|_{L^2} }
 \lesssim  \tfrac{ \| \sol \|_{\HonekappaSpace} }{\| \sol \|_{L^2} - \| \sol - \Ltwoproj\sol \|_{L^2} }
 \lesssim  ( 1- \tfrac{ \| \sol - \Ltwoproj\sol \|_{L^2} }{\| \sol \|_{L^2}} )^{-1}.
\end{eqnarray*}
Further, we conclude $\| \sol - \Ltwoproj\sol \|_{L^2} \lesssim \kappa H \Honekappa{\sol} \lesssim  \kappa H \| \sol \|_{L^2} $, and thus by the resolution condition $\kappa H \lesssim 1$, the desired estimate \eqref{est-Ritzprojsoldelta-2-new} follows.
\end{proof}

Next, we define a simultaneous projection for the order parameter and the vector potential based on the inner product $\langle E^{\prime\prime}(u,\bfA) \cdot ,\cdot \rangle$ on $\Honeperp \times \HonenSpace(\Omega)$.
\begin{definition}[$E^{\prime\prime}(u,\bfA)$-Ritz-projection]
\label{definition:Eprimeprime-Ritz}
Let $(\sol,\MagPot) \in  H^1(\Omega) \times \HonenSpace(\Omega) $ be a minimizer of \eqref{minimization-problem} and suppose that Assumption \ref{ass:local_uniq} holds. We define the Ritz-projection
$$
\textup{\bf R}_{H,h} :=  (\textup{R}_{H},\textup{R}_{h})  \,\,:\,\,  \Honeperp \times \HonenSpace(\Omega) \,\, \rightarrow \,\, (X_{\deltaH}\cap \Honeperp) \times \mathbf{X}_{\deltah}
$$
for $(\varphi , \mathbf{B}) \in \Honeperp \times \HonenSpace(\Omega) $ by
\begin{eqnarray*}
\langle E^{\prime\prime}(u,\bfA) \, \textup{\bf R}_{H,h} (\varphi , \mathbf{B}), 
(\psi_H , \mathbf{C}_h ) \rangle = \langle E^{\prime\prime}(u,\bfA) \, (\varphi , \mathbf{B}), 
(\psi_H , \mathbf{C}_h ) \rangle
\end{eqnarray*}
for all $(\psi_H , \mathbf{C}_h ) \in (X_{\deltaH}\cap \Honeperp) \times \mathbf{X}_{\deltah}$. The projection $\textup{\bf R}_{H,h}$ is well-defined by Proposition \ref{prop:coecivity_Epp}.
\end{definition}
To work with $\textup{\bf R}_{H,h}$, we need to show that $ \textup{\bf R}_{H,h} (\varphi , \mathbf{B})$ is a quasi-best-approximation in $H^1(\Omega) \times \HonenSpace(\Omega)$. We have the following result.
\begin{lemma}[Approximation properties of $\textup{\bf R}_{H,h}$]
\label{lemma-ritz-proj-est}
Suppose we are in the setting of Definition \ref{definition:Eprimeprime-Ritz} and let the assumptions of Lemma~\ref{lemma:ritz-projections:abstract-new} hold. 
Then, 
 we have for arbitrary $(\varphi , \mathbf{B}) \in \Honeperp \times \HonenSpace(\Omega) $ 
\begin{eqnarray*}
\lefteqn{ \Honecombi{ (\varphi , \mathbf{B}) - \textup{\bf R}_{H,h} (\varphi , \mathbf{B}) } }\\
&\lesssim&  
\inf_{(\varphi_H, \mathbf{B}_h) \in X_{\deltaH} \times \mathbf{X}_{\deltah}}
\Honecombi{ (\varphi -\varphi_H, \mathbf{B}-\mathbf{B}_h) } 
+ 
\inf_{(z_H, \mathbf{Z}_h) \in X_{\deltaH} \times \mathbf{X}_{\deltah}}
\Honecombi{ (z -z_H, \mathbf{Z}-\mathbf{Z}_h) }, 
\end{eqnarray*}
where $(z,\mathbf{Z}) \in \Honeperp \times \HonenSpace(\Omega)$ solves
\begin{eqnarray*}
\langle E'' (u,\bfA) ( z,\mathbf{Z} ) , (\psi, \mathbf{C} ) \rangle 
&= ( (\varphi , \mathbf{B}) \hspace{-2pt}-\hspace{-2pt} \textup{\bf R}_{H,h} (\varphi , \mathbf{B}) , (\psi, \mathbf{C} ))_{L^2 \times \mathbf{L}^2}
\end{eqnarray*}
for all $(\psi, \mathbf{C} ) \in \Honeperp \times \HonenSpace(\Omega)$.
\end{lemma}

\begin{proof}
To shorten the notation, we write  $(e_{\varphi},\mathbf{e}_{ \mathbf{B}}) := (\varphi , \mathbf{B}) - \textup{\bf R}_{H,h} (\varphi , \mathbf{B})$.
In the proof of Proposition \ref{prop:coecivity_Epp}, we verified the G{\aa}rding inequality
\begin{eqnarray*}
	\dualp{\energystab''(\sol,\MagPot) (e_{\varphi},\mathbf{e}_{ \mathbf{B}}) }{  (e_{\varphi},\mathbf{e}_{ \mathbf{B}}) }
	&\geq& C_1
	\Honecombi{(e_{\varphi},\mathbf{e}_{ \mathbf{B}})}^2
	-
	C_2
	\norm{L^2 \times \mathbf{L}^2}{(e_{\varphi},\mathbf{e}_{ \mathbf{B}})}^2
\end{eqnarray*}
for some $\kappa$-independent constants $C_1,C_2 \ge 0$. Following the \quotes{Schatz argument} \cite{Sch74}, we let $(z,\mathbf{Z}) \in \Honeperp \times \HonenSpace(\Omega)$ denote the unique solution to
\begin{eqnarray*}
\langle E'' (u,\bfA) ( z,\mathbf{Z} ) , (\psi, \mathbf{C} ) \rangle 
&= ( ( e_{\varphi} , \mathbf{e}_{\mathbf{B}} ) , (\psi, \mathbf{C} ))_{L^2 \times \mathbf{L}^2}, 
\quad \text{ for all } (\psi, \mathbf{C} ) \in \Honeperp \times \HonenSpace(\Omega).
\end{eqnarray*}
Then, we have by linear combination
\begin{eqnarray*}
\Honecombi{( e_{\varphi} , \mathbf{e}_{\mathbf{B}}  )}^2 
&\lesssim&  \dualp{\energystab''(\sol,\MagPot) (( e_{\varphi} , \mathbf{e}_{\mathbf{B}} ) + ( z,\mathbf{Z} ) ) }{( e_{\varphi} , \mathbf{e}_{\mathbf{B}} )}.
\end{eqnarray*}
Exploiting the symmetry of $\langle E'' (u,\bfA) \, \cdot ,\cdot\rangle$ and the Galerkin-orthogonality for $ \textup{\bf R}_{H,h}$, we further obtain for arbitrary $(\varphi_H , \mathbf{B}_h), \,(z_{\deltaH}, \mathbf{Z}_{\deltah}) \in (X_{\deltaH} \cap \Honeperp) \times \mathbf{X}_{\deltah}$ that
\begin{eqnarray*}
\lefteqn{ \Honecombi{( e_{\varphi} , \mathbf{e}_{\mathbf{B}}  )}^2 
\,\,\,\lesssim\,\,\, \dualp{\energystab''(\sol,\MagPot) (( e_{\varphi} , \mathbf{e}_{\mathbf{B}} ) + ( z,\mathbf{Z} ) ) }{( e_{\varphi} , \mathbf{e}_{\mathbf{B}} )}  } \\
&=&
\dualp{\energystab''(\sol,\MagPot) ( e_{\varphi} , \mathbf{e}_{\mathbf{B}} ) }{  (\varphi - \varphi_H , \mathbf{B} -\mathbf{B}_h)  }
+
 \dualp{\energystab''(\sol,\MagPot) ( e_{\varphi} , \mathbf{e}_{\mathbf{B}} ) }{ ( z-z_H,\mathbf{Z}-\mathbf{Z}_h ) }
.
\end{eqnarray*}
The continuity of $\energystab''(\sol,\MagPot)$ in Proposition \ref{prop:coecivity_Epp} together with 
estimate \eqref{est-Ritzprojsoldelta-2-new} in Lemma \ref{lemma:ritz-projections:abstract-new} finish the proof. 
\end{proof}
With the approximation properties of the Ritz-projection $\textup{\bf R}_{H,h} =  (\textup{R}_{H},\textup{R}_{h})$ at hand, we split the error as 
\begin{eqnarray}
\label{abstract-error-splitting}\lefteqn{ \| (\sol - \sol_{\deltaH}, \MagPot - \MagPot_{\deltah}) \|_{\HonekappaSpace \times \mathbf{H}^1} } \\
\nonumber &\le& 
\| (\sol - \textup{R}_{H} \sol, \MagPot - \textup{R}_{h} \MagPot ) \|_{\HonekappaSpace \times \mathbf{H}^1} + 
\| (\sol_{\deltaH} - \textup{R}_{H} \sol, \MagPot_{\deltah} - \textup{R}_{h} \MagPot ) \|_{\HonekappaSpace \times \mathbf{H}^1},
\end{eqnarray}
where it remains to provide an abstract estimate for the defect, i.e., the second term.
To study the defect, we exploit the coercivity of $\energystab''(\sol,\MagPot)$ and consider the term 
\begin{eqnarray}
\label{def-eps-2}
\sigma(\psi_{\deltaH},\mathbf{C}_{\deltah}) &:=& \dualp{\energystab''(\sol,\MagPot) (\sol_{\deltaH} - \textup{R}_{H} \sol, \MagPot_{\deltah} - \textup{R}_{h} \MagPot )  }{(\psi_{\deltaH},\mathbf{C}_{\deltah})} 
\end{eqnarray}
for $(\psi_{\deltaH},\mathbf{C}_{\deltah}) \in (X_{\deltaH}\cap \Honeperp) \times \mathbf{X}_{\deltah}$. By Galerkin orthogonality, we can write 
\begin{eqnarray*}
\sigma(\psi_{\deltaH},\mathbf{C}_{\deltah}) &=& 
\dualp{\energystab''(\sol,\MagPot) (\sol - \sol_{\deltaH} , \MagPot - \MagPot_{\deltah}) }{(\psi_{\deltaH},\mathbf{C}_{\deltah})}  
\end{eqnarray*}
which yields the following estimate.
\begin{lemma} \label{lem:epsilon_2_bound}
Let $(\sol,\MagPot) \in  H^1(\Omega) \times \HonenSpace(\Omega) $ be a minimizer of \eqref{minimization-problem} and $\sigma(\psi_{\deltaH},\mathbf{C}_{\deltah})$ be defined as in \eqref{def-eps-2}. It holds
\begin{eqnarray}
\label{est:epsilon_2_bound}
	| \sigma(\psi_h,\mathbf{C}_h) |
	&\lesssim&
	\left( 
	\|  \sol - \sol_{\deltaH} \|_{L^6}^2  +  \|  \sol - \sol_{\deltaH} \|_{\HonekappaSpace}^2 + \| \MagPot - \MagPot_{\deltah} \|_{H^1}^2 
	\right) \, \| (\psi_{\deltaH},\mathbf{C}_{\deltah}) \|_{H^1_{\kappa}\times \mathbf{H}^1}.
	\end{eqnarray}
\end{lemma}
The proof is technical and given in Appendix~\ref{sec:app_computations}.

As a direct conclusion from Lemma \ref{lemma-ritz-proj-est} and Lemma \ref{lem:epsilon_2_bound}, we have the following theorem.
\begin{theorem}\label{thm:Honecombi_bound}
Let $(\sol,\MagPot) \in  H^1(\Omega) \times \HonenSpace(\Omega) $ be a minimizer of \eqref{minimization-problem} that is locally quasi-unique in the sense of Definition \ref{definition-local-quasi-uniquness} and let $(\sol_{\deltaH},\MagPot_{\deltah}) \in X_{\deltaH} \times \mathbf{X}_{\deltah}$ be a (discrete) minimizer of \eqref{eq:energy_functional_approx}.
Under the assumptions on $X_{\deltaH} $ in 
Lemma~\ref{lemma:ritz-projections:abstract-new}, 
it holds
\begin{eqnarray}
\label{est:abstract-disc-estimate-full}
\nonumber\lefteqn{ \| (\sol - \sol_{\deltaH}, \MagPot - \MagPot_{\deltah}) \|_{\HonekappaSpace \times \mathbf{H}^1} } \\
\nonumber&\lesssim& 
\inf_{(\varphi_H, \mathbf{B}_h) \in X_{\deltaH} \times \mathbf{X}_{\deltah}}
\Honecombi{ (u -\varphi_H, \mathbf{A}-\mathbf{B}_h) } 
+ 
\inf_{(z_H, \mathbf{Z}_h) \in X_{\deltaH} \times \mathbf{X}_{\deltah}}
\Honecombi{ (z -z_H, \mathbf{Z}-\mathbf{Z}_h) } \\
&\enspace& + \,\, \Csol \left(  \|  \sol - \sol_{\deltaH} \|_{L^6}^2  \,\,+\,\,  \|  \sol - \sol_{\deltaH} \|_{\HonekappaSpace}^2 \,\,+\,\, \| \MagPot - \MagPot_{\deltah} \|_{H^1}^2\right)
\end{eqnarray}
where $(z,\mathbf{Z}) \in \Honeperp \times \HonenSpace(\Omega)$ solves
\begin{eqnarray}
\label{def-z-proj-error}
\langle E'' (u,\bfA) ( z,\mathbf{Z} ) , (\psi, \mathbf{C} ) \rangle 
&= ( (u , \mathbf{A}) \hspace{-2pt}-\hspace{-2pt} \textup{\bf R}_{H,h} (u , \mathbf{A}) , (\psi, \mathbf{C} )_{L^2 \times \mathbf{L}^2}
\end{eqnarray}
for all $(\psi, \mathbf{C} ) \in \Honeperp \times \HonenSpace(\Omega)$.
\end{theorem}
\begin{proof} 
We consider the error splitting \eqref{abstract-error-splitting}, where the first term is estimated by Lemma \ref{lemma-ritz-proj-est}. For the second term, we have with the coercivity of $\energystab''(\sol,\MagPot)$ and the definition of $\sigma$ in \eqref{def-eps-2}
\begin{eqnarray*}
\lefteqn{  \Csolinv \,\, \| (\sol_{\deltaH} - \textup{R}_{H} \sol, \MagPot_{\deltah} - \textup{R}_{h} \MagPot ) \|_{\HonekappaSpace \times \mathbf{H}^1}^2 } \\
&\lesssim&  \dualp{\energystab''(\sol,\MagPot) (\sol_{\deltaH} - \textup{R}_{H} \sol, \MagPot_{\deltah} - \textup{R}_{h} \MagPot ) }{ (\sol_{\deltaH} - \textup{R}_{H} \sol, \MagPot_{\deltah} - \textup{R}_{h} \MagPot ) } \\
&=& \sigma(\sol_{\deltaH} - \textup{R}_{H} \sol, \MagPot_{\deltah} - \textup{R}_{h} \MagPot ).
\end{eqnarray*}
Lemma \ref{lem:epsilon_2_bound} finishes the proof.
\end{proof}
Even though Theorem~\ref{thm:Honecombi_bound} looks promising on its own, it only gets meaningful when used in combination with Proposition \ref{prop:abstract_convergence}. The reason is that $(\sol , \MagPot )$ and $( \sol_{\deltaH}, \MagPot_{\deltah})$ might not be unique (even aside from gauge transformations). Hence, the higher order remainder term $\|  \sol - \sol_{\deltaH} \|_{L^6}^2  \,\,+\,\,  \|  \sol - \sol_{\deltaH} \|_{\HonekappaSpace}^2 \,\,+\,\, \| \MagPot - \MagPot_{\deltah} \|_{H^1}^2$ is not necessarily small. However, Proposition \ref{prop:abstract_convergence} guarantees that for each discrete minimizer $( \sol_{\deltaH}, \MagPot_{\deltah})$ there is a corresponding exact minimizer  $(\sol , \MagPot )$, such that $ \| (\sol - \sol_{\deltaH}, \MagPot - \MagPot_{\deltah}) \|_{\HonekappaSpace \times \mathbf{H}^1}$ becomes arbitrarily small for $(\deltaH(\delta),\deltah(\delta)) \rightarrow 0$. Hence, we can absorb the higher order term on the right-hand side of \eqref{est:abstract-disc-estimate-full} into the left-hand side. With this, we obtain the following conclusion. 
\begin{conclusion}%
\label{concl:Honecombi_bound}
Let Assumption \ref{ass:local_uniq} hold and let $(\sol_{\deltaH},\MagPot_{\deltah}) \in X_{\deltaH} \times \mathbf{X}_{\deltah}$ be a discrete minimizer of \eqref{eq:energy_functional_approx}, where $X_{\deltaH}$ has the properties as in
 Lemma~\ref{lemma:ritz-projections:abstract-new} 
.

Then, for all sufficiently small $(\deltaH(\delta),\deltah(\delta))$, there exists a minimizer $(\sol,\MagPot) \in  H^1(\Omega) \times \HonenSpace(\Omega)$ of \eqref{minimization-problem} such that $\sol_{\deltaH} \in \Honeperp$ and
\begin{eqnarray*}
\lefteqn{ \| (\sol - \sol_{\deltaH}, \MagPot - \MagPot_{\deltah}) \|_{\HonekappaSpace \times \mathbf{H}^1} } \\
&\lesssim& 
\inf_{(\varphi_H, \mathbf{B}_h) \in X_{\deltaH} \times \mathbf{X}_{\deltah}}
\Honecombi{ (u -\varphi_H, \mathbf{A}-\mathbf{B}_h) } 
+ 
\inf_{(z_H, \mathbf{Z}_h) \in X_{\deltaH} \times \mathbf{X}_{\deltah}}
\Honecombi{ (z -z_H, \mathbf{Z}-\mathbf{Z}_h) }
\end{eqnarray*}
with constants independent of $\kappa$ and $(\deltaH,\deltah)$ and where $(z,\mathbf{Z}) \in \Honeperp \times \HonenSpace(\Omega)$ denotes the solution to \eqref{def-z-proj-error}.
\end{conclusion}
We finish the section on abstract error estimates with an estimate for the error in energy. 
\begin{theorem} \label{thm:energy_bound}
Let $(\sol,\MagPot) \in  H^1(\Omega) \times \HonenSpace(\Omega) $ be an arbitrary minimizer of \eqref{minimization-problem} and $(\sol_{\deltaH},\MagPot_{\deltah}) \in X_{\deltaH} \times \mathbf{X}_{\deltah}$ an arbitrary discrete minimizer of \eqref{eq:energy_functional_approx}. 
Then, the error in the energy is bounded by
	\begin{eqnarray*}
		\lefteqn{ \energystab(\sol_{\deltaH},\MagPot_{\deltah}) - \energystab(\sol,\MagPot) }\\
		&\lesssim&
		\inf_{ (\varphi_{\deltaH},\mathbf{B}_{\deltah}) \in X_{\deltaH} \times \mathbf{X}_{\deltah} } \left(
		\Honecombi{ (\sol - \varphi_{\deltaH},\MagPot - \mathbf{B}_h )}^2
                 + \| \sol- \varphi_{\deltaH}\|_{L^6} \| \sol- \varphi_{\deltaH} \|_{L^2}^2 +  \| \sol- \varphi_{\deltaH} \|_{L^4}^4 \right)
	\end{eqnarray*}
	with constants independent of $\kappa$ and $(\deltaH,\deltah)$. Note here that $\energystab(\sol,\MagPot) \leq\energystab(\sol_{\deltaH},\MagPot_{\deltah})$.\\[0.2em]
	The result essentially compares the exact minimal energy level with the minimal energy level in the discrete space $X_{\deltaH} \times \mathbf{X}_{\deltah}$. Hence, the estimate requires that $(\sol_{\deltaH},\MagPot_{\deltah})$ converges to $(\sol,\MagPot)$. 
\end{theorem}
\begin{proof}
The proof follows similar arguments as in \cite[Lemma~5.9]{DoeH24}, but in contrast to the reference, we do not exploit the previously established error estimate for $\Honecombi{ (\sol - \sol_{\deltaH},\MagPot - \MagPot_{\deltah})}$. Let $(\varphi_{\deltaH},\mathbf{B}_{\deltah}) \in X_{\deltaH} \times \mathbf{X}_{\deltah}$ be arbitrary.  Then by the fact that  $(\sol_{\deltaH},\MagPot_{\deltah}) \in X_{\deltaH} \times \mathbf{X}_{\deltah}$  is a discrete minimizer, we have $\energystab(\sol_{\deltaH},\MagPot_{\deltah})\le \energystab(\varphi_{\deltaH},\mathbf{B}_{\deltah})$ and therefore it is sufficient to estimate $ \energystab(\varphi_{\deltaH},\mathbf{B}_{\deltah}) - \energystab(\sol,\MagPot) $ to prove the result. With the notation 
\begin{equation}
\mathcal{N}(\sol,\MagPot)  \coloneqq 
\tfrac{1}{2} \int_\Omega 
\tfrac{1}{2} \bigl( 1- |\sol|^2 \bigr)^2
+
|\curl \MagPot - \mathbf{H}|^2 - |\curl \MagPot|^2
\dint{x},
\end{equation}
we can write the energy as
	\begin{equation}
	\energystab(\sol,\MagPot) = 
	\tfrac{1}{2} a_{\bfA}( \sol , \sol) 
	+ \tfrac{1}{2} \abil{\MagPot}{\MagPot}
	+\mathcal{N}(\sol,\MagPot) .
	\end{equation}
	With this, we expand the energy error as
\begin{eqnarray*}
\lefteqn{ \energystab(\varphi_{\deltaH},\mathbf{B}_{\deltah}) - \energystab(\sol,\MagPot) }\\
&=& 
\tfrac{1}{2} a_{\bfA_h} ( \varphi_{\deltaH} , \varphi_{\deltaH} )
+ \tfrac{1}{2} \abil{\mathbf{B}_{\deltah}}{\mathbf{B}_{\deltah}}
+\mathcal{N}(\varphi_{\deltaH},\mathbf{B}_{\deltah}) 
-
\tfrac{1}{2} a_{\bfA} ( \sol , \sol) 
- \tfrac{1}{2} \abil{\MagPot}{\MagPot}
-\mathcal{N}(\sol,\MagPot) 
\\
&=& \underbrace{ \tfrac{1}{2} a_{\bfA} ( \varphi_{\deltaH} , \varphi_{\deltaH} )
+ \tfrac{1}{2} \abil{\mathbf{B}_{\deltah}}{\mathbf{B}_{\deltah}}
- \tfrac{1}{2} a_{\bfA} ( \sol , \sol) 
- \tfrac{1}{2} \abil{\MagPot}{\MagPot} }_{=:\,I_1}
\\
&\enspace&+\underbrace{\,\,\tfrac{1}{2} a_{\bfA_h} ( \varphi_{\deltaH} , \varphi_{\deltaH} )
- \tfrac{1}{2}  a_{\bfA} ( \varphi_{\deltaH} , \varphi_{\deltaH} )  }_{=:I_2}
 \,\,+\,\,\underbrace{\mathcal{N}(\varphi_{\deltaH},\mathbf{B}_{\deltah}) 
- \mathcal{N}(\sol,\MagPot)}_{=:I_3}.
\end{eqnarray*}
We now treat these terms separately.\\[0.4em]
(a) With Lemma~\ref{lem:Frechet_der_1} we write
\begin{eqnarray*}
	I_1
	&=&
	\tfrac{1}{2} a_{\bfA} ( \sol- \varphi_{\deltaH} , \sol- \varphi_{\deltaH} )
	+
	 \tfrac{1}{2} \abil{\MagPot - \mathbf{B}_{\deltah}}{\MagPot - \mathbf{B}_{\deltah}} 
	-
	a_{\bfA} ( \sol , \sol- \varphi_{\deltaH} )
	-
	\abil{\MagPot}{\MagPot - \mathbf{B}_{\deltah}} 
	\\
	&\overset{E^{\prime}(\sol,\MagPot)=0}{=}&
	\tfrac{1}{2} a_{\bfA} ( \sol- \varphi_{\deltaH} , \sol- \varphi_{\deltaH} )
	+
	\tfrac{1}{2} \abil{\MagPot - \mathbf{B}_{\deltah}}{\MagPot - \mathbf{B}_{\deltah}} 
	\,\, + \,\, \Real  \int_\Omega (|\sol|^2 - 1) \sol (\sol- \varphi_{\deltaH})^* \dint{x}
\\
&\enspace&\quad+
\Real  \int_\Omega |\sol|^2 \MagPot \cdot (\MagPot - \mathbf{B}_{\deltah})  +
\tfrac{\ci}{\kappa}   \bigl(  \sol^* \nabla \sol \cdot (\MagPot - \mathbf{B}_{\deltah})  \bigr) 
-
\mathbf{H} \cdot \curl (\MagPot - \mathbf{B}_{\deltah})
\dint{x} .
\end{eqnarray*}
(b) Concerning $I_2$ we compute 
\begin{eqnarray*}
\lefteqn{ I_2 \,\, = \,\, \tfrac{1}{2} a_{\bfA_h} ( \varphi_{\deltaH} , \varphi_{\deltaH})
- \tfrac{1}{2} a_{\bfA}( \varphi_{\deltaH} , \varphi_{\deltaH}) 
}
\\
&=&\tfrac{1}{2} \int_\Omega 
|\tfrac{\ci}{\kappa} \nabla \varphi_{\deltaH} +   \mathbf{B}_{\deltah} \varphi_{\deltaH} |^2 
-
|\tfrac{\ci}{\kappa} \nabla \varphi_{\deltaH} +  \MagPot \varphi_{\deltaH} |^2 
\dint{x}
\\
&=&
\Real \int_\Omega  \tfrac{\ci}{\kappa} \nabla \varphi_{\deltaH}  \cdot  (\mathbf{B}_{\deltah} - \MagPot) \varphi_{\deltaH}^{\ast} 
+  \tfrac{1}{2}  |\mathbf{B}_{\deltah}|^2\,|\varphi_{\deltaH} |^2 -   \tfrac{1}{2}  |\MagPot|^2\,|\varphi_{\deltaH} |^2
\dint{x} .
\end{eqnarray*}

(c) We now turn to $I_3$. In the proof of \cite[Lemma~5.9]{DoeH24}  it was shown that for $\errh = \sol - \varphi_{\deltaH}$
\begin{eqnarray*}
 \bigl( 1- |\varphi_{\deltaH}|^2 \bigr)^2 -   \bigl( 1- |\sol|^2 \bigr)^2 =  4(1 - |\sol|^2) \Real(\sol(\sol - \varphi_{\deltaH} )^{\ast}) \,\,+\,\, \mbox{h.o.t}\,(\errh) ,
\end{eqnarray*}
with $\,\,\mbox{h.o.t}\,(\errh) \coloneqq 2(|\sol|^2-1)|\errh|^2 +(|\errh|^2- \Real(u( \errh )^{\ast}))^2\,\,$ being a higher order term. 
 With this we obtain
\begin{eqnarray*}
\lefteqn{ I_3 \,\,=\,\,  \tfrac{1}{2} \int_\Omega 
\tfrac{1}{2} \bigl( 1- |\varphi_{\deltaH}|^2 \bigr)^2 - \tfrac{1}{2} \bigl( 1- |\sol|^2 \bigr)^2 
+  2 \, \mathbf{H} \cdot \curl (\MagPot - \mathbf{B}_{\deltah} )
\dint{x} } \\
&=&  \int_\Omega  (1 - |\sol|^2) \Real(\sol(\sol - \varphi_{\deltaH} )^{\ast}) + \mathbf{H} \cdot \curl (\MagPot - \mathbf{B}_{\deltah} ) \dint{x} 
+ \int_{\Omega} \mbox{h.o.t}\,(\errh) \dint{x} .
\end{eqnarray*}
We are ready to sum up $I_1$, $I_2$ and $I_3$. By noting that
\begin{eqnarray*}
\lefteqn{\tfrac{1}{2} |\mathbf{B}_{\deltah}|^2 |\varphi_{\deltaH}|^2
- \tfrac{1}{2} |\MagPot|^2 |\varphi_{\deltaH}|^2 + |\MagPot|^2 |\sol|^2 - ( \mathbf{B}_{\deltah} \cdot \MagPot ) |\sol|^2}\\
&=& \tfrac{1}{2} (|\mathbf{B}_{\deltah}|^2-|\MagPot|^2)(|\varphi_{\deltaH}|^2-|\sol|^2) +  \tfrac{1}{2} |\mathbf{B}_{\deltah} - \MagPot|^2 |\sol|^2
\end{eqnarray*}
we obtain
\begin{eqnarray*}
\lefteqn{ \energystab(\varphi_{\deltaH},\mathbf{B}_{\deltah}) - \energystab(\sol,\MagPot) 
\,\,=\,\,
\tfrac{1}{2} a_{\bfA} ( \sol- \varphi_{\deltaH}, \sol- \varphi_{\deltaH} )
	+
	\tfrac{1}{2} \abil{\MagPot - \mathbf{B}_{\deltah}}{\MagPot - \mathbf{B}_{\deltah}}  }\\
&\enspace&\quad+	
 \tfrac{1}{2} \int_\Omega 
 (|\mathbf{B}_{\deltah}|^2-|\MagPot|^2)(|\varphi_{\deltaH}|^2-|\sol|^2) + |\mathbf{B}_{\deltah} - \MagPot|^2 |\sol|^2
\dint{x}
\\
&\enspace&\quad+
\Real  \int_\Omega 
\tfrac{\ci}{\kappa}   \bigl(  (\sol^* -\varphi_{\deltaH}^*) \nabla \sol  
  + \varphi_{\deltaH}^* ( \nabla \sol  -  \nabla \varphi_{\deltaH})  \bigr)  \cdot (\MagPot - \mathbf{B}_{\deltah})  \dint{x} + \int_{\Omega} \mbox{h.o.t}\,(\errh)  \dint{x} \\
  &\lesssim& \| \sol- \varphi_{\deltaH} \|_{\HonekappaSpace}^2 + \| \mathbf{B}_{\deltah} - \MagPot\|_{H^1}^2
  + 
    \| \MagPot - \mathbf{B}_{\deltah} \|_{L^6}  \| |\MagPot| + |\mathbf{B}_{\deltah}| \|_{L^6} 
   \|  \sol -\varphi_{\deltaH} \|_{L^2}  \| | \sol | + |\varphi_{\deltaH} | \|_{L^6} 
  \\
 &\enspace& 
 + \| \MagPot - \mathbf{B}_{\deltah} \|_{L^4} \| \tfrac{1}{\kappa} \nabla \sol \|_{L^4} \| \sol -\varphi_{\deltaH} \|_{L^2} + \| \tfrac{1}{\kappa} \nabla (\sol - \varphi_{\deltaH}) \|_{L^2} \| \varphi_{\deltaH}\|_{L^4} \| \MagPot - \mathbf{B}_{\deltah} \|_{L^4}  \\
 &\enspace&  + \| \sol - \varphi_{\deltaH} \|_{L^2}^2 + \| \sol - \varphi_{\deltaH} \|_{L^4}^4 \\
 &\lesssim&  \| \sol- \varphi_{\deltaH} \|_{\HonekappaSpace}^2 + \| \mathbf{B}_{\deltah} - \MagPot\|_{H^1}^2 + \| \varphi_{\deltaH}\|_{L^6}^2 \| \sol- \varphi_{\deltaH} \|_{L^2}^2 +  \| \sol- \varphi_{\deltaH} \|_{L^4}^4 \\
&\lesssim& 
  \| \sol- \varphi_{\deltaH} \|_{\HonekappaSpace}^2 + \| \mathbf{B}_{\deltah} - \MagPot\|_{H^1}^2 + \left( \| \sol - \varphi_{\deltaH} \|_{L^6} + \| \sol \|_{L^6}  \right)^2 \| \sol- \varphi_{\deltaH} \|_{L^2}^2 +  \| \sol- \varphi_{\deltaH} \|_{L^4}^4
 \\
&\lesssim& 
  \| \sol- \varphi_{\deltaH} \|_{\HonekappaSpace}^2 + \| \mathbf{B}_{\deltah} - \MagPot\|_{H^1}^2 + \| \sol - \varphi_{\deltaH} \|_{L^6} \| \sol- \varphi_{\deltaH} \|_{L^2}^2 +  \| \sol- \varphi_{\deltaH} \|_{L^4}^4.
\end{eqnarray*}
\end{proof}

\section{Error analysis in the LOD space}
\label{sec:error_analysis}
We are prepared now for the error analysis in the LOD space and for proving the main results stated in Section \ref{sec:space_main}. For that, we shall apply the results from Section \ref{sec:abstract_error_analysis} with the choice $(X_{\deltaH},\mathbf{X}_{\deltah})=(\VShLOD,\VSAhzero{k})$. In the first step, we need an auxiliary result, which is an inverse inequality in $\VShLOD$. The proof is given in Appendix~\ref{sec:app_computations}.
\begin{lemma}[Inverse inequality in $\VShLOD$]
\label{lem:inverse-inequality-LOD}
Assume $\deltaH \lesssim \kappa^{-1}$, then it holds
\begin{eqnarray*}
\| \nabla \varphi_{H}^{\LOD} \|_{L^2} \lesssim \tfrac{1}{\deltaH} \| \varphi_{H}^{\LOD} \|_{L^2} \qquad  \mbox{for all } \varphi_{H}^{\LOD} \in \VShLOD.
\end{eqnarray*}
\end{lemma}
An important assumption in the abstract error analysis was the $H^1_{\kappa}$-stability of the $L^2$-projection onto the discrete space which is now set to be the LOD space. Using the above inverse inequality, the following result can be established which verifies that the assumptions in 
Lemma~\ref{lemma:ritz-projections:abstract-new} 
are indeed fulfilled.
\begin{lemma}[$H^1_{\kappa}$-stability of the $L^2$-projection on $\VShLOD$]
\label{lemma-properties-LOD-space}
Assume that $\deltaH \lesssim \kappa^{-1}$ and let $\LtwoprojLOD : H^1(\Omega) \rightarrow \VShLOD$ denote the $L^2$-projection. It holds
\begin{eqnarray*}
\| \LtwoprojLOD \varphi \|_{H^1_{\kappa}} &\lesssim& \|  \varphi \|_{H^1_{\kappa}}  \qquad \mbox{for all } \varphi \in H^1(\Omega).
\end{eqnarray*}
Furthermore, it holds
\begin{eqnarray*}
\| u - \LtwoprojLOD u \|_{L^2} &\lesssim& \kappa H \, \| u \|_{H^1_{\kappa}}
\end{eqnarray*}

\end{lemma}

\begin{proof}
Consider $(1 - \mathcal{C})\LtwoprojFEM\varphi \in \VShLOD$, then the $H^1_{\kappa}$-stability of $\mathcal{C}$ and $\LtwoprojFEM$ (on quasi-uniform meshes) imply that 
\begin{eqnarray}
\label{H1kappa-stab-corr-L2proj}
\| (1 - \mathcal{C})\LtwoprojFEM \varphi\|_{H^1_{\kappa}} \lesssim  \| \varphi \|_{H^1_{\kappa}}. 
\end{eqnarray}
Furthermore, as $(1 - \mathcal{C})\LtwoprojFEM \varphi - \varphi \in W$, we also have
\begin{eqnarray}
\label{est-LOD-func-L2-to-H1kap}
\hspace{30pt}\| (1 - \mathcal{C})\LtwoprojFEM \varphi - \varphi \|_{L^2} 
\,=\,  \| (1-\LtwoprojFEM) \left( (1 - \mathcal{C})\LtwoprojFEM \varphi - \varphi \right) \|_{L^2}
\,\lesssim\, \kappa \, H \|  \varphi  \|_{H^1_{\kappa}}.
\end{eqnarray}
We obtain with the inverse inequality from Lemma \ref{lem:inverse-inequality-LOD}
\begin{eqnarray*}
\| \LtwoprojLOD \varphi \|_{\HonekappaSpace} &\le&  \| \varphi - (1 - \mathcal{C})\LtwoprojFEM \varphi \|_{\HonekappaSpace} + \| (1 - \mathcal{C})\LtwoprojFEM \varphi - \LtwoprojLOD \varphi \|_{\HonekappaSpace} + \| \varphi \|_{\HonekappaSpace} \\
&\lesssim& \| \varphi \|_{\HonekappaSpace}  + \tfrac{1}{\kappa \deltaH } \| (1 - \mathcal{C})\LtwoprojFEM \varphi - \LtwoprojLOD \varphi \|_{L^2}  \\
&\lesssim& \| v \|_{\HonekappaSpace}  + \tfrac{1}{\kappa \deltaH}  \| (1 - \mathcal{C})\LtwoprojFEM \varphi - \varphi \|_{L^2} + \tfrac{1}{\kappa \deltaH}  \| \LtwoprojLOD \varphi - \varphi \|_{L^2} \\
&\le& \| \varphi \|_{\HonekappaSpace}  + \tfrac{2}{\kappa \deltaH}  \| (1 - \mathcal{C})\LtwoprojFEM \varphi - \varphi \|_{L^2}  \,\, \overset{\eqref{est-LOD-func-L2-to-H1kap}}{\lesssim} \,\,  \| \varphi \|_{\HonekappaSpace},
\end{eqnarray*}
The second estimate readily follows with \eqref{est-LOD-func-L2-to-H1kap} for $\varphi = u$.
\end{proof}
Next, we quantify the best-approximation error in the LOD space.
\begin{lemma}
\label{lemma:bestapprox-LOD}
$(a)$ Let $(\sol,\MagPot) \in  H^1(\Omega) \times \HonenSpace(\Omega)$ denote an arbitrary minimizer of \eqref{minimization-problem}. If $H\lesssim \kappa^{-1}$, then it holds
\begin{eqnarray}
\label{bestapprox-LOD-est-1}
\lefteqn{ \inf_{\varphi_{H}^{\LOD} \in \VShLOD } \| u  - \varphi_{H}^{\LOD}\|_{\HonekappaSpace} }\\
\nonumber&\lesssim&
(\kappa H)^3 \, + \, \kappa H \, \inf_{\varphi_H \in \VSh} \| (\tfrac{\ci}{\kappa}  \nabla u + (\bfAapp\hspace{-1pt}+\hspace{-1pt}\bfA)u ) \cdot (\bfAapp-\bfA) - \varphi_H \|_{L^2}.
\end{eqnarray}
$(b)$ If $\bfAapp$ fulfills $\| \bfAapp \|_{L^\infty} + \| \bfAapp \|_{W^{1,3}} \lesssim 1$, then the estimate can be further bounded as 
\begin{eqnarray*}
\inf_{\varphi_{H}^{\LOD} \in \VShLOD } \| u  - \varphi_{H}^{\LOD}\|_{\HonekappaSpace} 
&\lesssim& \,(\kappa H)^2\left( (\kappa H) +  \|\bfAapp-\bfA  \|_{H^1}  \right) .
\end{eqnarray*}
$(c)$ Finally, if $\bfAapp$ admits the same regularity $\kappa$-dependent stability bounds as $\bfA$, i.e., $\| \bfAapp \|_{H^2} \lesssim 1$ and $\| \bfAapp \|_{H^3} \lesssim \kappa$, then optimal order convergence is obtained with
\begin{eqnarray}
\label{bestapprox-LOD-est-2}
\inf_{\varphi_{H}^{\LOD} \in \VShLOD } \| u  - \varphi_{H}^{\LOD}\|_{\HonekappaSpace} &\lesssim& (\kappa H)^3 .
\end{eqnarray}
\end{lemma}
Before proving the result, we note that estimate \eqref{bestapprox-LOD-est-1} only exploits $H^2$-regularity of $u$, whereas the refined estimate exploits both the $H^3$-regularity of $u$ and $\bfA$, unless $\bfAapp=\bfA$. Recalling that $H^3$-regularity can only be guaranteed on box-shaped domains, we see that \eqref{bestapprox-LOD-est-2} might not hold on general domains, unless the distance between $\bfAapp$ and $\bfA$ is sufficiently small. Hence, to keep the second error contribution in \eqref{bestapprox-LOD-est-1} small, we practically want to select $\bfAapp$ such that $\| \bfA - \bfAapp \|_{L^2}$ is small.
\begin{proof}[Proof of Lemma \ref{lemma:bestapprox-LOD}]
(a) We consider the function $\hat{u}_H^\LOD := (1 - \mathcal{C})\LtwoprojFEM(u) \in \VShLOD$ for which we have  $\hat{u}_H^\LOD-u \in W$. Since the error is an element of the detail space, we can apply Lemma \ref{lemma:coercivity-abil-on-W} and obtain
\begin{eqnarray*}
\lefteqn{ \tfrac{1}{2} \| u - \hat{u}_H^\LOD  \|_{\HonekappaSpace}^2
\,\,\,\le\,\,\,  a_{\bfAapp}( u - \hat{u}_H^\LOD  , u - \hat{u}_H^\LOD ) 
\,\,\, \overset{\eqref{eqn:def-ideal-correctors}}{=} \,\,\, a_{\bfAapp}( u  , u - \hat{u}_H^\LOD ) } \\
&=& a_{\bfA} ( u  , u - \hat{u}_H^\LOD ) \, + \, (a_{\bfAapp}-a_{\bfA}) ( u  , u - \hat{u}_H^\LOD ) \\
&=& ((1-|u|^2) u , u - \hat{u}_H^\LOD )_{L^2} \,+\, 
(  \tfrac{\ci}{\kappa}  \nabla u,   (\bfAapp-\bfA) \, (u - \hat{u}_H^\LOD)  )_{L^2}\\
&\enspace&\, + \, ( (\bfAapp-\bfA) u  ,  \tfrac{\ci}{\kappa}  \nabla (u - \hat{u}_H^\LOD) )_{L^2} +  ( (|\bfAapp|^2-|\bfA|^2) u  , u - \hat{u}_H^\LOD  )_{L^2} \\
&=& ((1-|u|^2 + |\bfAapp|^2-|\bfA|^2) u , u - \hat{u}_H^\LOD )_{L^2} \,+\, 
 (  \tfrac{2\ci}{\kappa}  \nabla u \cdot  (\bfAapp-\bfA) ,  (u - \hat{u}_H^\LOD)  )_{L^2},
\end{eqnarray*}
where we used in the last step that \,\,$\div \bfAapp  = \div \bfA = 0$. By defining for brevity 
$$
f_{\ast} := (\tfrac{2\ci}{\kappa}  \nabla u + u(\bfAapp + \bfA) ) \cdot (\bfAapp-\bfA),
$$
we conclude with the properties of the $L^2$-projection $\LtwoprojFEM$ that
\begin{eqnarray*}
\lefteqn{ \tfrac{1}{2} \| u - \hat{u}_H^\LOD  \|_{\HonekappaSpace}^2 } \\
&\le& ( \, (1-\LtwoprojFEM) \hspace{-2pt}\left( (1-|u|^2) u \right)  +  (1-\LtwoprojFEM)f_{\ast}  \,,\, (1-\LtwoprojFEM)( u - \hat{u}_H^\LOD) )_{L^2}  \\
&\overset{\eqref{L2-proj-est}}{\lesssim}& \,\kappa\, H \, \left( \, \| (1-\LtwoprojFEM) \hspace{-2pt}\left( (1-|u|^2) u \right) \hspace{-1pt}\|_{L^2} + \| (1-\LtwoprojFEM)f_{\ast}\|_{L^2} \, \right)  \| u - \hat{u}_H^\LOD  \|_{\HonekappaSpace}.
\end{eqnarray*}
Exploiting the $H^2$-regularity of $u$ and the corresponding $\kappa$-dependent stability bounds yields $(1-|u|^2) u \in H^2(\Omega)$ with $| (1-|u|^2) u |_{H^2(\Omega)} \lesssim \kappa^2$. Consequently, we have
\begin{eqnarray*}
\| u - \hat{u}_H^\LOD  \|_{\HonekappaSpace} 
&\lesssim& \,\kappa H \left( (\kappa H)^2 + \| (1-\LtwoprojFEM)f_{\ast}\|_{L^2} \right) .
\end{eqnarray*}
This yields the first estimate \eqref{bestapprox-LOD-est-1}.

(b) \& (c)
For the refined estimates
we use $ \| (1-\LtwoprojFEM)f_{\ast}\|_{L^2} \lesssim  H \, | f_{\ast} |_{H^1}$ and $ \| (1-\LtwoprojFEM)f_{\ast}\|_{L^2} \lesssim H^2 | f_{\ast} |_{H^2}$ respectively, where it remains to bound $f_{\ast}=\tfrac{\ci}{\kappa}  \nabla u \cdot (\bfAapp-\bfA) + u(\bfAapp + \bfA) \cdot (\bfAapp-\bfA)$ in the $H^1$- and $H^2$-semi-norms.
For brevity, we use the notation 
$$
|D^{(m)}v|^2 := \underset{|\boldsymbol{\alpha}|=m}{\sum_{\boldsymbol{\alpha} \in \mathbb{N}_0^d }} |\partial^{\boldsymbol{\alpha}} v|^2
$$
to denote the pointwise norm of all $m$'th partial derivatives of some given function $v \in H^m(\Omega)$. 
With this, we obtain for $\ell=1,2$ that
\begin{eqnarray*}
 | \tfrac{\ci}{\kappa}  \nabla u  \cdot (\bfAapp-\bfA) |_{H^\ell}
\,\,\,\lesssim\,\,\, \tfrac{1}{\kappa} \sum_{i=0}^\ell \| |D^{(1+i)} u| \, |D^{(\ell-i)} (\bfAapp-\bfA)|\, \|_{L^2}
\end{eqnarray*}
and
\begin{eqnarray*}
 | u(\bfAapp + \bfA) \cdot (\bfAapp-\bfA) |_{H^\ell}
\,\,\,\lesssim\,\,\, \sum_{i=0}^\ell \sum_{j=0}^i \| |D^{(\ell-i)} u| \, |D^{(i-j)} (\bfAapp+\bfA)| \, |D^{(j)} (\bfAapp-\bfA)|\, \|_{L^2}.
\end{eqnarray*}
For $\ell=1$ we use $\| \bfAapp \|_{L^\infty} + \| \bfAapp \|_{W^{1,3}} \lesssim 1$ to get
\begin{eqnarray*}
\| \nabla f_{\ast}\|_{L^2} 
&\lesssim& 
\kappainv \| u  \|_{W^{2,4}}  \| (\bfAapp-\bfA) \|_{L^4} 
+
\kappainv \| \nabla u \|_{L^\infty}  \| \nabla   (\bfAapp-\bfA) \|_{L^2} 
\\
&\enspace&\quad
+
\| \nabla u   \|_{L^4} 
 \| \bfAapp + \bfA  \|_{L^4} 
  \| \bfAapp-\bfA   \|_{L^4} 
+
\|  u   \|_{L^\infty} 
 \| \nabla (\bfAapp + \bfA)  \|_{L^3} 
  \| \bfAapp-\bfA  \|_{L^6} 
\\
&\enspace&\quad
  +
\|  u   \|_{L^\infty} 
 \| (\bfAapp + \bfA)  \|_{L^\infty} 
  \|\nabla (\bfAapp-\bfA)   \|_{L^2} 
\\
&\lesssim& 
\kappa   \|\bfAapp-\bfA  \|_{H^1} 
\end{eqnarray*}
and thus $\| u - \hat{u}_H^\LOD  \|_{\HonekappaSpace}  \,\,\,\lesssim \,\,\,(\kappa H)^2\left( (\kappa H) +  \|\bfAapp-\bfA  \|_{H^1}  \right)$.

For $\ell=2$ we use the estimates $\| \bfA \|_{H^2} \lesssim 1$ and $\| \bfA \|_{H^3} \lesssim \kappa$ to see that
\begin{eqnarray*}
\lefteqn{ | \tfrac{\ci}{\kappa}  \nabla u  \cdot (\bfAapp-\bfA) |_{H^2}
\,\,\,\lesssim\,\,\, \tfrac{1}{\kappa} \sum_{i=0}^2 \| |D^{(1+i)} u| \, |D^{(2-i)} (\bfAapp-\bfA)|\, \|_{L^2} }\\
&\lesssim&  \tfrac{1}{\kappa} \left( \|  u\|_{W^{1,4}}  \| \bfAapp-\bfA \|_{W^{2,4}} +  \|  u\|_{W^{2,4}}  \| \bfAapp-\bfA \|_{W^{1,4}} +  \|  u\|_{H^{3}}  \| \bfAapp-\bfA \|_{L^{\infty}}   \right) \\
&\lesssim&  \| \bfAapp-\bfA \|_{W^{2,4}} +  \kappa \| \bfAapp-\bfA \|_{W^{1,4}} + \kappa^2 \| \bfAapp-\bfA \|_{L^{\infty}}  \,\,\,\lesssim\,\,\, \kappa^2.
\end{eqnarray*}
In a similar fashion, we also have
\begin{eqnarray*}
 | u(\bfAapp + \bfA) \cdot (\bfAapp-\bfA) |_{H^2}
\,\,\,\lesssim\,\,\, \sum_{i=0}^2\sum_{j=0}^i \| |D^{(2-i)} u| \, |D^{(i-j)} (\bfAapp+\bfA)| \, |D^{(j)} (\bfAapp-\bfA)|\, \|_{L^2}  \,\,\,\lesssim\,\,\, \kappa^2.
\end{eqnarray*}
We conclude $|f_{\ast} |_{H^2}\lesssim \kappa^2$. This proves \eqref{bestapprox-LOD-est-2}.
\end{proof}
With the previous lemma, we can directly quantify the error in energy.
\begin{conclusion}
\label{conclusion:energy-error}
Assume that $\bfAapp$ admit the same regularity and $\kappa$-dependent stability bounds as $\bfA$ and that $\kappa H \lesssim 1$. 
Let $(\sol,\MagPot) \in  H^1(\Omega) \times \HonenSpace(\Omega) $ be an arbitrary minimizer of \eqref{minimization-problem} and $(\sol_{H}^{\LOD},\mathbf{A}_{h,k}^{\mbox{\tiny FEM}})  \in \VShLOD \times \VSAhzero{k}$ a discrete minimizer fulfilling \eqref{eq:disc_approx_LOD} for $k=1,2$. Then, it holds
\begin{eqnarray*}
 \energystab(\sol_{H}^{\LOD},\mathbf{A}_{h,k}^{\mbox{\tiny FEM}}) - \energystab(\sol,\MagPot) 
		&\lesssim& (\kappa \deltaH)^6  + \kappa^{2k-2} \deltah^{2k} 
	\end{eqnarray*}
with constants independent of $\kappa$ and $(\deltaH,\deltah)$.
\end{conclusion}

\begin{proof}
We employ Theorem \ref{thm:energy_bound} and use Lemma~\ref{lemma:bestapprox-LOD} as well as \eqref{eq:approx_FEM_Lag} to bound the first term. Thus, it remains to bound the remaining terms in $L^p$.
Since the $L^2$-projection $\LtwoprojFEM$ is $L^p$-stable on quasi-uniform meshes for any $p\in [1,\infty]$ (cf. \cite{DouglasDupontWahlbin74,DST21}), i.e.
\begin{eqnarray*}
\|  \LtwoprojFEM \varphi \|_{L^p} \,\lesssim\, \|  \varphi \|_{L^p} \qquad \mbox{for all } \varphi \in L^p(\Omega),
\end{eqnarray*}
we also have the best-approximation property in $L^p$. Together with the standard approximation results for finite elements (cf. \cite[Sec. 4.4]{BrennerScott}), this yields
\begin{eqnarray}
\label{proj-error-Lp}
\| \varphi - \LtwoprojFEM \varphi \|_{L^p} \, \lesssim (\kappa H)^{\ell} \, \| \varphi \|_{H^{\ell}_{\kappa}} \qquad 
\mbox{for all } \varphi \in L^p(\Omega) \cap H^{\ell}(\Omega),
\end{eqnarray}
where $\ell=1,2$. For the trial function $(1 - \mathcal{C})\LtwoprojFEM u \in \VShLOD$ this implies for any $p\in [2,6]$
\begin{eqnarray*}
\lefteqn{ \| u - (1 - \mathcal{C})\LtwoprojFEM u  \|_{L^p}
\,\,\,\le\,\,\, \| u - \LtwoprojFEM u \|_{L^p} + \|  (1-\LtwoprojFEM) \mathcal{C}\LtwoprojFEM u  \|_{L^p} }\\
 &\overset{\eqref{proj-error-Lp}}{\lesssim}&  (\kappa H)^2 \| u\|_{H^2_{\kappa}} +   \kappa H \| \mathcal{C}\LtwoprojFEM u  \|_{H^1_{\kappa}} \\
 &\le&  (\kappa H)^2 \| u\|_{H^2_{\kappa}} +   \kappa H \| u - (1 - \mathcal{C})\LtwoprojFEM u  \|_{H^1_{\kappa}} 
 +   \kappa H \| (1-\LtwoprojFEM) u   \|_{H^1_{\kappa}} \\
 &\overset{\eqref{bestapprox-LOD-est-2}}{\lesssim}&  (\kappa H)^2 +   (\kappa H)^4 \,\,\, \lesssim (\kappa H)^2,
 \end{eqnarray*}
where the penultimate step also exploited the $H^1_{\kappa}$-stability of $\LtwoprojFEM$. Using these estimates for $p=2$, $p=4$, and $p=6$ gives the claim.
\end{proof}

We now turn to the error estimates for the discrete minimizers. The estimates in Lemma \ref{lemma:bestapprox-LOD} are sufficient to control the first term in Conclusion \ref{concl:Honecombi_bound}. However, for an $H^1$-error estimate we also need to study the approximation properties of the LOD space with respect to the solution $(z,\mathbf{Z}) \in \Honeperp \times \HonenSpace(\Omega)$ to the auxiliary problem \eqref{def-z-proj-error}. This will be the main task in the proof to establish the following main result.

\begin{proposition}
\label{proposition-H1-est}
Let Assumption \ref{ass:local_uniq} hold and let  $\bfAapp$ admit the same regularity and $\kappa$-dependent stability bounds as $\bfA$.
If $(\sol,\MagPot) \in  H^1(\Omega) \times \HonenSpace(\Omega) $ is a minimizer of \eqref{minimization-problem} and if the mesh size $(H,h)$ is sufficiently small, in particular such that
\begin{eqnarray}
\label{resolution-condition-1}
((\kappa H)^2 + h) \,\,
\kappa^{\varepsilon} \, \Csol \,\,\, \lesssim \,\,\, 1
\end{eqnarray}
then it holds
\begin{eqnarray*}
\Honecombi{ (u , \bfA ) - \textup{\bf R}_{H,h} (u , \bfA) } 
&\lesssim&
\inf_{(\varphi_H^{\LOD}, \mathbf{B}_h) \in \VShLOD \times \VSAhzero{k}}
\Honecombi{ (u -\varphi_H^{\LOD}, \bfA -\mathbf{B}_h) }
\end{eqnarray*}
for the $E^{\prime\prime}(u,\bfA)$-Ritz-projection $\textup{\bf R}_{H,h}$ in Definition \ref{definition:Eprimeprime-Ritz}.

Conversely, if $(\sol_{H}^{\LOD},\mathbf{A}_{h,k}^{\mbox{\tiny FEM}})  \in \VShLOD \times \VSAhzero{k}$ is a discrete minimizer fulfilling \eqref{eq:disc_approx_LOD} for $k=1,2$, then, for all sufficiently small mesh sizes $(\deltaH,\deltah)$ consistent with \eqref{resolution-condition-1}, there exists a minimizer $(\sol,\MagPot) \in  H^1(\Omega) \times \HonenSpace(\Omega)$ of \eqref{minimization-problem} such that $\sol_{H}^{\LOD} \in \Honeperp$ fulfills 
\begin{eqnarray*}
 \| (\sol - \sol_{H}^{\LOD} , \MagPot - \mathbf{A}_{h,k}^{\mbox{\tiny FEM}}) \|_{\HonekappaSpace \times \mathbf{H}^1} 
&\lesssim& 
\kappa^3 \deltaH^3  + \kappa^{k-1} \deltah^k.
\end{eqnarray*}
\end{proposition}

\begin{proof}
By the definition of $(z,\mathbf{Z}) \in \Honeperp \times \HonenSpace(\Omega)$, Lemma \ref{lem:inf_sup_cond_exact_bounds} directly yields
\begin{eqnarray}
\label{H1est-z}
		\Honecombi{(z, \mathbf{Z} )}\,\,\lesssim\,\, \Csol \, \norm{L^2 \times {\mathbf{L}}^2}{(u - \textup{R}_{H}u  , \bfA - \textup{R}_{h} \bfA ) }.
\end{eqnarray}
Furthermore, we recall from Proposition \ref{prop:H2_reg_E_primeprime} that
\begin{eqnarray}
\label{H2est-z}
\kappa^{\varepsilon} \norm{\HtwokappaSpace}{z}
+ \norm{H^2}{\mathbf{Z}} &\lesssim&
		\kappa^{\varepsilon} \, \Csol \,
		\norm{L^2 \times {\mathbf{L}}^2}{(u - \textup{R}_{H}u  , \bfA - \textup{R}_{h} \bfA ) }.
\end{eqnarray}
Next, we test in the equation for $(z,\mathbf{Z})$, i.e. in \eqref{def-z-proj-error}, with a test function $(\psi, \mathbf{0} ) \in \Honeperp \times \HonenSpace(\Omega)$. Using furthermore the representation of $\energystab''(\sol,\MagPot)$ in Lemma \ref{lem:energy2prim_perturbation_v2} we obtain that  $(z,\mathbf{Z})$ fulfills
\begin{eqnarray*}
\abilmag{z}{\psi} \,\,+\,\,
 (  \bigl( |\sol|^2\hspace{-2pt}-\hspace{-2pt} 1 \bigr)  z +  2 \Real( \sol z^*) \sol  , \psi )_{L^2}
\,\,+\,\, 
		g \bigl( (z,\mathbf{Z}) , (\psi,\mathbf{0}) \bigr)
&=& ( u \hspace{-2pt}-\hspace{-2pt} \textup{R}_{H} u ,  \psi  )_{L^2}
\end{eqnarray*}
where $g$ is given by
\begin{eqnarray*}
 g \bigl( (z,\mathbf{Z}) , (\psi,\mathbf{0}) \bigr)
&=&
\Real \int_\Omega 
2 \, \sol (\MagPot \cdot \mathbf{Z} ) \, \psi^* 
+
\tfrac{\ci}{\kappa}   \bigl(  
\sol^* \nabla \psi
+
 \psi^* \nabla \sol \bigr) 
\cdot \mathbf{Z}  
\dint{x} \\
&\overset{\mathbf{Z} \cdot \mathbf{n}\vert_{\partial \Omega} = 0}{=}& \Real \int_\Omega 
2 \, \sol (\MagPot \cdot \mathbf{Z} ) \, \psi^* 
+
2 \tfrac{\ci}{\kappa} 
 ( \nabla \sol 
\cdot \mathbf{Z}  ) \psi^*
+ \tfrac{\ci}{\kappa} \sol \, \div \mathbf{Z} \, \psi^*
\dint{x}.
\end{eqnarray*}
By defining
\begin{eqnarray*}
g_{\ast} &:=& - \bigl( |\sol|^2\hspace{-2pt}-\hspace{-2pt} 1 \bigr)  z -  2 \Real( \sol z^*) \sol 
- 2 \, \sol (\MagPot \cdot \mathbf{Z} ) \, 
- 2 \tfrac{\ci}{\kappa} ( \nabla \sol  \cdot \mathbf{Z}  ) 
- \tfrac{\ci}{\kappa} \sol \, \div \mathbf{Z} 
\,\, + \,\,  ( u \hspace{-2pt}-\hspace{-2pt} \textup{R}_{H} u),
\end{eqnarray*}
we see that $z \in \Honeperp$ fulfills $\abilmag{z}{\psi} = (g_{\ast},\psi)_{L^2}$ for all $\psi \in \Honeperp$. By the stability estimates for $u$ and $\bfA$, the source term fulfills 
\begin{eqnarray*}
\| g_{\ast} \|_{L^2} \,\,\,\lesssim\,\,\, \Honecombi{(z, \mathbf{Z} )} + \|  u \hspace{-2pt}-\hspace{-2pt} \textup{R}_{H} u \|_{L^2} \,\,\,\lesssim\,\,\, \Csol \, \norm{L^2 \times {\mathbf{L}}^2}{(u - \textup{R}_{H}u  , \bfA - \textup{R}_{h} \bfA ) }.
\end{eqnarray*}
Furthermore, we have with the stability and regularity bounds for $u$ and $\bfA$ that
\begin{eqnarray*}
\lefteqn{ \| \nabla g_{\ast} \|_{L^2} }\\
&\lesssim&  \| \nabla (u - \textup{R}_{H}u) \|_{L^2} +  \| \nabla u \|_{L^{\infty}} \| u \|_{L^{\infty}} \|z \|_{L^2} + (1 + \| u\|_{L^{\infty}}^2) \| \nabla z \|_{L^2} \\
&\enspace& + \| \nabla u \|_{L^{4}} \| \bfA \|_{L^{\infty}} \| \mathbf{Z} \|_{L^4}
+ \| u \|_{L^{\infty}} \| \nabla \bfA \|_{L^{4}} \| \mathbf{Z} \|_{L^4}
+  \| u \|_{L^{\infty}} \| \bfA \|_{L^{\infty}} \| \nabla \mathbf{Z} \|_{L^2} \\
&\enspace& \tfrac{1}{\kappa} \| \sol \|_{W^{2,4}}  \| \mathbf{Z} \|_{L^{4}} +  \tfrac{1}{\kappa} \| \nabla \sol \|_{L^{\infty}}  \| \nabla \mathbf{Z} \|_{L^{2}}
+ \tfrac{1}{\kappa} \| \sol \|_{L^{\infty}} \,  \| \mathbf{Z} \|_{\mathbf{H}^2} +  \tfrac{1}{\kappa} \| \nabla \sol \|_{L^{\infty}} \,  \| \div \mathbf{Z} \|_{L^2} \\
&\lesssim&  \| \nabla (u - \textup{R}_{H}u) \|_{L^2}  + 
\kappa^{1+\varepsilon} \|z \|_{L^2} + \kappa \| z \|_{H^1_{\kappa}} + \kappa \| \mathbf{Z} \|_{\mathbf{H}^1} 
+ \underbrace{ 
\kappa^{\varepsilon}}_{\le \kappa} \|  \mathbf{Z} \|_{\mathbf{H}^{1}}
+ \tfrac{1}{\kappa}  \| \mathbf{Z} \|_{\mathbf{H}^2}.
\end{eqnarray*}
Hence with \eqref{H1est-z} and \eqref{H2est-z}
\begin{eqnarray}
\label{gradgstar-est}
\tfrac{1}{\kappa} \| \nabla g_{\ast} \|_{L^2} %
&\lesssim& \tfrac{1}{\kappa}  \| \nabla (u - \textup{R}_{H}u) \|_{L^2} +  
\kappa^{\varepsilon} \, \|z \|_{L^2} + \| z \|_{H^1_{\kappa}} + \| \mathbf{Z} \|_{\mathbf{H}^1} 
+ \tfrac{1}{\kappa^2}  \| \mathbf{Z} \|_{\mathbf{H}^2} \\
\nonumber&\lesssim& \| u - \textup{R}_{H}u \|_{H^1_{\kappa}} + 
\kappa^{\varepsilon} \, \Csol \, \norm{L^2 \times {\mathbf{L}}^2}{(u - \textup{R}_{H}u  , \bfA - \textup{R}_{h} \bfA ) }.
\end{eqnarray}
Before we can proceed, we have to take care of the fact that the test functions for defining $z$ only involve $\Honeperp$. For that reason, we write an arbitrary $\psi \in H^1(\Omega)$ uniquely as $\psi = \psi^{\perp} + \tfrac{(\ci u, \psi)_{L^2}}{(\ci u,\ci u)_{L^2}} \ci u$ for $\psi^{\perp} \in \Honeperp$ and hence
\begin{eqnarray}
\label{new-def-z} \hspace{40pt}
\abilmag{z}{\psi} = (g_{\ast},\psi^{\perp})_{L^2} +  \tfrac{\abilmag{z}{\ci u}}{(\ci u,\ci u)_{L^2}} (\ci u,\psi)_{L^2} 
= (g_{\ast},\psi)_{L^2}+ \underbrace{\left(
 \tfrac{\abilmag{z}{\ci u} - (g_{\ast}, \ci u )_{L^2}}{\| u\|^2_{L^2}}  \right)}_{=:\zeta} (\ci u,\psi)_{L^2}.
\end{eqnarray}
With $ \| u \|_{H^1_{\kappa}} \lesssim \| u\|_{L^2}$ we have
\begin{eqnarray*}
|\zeta| = \left|  \tfrac{\abilmag{z}{\ci u} - (g_{\ast}, \ci u )_{L^2}}{\| u\|^2_{L^2}} \right| 
\lesssim  \tfrac{\| z \|_{H^1_{\kappa}} + \| g_{\ast} \|_{L^2} }{ \| u\|_{L^2} }
\lesssim \| u\|_{L^2}^{-1} \, \Csol \, \norm{L^2 \times {\mathbf{L}}^2}{(u - \textup{R}_{H}u  , \bfA - \textup{R}_{h} \bfA ) }
\end{eqnarray*}
and hence again by $ \| u \|_{H^1_{\kappa}} \lesssim \| u\|_{L^2}$
\begin{eqnarray}
\label{zeta-nablaiu-est}
\tfrac{1}{\kappa}\| \zeta \, \ci \nabla u\|  
&\lesssim& \Csol \, \norm{L^2 \times {\mathbf{L}}^2}{(u - \textup{R}_{H}u  , \bfA - \textup{R}_{h} \bfA ) } .
\end{eqnarray}
With this, we now proceed similar as in the proof of Lemma \ref{lemma:bestapprox-LOD} and consider the function 
$$
z_H^\LOD := (1 - \mathcal{C})\LtwoprojFEM(z) \in \VShLOD
$$
which fulfills again $z_H^\LOD-z \in W$. Hence, we obtain
\begin{eqnarray*}
\lefteqn{ \tfrac{1}{2} \| z - z_H^\LOD  \|_{\HonekappaSpace}^2
\,\,\,\le\,\,\,  a_{\bfAapp}( z - z_H^\LOD  , z - z_H^\LOD ) 
\,\,\, \overset{\eqref{eqn:def-ideal-correctors}}{=} \,\,\, a_{\bfAapp}( z  ,  z - z_H^\LOD ) } \\
&=& a_{\bfA} ( z  , z - z_H^\LOD ) \, + \, (a_{\bfAapp}-a_{\bfA}) ( z  , z - z_H^\LOD ) \\
&\overset{\eqref{new-def-z}}{=}& ( g_{\ast} + \zeta \, \ci u,  z - z_H^\LOD )_{L^2} \,+\, 
(  \tfrac{\ci}{\kappa}  \nabla z,   (\bfAapp-\bfA) \,  (z - z_H^\LOD)  )_{L^2}\\
&\enspace&\, + \, ( (\bfAapp-\bfA) z  ,  \tfrac{\ci}{\kappa}  \nabla (z - z_H^\LOD) )_{L^2} +  ( (|\bfAapp|^2-|\bfA|^2) z  , z - z_H^\LOD )_{L^2} \\
&=& (g_{\ast} + \zeta \, \ci u + \underbrace{ ( |\bfAapp|^2-|\bfA|^2) z  +  
 \tfrac{2\ci}{\kappa}  \nabla z \cdot  (\bfAapp-\bfA)}_{ =: f_{\ast}}  \,,\,  z - z_H^\LOD )_{L^2} \\
 &=& ( \,(1 - \LtwoprojFEM)(g_{\ast} + \zeta \, \ci u + f_{\ast})\, ,\, (1 - \LtwoprojFEM) (z - z_H^\LOD) \,)_{L^2} \\
 &\overset{\eqref{L2-proj-est}}{\lesssim}&  (\kappa\,H)^2\,  \tfrac{1}{\kappa}\| \nabla g_{\ast}  + \zeta \, \ci \nabla u + \nabla f_{\ast} \|_{L^2} \, \| z - z_H^\LOD  \|_{\HonekappaSpace} 
\end{eqnarray*}
Hence, dividing by $\| z - z_H^\LOD  \|_{\HonekappaSpace} $ and using \eqref{gradgstar-est} and \eqref{zeta-nablaiu-est} we have
\begin{eqnarray*}
\lefteqn{ \| z - z_H^\LOD  \|_{\HonekappaSpace} } \\
&\lesssim&  (\kappa\,H)^2 \left( \| u - \textup{R}_{H}u \|_{H^1_{\kappa}}  +
\kappa^{\varepsilon} \, \Csol \, \norm{L^2 \times {\mathbf{L}}^2}{(u - \textup{R}_{H}u  , \bfA - \textup{R}_{h} \bfA ) }  + \tfrac{1}{\kappa}\| \nabla f_{\ast} \|_{L^2} \right). 
\end{eqnarray*}
It remains to bound $\| \nabla f_{\ast} \|_{L^2}$ where we obtain with the bounds for $\bfA$ and $\bfAapp$ that
\begin{eqnarray*}
\lefteqn{ \| \nabla f_{\ast} \|_{L^2} } \\
&\lesssim& \| \bfAapp + \bfA \|_{L^{\infty}} \| \bfAapp - \bfA \|_{L^{\infty}} \| \nabla z \|_{L^2}
+  \| \bfAapp + \bfA \|_{W^{1,4}} \| \bfAapp - \bfA \|_{L^{\infty}} \| z \|_{L^4} \\
&\enspace&\,\,+  \| \bfAapp + \bfA \|_{L^{\infty}} \| \bfAapp - \bfA \|_{W^{1,4}} \| z \|_{L^4} \\
&\enspace&\,\,+ \tfrac{1}{\kappa} | z |_{H^2} \| \bfAapp-\bfA \|_{L^{\infty}} 
+ \tfrac{1}{\kappa} \| \nabla z \|_{L^4} \| \bfAapp-\bfA \|_{W^{1,4}} \\
&\overset{\eqref{H1est-z},\eqref{H2est-z}}{\lesssim}& 
\kappa \, \Csol \, \norm{L^2 \times {\mathbf{L}}^2}{(u - \textup{R}_{H}u  , \bfA - \textup{R}_{h} \bfA ) }.
\end{eqnarray*}
In conclusion we have
\begin{eqnarray*}
 \| z - z_H^\LOD  \|_{\HonekappaSpace}  &\lesssim&  (\kappa\,H)^2 \left( \| u - \textup{R}_{H}u \|_{H^1_{\kappa}}  + 
 \kappa^{\varepsilon} \, \Csol \, \norm{L^2 \times {\mathbf{L}}^2}{(u - \textup{R}_{H}u  , \bfA - \textup{R}_{h} \bfA ) } \right). 
\end{eqnarray*}
On the other hand, we obtain with the approximation properties of $\VSAhzero{k}$ straightforwardly 
\begin{eqnarray*}
\inf_{\mathbf{Z}_h \in \VSAhzero{k}} \| \mathbf{Z} -  \mathbf{Z}_h \|_{\mathbf{H}^1} 
&\lesssim& \, h\,  \| \mathbf{Z} \|_{\mathbf{H}^2} \,\,\, \overset{\eqref{H2est-z}}{\lesssim} \,\,\,
\,h\, 
 \kappa^{\varepsilon} \, \Csol \,
		\norm{L^2 \times {\mathbf{L}}^2}{(u - \textup{R}_{H}u  , \bfA - \textup{R}_{h} \bfA ) },
\end{eqnarray*}
which leads to the estimate
\begin{eqnarray} \label{eq:est_zZ_best}
\lefteqn{ \inf_{(z_H , \mathbf{Z}_h) \in \VShLOD \times \VSAhzero{k}} \Honecombi{ (z - z_H , \mathbf{Z} -  \mathbf{Z}_h ) }  }\\
&\lesssim& 
(\kappa\,H)^2  \| u - \textup{R}_{H}u \|_{H^1_{\kappa}}  \,+\,
( (\kappa\,H)^2 + h) 
\, \kappa^{\varepsilon} \, \Csol \, \norm{L^2 \times {\mathbf{L}}^2}{ (u , \mathbf{A}) \hspace{-2pt}-\hspace{-2pt} \textup{\bf R}_{H,h} (u , \mathbf{A})  } . \notag
\end{eqnarray}
Note that for this estimate only the condition $\kappa H\lesssim 1$ is required. Recalling that Lemma \ref{lemma-properties-LOD-space} verifies the assumptions of 
Lemma~\ref{lemma:ritz-projections:abstract-new} and \ref{lemma-ritz-proj-est}, 
we conclude that 
\begin{eqnarray*}
\lefteqn{ \Honecombi{ (u , \bfA ) - \textup{\bf R}_{H,h} (u , \bfA) } 
\,\,\,\lesssim\,\,\,  
\inf_{(\varphi_H^{\LOD}, \mathbf{B}_h) \in \VShLOD \times \VSAhzero{k}}
\Honecombi{ (u -\varphi_H^{\LOD}, \bfA -\mathbf{B}_h) } }\\
&\enspace&\quad + 
((\kappa H)^2 + h) \, 
\kappa^{\varepsilon} \, \Csol \,  \Honecombi{ (u , \bfA ) - \textup{\bf R}_{H,h} (u , \bfA) }.\hspace{100pt}
\end{eqnarray*}
Hence, for $((\kappa H)^2 + h) \, 
\kappa^{\varepsilon} \, \Csol \lesssim 1$ (sufficiently small), the second term can be absorbed into the left-hand side, which gives the first assertion.
In particular, under the above resolution condition we have 
\begin{eqnarray*}
\lefteqn{ \inf_{(z_H^{\LOD}, \mathbf{Z}_h) \in \VShLOD \times \VSAhzero{k} }
\Honecombi{ (z -z_H, \mathbf{Z}-\mathbf{Z}_h) }
\,\,\,\lesssim\,\,\,  \Honecombi{ (u , \bfA ) - \textup{\bf R}_{H,h} (u , \bfA) } } \\
&\lesssim&
\inf_{(\varphi_H^{\LOD}, \mathbf{B}_h) \in \VShLOD \times \VSAhzero{k}}
\Honecombi{ (u -\varphi_H^{\LOD}, \bfA -\mathbf{B}_h) }. \hspace{120pt}
\end{eqnarray*}
The proof is finished by combining the above estimate with Conclusion \ref{concl:Honecombi_bound}, the estimates from Lemma \ref{lemma:bestapprox-LOD} and the standard approximation properties \eqref{eq:approx_FEM_Lag}of $\VSAhzero{k}$, i.e., 
\begin{eqnarray*}
\inf_{ \mathbf{B}_h \in \VSAhzero{1}}
\| \bfA -\mathbf{B}_h \|_{\mathbf{H}^1} &\lesssim& h \| \bfA \|_{\mathbf{H}^2} \,\,\, \lesssim \,\,\, h \quad \mbox{and} \\
\inf_{ \mathbf{B}_h \in \VSAhzero{2}}
\| \bfA -\mathbf{B}_h \|_{\mathbf{H}^1} &\lesssim& h^2 \| \bfA \|_{\mathbf{H}^3} \,\,\, \lesssim \,\,\, \kappa h^2.
\end{eqnarray*} 
\end{proof}

It remains to turn to the $L^2$-error estimate.
For that, we first bound the Ritz-projection error in the $L^2$-norm by the corresponding error in the $H^1$-norm.
\begin{lemma}
\label{lemma:L2-est-Ritz-proj}
Let Assumption \ref{ass:local_uniq} hold and let  $\bfAapp$ admit the same regularity and $\kappa$-dependent stability bounds as $\bfA$. 
If $\kappa H\lesssim 1$ and if $(\sol,\MagPot) \in  H^1(\Omega) \times \HonenSpace(\Omega) $ is a minimizer of \eqref{minimization-problem} then it holds
\begin{eqnarray*}
\lefteqn{ \|  (u , \mathbf{A}) \hspace{-2pt}-\hspace{-2pt} \textup{\bf R}_{H,h} (u , \mathbf{A})  \|_{L^2 \times \mathbf{L}^2} }\\
&\lesssim& \left( \kappa\,H  \,+\, ( (\kappa\,H)^2 + h) 
\, \kappa^{\varepsilon} \, \Csol \right) \Honecombi{ (u , \bfA ) - \textup{\bf R}_{H,h} (u , \bfA) }.
\end{eqnarray*}
\end{lemma}

\begin{proof}
We consider again the auxiliary problem \eqref{def-z-proj-error}. The corresponding solution $(z,\mathbf{Z}) \in \Honeperp \times \HonenSpace(\Omega)$ allows to express the $L^2$-error as 
\begin{eqnarray*}
\| (u , \mathbf{A}) \hspace{-2pt}-\hspace{-2pt} \textup{\bf R}_{H,h} (u , \mathbf{A}) \|_{L^2 \times \mathbf{L}^2}^2
&=& 
\langle E'' (u,\bfA) ( z,\mathbf{Z} ) , (u , \mathbf{A}) \hspace{-2pt}-\hspace{-2pt} \textup{\bf R}_{H,h} (u , \mathbf{A}) \rangle  \\
&=& 
\langle E'' (u,\bfA) ( z - z_H,\mathbf{Z} - \mathbf{Z}_H ) , (u , \mathbf{A}) \hspace{-2pt}-\hspace{-2pt} \textup{\bf R}_{H,h} (u , \mathbf{A}) \rangle,
\end{eqnarray*}
for arbitrary $(z_H, \mathbf{Z}_H ) \in (\VShLOD \cap \Honeperp) \times \VSAhzero{k}$. With the continuity of $ E'' (u,\bfA)$ and Lemma \ref{lemma:ritz-projections:abstract-new} we therefore have
\begin{eqnarray*}
\lefteqn{ \| (u , \mathbf{A}) \hspace{-2pt}-\hspace{-2pt} \textup{\bf R}_{H,h} (u , \mathbf{A}) \|_{L^2 \times \mathbf{L}^2}^2 }\\
&\lesssim& \Honecombi{ (u , \bfA ) - \textup{\bf R}_{H,h} (u , \bfA) } \inf_{(z_H , \mathbf{Z}_h) \in \VShLOD \times \VSAhzero{k}} \Honecombi{ (z - z_H , \mathbf{Z} -  \mathbf{Z}_h ) } .
\end{eqnarray*}
The latter term was already estimated in \eqref{eq:est_zZ_best},
such that combining this with the previous estimate
 and applying Young's inequality afterwards yields for any $\delta>0$
\begin{eqnarray*}
\lefteqn{ \|  (u , \mathbf{A}) \hspace{-2pt}-\hspace{-2pt} \textup{\bf R}_{H,h} (u , \mathbf{A})  \|_{L^2 \times \mathbf{L}^2}^2 }\\
&\lesssim& (\kappa\,H)^2  \Honecombi{ (u , \bfA ) - \textup{\bf R}_{H,h} (u , \bfA) }^2  \\
&\enspace&\enspace+  ( (\kappa\,H)^2 + h) 
\, \kappa^{\varepsilon} \, \Csol \Honecombi{ (u , \bfA ) - \textup{\bf R}_{H,h} (u , \bfA) } \, \norm{L^2 \times {\mathbf{L}}^2}{(u , \mathbf{A}) \hspace{-2pt}-\hspace{-2pt} \textup{\bf R}_{H,h} (u , \mathbf{A})   } \\
&\lesssim& (\kappa\,H)^2  \Honecombi{ (u , \bfA ) - \textup{\bf R}_{H,h} (u , \bfA) }^2  \\
&\enspace&\enspace+  \tfrac{1}{\delta} \left( ( (\kappa\,H)^2 + h) 
\, \kappa^{\varepsilon} \, \Csol \Honecombi{ (u , \bfA ) - \textup{\bf R}_{H,h} (u , \bfA) } \right)^2 + \delta \,\norm{L^2 \times {\mathbf{L}}^2}{(u , \mathbf{A}) \hspace{-2pt}-\hspace{-2pt} \textup{\bf R}_{H,h} (u , \mathbf{A}) }^2.
\end{eqnarray*}
For sufficiently small $\delta$ we can absorbe the $ \norm{L^2 \times {\mathbf{L}}^2}{(u , \mathbf{A}) \hspace{-2pt}-\hspace{-2pt} \textup{\bf R}_{H,h} (u , \mathbf{A}) }$-contribution into the left-hand side to obtain
\begin{eqnarray*}
\lefteqn{ \|  (u , \mathbf{A}) \hspace{-2pt}-\hspace{-2pt} \textup{\bf R}_{H,h} (u , \mathbf{A})  \|_{L^2 \times \mathbf{L}^2} }\\
&\lesssim& \left( \kappa\,H  \,+\, ( (\kappa\,H)^2 + h) 
\, \kappa^{\varepsilon} \, \Csol \right) \Honecombi{ (u , \bfA ) - \textup{\bf R}_{H,h} (u , \bfA) },
\end{eqnarray*}
which gives the claim.
\end{proof}
With this, we are ready to finalize the $L^2$-error estimate

\begin{proposition}[$L^2$-error estimate]
\label{prop:L2-error-estimates}
In the setting of Proposition \ref{proposition-H1-est}, the $L^2$-error between a discrete minimizer $(\sol_{H}^{\LOD},\mathbf{A}_{h,k}^{\mbox{\tiny FEM}})  \in \VShLOD \times \VSAhzero{k}$ and its corresponding exact minimizer $(\sol,\MagPot) \in  H^1(\Omega) \times \HonenSpace(\Omega)$ (in the same phase, i.e. $\sol_{H}^{\LOD} \in \Honeperp$) can be bounded by 
\begin{eqnarray*}
 \lefteqn{ \| (\sol - \sol_{H}^{\LOD} , \MagPot - \mathbf{A}_{h,k}^{\mbox{\tiny FEM}}) \|_{L^2 \times \mathbf{L}^2} \,\,\, \lesssim \,\,\, (\kappa\,H)^4 + \kappa^k H \deltah^k  }\\
&\enspace&
\quad+ 
\, \kappa^{\varepsilon} \, \Csol\, \left(\kappa^3 \deltaH^3  + \kappa^{k-1} \deltah^k\right) ((\kappa\,H)^2 + h) + \Csol \left( \kappa^4 H^3  + \kappa \, h^k \right)^2 
\end{eqnarray*}
where $k=1,2$.
\end{proposition}

\begin{proof}
We split the error as
\begin{eqnarray*}
 \lefteqn{ \| (\sol - \sol_{H}^{\LOD} , \MagPot - \mathbf{A}_{h,k}^{\mbox{\tiny FEM}}) \|_{L^2 \times \mathbf{L}^2} }\\
&\le& 
 \| (\sol - \textup{R}_{H} u  , \MagPot -  \textup{R}_{h} \bfA )\|_{L^2 \times \mathbf{L}^2}
 +  \| (\textup{R}_{H} u - \sol_{H}^{\LOD} , \textup{R}_{h} \bfA - \mathbf{A}_{h,k}^{\mbox{\tiny FEM}}) \|_{L^2 \times \mathbf{L}^2} \\
&\le& 
 \| (\sol - \textup{R}_{H} u  , \MagPot -  \textup{R}_{h} \bfA )\|_{L^2 \times \mathbf{L}^2}
 +  \| (\textup{R}_{H} u - \sol_{H}^{\LOD} , \textup{R}_{h} \bfA - \mathbf{A}_{h,k}^{\mbox{\tiny FEM}}) \|_{H^1_{\kappa} \times \mathbf{H}^1} ,
\end{eqnarray*}
where the first term is controlled by Lemma \ref{lemma:L2-est-Ritz-proj} in combination with Proposition \ref{proposition-H1-est} as
\begin{eqnarray*}
 \lefteqn{  \| (\sol - \textup{R}_{H} u  , \MagPot -  \textup{R}_{h} \bfA )\|_{L^2 \times \mathbf{L}^2}  }\\
&\lesssim& (\kappa\,H)^4 + \kappa^k H \deltah^k
+  
\, \kappa^{\varepsilon} \, \Csol\, \left(\kappa^3 \deltaH^3  + \kappa^{k-1} \deltah^k\right) ((\kappa\,H)^2 + h).
\end{eqnarray*}
For the second term, we use the coercivity of $E^{\prime\prime}(u,\bfA)$ to estimate
\begin{eqnarray*}
 \lefteqn{ \Csolinv \| (\textup{R}_{H} u - \sol_{H}^{\LOD} , \textup{R}_{h} \bfA - \mathbf{A}_{h,k}^{\mbox{\tiny FEM}}) \|_{H^1_{\kappa} \times \mathbf{H}^1}^2 }\\
&\le& \dualp{\energystab''(\sol,\MagPot) (\textup{R}_{H} u - \sol_{H}^{\LOD} , \textup{R}_{h} \bfA - \mathbf{A}_{h,k}^{\mbox{\tiny FEM}}) }{ (\textup{R}_{H} u - \sol_{H}^{\LOD} , \textup{R}_{h} \bfA - \mathbf{A}_{h,k}^{\mbox{\tiny FEM}}) } \\
  &\overset{\eqref{def-eps-2}}{=}& \sigma(\textup{R}_{H} u - \sol_{H}^{\LOD} , \textup{R}_{h} \bfA - \mathbf{A}_{h,k}^{\mbox{\tiny FEM}}) \\
    &\overset{\eqref{est:epsilon_2_bound}}{\lesssim}&
    	\left( 
	\|  \sol - \sol_{H}^{\LOD} \|_{L^6}^2  +  \|  \sol - \sol_{H}^{\LOD} \|_{\HonekappaSpace}^2 + \| \MagPot - \mathbf{A}_{h,k}^{\mbox{\tiny FEM}} \|_{H^1}^2 
	\right) \, \| (\textup{R}_{H} u - \sol_{H}^{\LOD} , \textup{R}_{h} \bfA - \mathbf{A}_{h,k}^{\mbox{\tiny FEM}}) \|_{H^1_{\kappa}\times \mathbf{H}^1}.
\end{eqnarray*}
By combining the previous estimates we obtain
\begin{eqnarray*}
 \lefteqn{ \| (\sol - \sol_{H}^{\LOD} , \MagPot - \mathbf{A}_{h,k}^{\mbox{\tiny FEM}}) \|_{L^2 \times \mathbf{L}^2} }\\
&\lesssim& 
(\kappa\,H)^4 + \kappa^k H \deltah^k
+ 
\, \kappa^{\varepsilon} \, \Csol\, \left(\kappa^3 \deltaH^3  + \kappa^{k-1} \deltah^k\right) ((\kappa\,H)^2 + h) \\
&\enspace&+ \Csol \, \left( 
	\|  \sol - \sol_{H}^{\LOD} \|_{L^6}^2  +  \|  \sol - \sol_{H}^{\LOD} \|_{\HonekappaSpace}^2 + \| \MagPot - \mathbf{A}_{h,k}^{\mbox{\tiny FEM}} \|_{H^1}^2 
	\right).
\end{eqnarray*}
The last term can be further estimated by using $\|  \sol - \sol_{H}^{\LOD} \|_{L^6} \lesssim \kappa \|  \sol - \sol_{H}^{\LOD} \|_{H^1_{\kappa}}$ and the previously established estimate for $ \| (\sol - \sol_{H}^{\LOD} , \MagPot - \mathbf{A}_{h,k}^{\mbox{\tiny FEM}}) \|_{\HonekappaSpace \times \mathbf{H}^1} $ in Proposition \ref{proposition-H1-est} to obtain
\begin{eqnarray*}
 \lefteqn{ \| (\sol - \sol_{H}^{\LOD} , \MagPot - \mathbf{A}_{h,k}^{\mbox{\tiny FEM}}) \|_{L^2 \times \mathbf{L}^2} \,\,\, \lesssim \,\,\, (\kappa\,H)^4 + \kappa^k H \deltah^k  }\\
&\enspace&
\quad+ 
\, \kappa^{\varepsilon} \, \Csol\, \left(\kappa^3 \deltaH^3  + \kappa^{k-1} \deltah^k\right) ((\kappa\,H)^2 + h) + \Csol \left( \kappa^4 H^3  + \kappa \, h^k \right)^2 
\end{eqnarray*}
\end{proof}

\section{Numerical experiments}
\label{sec:num_exp}

In this section we verify our theoretical results from Theorem \ref{thrm:main-results} in numerical experiments and investigate the optimality of the convergence w.r.t. the mesh size $H$ for the order parameter $u$ and the mesh size $h$ for the vector potential $\MagPot$ as well as the scaling of the convergence w.r.t. the GL parameter $\kappa$. The implementation for our experiments is available as a MATLAB Code on \url{https://github.com/cdoeding/fullGLmodelLOD}.

For the experiments we choose a LOD approximation for the order parameter $u$ and quadratic FE ($k = 2$) for the vector potential $\mathbf{A}$. The latter choice has two reasons: First, considering either $k = 1$ or $k = 2$ is sufficient to demonstrate the main results numerically, since the difference in the analysis is based on standard properties of Lagrange finite elements and there is no reason to expect a different behavior in the case $k = 1$,  except for the lower order w.r.t. $h$ and $\kappa$. Second, due to the higher convergence in the case of quadratic FE, it is easier to extract the expected third order convergence in the LOD space by taking sufficiently small $h$ in the experiments. For brevity, we will not compare the LOD approximation in the order parameter with the FE approximation in this section, since this has already been done for a simplified Ginzburg–Landau model in \cite{BDH24}. In \cite{BDH24}, the advantages of the LOD approach compared to classical FE were demonstrated, and no deviations are expected for the full Ginzburg–Landau model.

Another simplification we make is a reduction to two dimensions to keep the complexity and runtimes reasonable in our experiments. We emphasize that the main results and the analysis in this work are not restricted to the three-dimensional case, but can be modified to the problem in two dimensions. To derive the corresponding Ginzburg--Landau model in $2d$, one intuitively considers a $3d$ external magnetic field $\MagF$ which is perpendicular to the $x_1$-$x_2$ plane, i.e., $\MagF = (0,0,\MagF_3)^T$ for some scalar function $\MagF_3$. 
Then the order parameter and the vector potential should not vary in the $x_3$-direction as far as we stay away from the $x_3$-boundary of the superconductor, i.e., $\partial_{x_3} u = 0$ and $\partial_{x_3} \mathbf{A} = \mathbf{0}$. This implies that the third component of the vector potential has to vanish and one derives the Ginzburg--Landau free energy in two dimensions 
\begin{eqnarray*}
	E_{\mathrm{GL},2d}(\sol,\MagPot) 
	&\coloneqq&
	\frac{1}{2} \int_\Omega 
	|\kappainvci \nabla \sol +   \MagPot \sol |^2 
	+ 
	\frac{1}{2} \bigl( 1- |\sol|^2 \bigr)^2
	+
	|\curl_{2d} \MagPot - \MagF_3 |^2
	\dx
\end{eqnarray*}
for the reduced $2d$-vector potential $\mathbf{A}: \Omega \rightarrow \mathbb{R}^2$ and the reduced order parameter $u: \Omega \rightarrow \C$ where now $\Omega \subset \mathbb{R}^2$ is a rectangle. Here $\curl_{2d}$ denotes the conventional $2d$-$\curl$-operator mapping vector fields to scalar functions. Setting up the previous analysis in two dimensions and introducing the stabilized energy
\begin{eqnarray*}
	E_{2d}(\sol,\MagPot)  := E_{\mathrm{GL},2d}(\sol,\MagPot) + \frac{1}{2} \int_\Omega |\div \MagPot|^2 \dx
\end{eqnarray*}
one can derive the corresponding results of Theorem \ref{thrm:main-results} in $2d$ using fully analogue arguments. For the sake of readability and brevity we omit this case in the analysis of this work and drop the $2d$ subindices in the following and write $\MagF = \MagF_3$.
\subsection{Gradient descent method}
\label{subsec:gradient_descent}

Although our theoretical results focus on approximability results for minimizers in discrete spaces, verifying these results numerically requires computing the minimizers themselves. One approach is to compute these minimizers using an implicit Euler discretization of the $L^2$-gradient flow associated with the GL energy functional \eqref{eq:energy_functional_stab}, as it was done, for instance, in \cite{BDH24}. However, due to the small coercivity constant $\Csol^{-1}$ from Proposition \ref{prop:coecivity_Epp}, particularly for large $\kappa$, the $L^2$-gradient flow typically converges only slowly to the desired minimizer. Therefore, we employ a more sophisticated approach here: a discretization of a suitable Sobolev-gradient flow allowing for energy-adaptive step sizes. Sobolev gradients $\nabla_X E$ represent $E^{\prime}$ with respect to an $X$-metric induced by an inner product $(\cdot,\cdot)_X$. By choosing the inner product in a problem-specific way, the corresponding gradient flow can be significantly accelerated; see Neuberger \cite{Neu97} for an introduction.

We denote by $X$ the tensor space $H^1(\Omega) \times \HonenSpace(\Omega)$. The Sobolev gradient of the Ginzburg-Landau energy $E$ at any point $(u,\mathbf{A}) \in X$, not necessarily a minimizer, is defined as the solution $\nabla_X E(u,\mathbf{A}) \in X$ of
\begin{eqnarray*}
	\big( \nabla_X E(u,\mathbf{A}), (v,\mathbf{B}) \big)_X = \partial_u E(u,\mathbf{A})v + \partial_{\mathbf{A}} E(u,\mathbf{A}) \mathbf{B} \quad \forall (v,\mathbf{B}) \in X.
\end{eqnarray*} 
Here we choose the inner product $(\cdot,\cdot)_{X}$ as
\begin{eqnarray*}
	\big((v,\mathbf{B}),(w,\mathbf{C}) \big)_X := \tilde{a}_{u,\mathbf{A}}(v,w) + \tilde{b}_{u,\mathbf{A}}(\mathbf{B},\mathbf{C}) 
\end{eqnarray*}
with
\begin{align}
	\tilde{a}_{u,\mathbf{A}}(v,w) & := a_{\mathbf{A}}(v,w) + \big((\beta + |u|^2 + |\mathbf{A}|^2) v,w \big) ,\\
	\tilde{b}_{u,\mathbf{A}}(\mathbf{B},\mathbf{C})  & := b(\mathbf{B}, \mathbf{C}) + \big\langle (\beta + |u|^2) \mathbf{B}, \mathbf{C} \big\rangle,
\end{align}
and a stabilization parameter $\beta > 0$. It is straightforward to verify that for every $\beta > 0$ the bilinear form $(\cdot,\cdot)_X$ defines an inner product on $X$. This choice enhances coercivity and ensures stability of the resulting gradient flow even when $|u|$ or $|\mathbf{A}|$ exhibit large variations. Introducing the elliptic operators
\begin{align}
	\tilde{\mathcal{A}}_{u,\mathbf{A}}: H^1(\Omega) \rightarrow (H^1(\Omega))^*, \quad \langle \tilde{\mathcal{A}}_{u,\mathbf{A}} v, w \rangle = \tilde{a}_{u,\mathbf{A}}(v,w)  \quad \forall \, v,w \in H^1(\Omega),
\end{align} 
and
\begin{align}
	\tilde{\mathcal{B}}_{u,\mathbf{A}}: \HonenSpace(\Omega) \rightarrow (\HonenSpace(\Omega))^*, \quad \langle \tilde{\mathcal{B}}_{u,\mathbf{A}} \mathbf{B}, \mathbf{C} \rangle = \tilde{b}_{u,\mathbf{A}}(\mathbf{B},\mathbf{C})  \quad \forall \, \mathbf{B},\mathbf{C} \in \HonenSpace(\Omega),
\end{align} 
the Sobolev gradient can be expressed as $\nabla_X E(u,\mathbf{A}) = \big( \nabla_{X,u} E(u,\mathbf{A}), \nabla_{X,\mathbf{A}} E(u,\mathbf{A})\big) \in X$, where
\begin{align}
	\nabla_{X,u} E(u,\mathbf{A}) & = u - \tilde{\mathcal{A}}_{u,\mathbf{A}}^{-1} \big( (1 + \beta + |\mathbf{A}|^2 ) u \big), \\
	\nabla_{X,\mathbf{A}} E(u,\mathbf{A}) & = \mathbf{A} - \tilde{\mathcal{B}}_{u,\mathbf{A}}^{-1} \Big( \beta \mathbf{A} - \kappainv  \Real ( \ci  u^* \nabla u ) + \curl^* \MagF \Big),
\end{align}
with $\MagF = \curl \mathbf{A}$ denoting the magnetic field. The operator $\curl^*: \mathbf{H}^1(\Omega) \to \mathbf{H}^1(\Omega)^*$ is defined by
\begin{eqnarray*}
\langle \curl^* \MagF, \mathbf{B} \rangle = \int_{\Omega} \MagF \cdot \curl \mathbf{B} , \mathrm{d}x \quad \forall \mathbf{B} \in \mathbf{H}^1(\Omega).
\end{eqnarray*}
Here, $\tilde{\mathcal{A}}_{u,\mathbf{A}}^{-1}$ and $\tilde{\mathcal{B}}_{u,\mathbf{A}}^{-1}$ denote the inverse operators in the corresponding Sobolev spaces, defined via the Riesz representation. This formulation provides a problem-adapted metric structure that typically leads to much faster convergence of the gradient flow compared to the standard $L^2$-based formulation.

At the discrete level, we select $\bfAapp \in \HonenSpaceDiv(\Omega) \cap \mathbf{L}^{\infty}(\Omega)$ and construct the LOD space $\VShLOD$ based on this $\bfAapp$. As mentioned in Section \ref{sec:space_main}, a natural choice is to select $\bfAapp \in \VSAhzero{k}$ as the solution to $\curl \bfAapp = \MagF$ with $\div \bfAapp = 0$ which can be easily computed. Unless otherwise stated, we use this choice for $\bfAapp$. The desired iteration for computing a minimizer by the $\nabla_{X}E$-gradient descent now seeks $(u^{n}_H, \mathbf{A}^{n}_h) \in \VShLOD \times \VSAhzero{k}$, $n \in \mathbb{N}$,  satisfying
\begin{align}
	u^{n+1}_H & = u_H^n - \tau_n \nabla_{X,u^n_H} E(u^{n}_H,\mathbf{A}^n_h), \\
	\mathbf{A}^{n+1}_h & = \mathbf{A}^{n}_h - \tau_n \nabla_{X,\mathbf{A}^n_h} E(u^{n}_H,\mathbf{A}^n_h)
\end{align}
for some step size $\tau_n > 0$ and suitable initial values $(u^0_H, \mathbf{A}^0_h) \in \VShLOD \times \VSAhzero{k}$. The step size $\tau_n > 0$ in each iterations is adaptively chosen such that $E(u^{n+1}_H, \mathbf{A}^{n+1}_h)$ gets minimal. Since $E(u^{n+1}_H, \mathbf{A}^{n+1}_h)$ is a fourth order polynomial in $\tau_n$ this is achieved numerically e.g. by line search in $\tau_n$.

Clearly, every (local) minimizer is a stationary point of the iterative scheme, and vice versa: every stationary point satisfies the first-order condition for a local minimum. For an optimal $\tau_n$, the iteration is expected to reduce the energy, and we may find a local minimizer of the GL energy in $\VShLOD \times \VSAhzero{k}$ as a limit of the iteration. A rigorous convergence analysis of the scheme is much more involved and is left for future research. In practice, the optimal $\tau_n$ can become very small during the iteration process, which may cause the iteration to converge slowly due to small updates per iteration. To avoid very small steps, we set a lower bound of $0.1$ for $\tau_n$ in our experiments. This comes at the cost of losing the energy-diminishing property in every iteration step, but results in a more robust, quicker-converging overall computation. The stabilization parameter in the definition of the Sobolev gradient is set to $\beta = 0.1$ in our experiments. Finally, we terminate the iteration when the difference in energy of two iterates, $|\energy(\sol^{n+1}, \MagPot^{n+1}) - \energy(\sol^n, \MagPot^n) |$, approaches a given tolerance, $\varepsilon_{\mathrm{tol}} > 0$. 

\subsection{Model problem and discrete minimizers}
\label{subsec:model}

The model of our numerical experiments considers the GL energy on the two dimensional unit square $\Omega = (0,1)^2 \subset \mathbb{R}^2$, with the external magnetic field
\begin{align}
	\MagF(x) = 10 \sin(\pi x_1) \sin(\pi x_2), \quad x = (x_1,x_2) \in \Omega,
\end{align}
and the particular values for the GL parameter $\kappa = 5, 10, 15, 20, 25, 30, 50, 100$. The range of GL parameter is chosen such that it covers the magnitudes of common type-II superconductors such as niobium ($\kappa \sim 1 $), magnesium diboride ($\kappa \sim 25$), or yttrium barium copper oxide ($\kappa \sim 100$), see \cite{Finnetal01,McCSe65,StZw97}. The discrete minimizers are computed with a LOD approximation for the order parameter $u$ and quadratic FE for the vector potential $\mathbf{A}$ with the iterative solver described in Section \ref{subsec:gradient_descent}. The mesh sizes are set to $H = 2^{-7}$ and $h = 2^{-7}$ and we compute the minimizers up to a tolerance of $\varepsilon_{\mathrm{tol}} = 10^{-12}$. The LOD space is constructed in all our experiments, unless otherwise stated, based on the vector potential $\bfAapp \in \mathbf{V}^2_{h,0}$ which is a $\mathcal{P}_2$-approximation of $\curl \bfAapp = \MagF$ with $\div \bfAapp = 0$ and fixed fine mesh size $h = 2^{-7}$ through all experiments . \\
For the practical realization of the LOD spaces we need to compute an associated basis of $\VShLOD$ which requires to solve the corrector problems \eqref{eqn:def-ideal-correctors} for suitable function $\varphi \in \VSh$. This is done using a standard $\mathcal{P}_1$-FEM discretization on a fine mesh with mesh size $h_{\mathrm{fine}} = 2^{-9}$. It is known, cf. \cite{MaP21,BDH24}, that a basis can be chosen such that each basis function decays exponentially. Therefore, the corrector problems can be restricted to local patches by allowing small localization errors and to obtain a feasibly computable and locally supported basis of $\VShLOD$. The patches are defined through an oversampling parameter $\ell$ describing the number of layers around the element of $\mathcal{T}_H$ associated with the patch. We refer to \cite{BDH24} for more details on the computation and for estimates of the occurring localization error in the context of the GL energy minimization problem. Unless otherwise stated, we choose $\ell = 10$ oversampling layers for each coarse element in our experiments. \\
To achieve the fine resolution of $H = 2^{-7}$ and $h = 2^{-7}$ at feasible computational cost we use a multilevel type approach that first computes a discrete minimizer on a coarse mesh which is then used as an initial value for the minimizing iteration on a finer mesh. This process is repeated across four levels so that the mesh parameters $(H,h,\ell)$ follow the sequence
\begin{align}
	 (2^{-3}, 2^{-4}, 3) \quad \rightarrow \quad (2^{-5}, 2^{-4}, 5) \quad \rightarrow \quad (2^{-7}, 2^{-4}, 10) \quad \rightarrow \quad (2^{-7}, 2^{-7}, 10) 
\end{align}
and with the constant fine mesh size $h_{\mathrm{fine}} = 2^{-9}$ for the LOD realization. For the first iteration we use a constant $0.8 + 0.6 \ci$ initial value in $\sol$ and $\bfAapp$ as an initial value in $\MagPot$. Finally, for the first three levels we choose $\varepsilon_{\mathrm{tol}} = 10^{-10}$ and the final iteration is then computed up to $\varepsilon_{\mathrm{tol}} = 10^{-12}$. \\
The order parameters $\sol$ of the computed discrete minimizers are shown in Figure \ref{fig:reference-minimizers} and Figure \ref{fig:reference-minimizers_2}. In Figure \ref{fig:reference-minimizers-A} the corresponding vector potential $\MagPot$ and its $\curl$ for $\kappa = 5$ is shown. This value is representative of all other values of $\kappa = 10, 15, 20, 25, 30, 50, 100$, since no visual differences are notable. In contrast, we plot the difference $\curl \MagPot - \MagF$ in Figure \ref{fig:reference-minimizers-curlA-H} for all values of $\kappa$ where small differences can be detected due to the vortex structure of the order parameter. Finally, the energy values for the original GL energy $\energy$ and the stabilized energy $E$ are given in Table \ref{tab:energy-reference-minimizers}. 

\begin{figure}[h!] 
\centering
\begin{minipage}[t]{0.24\textwidth}
\centering
\includegraphics[scale=0.19]{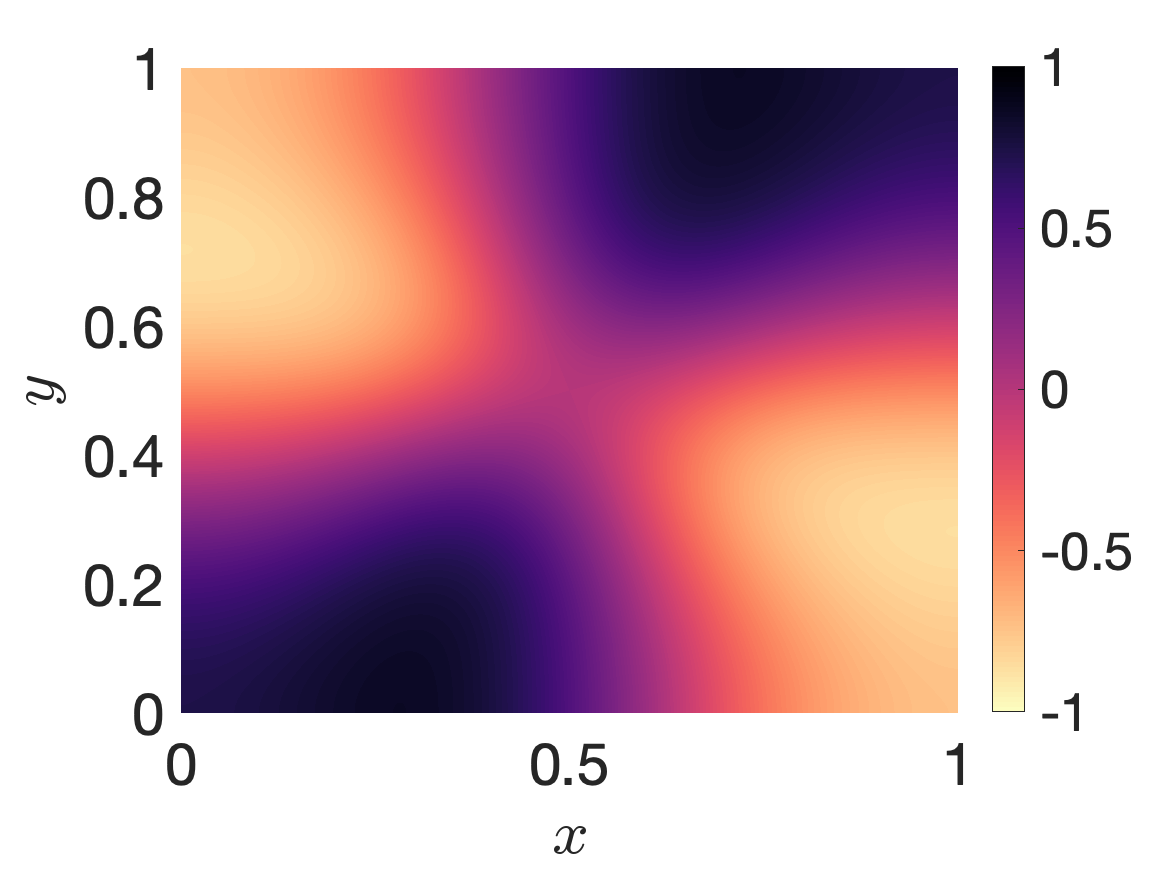}
\end{minipage}
\begin{minipage}[t]{0.24\textwidth}
\centering
\includegraphics[scale=0.19]{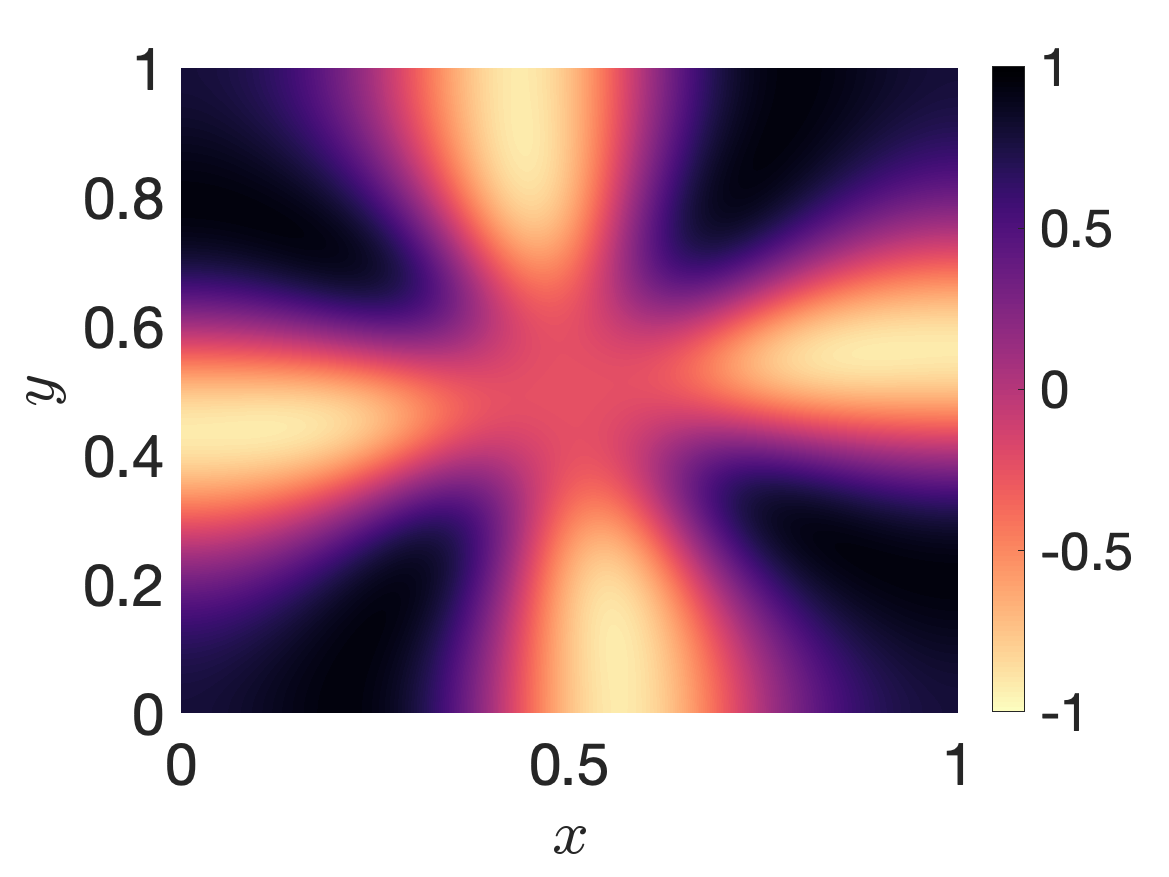}
\end{minipage}
\begin{minipage}[t]{0.24\textwidth}
\centering
\includegraphics[scale=0.19]{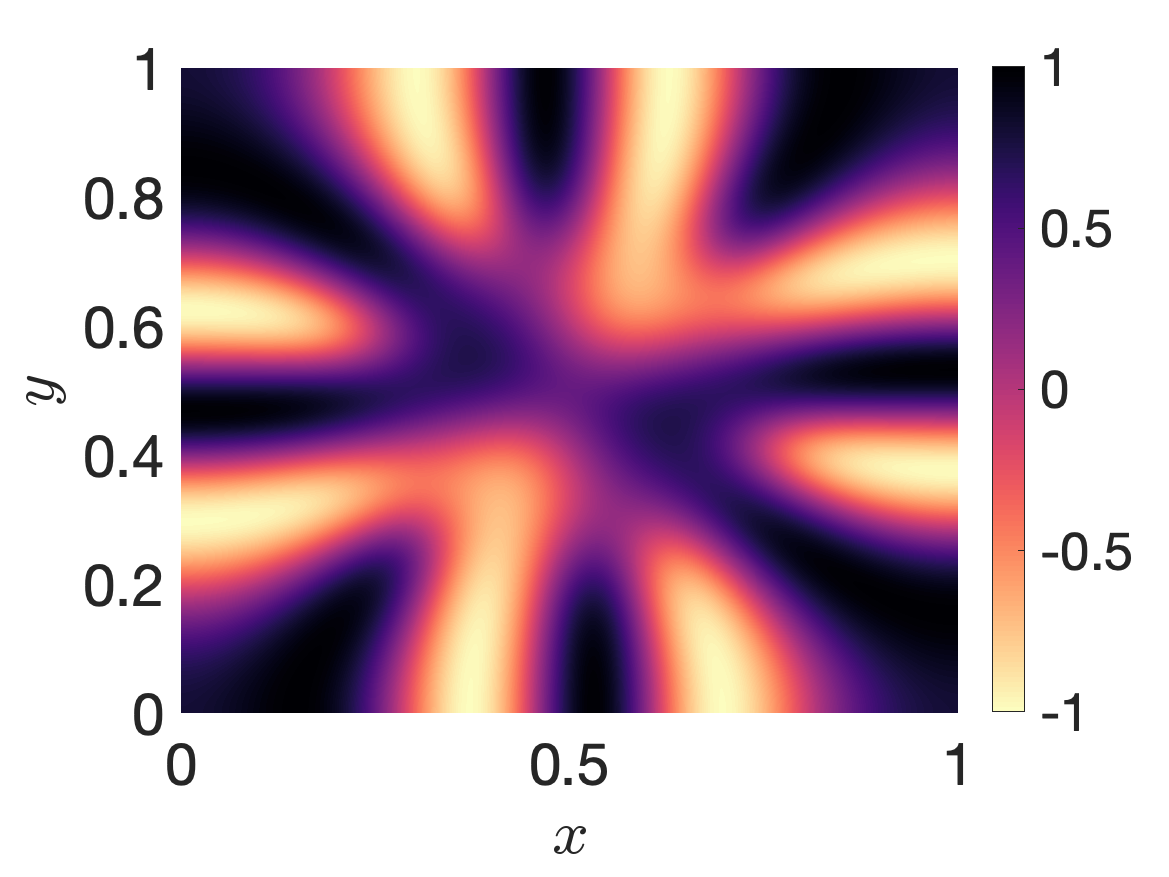}
\end{minipage}
\begin{minipage}[t]{0.24\textwidth}
\centering
\includegraphics[scale=0.19]{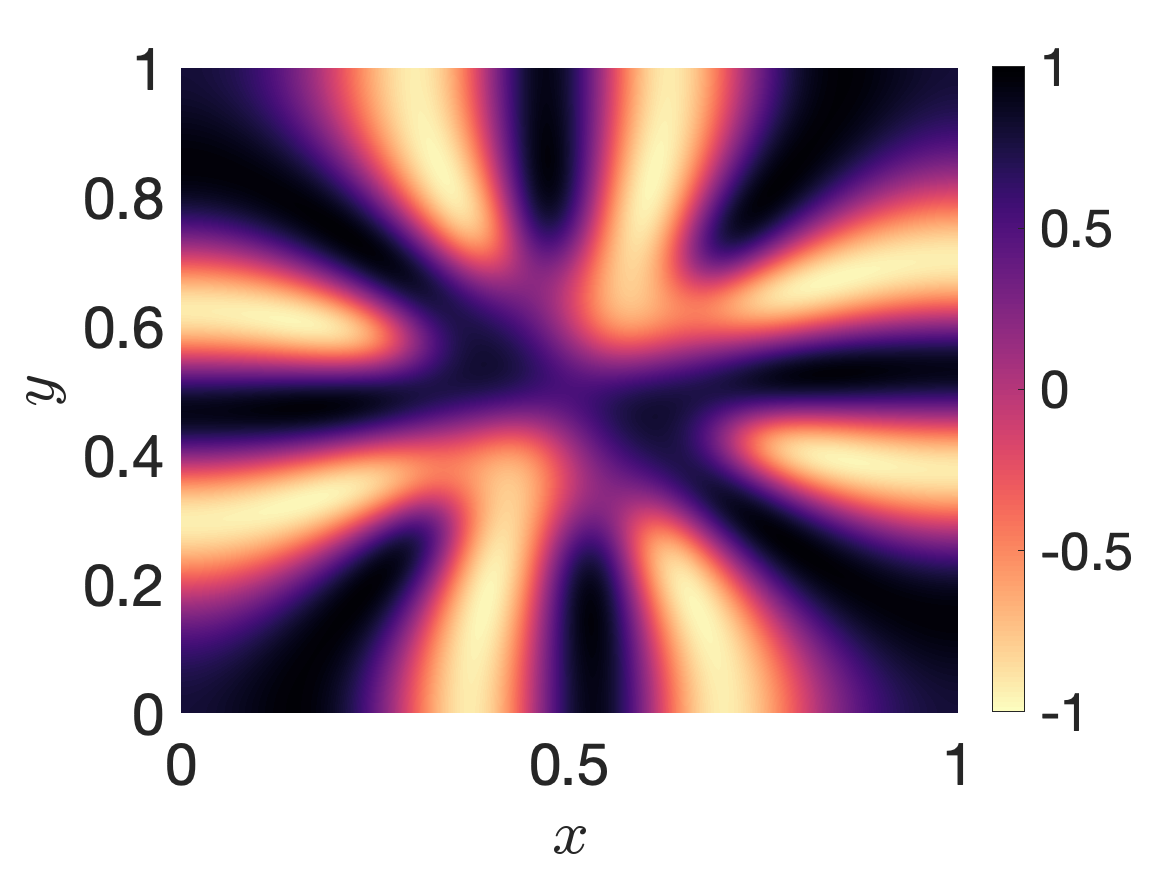}
\end{minipage} \\
\begin{minipage}[t]{0.24\textwidth}
\centering
\includegraphics[scale=0.19]{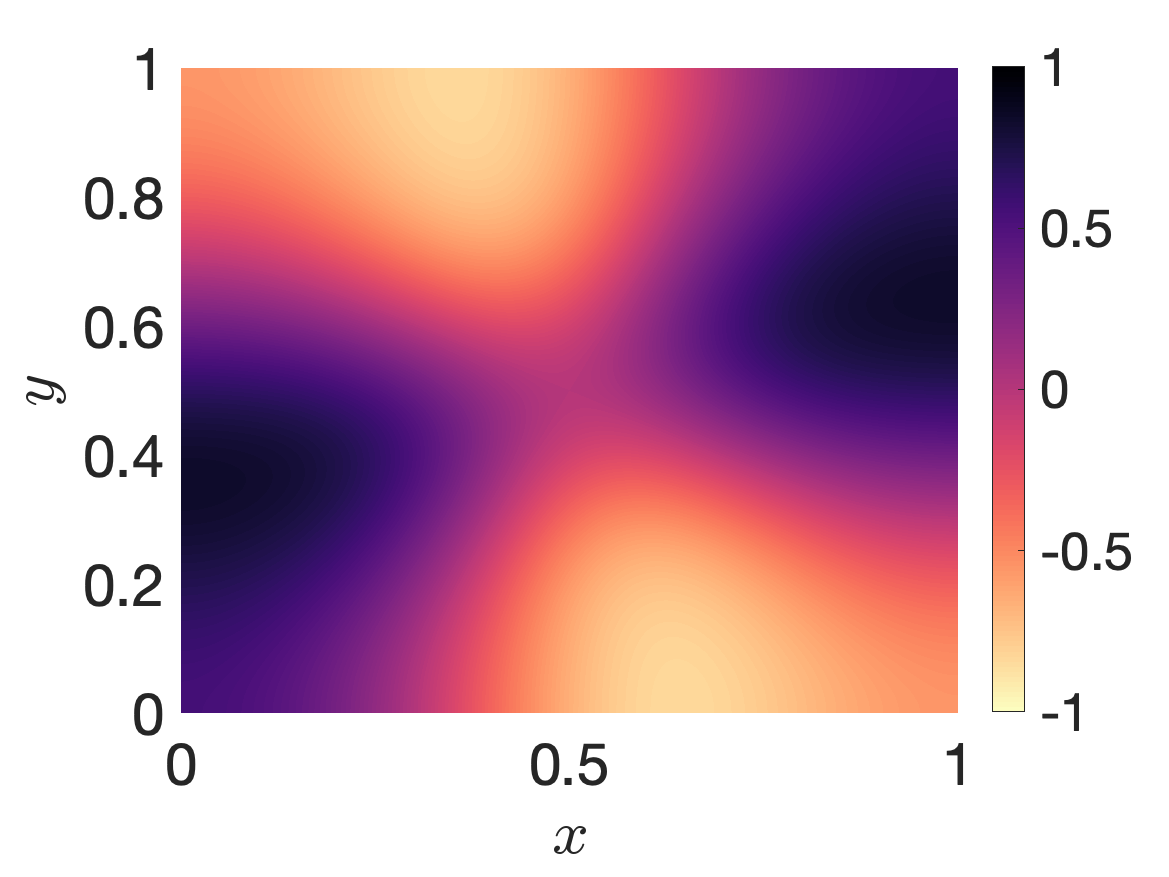}
\end{minipage}
\begin{minipage}[t]{0.24\textwidth}
\centering
\includegraphics[scale=0.19]{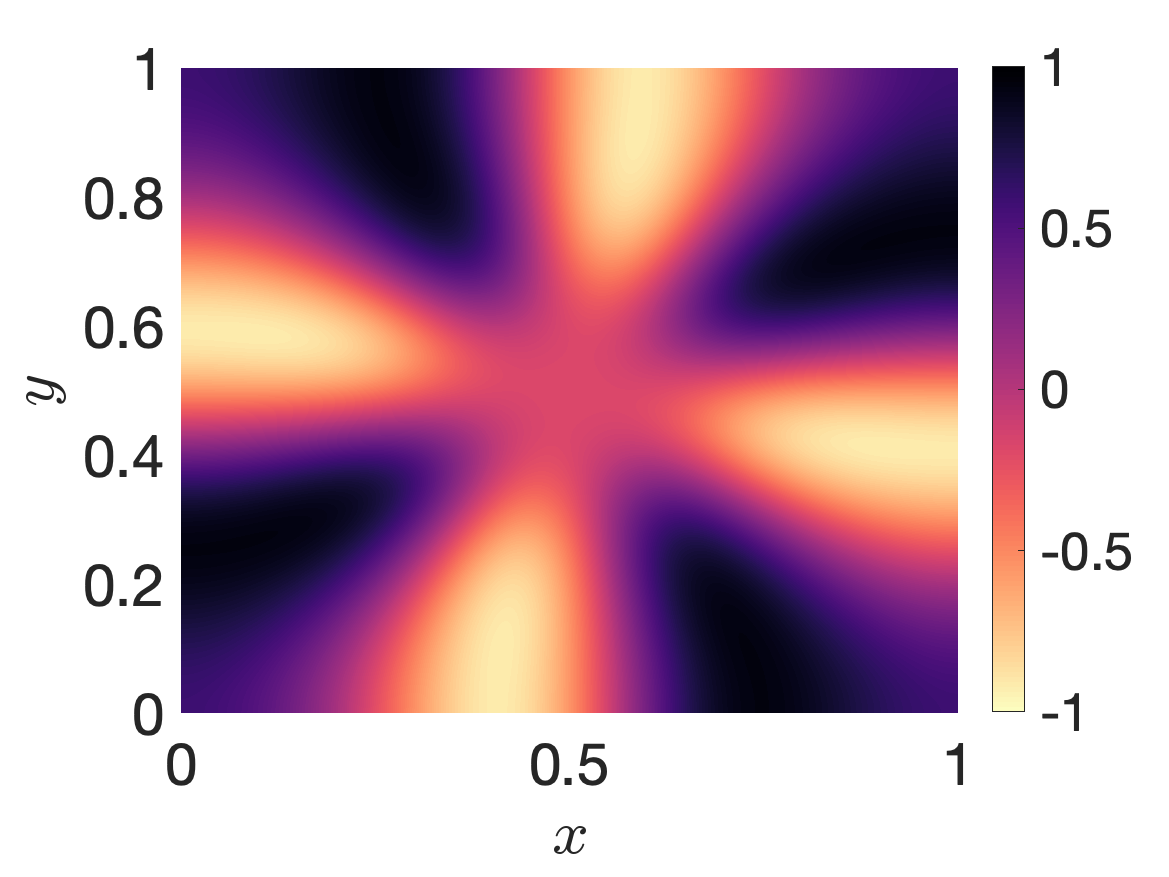}
\end{minipage}
\begin{minipage}[t]{0.24\textwidth}
\centering
\includegraphics[scale=0.19]{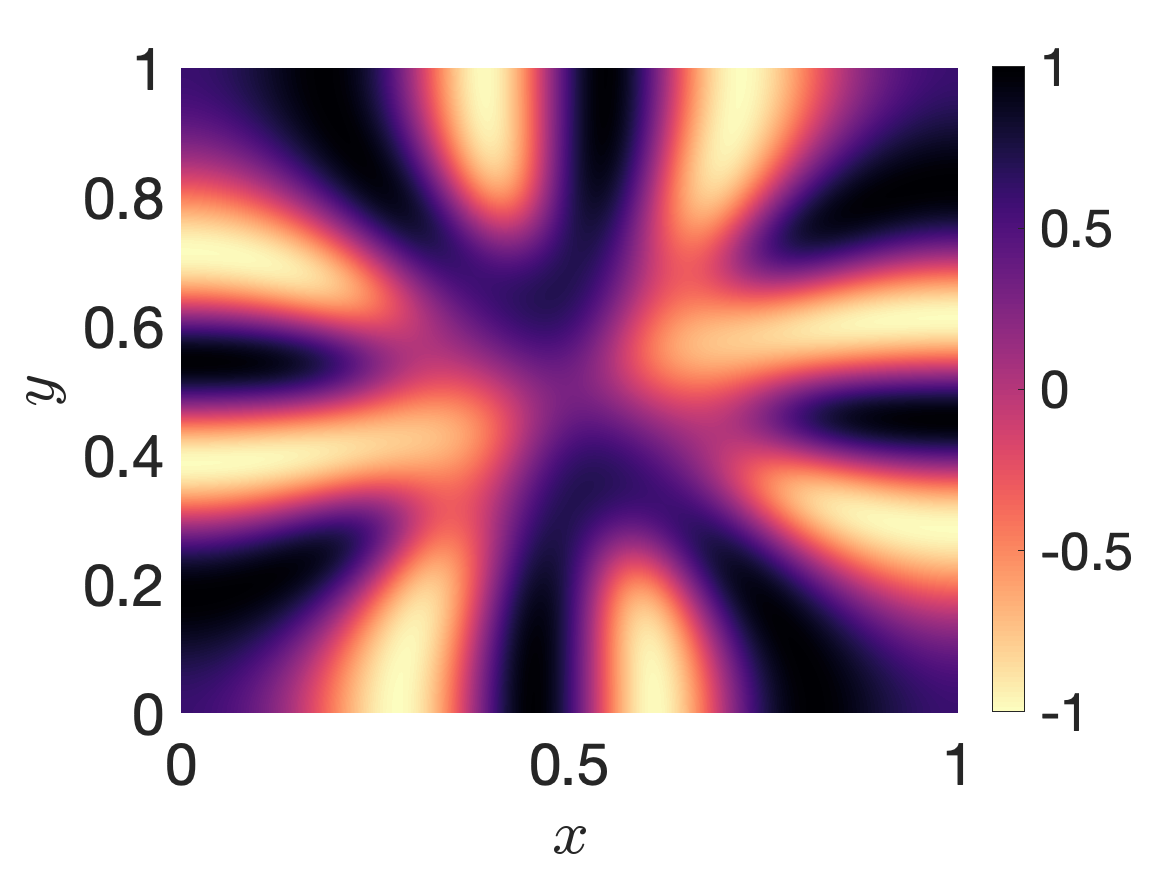}
\end{minipage}
\begin{minipage}[t]{0.24\textwidth}
\centering
\includegraphics[scale=0.19]{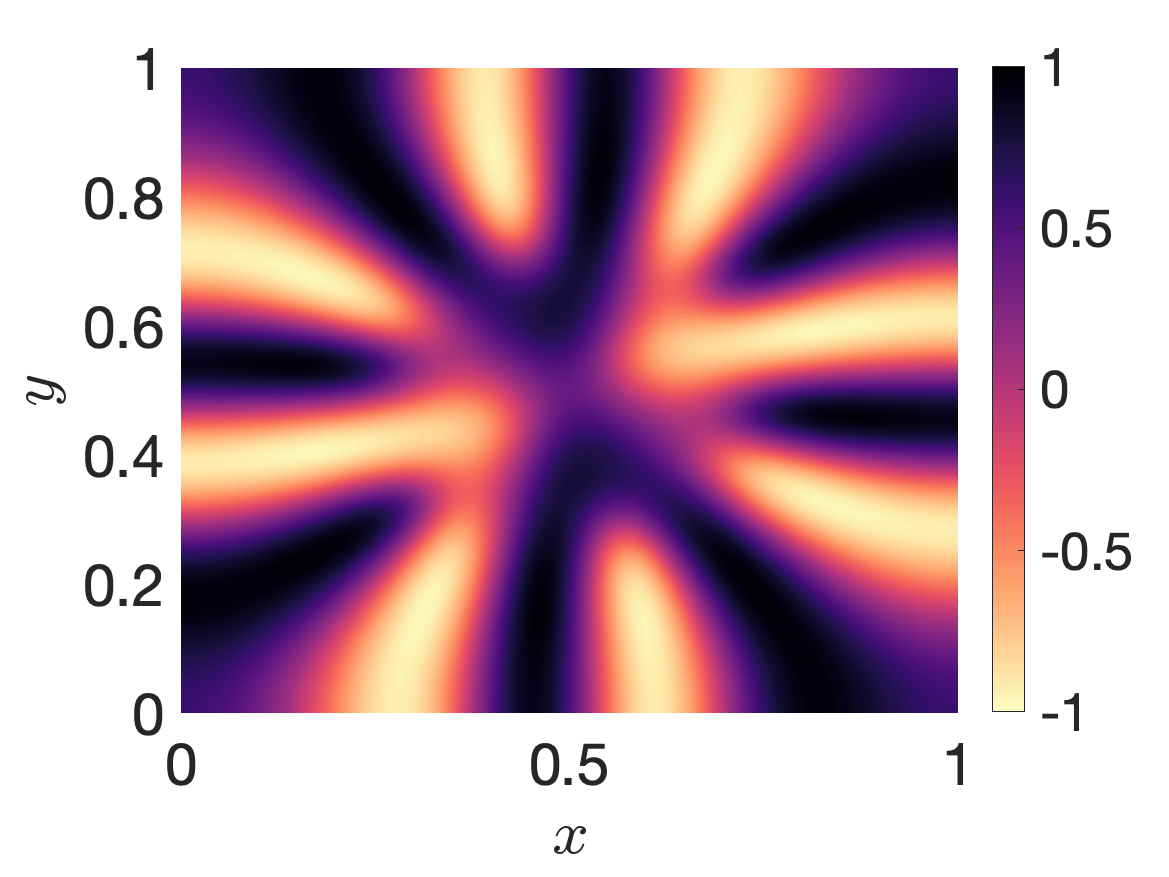}
\end{minipage} \\
\begin{minipage}[t]{0.24\textwidth}
\centering
\includegraphics[scale=0.19]{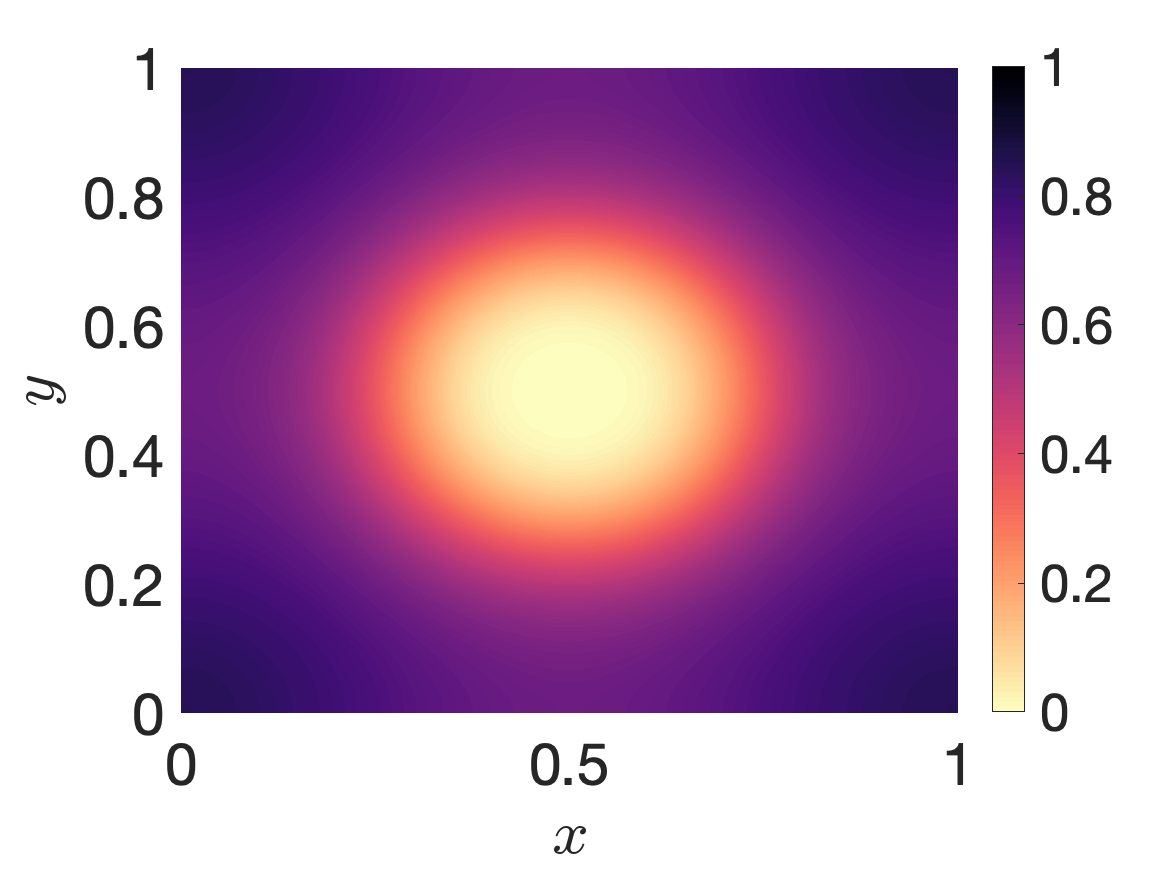}
\end{minipage}
\begin{minipage}[t]{0.24\textwidth}
\centering
\includegraphics[scale=0.19]{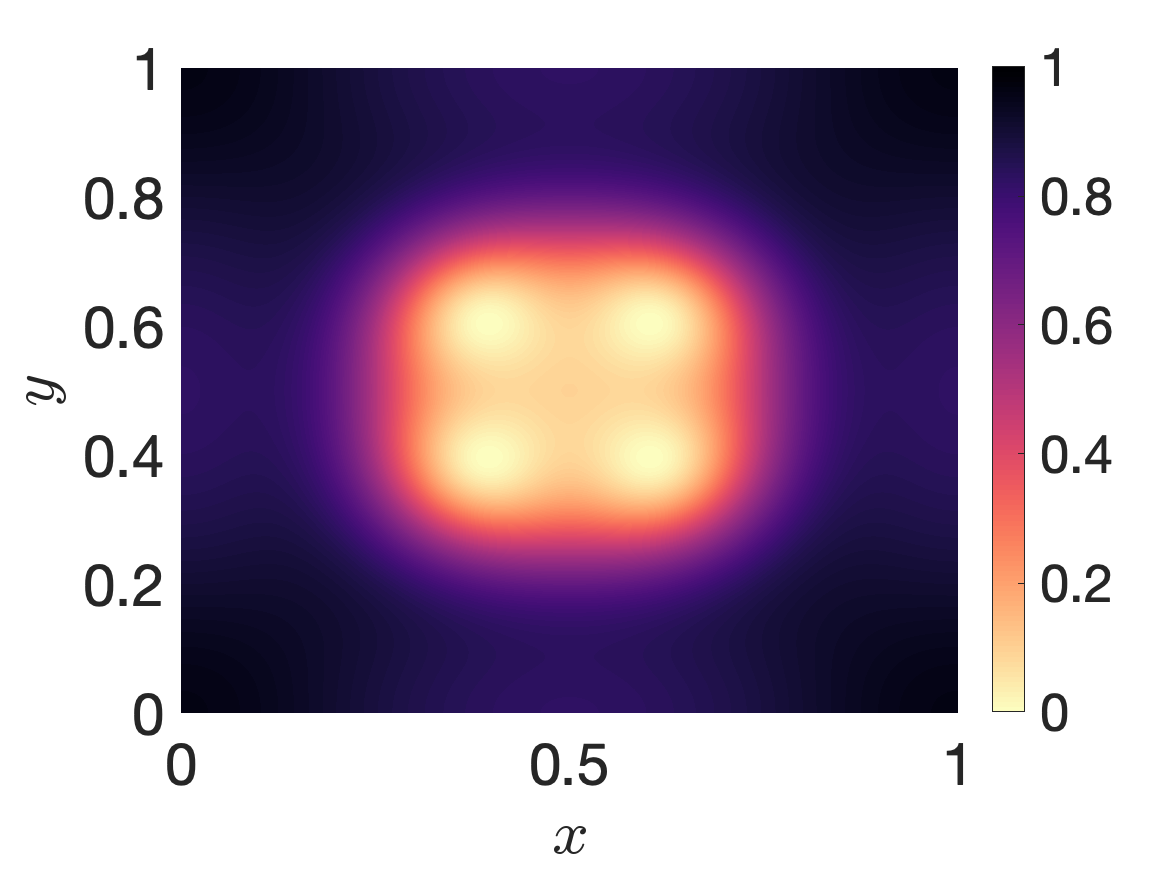}
\end{minipage}
\begin{minipage}[t]{0.24\textwidth}
\centering
\includegraphics[scale=0.19]{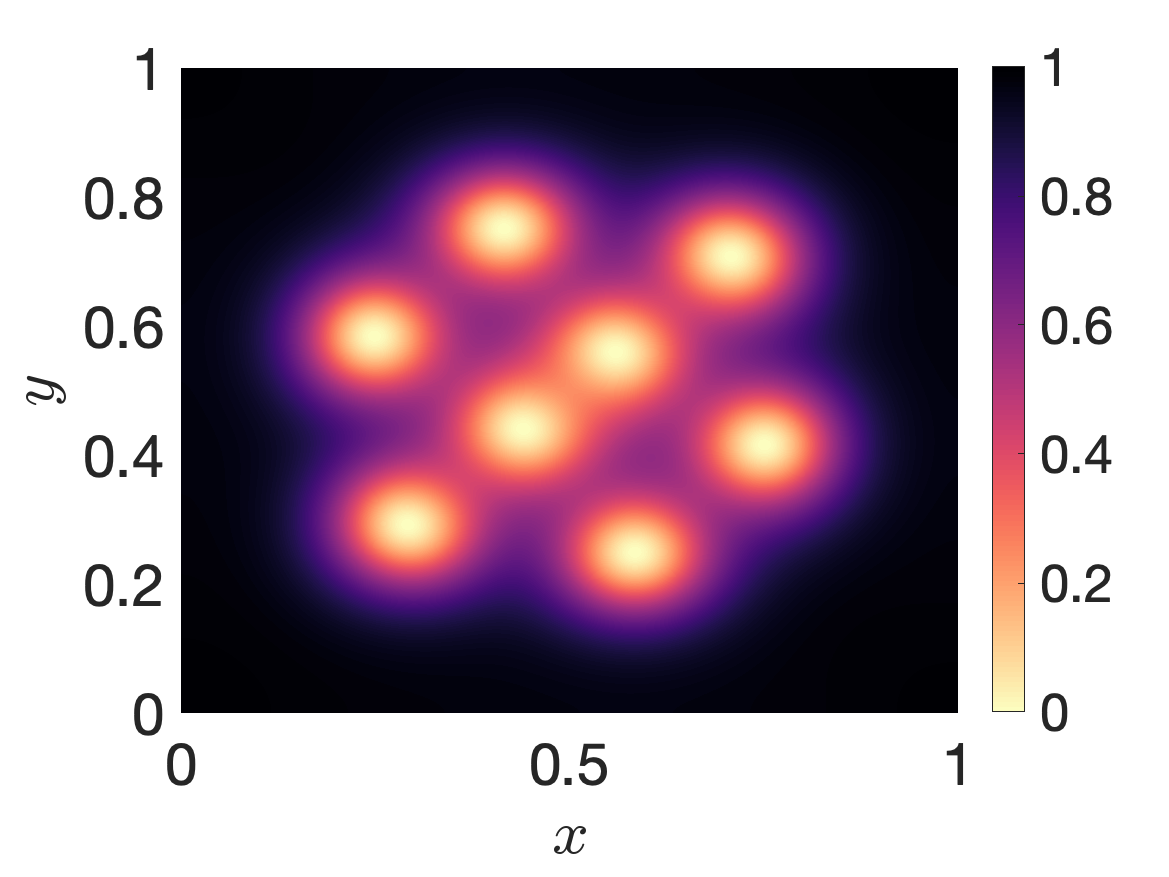}
\end{minipage}
\begin{minipage}[t]{0.24\textwidth}
\centering
\includegraphics[scale=0.19]{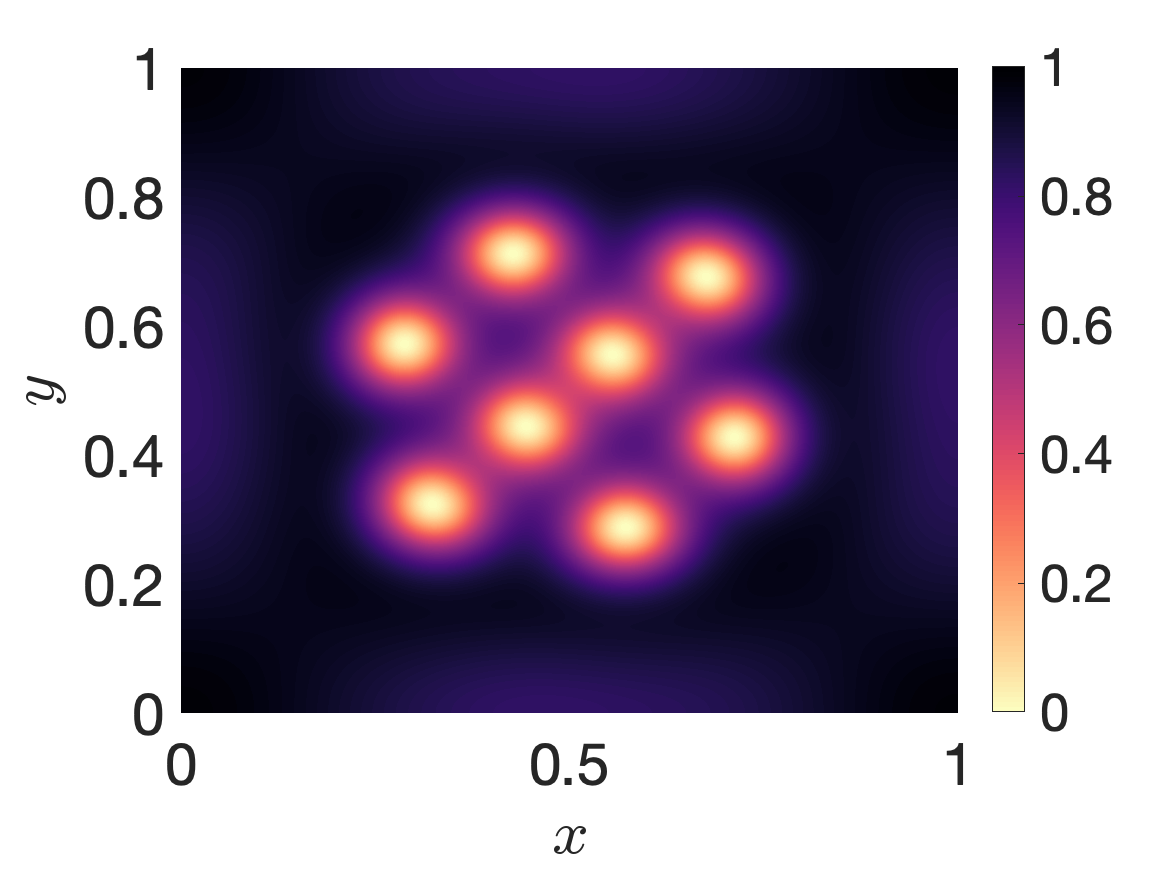}
\end{minipage}
\caption{Real part $\Real u$ (top row), imaginary part $\Imag u$ (middle row), and density $|u|^2$ (bottom row) of the order parameter component $\sol$ of the GL energy minimizer $(\sol,\MagPot)$ for $\kappa = 5, 10, 15, 20$ (left to right).}
\label{fig:reference-minimizers}
\end{figure}

\begin{figure}[h!] 
\centering
\begin{minipage}[t]{0.24\textwidth}
\centering
\includegraphics[scale=0.19]{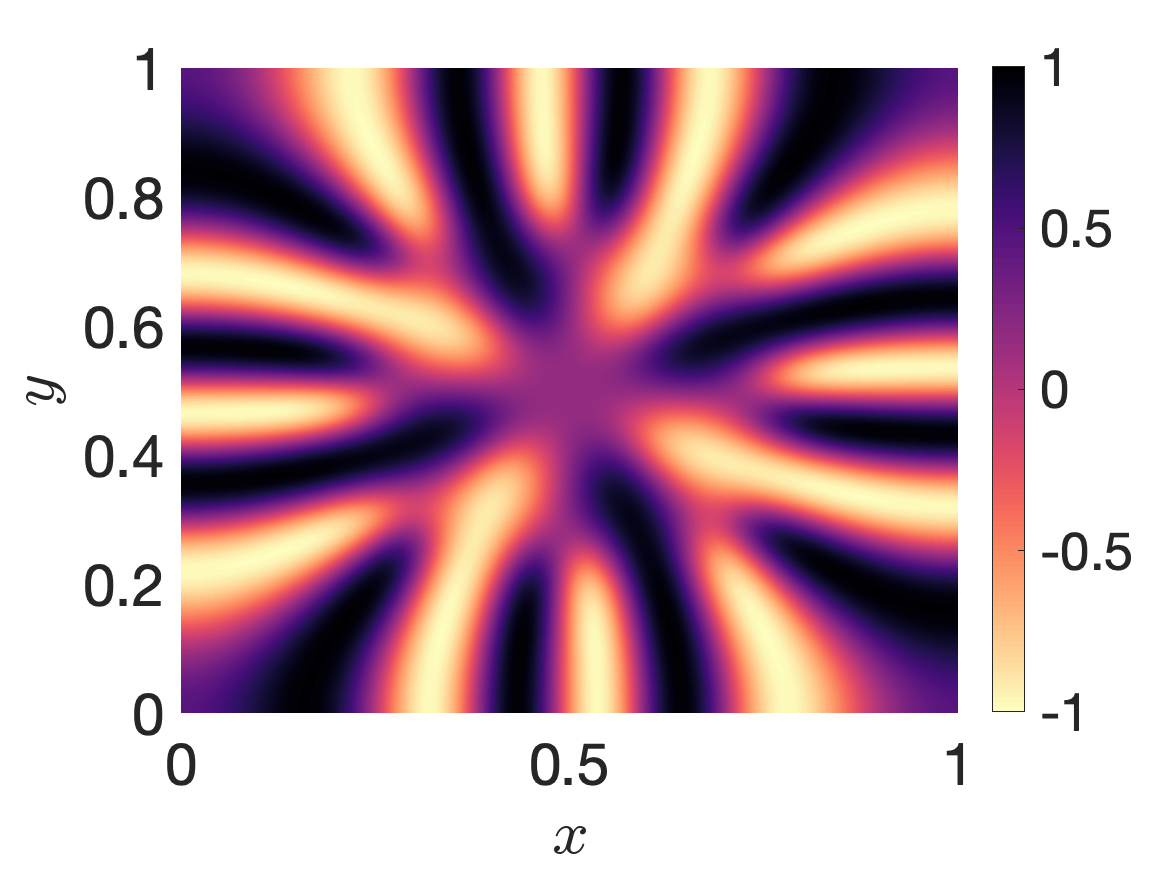}
\end{minipage}
\begin{minipage}[t]{0.24\textwidth}
\centering
\includegraphics[scale=0.19]{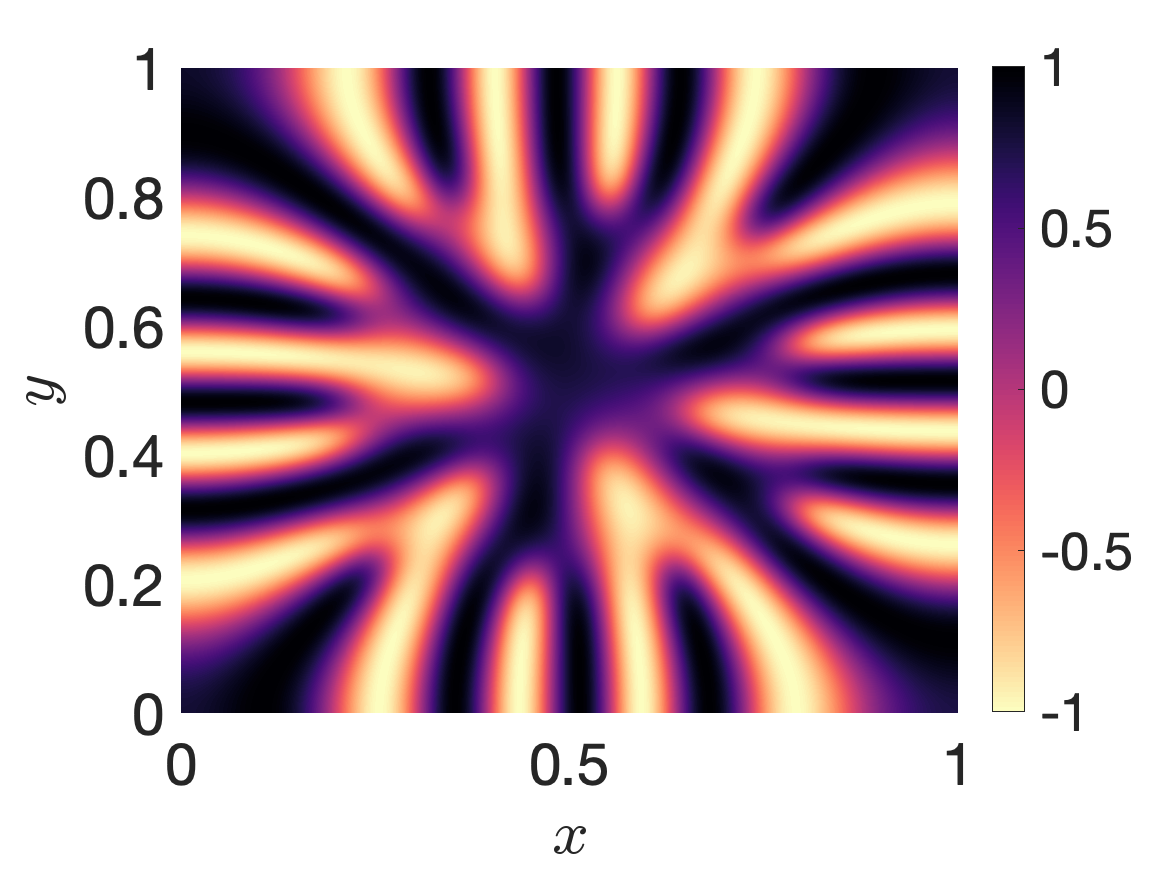}
\end{minipage}
\begin{minipage}[t]{0.24\textwidth}
\centering
\includegraphics[scale=0.19]{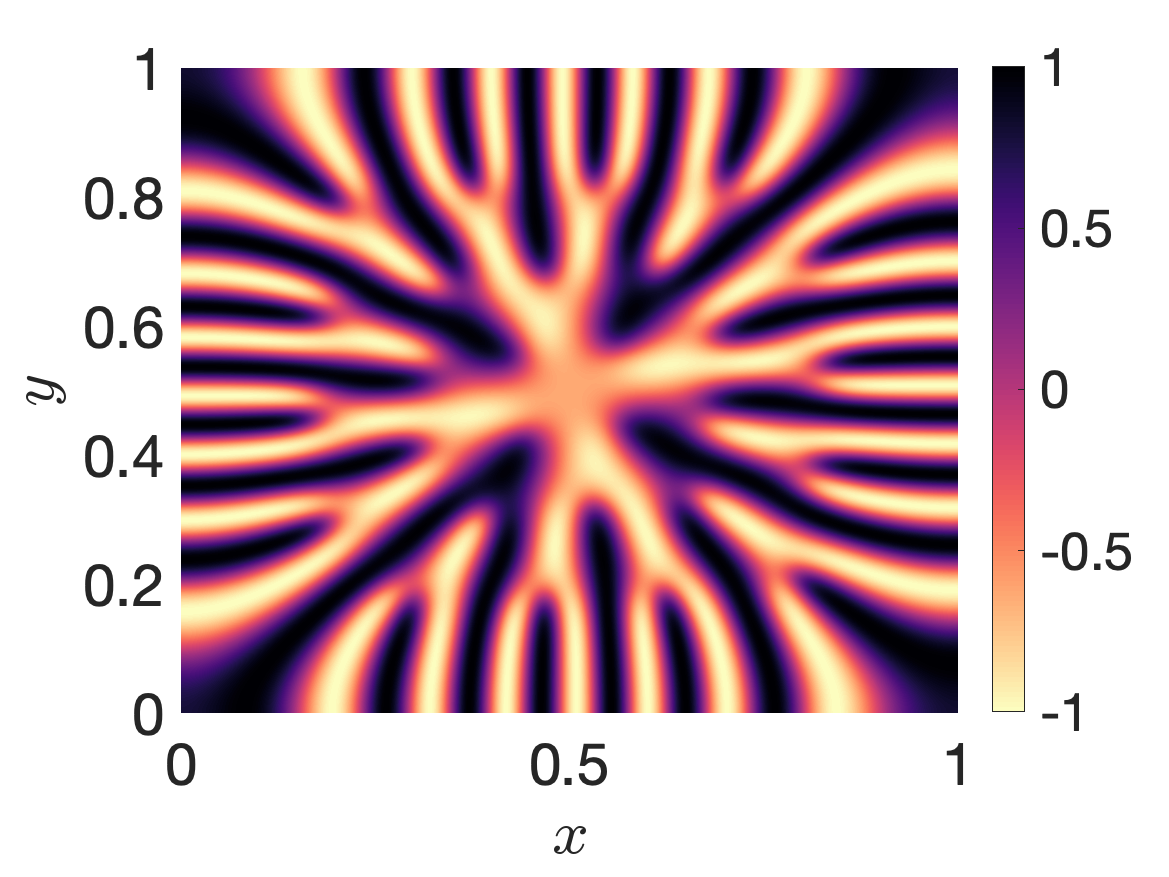}
\end{minipage}
\begin{minipage}[t]{0.24\textwidth}
\centering
\includegraphics[scale=0.19]{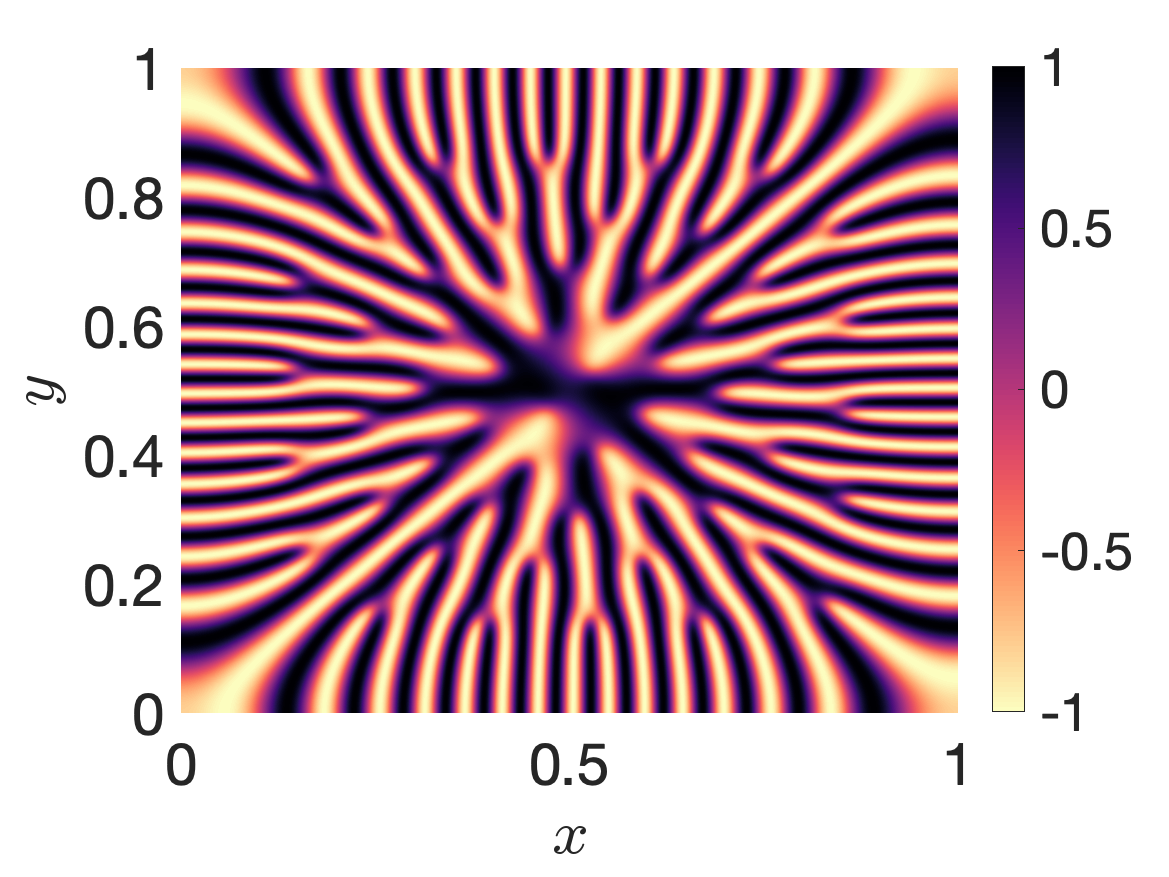}
\end{minipage} \\
\begin{minipage}[t]{0.24\textwidth}
\centering
\includegraphics[scale=0.19]{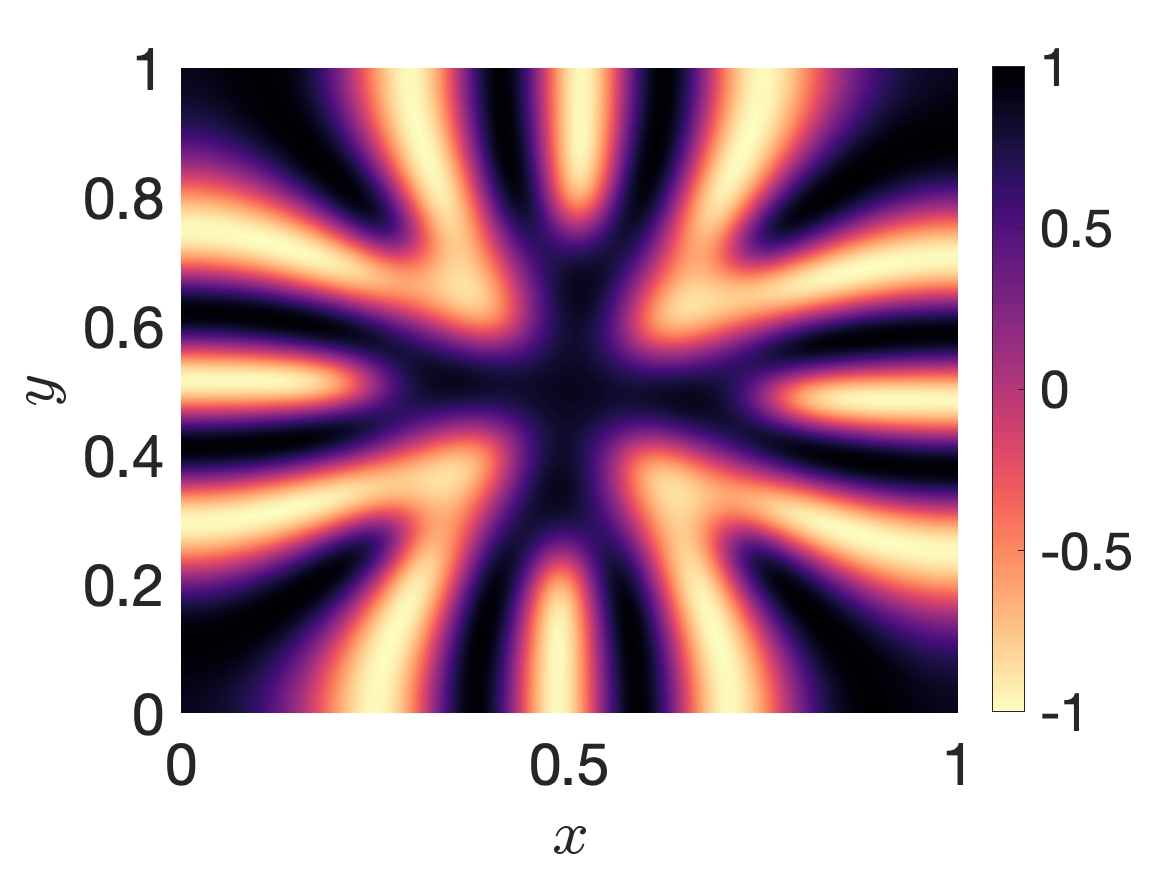}
\end{minipage}
\begin{minipage}[t]{0.24\textwidth}
\centering
\includegraphics[scale=0.19]{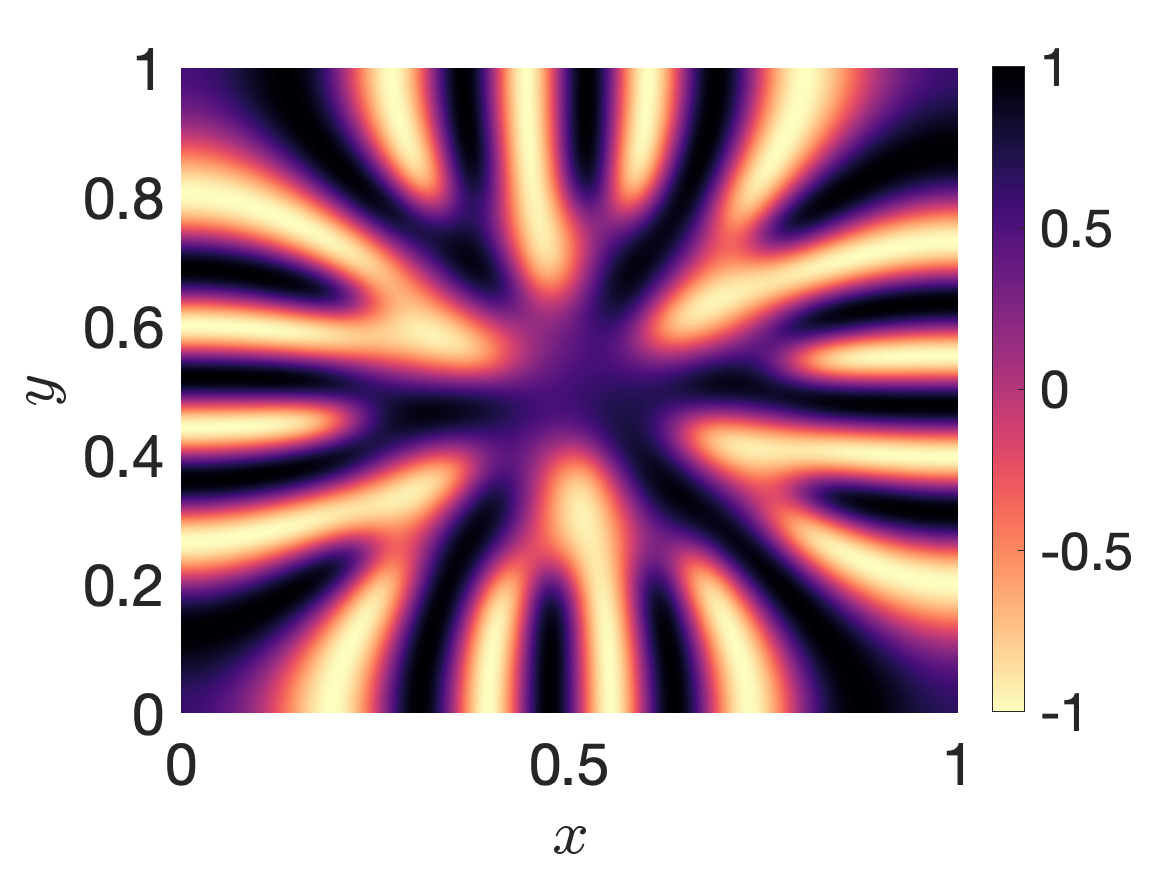}
\end{minipage}
\begin{minipage}[t]{0.24\textwidth}
\centering
\includegraphics[scale=0.19]{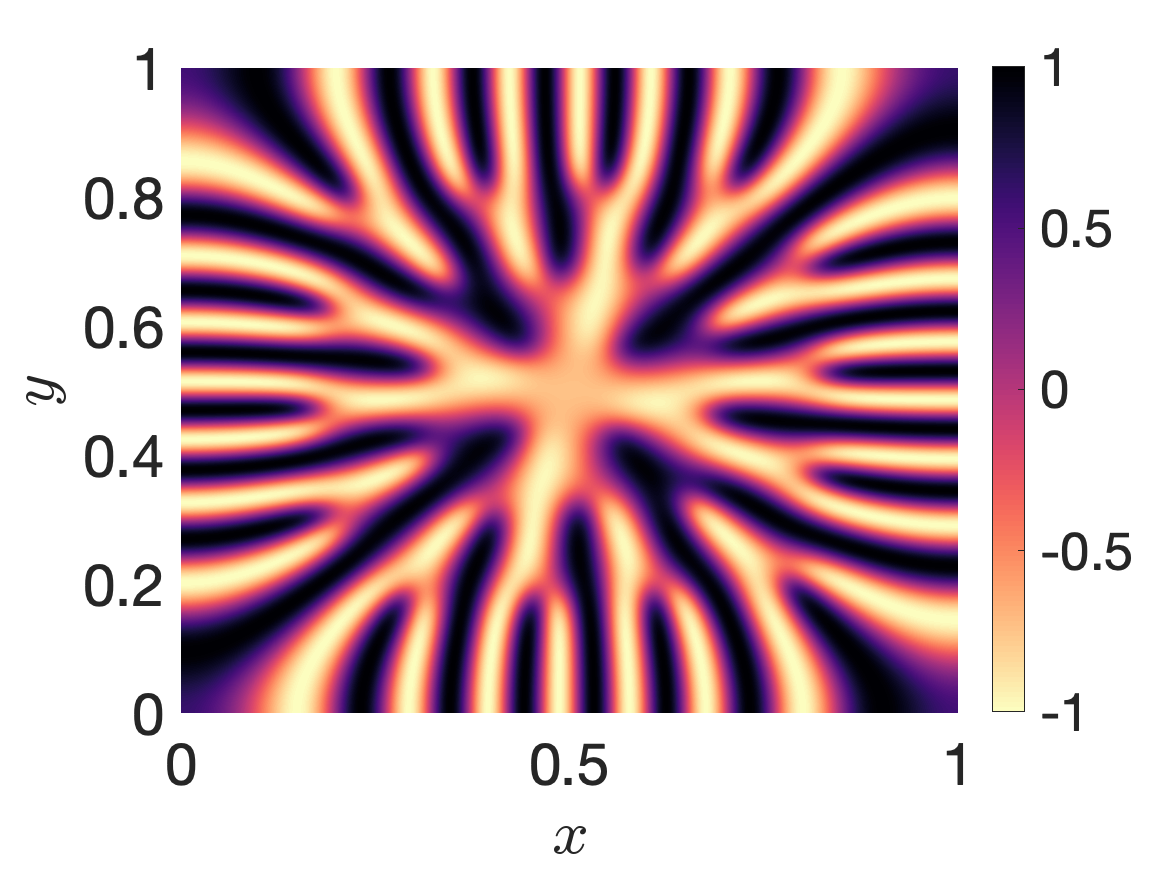}
\end{minipage}
\begin{minipage}[t]{0.24\textwidth}
\centering
\includegraphics[scale=0.19]{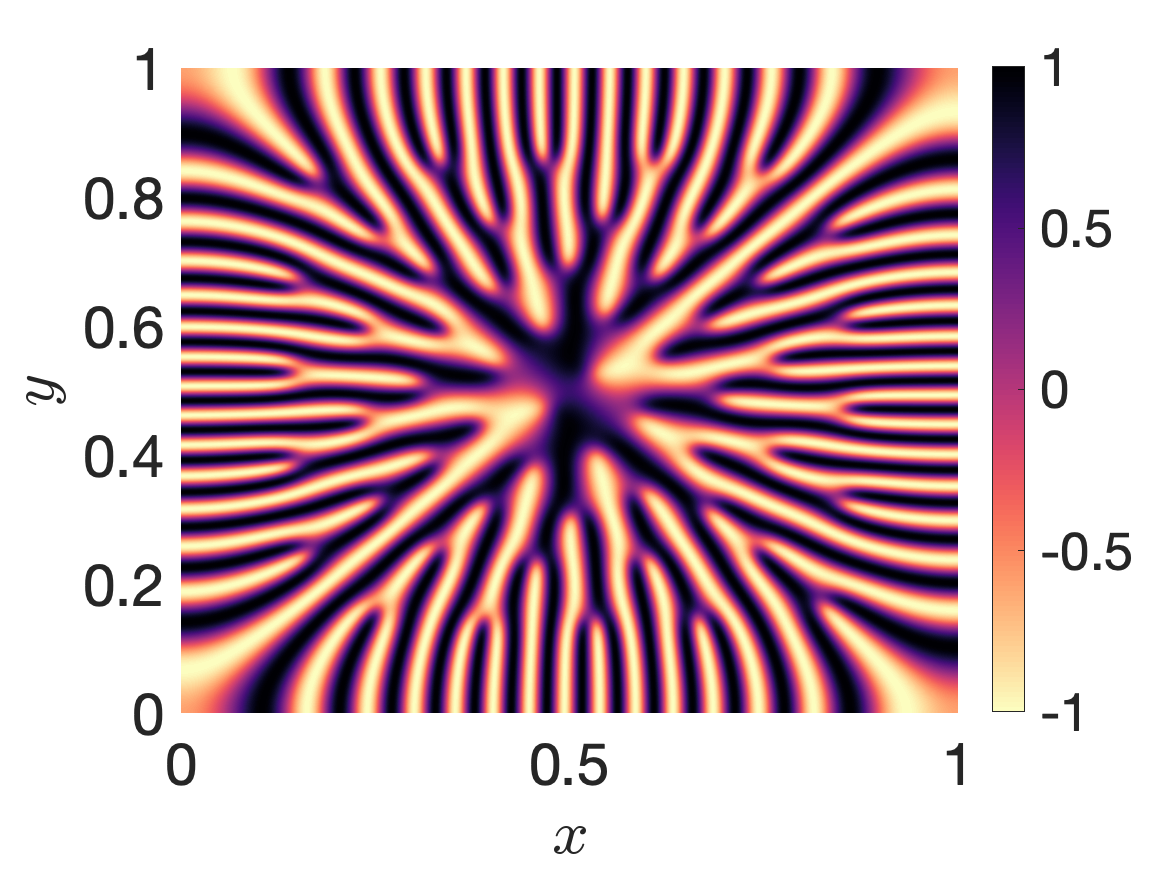}
\end{minipage} \\
\begin{minipage}[t]{0.24\textwidth}
\centering
\includegraphics[scale=0.19]{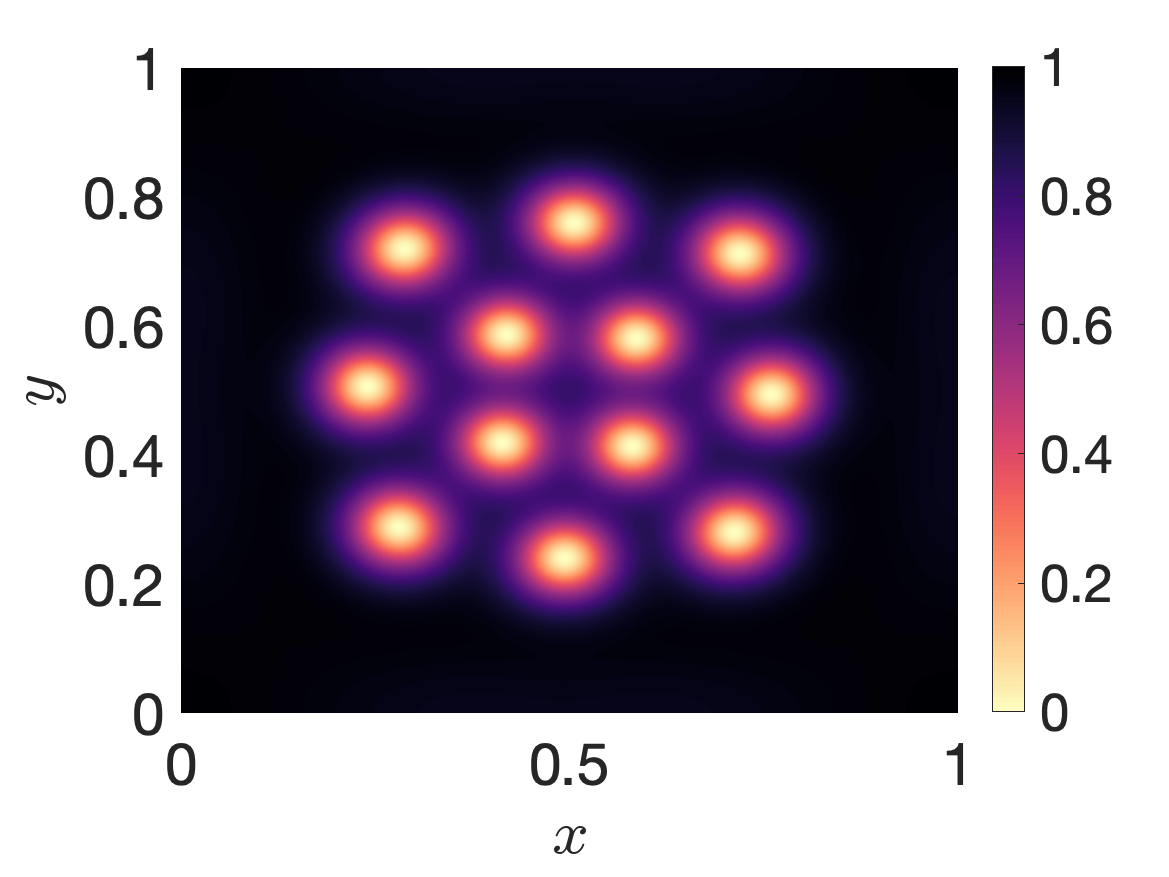}
\end{minipage}
\begin{minipage}[t]{0.24\textwidth}
\centering
\includegraphics[scale=0.19]{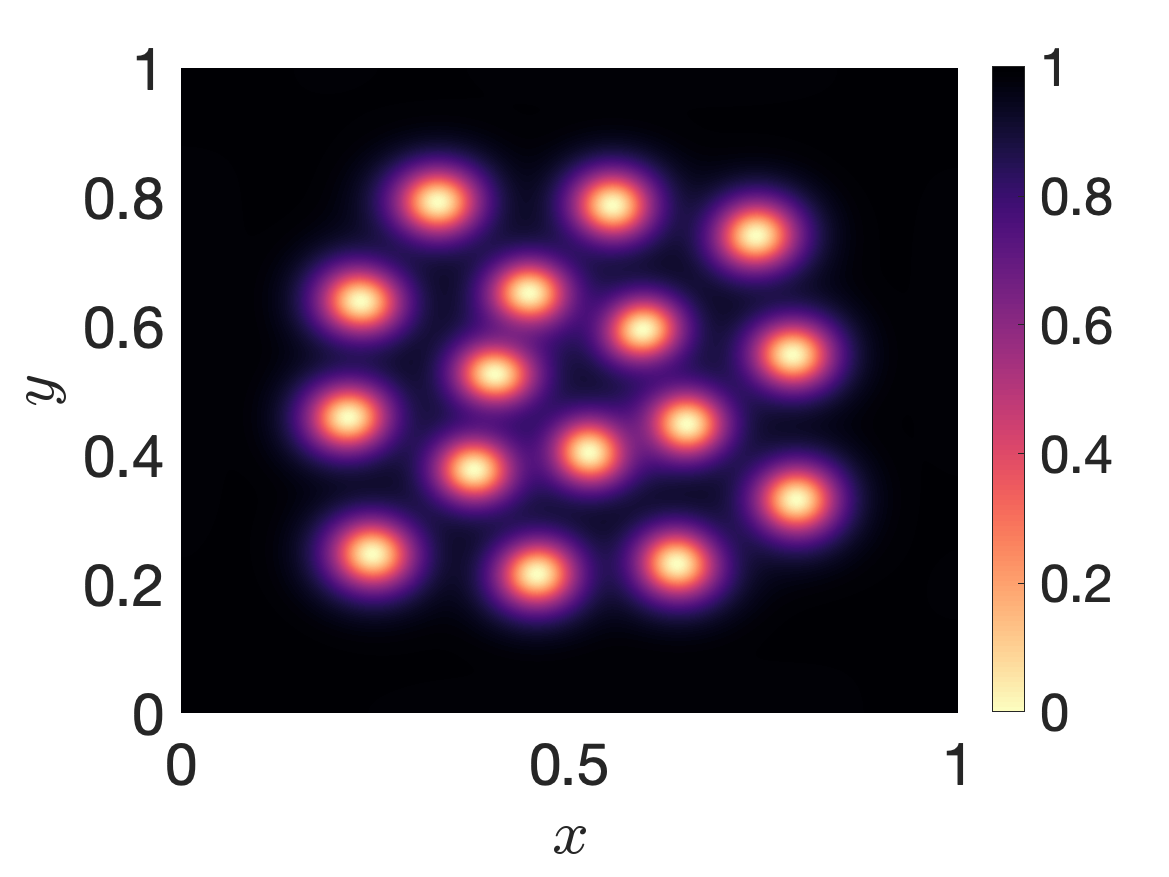}
\end{minipage}
\begin{minipage}[t]{0.24\textwidth}
\centering
\includegraphics[scale=0.19]{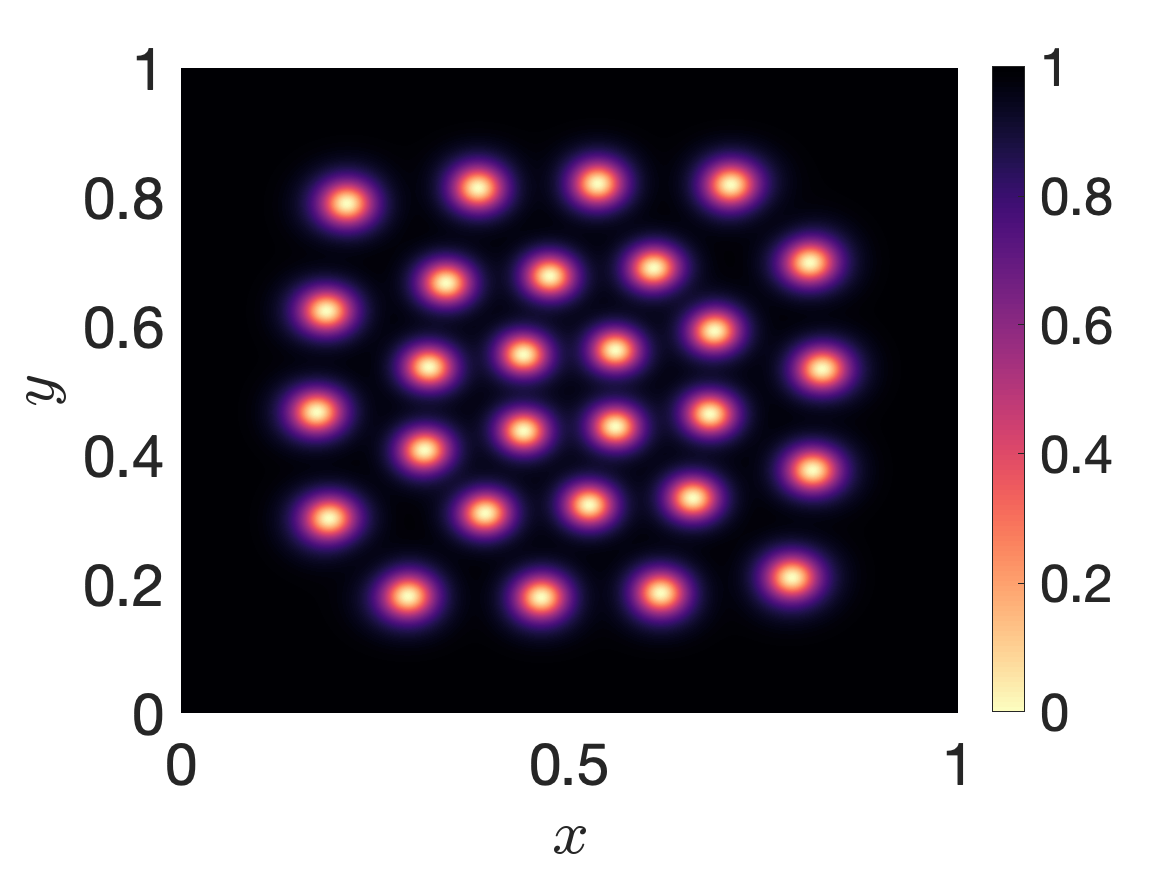}
\end{minipage}
\begin{minipage}[t]{0.24\textwidth}
\centering
\includegraphics[scale=0.19]{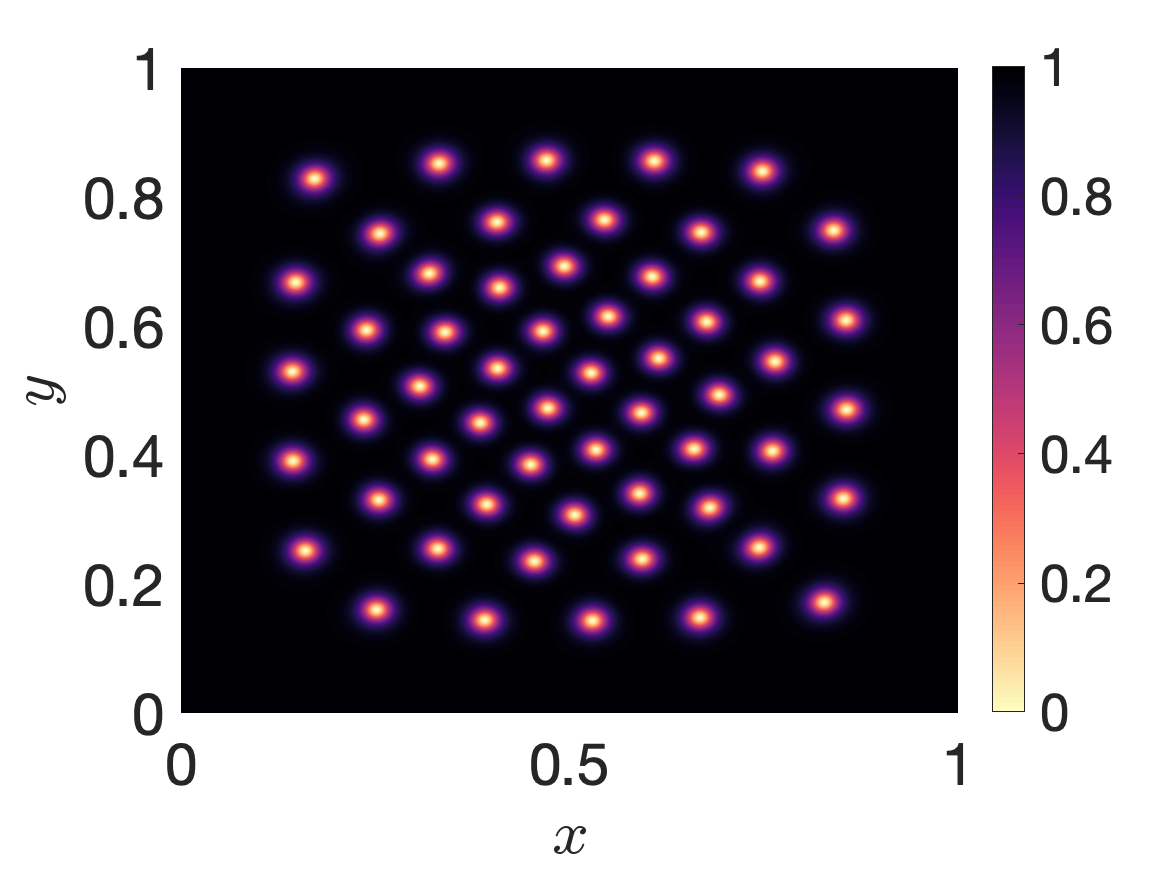}
\end{minipage}
\caption{Real part $\Real u$ (top row), imaginary part $\Imag u$ (middle row), and density $|u|^2$ (bottom row) of the order parameter component $\sol$ of the GL energy minimizer $(\sol,\MagPot)$ for $\kappa = 25, 30, 50, 100$ (left to right).}
\label{fig:reference-minimizers_2}
\end{figure}

\begin{figure}[h!] 
\centering
\begin{minipage}[t]{0.26\textwidth}
\centering
\includegraphics[scale=0.2]{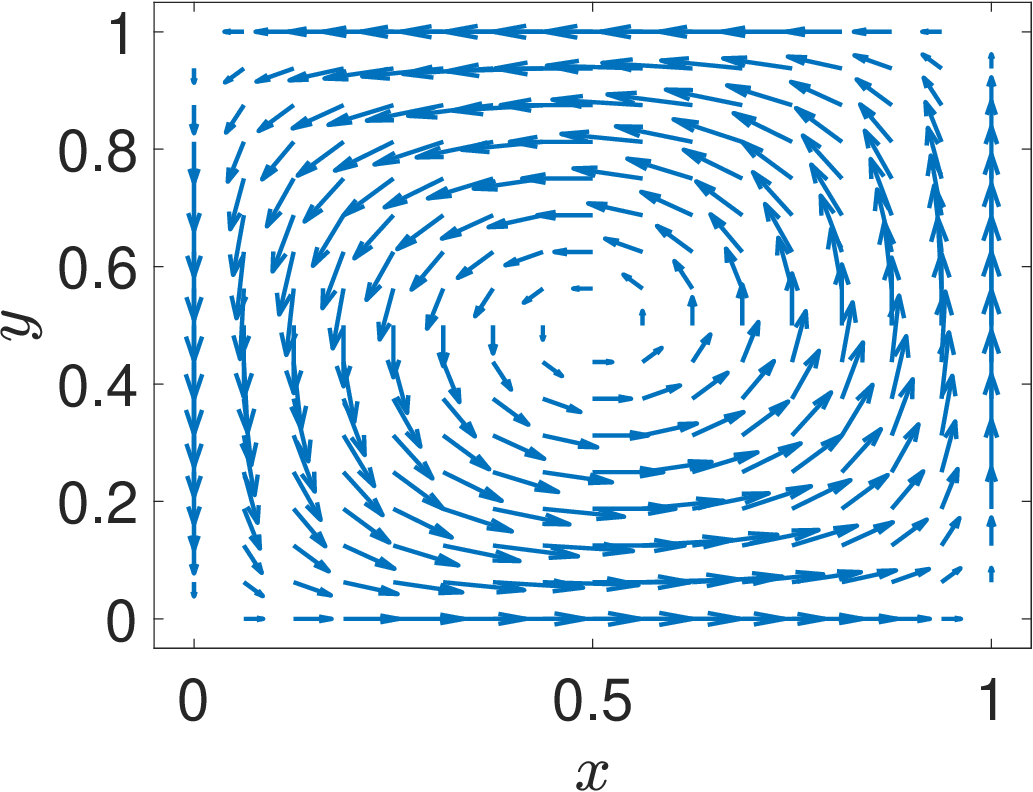}
\end{minipage}
\begin{minipage}[t]{0.26\textwidth}
\centering
\includegraphics[scale=0.2]{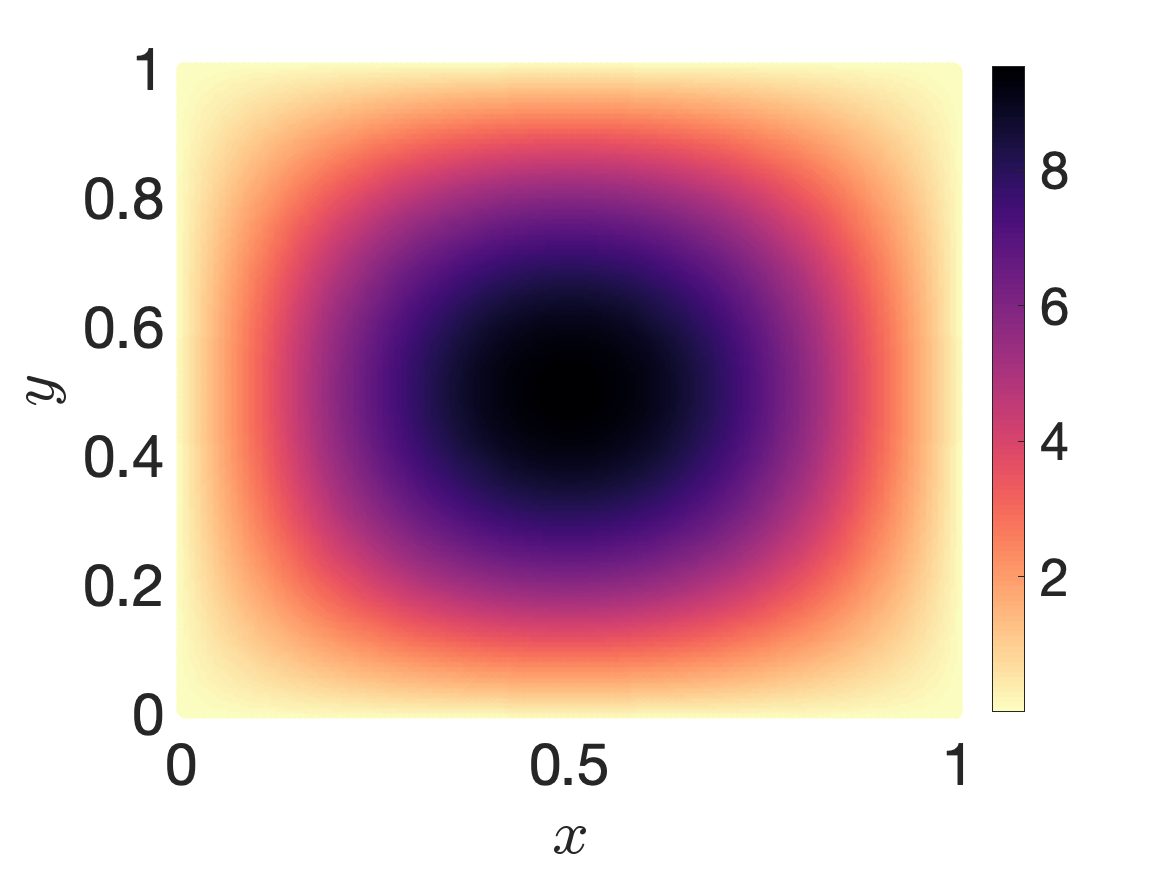}
\end{minipage}
\caption{The vector potential, $\MagPot$ (left, plotted on a coarse mesh), and $\curl \MagPot$ (middle), of the GL energy minimizer $(\sol,\MagPot)$ for $\kappa = 5$, representative of all values of $\kappa = 5, 10, 15, 20, 25, 30, 50, 100$.}
\label{fig:reference-minimizers-A}
\end{figure}

\begin{figure}[h!] 
\centering
\begin{minipage}[t]{0.24\textwidth}
\centering
\includegraphics[scale=0.19]{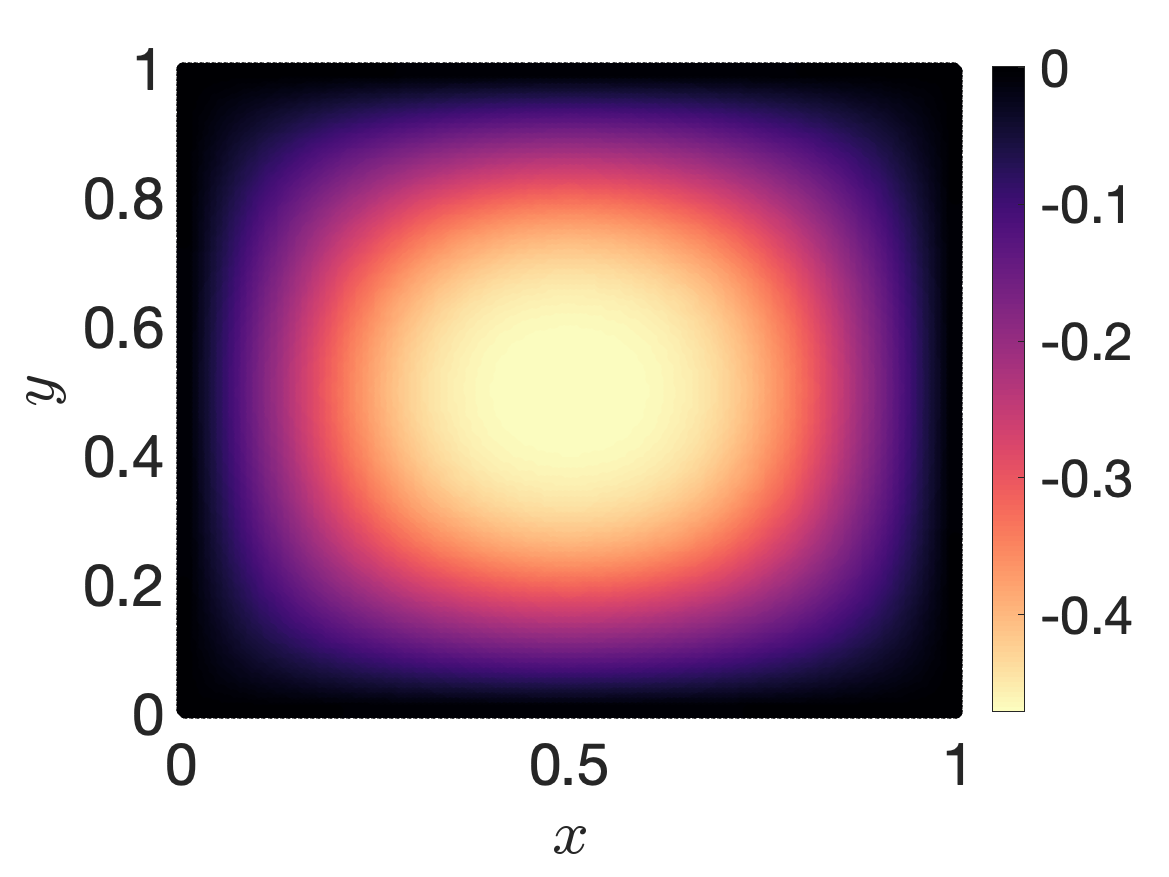}
\end{minipage}
\begin{minipage}[t]{0.24\textwidth}
\centering
\includegraphics[scale=0.19]{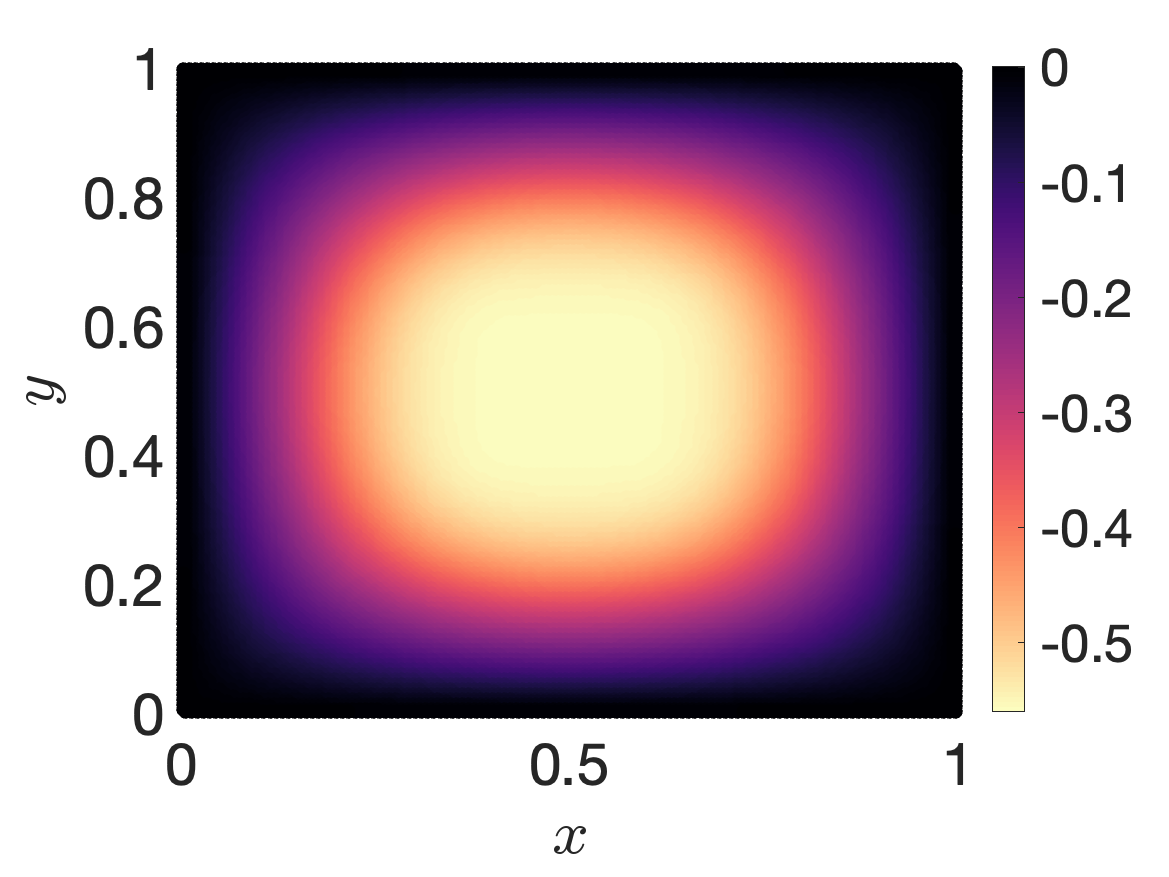}
\end{minipage}
\begin{minipage}[t]{0.24\textwidth}
\centering
\includegraphics[scale=0.19]{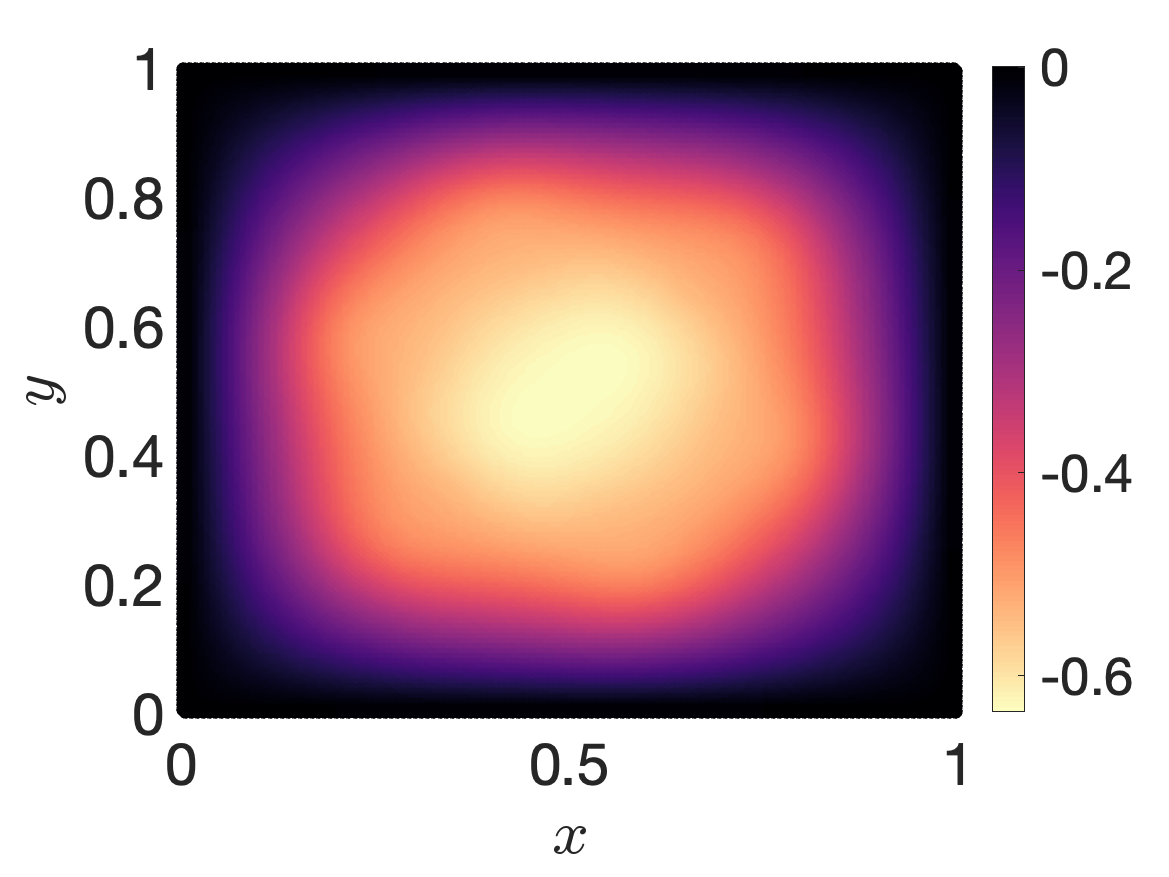}
\end{minipage}
\begin{minipage}[t]{0.24\textwidth}
\centering
\includegraphics[scale=0.19]{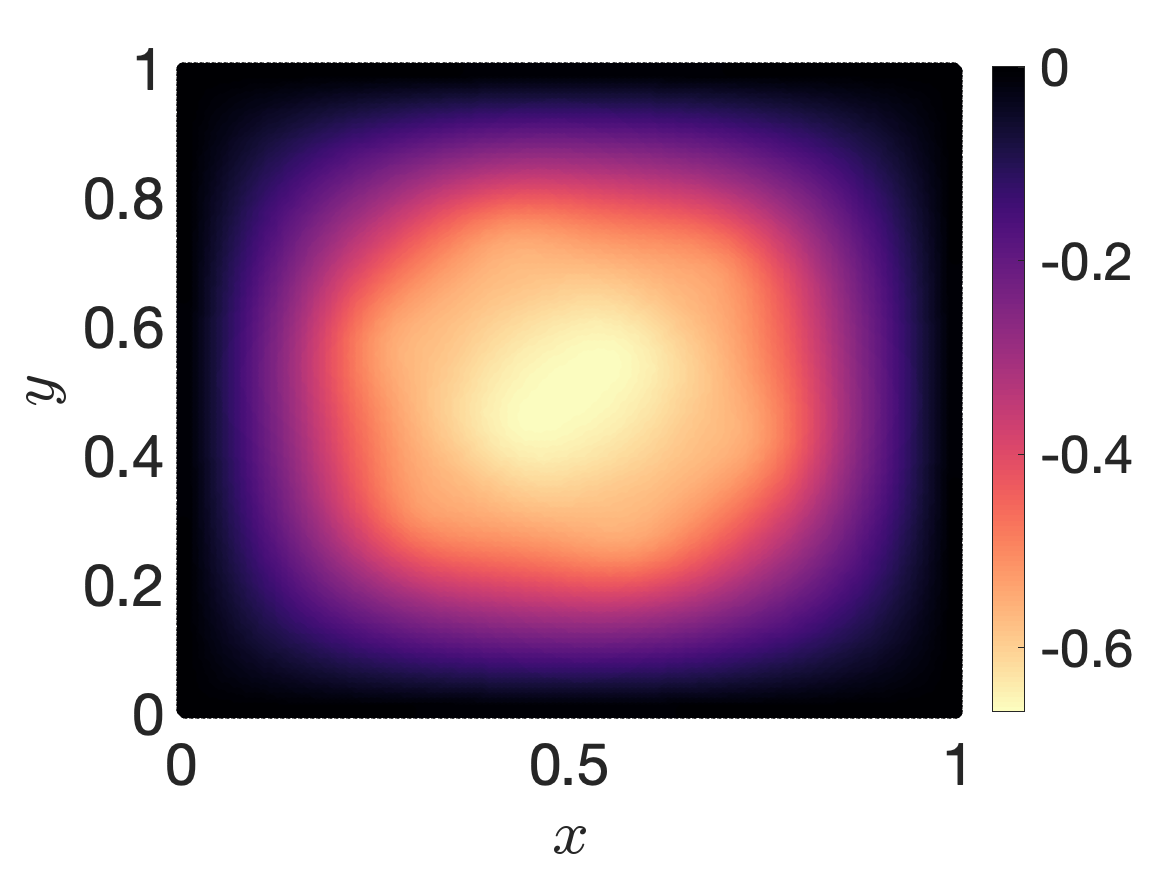}
\end{minipage} \\
\begin{minipage}[t]{0.24\textwidth}
\centering
\includegraphics[scale=0.19]{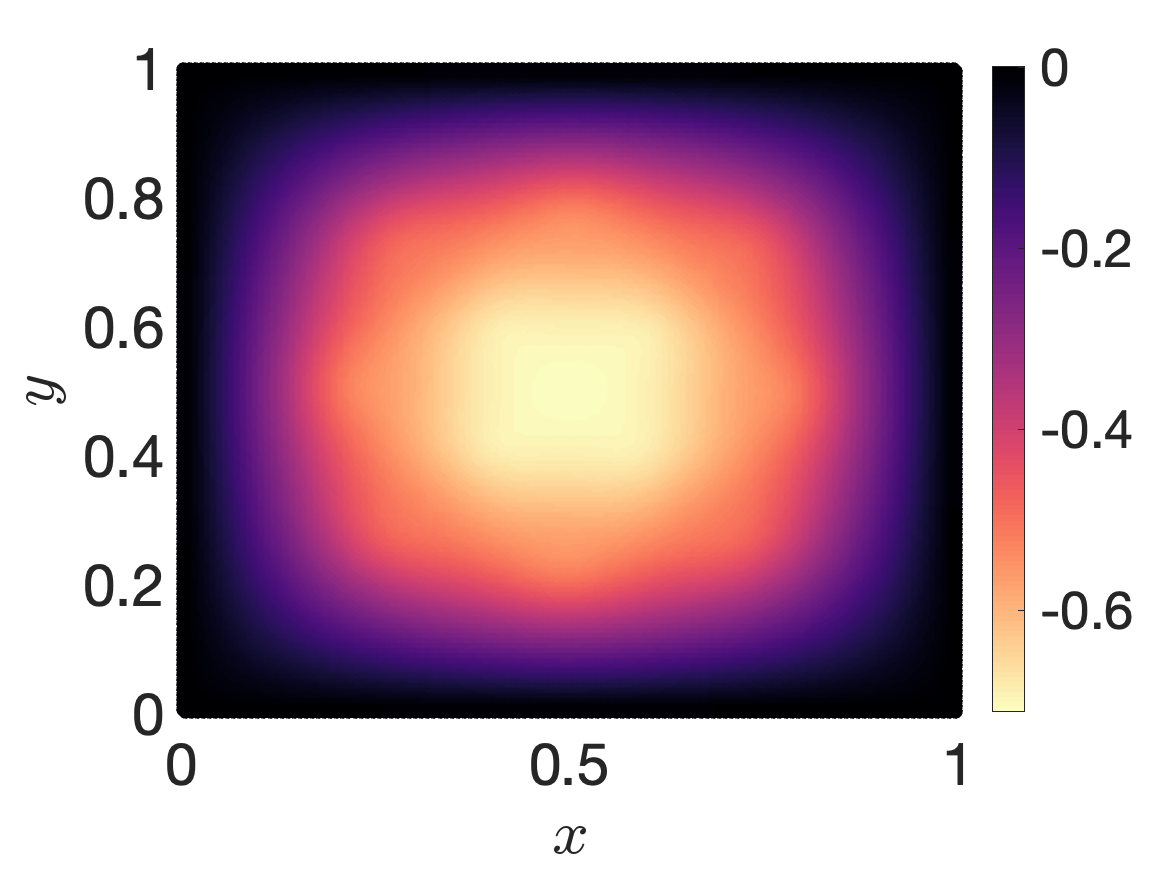}
\end{minipage}
\begin{minipage}[t]{0.24\textwidth}
\centering
\includegraphics[scale=0.19]{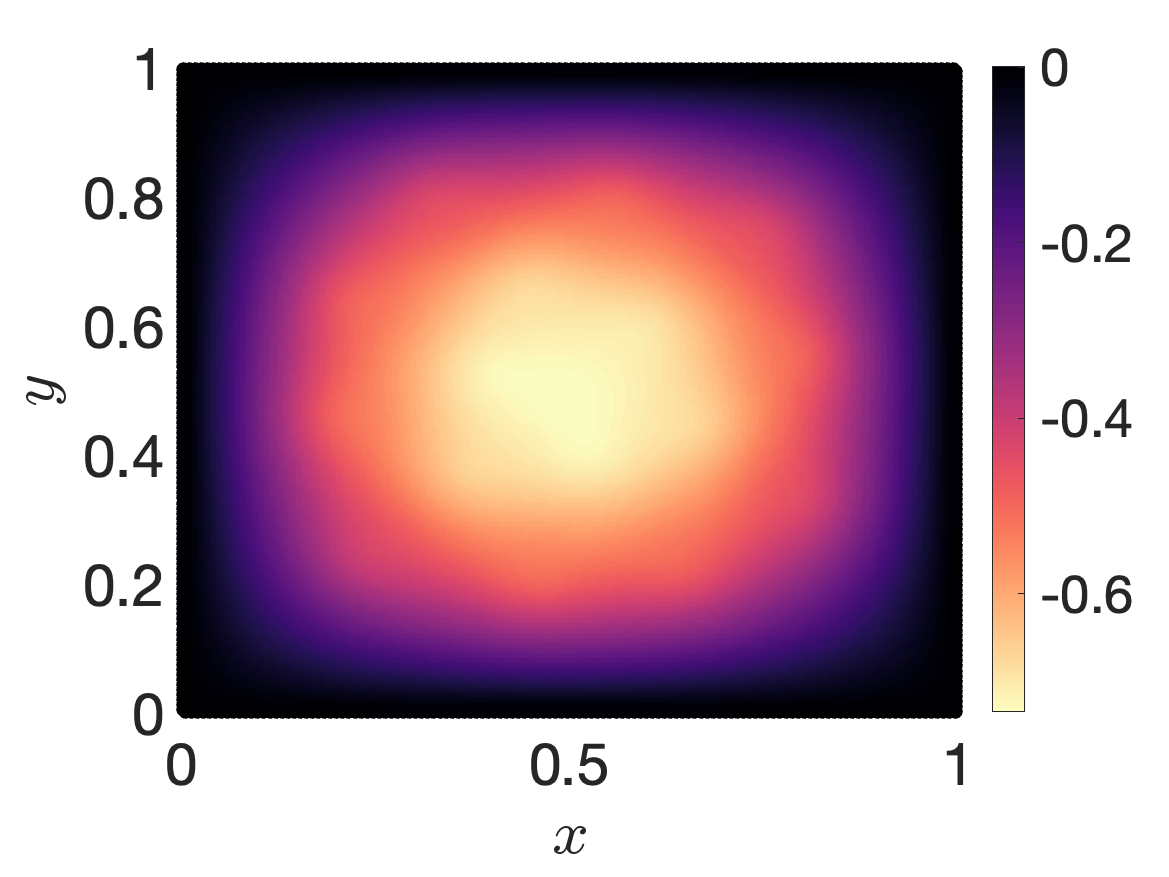}
\end{minipage}
\begin{minipage}[t]{0.24\textwidth}
\centering
\includegraphics[scale=0.19]{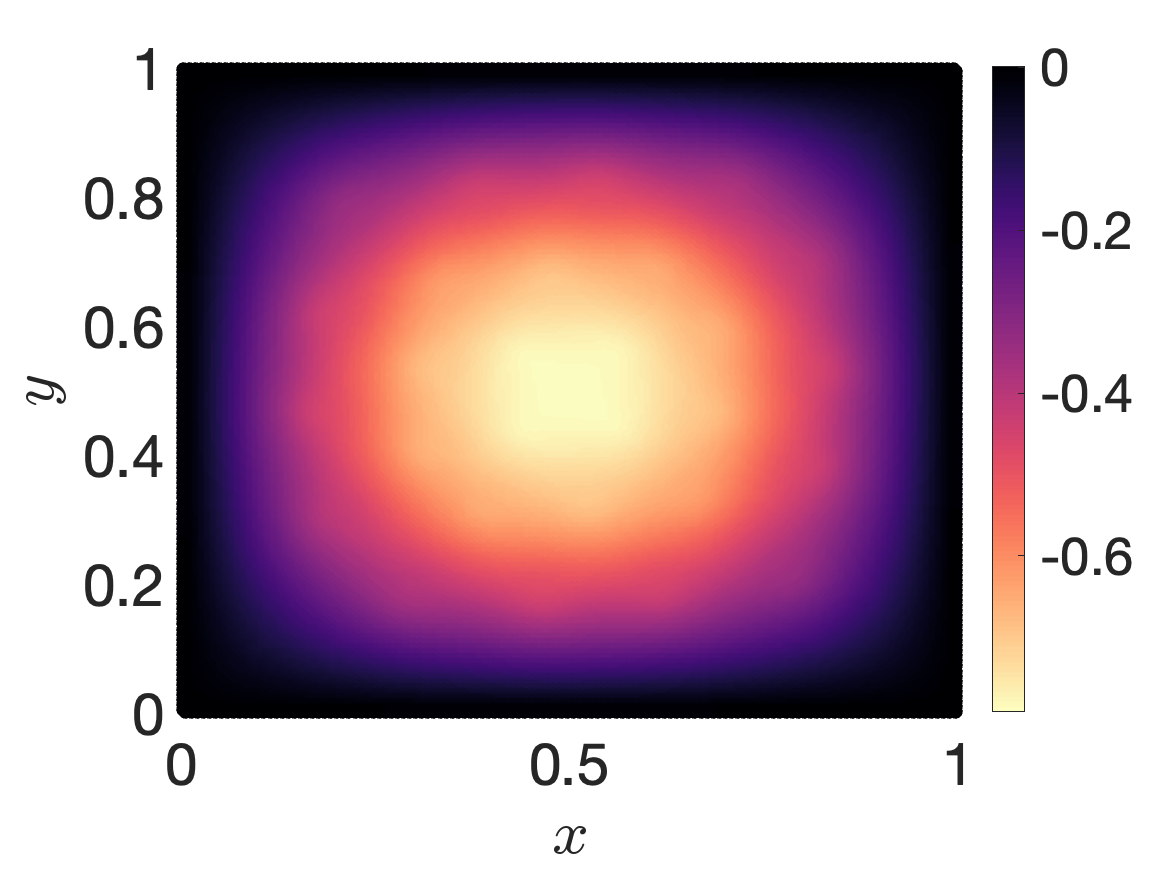}
\end{minipage}
\begin{minipage}[t]{0.24\textwidth}
\centering
\includegraphics[scale=0.19]{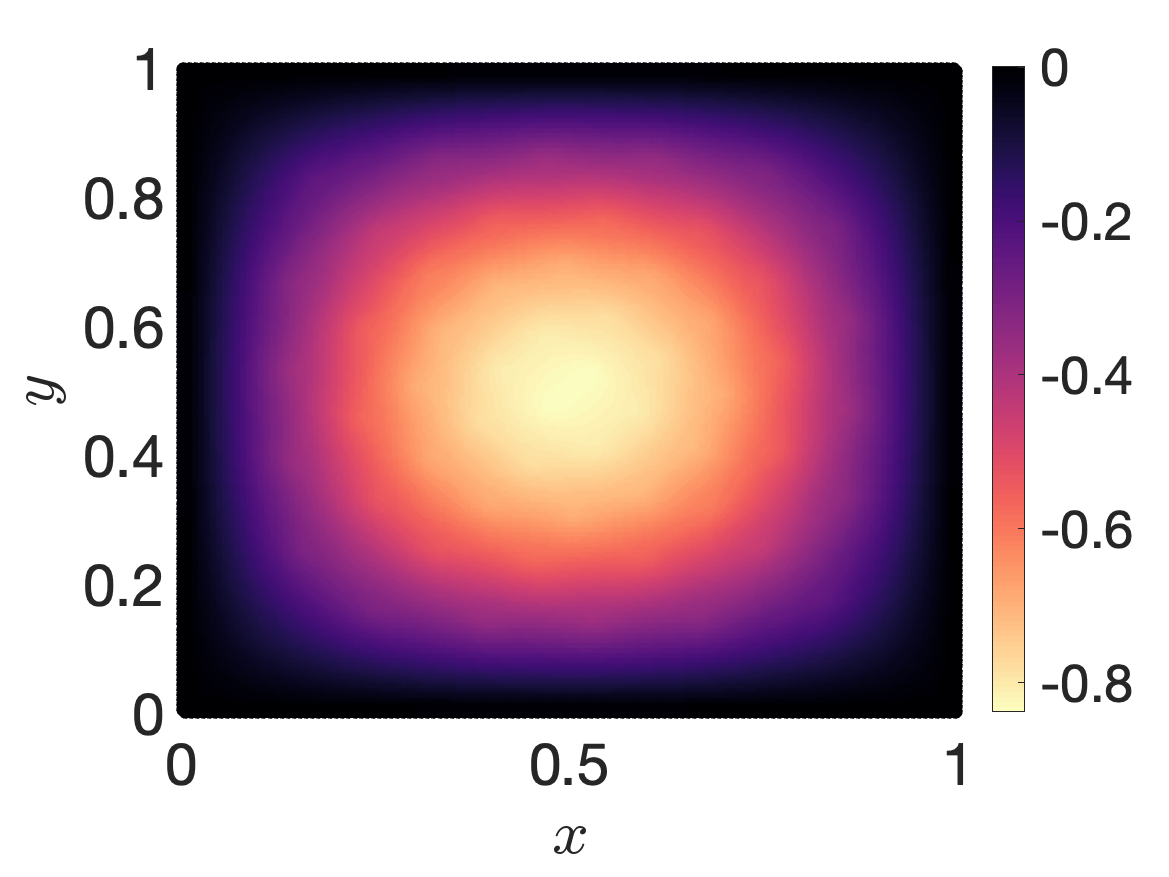}
\end{minipage}
\caption{Difference $\curl \MagPot - \MagF$ of the GL energy minimizer $(\sol,\MagPot)$ for $\kappa = 5, 10, 15, 20, 25, 30, 50, 100$ (left to right and top to bottom).}
\label{fig:reference-minimizers-curlA-H}
\end{figure}

\renewcommand*{\arraystretch}{1.2}
\begin{table}[h]
	\centering
	\begin{tabular}[h]{ |c|c|c|c|c|c|c|c|c|c| }
	\hline
		$\kappa$ 
		& 5 
		& 10 
		& 15 
		& 20 
		& 25 
		& 30
		& 50
		& 100 \\
		\hline
		$E(\sol,\MagPot)$ 
		& $0.192633$ 
		& $0.164499$ 
		& $0.164999$ 
		& $0.145422$ 
		& $0.142620$ 
		& $0.143023$ 
		& $0.133575$ 
		& $0.122268$ \\ 
		\hline
		$\energy(\sol,\MagPot)$ 
		& $0.192615$ 
		& $0.164492$ 
		& $0.164950$ 
		& $0.145396$ 
		& $0.142595$ 
		& $0.142997$ 
		& $0.133560$ 
		& $0.122261$ \\ 
		\hline 
	\end{tabular}
	\caption{Approximate energies $E(\sol,\MagPot)$ and $\energy(\sol,\MagPot)$ of the minimizers for $\kappa = 5, 10, 15, 20, 25, 30, 50, 100$.}
	\label{tab:energy-reference-minimizers}
\end{table}

In the density $|\sol|^2$ of the order parameter $\sol$ (bottom row of Figure \ref{fig:reference-minimizers} and Figure \ref{fig:reference-minimizers_2}) we see the expected vortex patterns, known as the Abrikosov lattice, that occur in type-II superconductors penetrated by external magnetic fields. Interestingly, the vortices in the density arise from the interplay between the real part (top row) and the imaginary part (middle row) of $\sol$. Both components exhibit a more intricate structure due to oscillations, particularly in regions where $|\sol| \approx 1$. In particular, the physical relevant vortices in the density are well resolved by the LOD approximation space, consistent with the observations of \cite{BDH24} for the reduced GL model with given vector potentials $\MagPot$. The number of vortices increases and their diameter decreases while the GL parameter $\kappa$ increases. Looking at the vector potential $\MagPot$ (Figure \ref{fig:reference-minimizers-A}), we clearly see that no special vortex-like structure appears in the vector potential. The vortices are only marginally detectable once we consider the difference $\curl \MagPot - \MagF$ (Figure \ref{fig:reference-minimizers-curlA-H}). This justifies the choice of a standard FE discretization for the vector potential $\MagPot$. The physically relevant observable is $\curl \MagPot$, which describes the magnetic field inside the superconductor. As expected, it is aligned with the external magnetic field $\MagF$. In general, we observed during the computation that the alignment of $\MagPot$ stabilizes after a few iterations, while the challenging variable is the order parameter as we have seen that it takes much more iterations (up to $\mathcal{O}(10^4)$ iterations) until the correct vortex-pattern appears and the Sobolev gradient flow converges. \\

\begin{figure}[h!] 
\centering
\begin{minipage}{0.49\textwidth}
\centering
\includegraphics[scale=0.35]{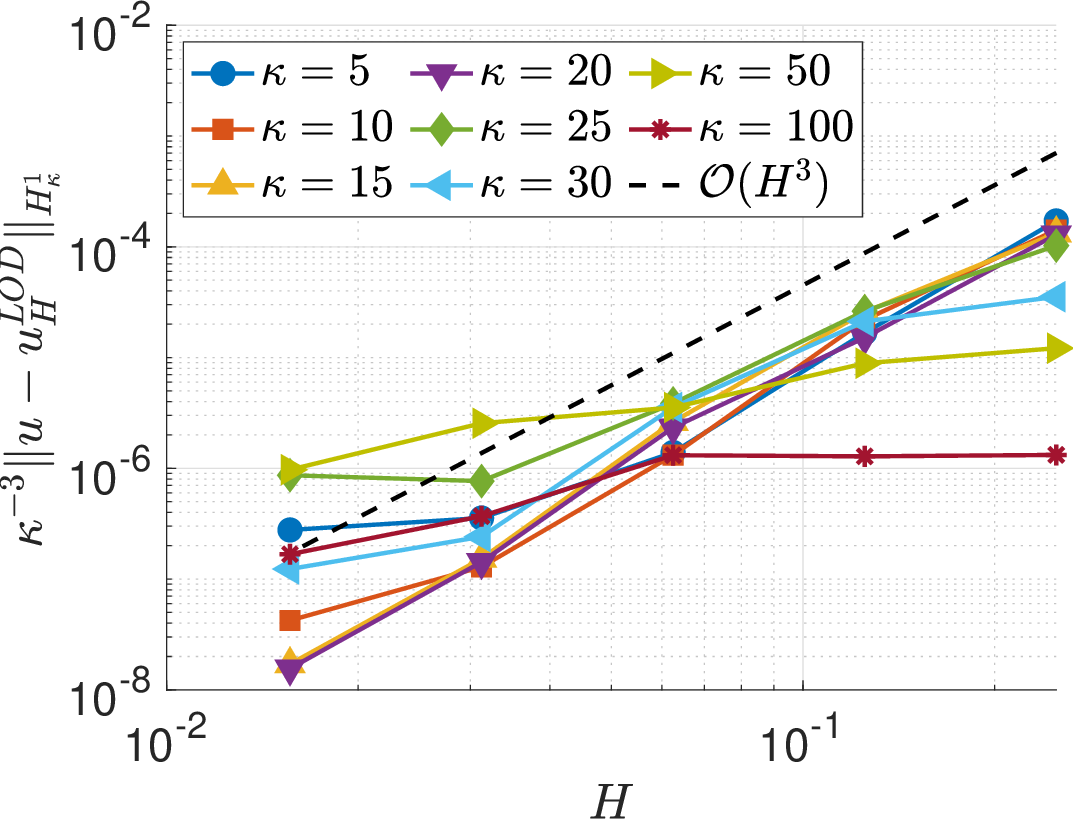}
\end{minipage}
\begin{minipage}{0.49\textwidth}
\centering
\includegraphics[scale=0.35]{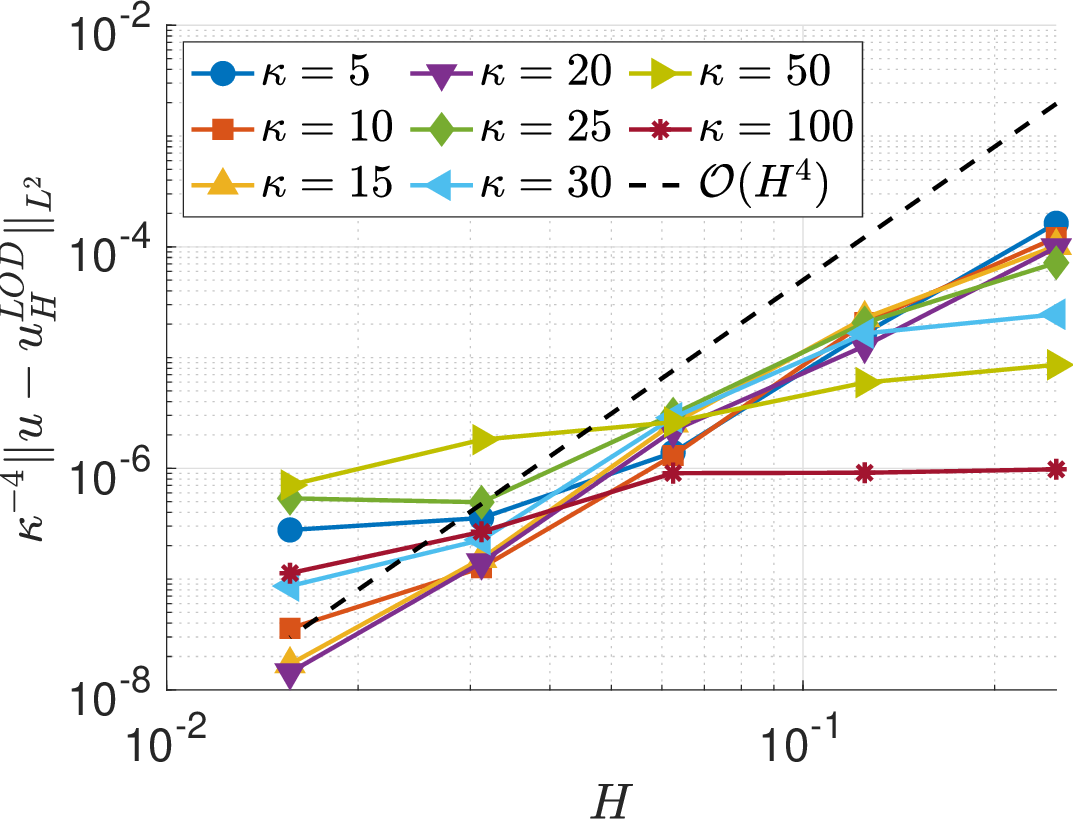}
\end{minipage}
\caption{Error of the order parameter $u$ for the mesh sizes $H = 2^{-\{ 2,3,4,5,6 \}}$, $h = 2^{-6}$ and LOD parameters $h_{\mathrm{fine}} = 2^{-9}$ and $\ell = 10$. Left: $\kappa$-scaled $H^1_\kappa$-error $\kappa^{-3}\| u - u^{\LOD}_H\|_{H^1_\kappa}$. Right:  $\kappa$-scaled $L^2$-error $\kappa^{-4}\| u - u^{\LOD}_H\|_{L^2}$.}
\label{fig:conv_u}
\end{figure}
\begin{figure}[h!] 
\centering
\begin{minipage}{0.49\textwidth}
\centering
\includegraphics[scale=0.35]{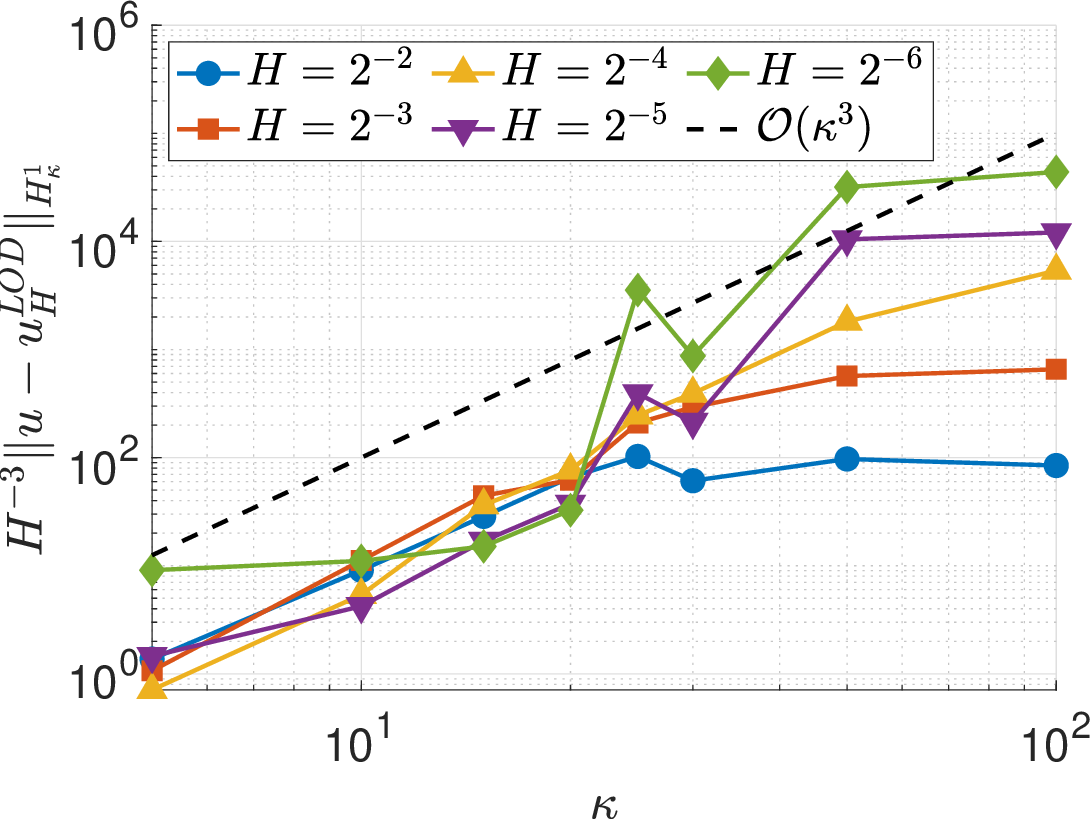}
\end{minipage}
\begin{minipage}{0.49\textwidth}
\centering
\includegraphics[scale=0.35]{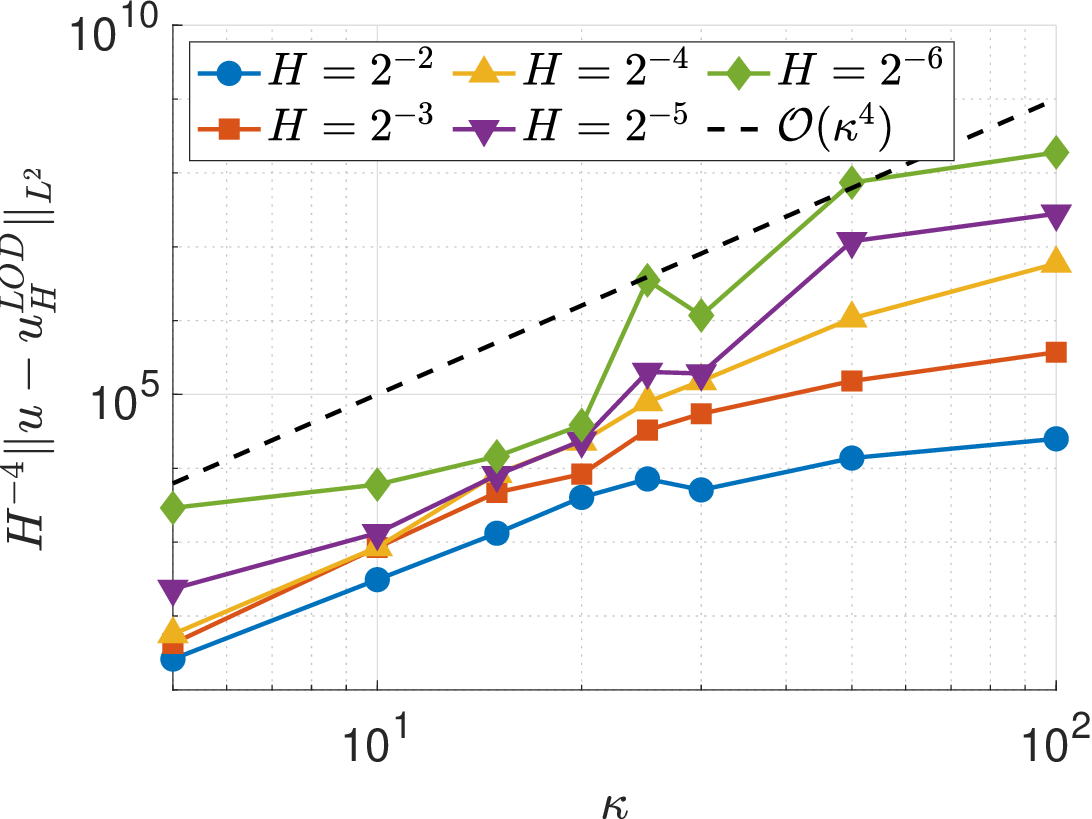}
\end{minipage}
\caption{Error of the order parameter $u$ relative to $\kappa$ for different mesh sizes $H = 2^{-\{ 2,3,4,5,6 \}}$, $h = 2^{-6}$ and LOD parameters $h_{\mathrm{fine}} = 2^{-9}$ and $\ell = 10$.\\ Left: $H^{-3}$-scaled $H^1_\kappa$-error $H^{-3}\| u - u^{\LOD}_H\|_{H^1_\kappa}$. Right:  $H^{-4}$-scaled $L^2$-error $H^{-4}\| u - u^{\LOD}_H\|_{L^2}$.}
\label{fig:conv_kappa_u}
\end{figure}

\subsection{Convergence rates}
Next we investigate the error of the approximations of the minimizers and its dependence on the mesh sizes $H,h$ and on the GL parameter $\kappa$ as stated in Theorem \ref{thrm:main-results}. For the first experiment, we choose a fixed fine mesh size $h = 2^{-6}$ for the vector potential $\MagPot$, different mesh sizes $H = 2^{-\{ 2, 3, 4, 5, 6\}}$ for the LOD approximation in the order parameter $u$, and compute the errors $\| u - u_{H}^\LOD \|_{H^1_{\kappa}}$ and $\| u - u_{H}^\LOD \|_{L^2}$ to the reference solution for the different values $\kappa = 5,10,20,25,30,50,100$. The minimizers $(\sol,\MagPot)$ computed in Section \ref{subsec:model} serve as reference solutions and all other parameters are chosen as in Section \ref{subsec:model}. For the given configuration, the minimizer is obtained using the iterative method described in Section \ref{subsec:gradient_descent}, starting from an initial value given by a projection of the reference solution onto the chosen approximation space. \\
Due to the fine mesh size in the vector potential we can expect that the overall error $(\sol - \sol_{H}^{\LOD} , \MagPot - \mathbf{A}_{h,2}^{\mbox{\tiny FEM}})$ is dominated by the error in the order parameter. The results are depicted in Figure \ref{fig:conv_u}. We observe after a pre-asymptotic phase an order three decay of the $H^1_\kappa$-error w.r.t. the mesh size and an order four decay of the $L^2$-error w.r.t. the mesh size. The pre-asymptotic phase is explained by the resolution condition $\deltaH \lesssim \kappa^{-1-\varepsilon/2} \Csol^{-1/2} $ which is needed for a quasi-best-approximation behavior or at least $\deltaH \lesssim \kappa^{-1-\varepsilon/5} \Csol^{-1/5}$ to compensate the additional higher order error terms as we discussed in Section \ref{sec:space_main} after Theorem \ref{thrm:main-results}. For $H = 2^{-5}$ both errors start to stagnate which is most-likely due to the fine scale discretization of order $\mathcal{O}(h_{\mathrm{fine}})$ that we used to solve the local corrector problems for the LOD space. This behavior was also observed in \cite{BDH24} and we refer to the reference for a more detailed discussion. For larger values of $\kappa$, namely $\kappa = 50$ and $\kappa = 100$, the pre-asymptotic regime is so large that the error directly turns into the stagnation phase without showing the predicted third-order or fourth-order convergence, respectively. However, this is in line with our theoretical results. Let us now turn to the $\kappa$-dependence of the error. We point out that in Figure \ref{fig:conv_u} the $H^1_\kappa$-error is scaled with a $\kappa^{-3}$ pre-factor and the $L^2$-error with a $\kappa^{-4}$ pre-factor respectively. We observe that the error curves are almost on top of each other as $\kappa$ varies and emphasize that a different $\kappa$-scaling of the error leads to a significant difference between the error curves. To examine the dependence on $\kappa$ in more detail, we plot the $H^1_\kappa$-error (resp. the $L^2$-error) against $\kappa$ for different coarse mesh sizes $H$ in Figure \ref{fig:conv_kappa_u}. Figure \ref{fig:conv_kappa_u} (right) shows the $H^1_\kappa$-error for small values of $\kappa$, clearly demonstrating the predicted optimal order of $\mathcal{O}(\kappa^3 H^3)$. Furthermore, due to the $H^{-3}$-scaling of the error, we can see that the curves lie almost on top of each other, which again shows the predicted optimal order of the $H^1_\kappa$-error, i.e. $\mathcal{O}(\kappa^3 H^3)$. However, for larger values of $\kappa$, the error stagnates at a plateau, most likely due to fine-scale discretization or localization errors caused by the practical realization of the LOD. Indeed, neglecting the $H^{-3}$-scaling of the error shows that all plateaus are at the same level, strongly indicating additional discretization errors in the LOD. Figure \ref{fig:conv_kappa_u} shows the $L^2$-error as a function of $\kappa$, and the situation is similar. In the first regime of small values of $\kappa$, we observe the expected $\kappa^4$-scaling of the error, followed by a stagnation regime, which is again most likely due to fine-scale discretization or localization error. However, considering the $H^{-4}$-scaling of the error, we can see that the curves are not clearly on top of each other. Indeed, a $H^{-3}$-scaling seems to fit as well, which might indicate pollution of the error by the term $\mathcal{O}(\kappa^3 H^3)$ in the $L^2$-error estimates of Theorem \ref{thrm:main-results}. However, the effect is too small to draw a clear conclusion. Apart from that, we can conclude that the numerical findings align with the decay of the $H^1_\kappa$-error (resp. $L^2$-error) in the order parameter $u$ as proved in Theorem \ref{thrm:main-results}, i.e. the order of decay is given by the terms $\mathcal{O}(\kappa^3 H^3)$ (resp. $\mathcal{O}(\kappa^4 H^4)$). \\

\begin{figure}[h!] 
\centering
\begin{minipage}{0.49\textwidth}
\includegraphics[scale=0.35]{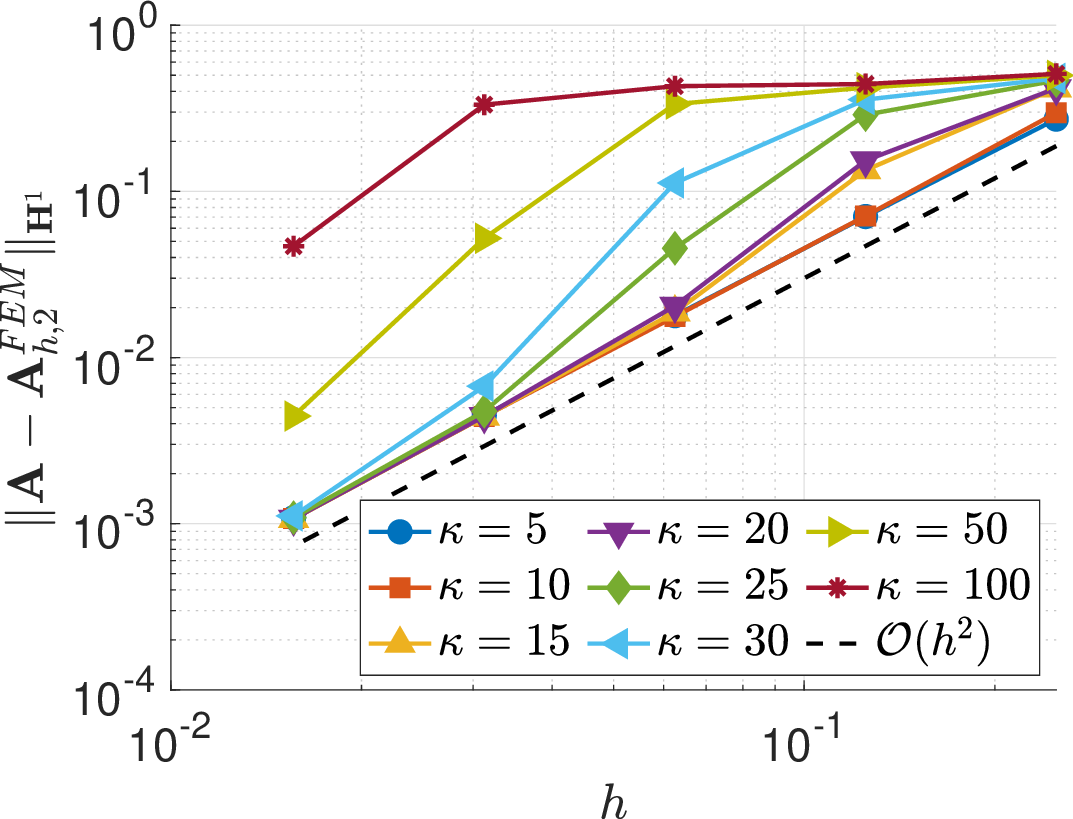}
\end{minipage}
\begin{minipage}{0.49\textwidth}
\includegraphics[scale=0.35]{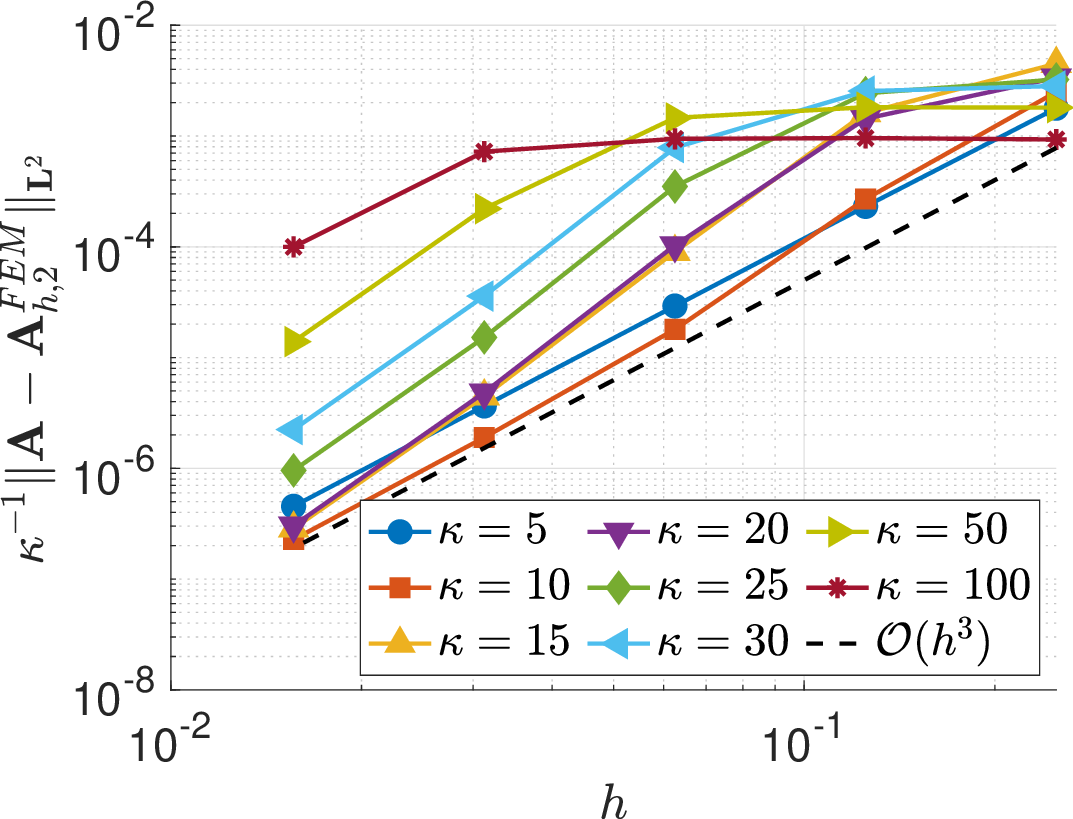}
\end{minipage}
\caption{Error of the vector potential $\MagPot$ for the mesh sizes $h = 2^{-\{ 2,3,4,5, 6\}}$, $H = 2^{-5}$ and LOD parameters $h_{\mathrm{fine}} = 2^{-9}$ and $\ell = 10$. Left: $\mathbf{H}^1$-error $\| \MagPot - \mathbf{A}_{h,2}^{\mbox{\tiny FEM}}\|_{\mathbf{H}^1}$. Right:  $\mathbf{L}^2$-error $\kappa^{-1} \| \MagPot - \mathbf{A}_{h,2}^{\mbox{\tiny FEM}} \|_{\mathbf{L}^2}$.}
\label{fig:conv_A_opt}
\end{figure}

In the next experiment we extract the convergence in the vector potential $\MagPot$. We fix a small mesh size $H = 2^{-5}$ for the order parameter, vary the mesh size $h = 2^{-\{ 2,3,4,5,6 \}}$ for the vector potential, and compute the errors $\| \MagPot - \mathbf{A}_{h,2}^{\mbox{\tiny FEM}}\|_{\mathbf{H}^1}$ and $\| \MagPot - \mathbf{A}_{h,2}^{\mbox{\tiny FEM}}\|_{\mathbf{L}^2}$. Again all other parameters are set as before and we can now expect that the overall error is dominated by the error in the vector potential $\MagPot$. The results are shown in Figure \ref{fig:conv_A_opt} where we can see that the $\mathbf{H}^1$-error decays in a clear asymptotic phase with a rate of two w.r.t. $h$ after a pre-asymptotic phase which increases in $\kappa$. This pre-asymptotic phase is now explained by the resolution condition $\deltah \lesssim \kappa^{-\varepsilon} \Csol^{-1}$ which is needed for a quasi-best-approximation or at least $\deltah \lesssim \kappa^{-(1 + \varepsilon)/3} \Csol^{-1/3}$ to compensate the additional higher order error terms. This coincides with our theoretical findings from Theorem \ref{thrm:main-results}. In view of the $\kappa$-dependence the situation is more unclear. First, Figure \ref{fig:conv_A_opt} left shows that for the small values of $\kappa$ and in the asymptotic phase the convergence curves lie almost exactly on top of each other, although the error is not scaled with one order of $\kappa$ as it would have been expected from Theorem \ref{thrm:main-results}. This indicates that our error estimates of order $\mathcal{O}(\kappa h^2)$ might be suboptimal w.r.t. $\kappa$ as an implication of a possibly suboptimal estimate of $\| \MagPot \|_{\mathbf{H}^3}$ in the analysis; see Remark \ref{rem:A_in_H3} for a more detailed discussion. We also observe a clear third-order convergence rate with respect to the mesh size for the $L^2$-error in the asymptotic phase, i.e., for small values of $h$ after a pre-asymptotic phase, which is again explained by the resolution condition. The dependence of the $L^2$-error on $\kappa$ is again much more complicated since we observe the $\mathcal{O}(\kappa^\varepsilon \Csol h^3)$ error term from Theorem \ref{thrm:main-results}, for which the dependence on $\kappa$ enters through the coercivity constant $\Csol$. We found that the error scales best with $\kappa^{-1}$ in a range of integer powers, but we cannot draw a precise conclusion about the dependence on $\kappa$ since it is unknown for the coercivity constant $\Csol$. We note that plotting the error over $\kappa$ does not provide any additional information that would allow us to draw a clear conclusion. At this point, we can conclude that the numerical findings align with our theoretical findings, though they may indicate that our estimates in the $\kappa$-scaling are suboptimal with regard to the vector potential.

\begin{figure}[h!] 
\centering
\begin{minipage}{0.49\textwidth}
\includegraphics[scale=0.35]{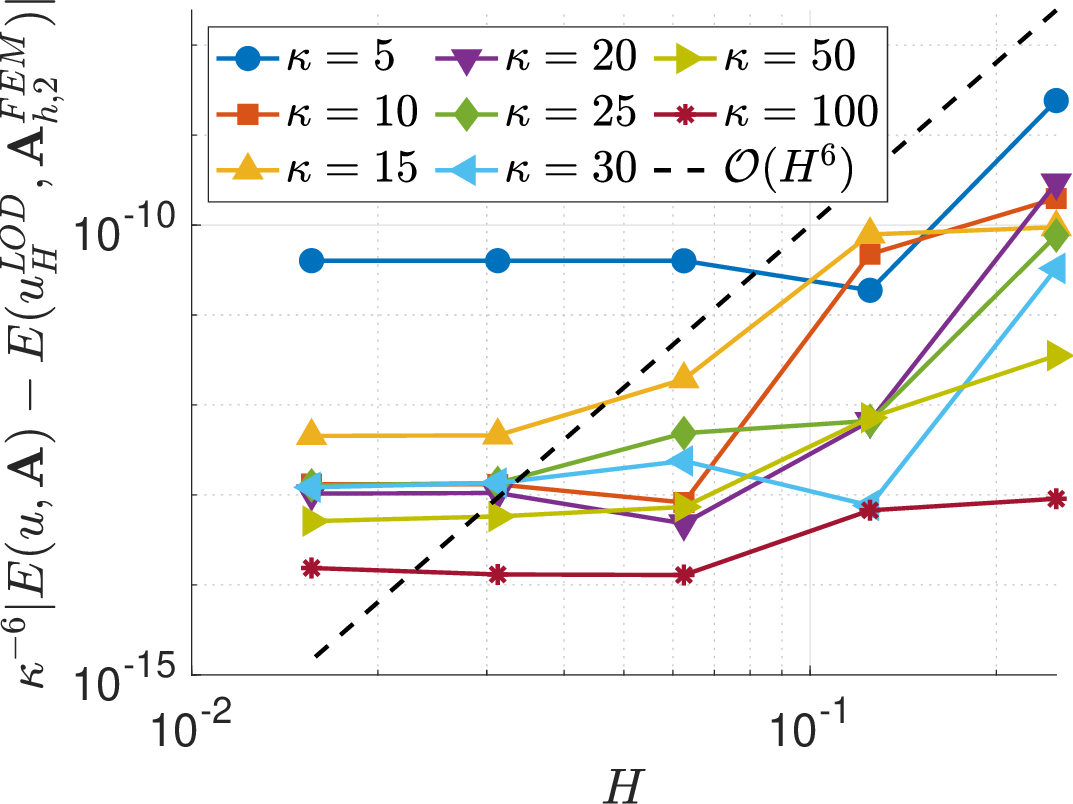}
\end{minipage}
\begin{minipage}{0.49\textwidth}
\includegraphics[scale=0.35]{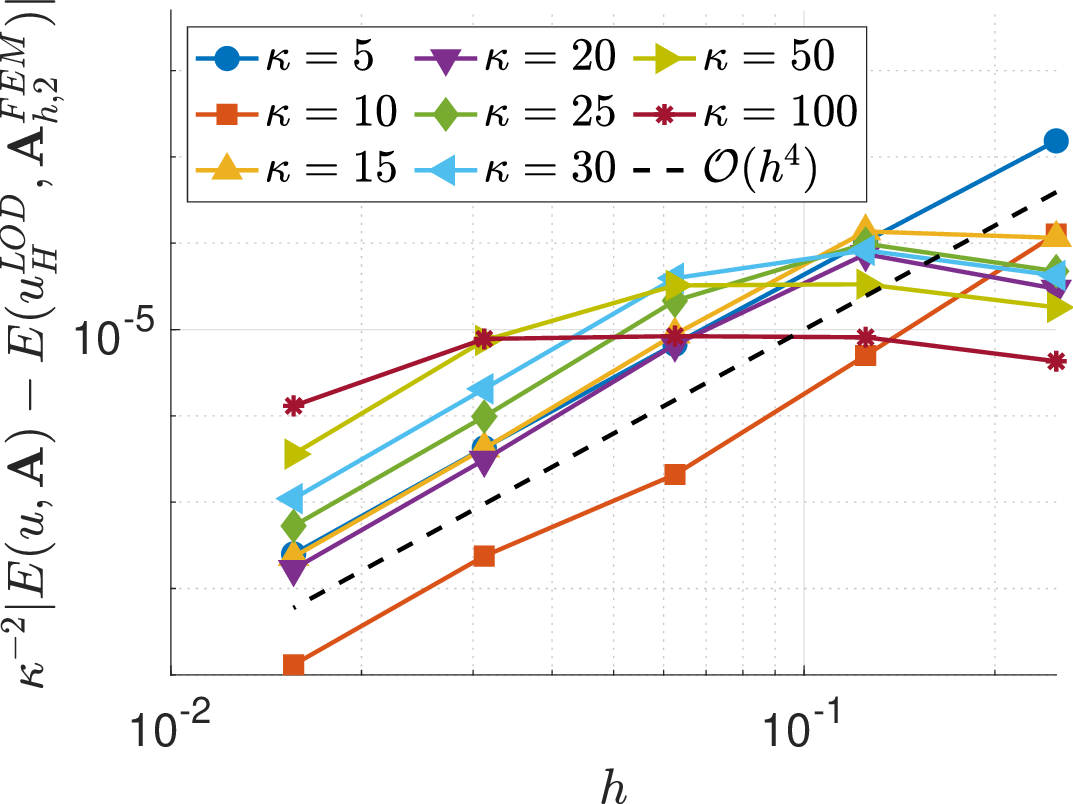}
\end{minipage}
\caption{Error of the energy $E$ with LOD parameters $h_{\mathrm{fine}} = 2^{-9}$ and $\ell = 10$. Left: $\kappa^{-6}|E(\sol, \MagPot) - E(u^{\LOD}_H, \mathbf{A}_{h,2}^{\mbox{\tiny FEM}})|$ for the mesh sizes $H = 2^{-\{ 2,3,4,5,6 \}}$, $h = 2^{-6}$. Right:  $\kappa^{-2} |E(\sol, \MagPot) - E(u^{\LOD}_H, \mathbf{A}_{h,2}^{\mbox{\tiny FEM}})|$ for the mesh sizes $h = 2^{-\{ 2,3,4,5, 6\}}$, $H = 2^{-5}$.}
\label{fig:conv_energy}
\end{figure}

In both experiments we additionally compute the errors in the GL energy which are shown in Figure \ref{fig:conv_energy}. The error in energy takes the convergence in both components into account and we see in both error plots a convergence in up to three phases: a pre-asymptotic phase, an asymptotic phase, and in the order parameter a stagnation phase. We first discuss the convergence in the order parameter (Figure \ref{fig:conv_energy} left). The pre-asymptotic regime is explained by the very weak resolution condition $\kappa H \lesssim 1$ as stated in Theorem \ref{thrm:main-results}. Indeed, this regime seems to be smaller compared to the $H^1_\kappa$-error, as seen most prominently for the $\kappa = 50$ curve. The pre-asymptotic regime is followed by an asymptotic convergent regime where we observe the predicted sixth order of convergence. However, due to the high order of convergence and the small magnitudes close to machine precision, the observations are less clear than for the $H^1_\kappa$ and $L^2$ errors. The same holds for the sixth-order scaling of the energy error w.r.t. $\kappa$, to which we approximate as closely as possible. The stagnation regime is again caused most likely by the fine-scale discretization error. For the error in the vector potential (Figure \ref{fig:conv_energy}, right), we observe the expected asymptotic fourth-order convergence. Moreover, the scaling with respect to $\kappa$ is approximately of order four, as anticipated. However, the convergence curves are not closely aligned, suggesting potential suboptimality with respect to $\kappa$, consistent with the behavior seen in the $H^1$ and $L^2$ errors. In addition, a pronounced pre-asymptotic regime is present -- particularly for $\kappa = 100$ and $\kappa = 50$ -- which is unexpected and appears, at first sight, to contradict our main results, since no resolution condition with respect to $h$ should be required. A closer examination of the computations revealed that, on coarser meshes (i.e., within the pre-asymptotic regime), the Sobolev gradient descent method converges to a different minimizer, as evidenced by a substantially different energy level. Unfortunately, we were unable to recover the corresponding reference minimizer on the coarser meshes through the Sobolev gradient descent iterations as the descent approach is not robust w.r.t. initial values. This numerical artifact accounts for the observed pre-asymptotic behavior. Except for this behavior -- and the aforementioned indication of suboptimality in the $\kappa$–scaling -- the numerical experiments are overall in good agreement with our theoretical results.

\subsection{Choice of $\bfAapp$} 
As a final experiment, we investigate numerically how the choice of $\bfAapp$ and therefore the choice of the LOD space affects our results, in particular the convergence in the order parameter $\sol$ in view of Lemma \ref{lemma:bestapprox-LOD} and our main results. We compare our standard choice -- where $\bfAapp$ is a $\mathcal{P}_2$ approximation of $\curl \bfAapp = \MagF$ and $\div \bfAapp = 0$ -- against alternatives. For simplicity, we pick the values $\kappa = 10$ and $\kappa = 20$ as representatives, since we investigated the influence of $\kappa$ numerically in the previous section, and retain all other model parameters. The following vector potentials $\bfAapp$ are considered:

\begin{description}
\item[Standard choice] $\mathcal{P}_2$ approximation of $\curl \bfAapp = \MagF$ with $\div \bfAapp = 0$ on a fine mesh size of $h = 2^{-7}$. This is easily pre-computed at low computational costs and expected to be a rough but reasonable approximation of $\MagPot$. \\
\item[Trivial choice] $\bfAapp = 0$. This allows for an efficient computation of the LOD space due to spatial symmetry at the cost of a poor approximation of $\MagPot$. \\
\item[Optimal choice] $\bfAapp = \MagPot$, where $\MagPot$ denotes the reference minimizer introduced in Section \ref{subsec:model}. This is the best available approximation of $\MagPot$ in our experiments but of course in general unknown. \\
\end{description}

Since the choice of the LOD space should not have a significant effect on the approximation of the vector potential $\MagPot$, we are only interested in the error w.r.t. the order parameter $\sol$. As in the previous section, we fix $h = 2^{-6}$ and vary the mesh size for the LOD space as $H = 2^{- \{ 2, 3, 4, 5, 6\} }$ with a fine scale resolution of $h_{\mathrm{fine}} = 2^{- 9}$ and an oversampling parameter of $\ell = 10$. We compute the error in $H^1_\kappa$ for each choice of $\bfAapp$. The results are shown in Figure \ref{fig:conv_Astar}.

\begin{figure}[h!] 
\centering
\begin{minipage}{0.49\textwidth}
\includegraphics[scale=0.35]{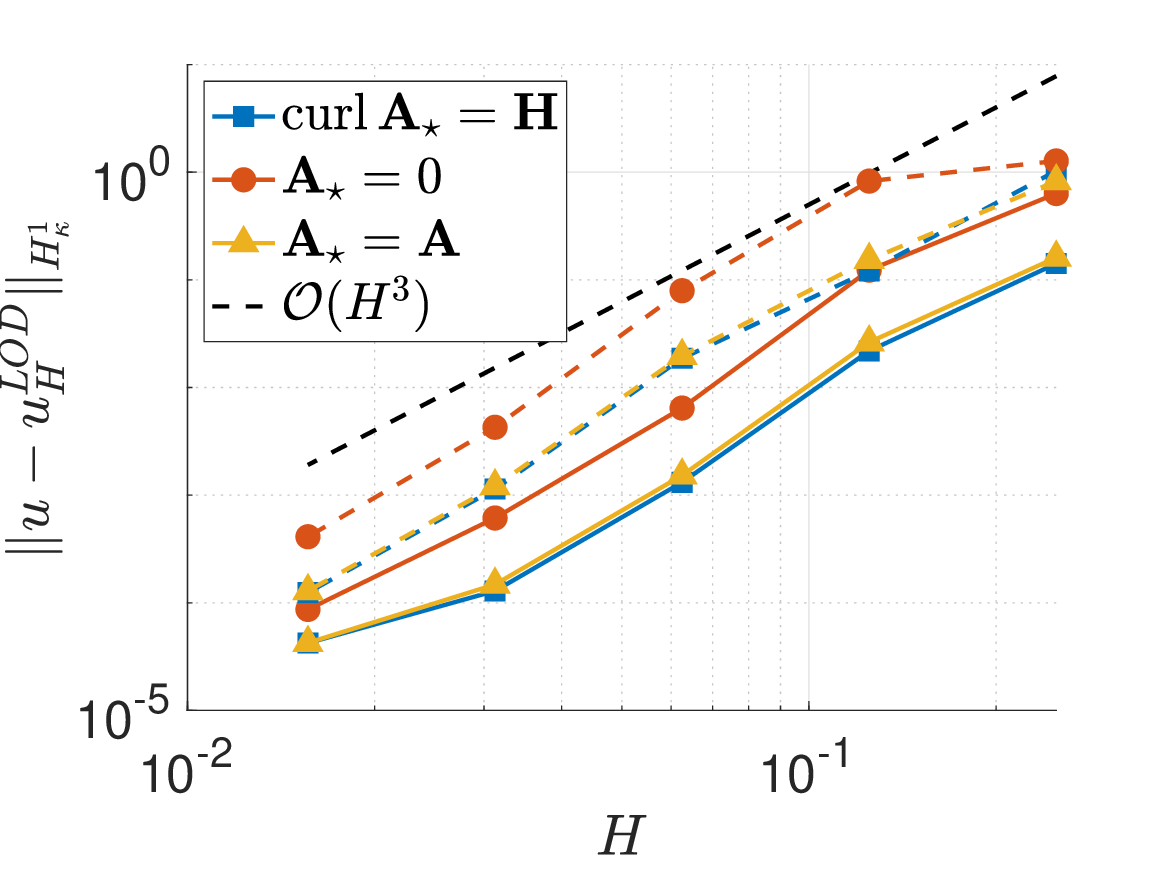}
\end{minipage}
\caption{Error of the order parameter $u$ for the different choices of $\bfAapp$, $\kappa = 10$ (solid lines) and $\kappa = 20$ (dashed lines) for the mesh sizes $H = 2^{-\{ 2,3,4,5,6 \}}$, $h = 2^{-6}$.}
\label{fig:conv_Astar}
\end{figure}

All three choices of $\bfAapp$ lead to optimal third-order convergence of the order parameter with respect to $H$ after the previously shown short pre-asymptotic regime. Every choice is therefore suited to some extent to approximate energy minimizers in the associated LOD spaces, which aligns with our main result. However, a direct comparison of the three choices reveals that the ``trivial choice'' ($\bfAapp = 0$) results in the largest error, making the ``standard choice'' preferable. The difference is almost one order of magnitude, and solving for $\curl \bfAapp = \MagF$ comes at a very low computational cost. In some cases, one might choose the ``trivial choice'' to compute the LOD space faster using spatial symmetries when $\bfAapp = 0$. However, in view of the computational cost of the Sobolev gradient descent, this advantage seems to be negligible. Surprisingly, our ``standard choice'' ($\curl \bfAapp = \MagF$) performs as well as the ``optimal choice'' ($\bfAapp = \MagPot$). This shows that the error committed by approximating $\bfA$ by $\bfAapp$ enters only weakly into the overall approximation error of GL energy minimizers. This justifies our standard choice in previous experiments, as it has a low computational cost and leads to good approximations of GL energy minimizers. Summarizing the numerical results confirm our main results.

\section*{Acknowledgment}

The authors would like to thank Wolfgang Reichel and Roland Schnaubelt for valuable discussions on the elliptic regularity.

\section*{Statements and Declarations}

The authors declare that they have no financial or non-financial competing interests directly or indirectly related to the work.


\appendix

\section{Proofs of the higher regularity results}
\label{sec:H2_reg_A}

In this section, we collect the proofs of Theorem~\ref{thm:reg_MagPot_H2_H3_gen} and Lemma~\ref{lemma-H3-regularity-u}. As we could not find any suitable reference which covers our cases, we present the proofs here in the appendix, even though these results might be known to many experts. 

For the sake of notation, we restrict ourselves to the unit cube $\Omega = \intervalop{0,1}^3$, but the case of general cuboids is easily derived by a linear transformation. For $\Omegaextk{s} = \intervalop{-s,s}^3$, $s \in \mathbb{R}^+$, the idea is to use reflections to extend the functions from $\Omega$ to the extended domain $\Omegaextk{1}$, and then periodically to any 
$\Omegaextk{2k+1}$ for $k \in \N$ while preserving its regularity.

The main intuition for this procedure comes from the eigenbasis of the Laplacian on a cube. For example, for homogeneous Dirichlet boundary conditions the basis on $\Omega$ consists of functions 
\begin{equation}
	\sin(\pi k x_1) \sin(\pi \ell x_2) \sin(\pi m x_3), \quad k,\ell,m \geq 1,
\end{equation}
and their natural extension is given by first performing an odd reflection on each face and then obtain a periodic function on $\Omegaextk{1}$. For Neumann boundary conditions, the same idea applies with the basis
\begin{equation}
	\cos(\pi k x_1) \cos(\pi \ell x_2) \cos(\pi m x_3), \quad k,\ell,m \geq 0,
\end{equation}
and hence even reflections on each face. For mixed problems the correct combination of sine and cosine enable us to extend this also to the mixed case.

\subsection{Neumann boundary conditions}

Let us consider the Neumann problem in Lemma~\ref{lemma-H3-regularity-u} given by
\begin{equation}
	- \Delta u = g \mbox{ in } \Omega\qquad \mbox{and} \qquad 
\nabla u \cdot \nu \vert_{\partial \Omega} = 0,
\end{equation}
for $g \in L^2(\Omega)$. For a function $f\in C(\bar{\Omega})$, we define the Neumann extension $	M_N \colon 
f \to 	
\fext$ with
\begin{align}
\fext(x_1,x_2,x_3) 
	=
	\begin{cases}
		f(x_1,x_2,x_3), \, &x_1 \in \intervalop{0,1}, x_2 \in \intervalop{0,1} ,x_3 \in \intervalop{0,1},
		\\
		f(x_1,x_2,-x_3), \quad &x_1 \in \intervalop{0,1}, x_2 \in \intervalop{0,1} ,x_3 \in \intervalop{-1,0},
		\\
		f(x_1,-x_2,-x_3), \quad &x_1 \in \intervalop{0,1}, x_2 \in \intervalop{-1,0} ,x_3 \in \intervalop{-1,0},
		\\
		f(x_1,-x_2,x_3), \quad &x_1 \in \intervalop{0,1}, x_2 \in \intervalop{-1,0} ,x_3 \in \intervalop{0,1},
		\\
		f(-x_1,x_2,x_3), \quad &x_1 \in \intervalop{-1,0}, x_2 \in \intervalop{0,1} ,x_3 \in \intervalop{0,1},
		\\
		f(-x_1,x_2,-x_3), \quad &x_1 \in \intervalop{-1,0}, x_2 \in \intervalop{0,1} ,x_3 \in \intervalop{-1,0},
		\\
		f(-x_1,-x_2,-x_3), \quad &x_1 \in \intervalop{-1,0}, x_2 \in \intervalop{-1,0} ,x_3 \in \intervalop{-1,0},
		\\
		f(-x_1,-x_2,x_3), \quad &x_1 \in \intervalop{-1,0}, x_2 \in \intervalop{-1,0} ,x_3 \in \intervalop{0,1}.
	\end{cases}
\end{align}
Without changing the notation, we extend the operator $M_N$ also to $L^2$-functions.

\begin{lemma} \label{lem:MN_H1}
	Let $f \in H^1(\Omega)$. Then the extension $M_N f$ satisfies:
	
	(a) $M_N f \in H^1(\Omegaextk{1})$ with 
	$\norm{H^1(\Omegaextk{1})}{M_N f}
	 \leq 
	 2^3 \norm{H^1(\Omega)}{M_N f}$. 
	
	(b) The periodic extension of $M_N f$ satisfies $M_N f \in H^1(\Omegaextk{2k+1}) $ for all $k\geq 1$ with
	\begin{equation}
	\norm{H^1(\Omegaextk{2k+1})}{M_N f}
\leq 
(2k+2)^3 \norm{H^1(\Omega)}{M_N f}.
	\end{equation}
\end{lemma}

\begin{proof}
(a)	Since $M_N f$ is in $H^1$ on each subdomain, it remains to check that the trace is continuous in the $L^2$ sense on the faces. However, the even reflection ensures this continuity. Since the $H^1$-norm on each subdomain is equal to the $H^1$-norm on $\Omega$, the estimate in (a) follows by counting cubes. 
	
(b)	For the periodic case, it is sufficient to note that the periodic extension from $\Omegaextk{1}$ is equivalent to iteratively performing even reflections on the outer faces. In particular, this implies continuity at all outer faces of $\Omegaextk{1}$, by repeating the calculation of the interior faces.
\end{proof}

We now turn to the case of preserving $H^2$-regularity.

\begin{lemma} \label{lem:MN_H2} 
	Let $f \in  H^2(\Omega)$ with $\nabla f \cdot \nu \vert_{\partial \Omega} =0$. Then the extension $M_N f$ satisfies:
	
	(a) $M_N f \in H^2(\Omegaextk{1})$ with 
	$\norm{H^2(\Omegaextk{1})}{M_N f}
	\leq 
	2^3 \norm{H^2(\Omega)}{M_N f}$. 
	
	(b) The periodic extension of $M_N f$ satisfies $M_N f \in H^2(\Omegaextk{2k+1}) $ for all $k\geq 1$ with
	\begin{equation}
		\norm{H^2(\Omegaextk{2k+1})}{M_N f}
		\leq 
		(2k+2)^3 \norm{H^2(\Omega)}{M_N f}.
		\end{equation}
\end{lemma}

\begin{proof}
	We note that it is sufficient to show that $\Delta \fext \in L^2(\Omegaextk{1})$ and elliptic regularity gives us the claim. We further note, that the periodic extension is handled as in Lemma~\ref{lem:MN_H1}.	
	
	Since we already know that $M_N f \in H^2$ on each subdomain and $M_N f \in H^1(\Omegaextk{1})$, in order to show that $\nabla \fext \in H(\div,\Omegaextk{1})$ holds, we have to prove that all normal traces of $\nabla \fext$ are continuous. 

	Computing the gradients on each subdomain, we obtain for the diagonal matrix
	$I_{a,b,c} = \diag(a,b,c)$ the expressions
	\begin{align}
		\nabla \fext(x_1,x_2,x_3) 
		=
		\begin{cases}
		I_{1,1,1}	\nabla	f \big|_{(x_1,x_2,x_3)}	, 
		\quad 
		&x_1 \in \intervalop{0,1}, x_2 \in \intervalop{0,1} ,x_3 \in \intervalop{0,1},
			\\
		I_{1,1,-1}	\nabla	f \big|_{(x_1,x_2,-x_3)},
		 \quad
		 &x_1 \in \intervalop{0,1}, x_2 \in \intervalop{0,1} ,x_3 \in \intervalop{-1,0},
			\\
		I_{1,-1,-1} \nabla	f \big|_{(x_1,-x_2,-x_3)}, 
		\quad &x_1 \in \intervalop{0,1}, x_2 \in \intervalop{-1,0} ,x_3 \in \intervalop{-1,0},
			\\
		I_{1,-1,1} \nabla	f \big|_{(x_1,-x_2,x_3)},
		\quad &x_1 \in \intervalop{0,1}, x_2 \in \intervalop{-1,0} ,x_3 \in \intervalop{0,1},
			\\
		I_{-1,1,1} \nabla	f \big|_{(-x_1,x_2,x_3)}, 
		\quad &x_1 \in \intervalop{-1,0}, x_2 \in \intervalop{0,1} ,x_3 \in \intervalop{0,1},
			\\
		I_{-1,1,-1} \nabla	f \big|_{(-x_1,x_2,-x_3)},
		 \quad &x_1 \in \intervalop{-1,0}, x_2 \in \intervalop{0,1} ,x_3 \in \intervalop{-1,0},
			\\
		I_{-1,-1,-1} \nabla	f \big|_{(-x_1,-x_2,-x_3)},
		 \quad &x_1 \in \intervalop{-1,0}, x_2 \in \intervalop{-1,0} ,x_3 \in \intervalop{-1,0},
			\\
		I_{-1,-1,1} \nabla	f \big|_{(-x_1,-x_2,x_3)}, 
		\quad &x_1 \in \intervalop{-1,0}, x_2 \in \intervalop{-1,0} ,x_3 \in \intervalop{0,1}.
		\end{cases}
	\end{align}
We only check the face $\{x_1 = 0, x_2 \in \intervalop{0,1}, x_3 \in \intervalop{0,1} \}$ with normal vector $\nu = e_1$ to obtain formally for all $x_2 \in \intervalop{0,1}, x_3 \in \intervalop{0,1}$
	\begin{align}
	\lim\limits_{x_1 \to 0^+} \partial_\nu \fext(x_1,x_2,x_3) 
	=
	\lim\limits_{x_1 \to 0^+} \partial_1 f(x_1,x_2,x_3) 
	=
	\partial_1 f(0,x_2,x_3) = 0 
\end{align}
as well as
\begin{align}
	\lim\limits_{x_1 \to 0^-} \partial_\nu \fext(x_1,x_2,x_3) 
	=
	\lim\limits_{x_1 \to 0^+} - \partial_1 f(-x_1,x_2,x_3) 
	=
	- \partial_1 f(0,x_2,x_3) = 0 .
\end{align}
For the other faces the very same computations can be performed. Thus, all normal traces of $\nabla \fext$ vanish on the inner faces, are thus in particular continuous, and we have shown $\nabla \fext \in H(\div,\Omegaextk{1})$.
\end{proof}

With this, we are in the position to prove Lemma~\ref{lemma-H3-regularity-u}.

\begin{proof}[Proof of Lemma~\ref{lemma-H3-regularity-u}]
First, we observe 
with Lemmas~\ref{lem:MN_H1} and \ref{lem:MN_H2} 
that for $\uext = M_N \sol$
\begin{equation}
-	\Delta \uext\big|_{(x_1,x_2, x_3)} 
=
- \Delta \sol \big|_{(\pm x_1,\pm x_2,\pm x_3)}
	=
	f  \big|_{(\pm x_1,\pm x_2,\pm x_3)}
	=
	\fext(x_1,x_2,x_3) ,
\end{equation}
with signs chosen accordingly to the definition of $M_N$. In particular,
$\uext$ solves the Neumann problem in Lemma~\ref{lemma-H3-regularity-u} also on $\Omegaextk{3}$ with right-hand side $\fext \in H^1(\Omegaextk{3})$. By interior regularity for elliptic problems (cf. \cite[Theorem~6.3.2]{Evans10}) we conclude $u = \uext\vert_{\Omega} \in H^3(\Omega)$ with
	\begin{eqnarray*}
		\| u \|_{H^3(\Omega)} &\lesssim& 
		\| \uext \|_{L^2(\Omegaextk{3})} + \| \fext \|_{H^1(\Omegaextk{3})} \,\,\,\lesssim \,\,\, \| u \|_{L^2(\Omega)} + \| f \|_{H^1(\Omega)}.
	\end{eqnarray*}
	Next, we turn towards the $W^{2,p}$-regularity of $u$, where we exploit again the extensions $ \uext$ and $\fext$ together with a Calder\'on--Zygmund estimate. 
	For that, let $B_{\ext}$ be a ball with radius $r=2$,
	 which is compactly contained in the extended domain $\Omegaextk{3}$. In particular, we have $\Omega \subset \subset B_{\ext} \subset \subset \Omegaextk{3}$ and a regular boundary $\partial B_{\ext} \in C^{1,1}$. We want to smoothly truncate $\uext$ to $B_{\ext}$ with a cut-off function $\eta \in C^{\infty}_0(B_{\ext})$ with $0\le \eta \le 1$. Furthermore, $\eta$ should not only be constant $1$ on $\Omega$, but also on a slightly enlarged box, that is, $\eta \equiv 1$ on
 		 $\Omegaext_{1+\delta}$
		  for a sufficiently small $\delta$ such that we still have $\Omegaext_{1+\delta} \subset \subset B_{\ext}$. Finally, assume that $\eta$ is selected such that $\| \nabla \eta \|_{L^{\infty}} \le C$ for some constant $C$ that only depends on $B_{\ext}$ and $\Omegaextk{3}$. We consider the function $\eta \, \uext \in H^1_0(B_{\ext})$ which solves
		$$
		- \Delta (\eta \, \uext) = - \eta \, \Delta \uext - 2 \nabla \uext \cdot \nabla \eta - \uext \, \Delta \eta
		= \eta \, \fext  - 2 \nabla \uext \cdot \nabla \eta - \uext \, \Delta \eta =: \tilde{f}_{\ext}.
		$$
		Since $\uext \in H^3(B_{\ext})$ (by interior regularity from the first part of the proof), Sobolev embeddings guarantee that we also have $\nabla \uext \in L^{\infty}(B_{\ext})$. Together with $\fext \in L^p(B_{\ext})$ (which directly follows from $f \in L^p(\Omega)$), we conclude that  $\eta \, \uext \in H^1_0(B_{\ext})$ is the unique solution to a Poisson problem on a smooth domain $B_{\ext}$, with homogeneous Dirichlet boundary condition and a right-hand side $\tilde{f}_{\ext} \in L^p$. By $L^p$-regularity theory for elliptic problems, cf. \cite[Chapt. 3, Thm. 6.3 \& Thm. 6.4]{ChenWuEllipticTheoryLp}, we conclude that this unique solution fulfills $\eta \, \uext \in W^{2,p}(B_{\ext})$. By construction of the cut-off function, we have $\uext \vert_{\Omegaextk{1+\delta}} = (\eta \, \uext)\vert_{\Omegaextk{1+\delta}} \in W^{2,p}(\Omegaextk{1+\delta})$. Furthermore, we still have $- \Delta \uext =  \fext$ in $\Omegaextk{1+\delta} \subset \Omegaextk{3}$. Using the Calder\'on--Zygmund estimate \cite[Theorem 9.11]{GiT01} with $\Omega \subset \subset \Omegaextk{1+\delta}$, we conclude that there exist constants (depending on $p$, $\Omega$ and $\delta$), such that
		\begin{align}
			\| u \|_{W^{2,p}(\Omega)}
				 =
				  \| \uext \|_{W^{2,p}(\Omega)} 
				  \lesssim  \bigl(
				  \|  \uext \|_{L^p(\Omegaextk{1+\delta})} 
				   +
				    \|  \fext \|_{L^p(\Omegaextk{1+\delta})} 
				    \bigr)
    	\lesssim
    		\bigl(
    	\|  u \|_{L^p(\Omega)} +	\|  f \|_{L^p(\Omega)} 
    	\bigr),
		\end{align}
	which proves the claim.
	\end{proof}

\subsection{Mixed boundary conditions}

We now turn to the regularity results of the vector potential $\MagPot$. As mentioned above, the $H^2$-regularity of a solution $\MagPotgen \in \HonenSpaceDiv(\Omega)$ of \eqref{eq:curlcurl_prob} follows from \cite[Lemma~3.7]{HocJS15}. In addition, the reference shows that the first component $\MagPotgen_1$ satisfies
\begin{equation} \label{eq:mixed_bc}
\begin{aligned}
	- \Delta \MagPotgen_1 &= f  \text{ in } \Omega  ,
	\\
	\MagPotgen_1 &= 0 \text{ on } 
	\Gamma_1 = \{ x \in \Omega \mid x_1 = 0 \text{ or } x_1 = 1\}, 
	\\
	\partial_2 \MagPotgen_1 &= 0 \text{ on }
	 \Gamma_2 = \{x \in \Omega \mid x_2 = 0 \text{ or } x_2 = 1\},
	\\
	\partial_3 \MagPotgen_1 &= 0 \text{ on }
	 \Gamma_3 = \{ x \in \Omega \mid x_3 = 0 \text{ or } x_3 = 1\},
\end{aligned}
\end{equation}
and similarly the other two components by interchanging the roles of the faces. We thus only study  the case of $\MagPotgen_1$. We follow the ideas of the Neumann case but now with the eigenbasis of the from
\begin{equation}
	\sin( \pi k x) \cos( \pi \ell y) \cos(\pi m z)
\end{equation}
in mind. This means odd reflections in $x_1$-direction and even reflections on $x_2$- and $x_3$-direction.
We therefore introduce the spaces
\begin{align}
	H_{0,1}^1(\Omega) &\coloneqq \{  \varphi \in H^1(\Omega) \mid \varphi  = 0 \text{ on } \Gamma_1  \},
	\\
	H_{0,1}^2(\Omega) &\coloneqq \{  \varphi \in H^2(\Omega) \mid \varphi  = 0 \text{ on } \Gamma_1 , \,
	\partial_\nu \varphi  = 0 \text{ on } \Gamma_2 \cup \Gamma_3  \}.
\end{align}
If $ \Omega$ is replaced by  a larger cube $\Omegaextk{s}$, we denote by $\Gamma_i \coloneqq \{ x \in \partial \Omegaextk{s} \mid x_i = -s \text{ or } x_i =s  \}$.
For a function $f\in C(\bar{\Omega})$, we define the mixed extension $M_{D,N} \colon  f \to\fext$ with
\begin{align}
\fext(x_1,x_2,x_3) 
	=
	\begin{cases}
		f(x_1,x_2,x_3), \quad &x_1 \in \intervalop{0,1}, x_2 \in \intervalop{0,1} ,x_3 \in \intervalop{0,1},
		\\
		f(x_1,x_2,-x_3), \quad &x_1 \in \intervalop{0,1}, x_2 \in \intervalop{0,1} ,x_3 \in \intervalop{-1,0},
		\\
		f(x_1,-x_2,-x_3), \quad &x_1 \in \intervalop{0,1}, x_2 \in \intervalop{-1,0} ,x_3 \in \intervalop{-1,0},
		\\
		f(x_1,-x_2,x_3), \quad &x_1 \in \intervalop{0,1}, x_2 \in \intervalop{-1,0} ,x_3 \in \intervalop{0,1},
		\\
	-	f(-x_1,x_2,x_3), \quad &x_1 \in \intervalop{-1,0}, x_2 \in \intervalop{0,1} ,x_3 \in \intervalop{0,1},
		\\
	-	f(-x_1,x_2,-x_3), \quad &x_1 \in \intervalop{-1,0}, x_2 \in \intervalop{0,1} ,x_3 \in \intervalop{-1,0},
		\\
		-f(-x_1,-x_2,-x_3), \quad &x_1 \in \intervalop{-1,0}, x_2 \in \intervalop{-1,0} ,x_3 \in \intervalop{-1,0},
		\\
		- f(-x_1,-x_2,x_3), \quad &x_1 \in \intervalop{-1,0}, x_2 \in \intervalop{-1,0} ,x_3 \in \intervalop{0,1}.
	\end{cases}
\end{align}
As in the Neumann case, this extension preserves the regularity of the inserted functions.

\begin{lemma} \label{lem:MDN_H1}
	Let $f \in H_{0,1}^1(\Omega)$ and $g \in H_{0,1}^2(\Omega)$.
	
	(a) $M_{D,N} f \in H_{0,1}^1 (\Omegaextk{1})$ with 
	$\norm{H^1(\Omegaextk{1})}{M_N f}
	\leq 
	2^3 \norm{H^1(\Omega)}{M_N f}$. 
	
	(b) The periodic extension of $M_{D,N} f$ satisfies $M_{D,N} f \in H_{0,1}^1(\Omegaextk{2k+1}) $ for all $k\geq 1$ with
	\begin{equation}
		\norm{H^1(\Omegaextk{2k+1})}{M_{D,N} f}
		\leq 
		(2k+2)^3 \norm{H^1(\Omega)}{M_{D,N} f}.
	\end{equation}

	(c) $M_{D,N} g \in H_{0,1}^2 (\Omegaextk{1})$ with 
	$\norm{H^2(\Omegaextk{1})}{M_{D,N} g}
	\leq 
	2^3 \norm{H^2(\Omega)}{M_{D,N} g}$. 
	
	(d) The periodic extension of $M_{D,N} f$ satisfies $M_N g \in H_{0,1}^2(\Omegaextk{2k+1}) $ for all $k\geq 1$ with
	\begin{equation}
		\norm{H^2(\Omegaextk{2k+1})}{M_{D,N} g}
		\leq 
		(2k+2)^3 \norm{H^2(\Omega)}{M_{D,N} g}.
	\end{equation}
\end{lemma}

\begin{proof}
	The claims on $H^1$ are easily verified, as we preserve continuity at all faces. Further, the computations for all faces where $x_1$ is either positive or negative are fully analogous to Lemma~\ref{lem:MN_H2} as all the normal  traces vanish. Hence, we check the conditions at $x_1 = 0$, and let for example $x_2,x_3 < 0$. Then, 
	\begin{align}
	\nabla \fext (x_1,x_2,x_3)	 &= 	I_{1,-1,-1}	\nabla	f \big|_{(x_1,-x_2,-x_3)}, \quad & &x_1 > 0 
	\\
	\nabla \fext (x_1,x_2,x_3)	 &=	- 	I_{-1,-1,-1} \nabla	f \big|_{(-x_1,-x_2,-x_3)} , \quad  & &x_1 < 0 ,
	\end{align} 
	and we obtain
		\begin{align}
		\lim\limits_{x_1 \to 0^+} \partial_\nu \fext(x_1,x_2,x_3) 
		=
		\lim\limits_{x_1 \to 0^+} \partial_1 f(x_1,-x_2,-x_3) 
		=
		\partial_1 f(0,-x_2,-x_3) 
	\end{align}
	as well as
	\begin{align}
		\lim\limits_{x_1 \to 0^-} \partial_\nu \fext(x_1,x_2,x_3) 
=
 \lim\limits_{x_1 \to 0^+} \partial_1 f(-x_1,-x_2,-x_3) 
=
\partial_1 f(0,-x_2,-x_3) ,
	\end{align}
and we also obtain here the continuity of the normal traces of the gradient. 
\end{proof}

We can then turn to the proof of the second part of Theorem~\ref{thm:reg_MagPot_H2_H3_gen}.

\begin{proof}[Proof of Theorem~\ref{thm:reg_MagPot_H2_H3_gen} (b)]
Let us recall that by part (a) $\MagPotgen \in \mathbf{H}^2(\Omega)$ solves 
\begin{equation}
	\Delta \MagPotgen = \mathbf{G},
	\qquad
	\MagPotgen \cdot \nu |_{\partial \Omega} = 
	\curl \MagPotgen \times  \nu |_{\partial \Omega} = 0,
\end{equation}
with $\mathbf{G} = \mathbf{F} + \curl \MagF$.
The regularity of $\MagF$ and $\mathbf{F}$
and the conditions
\begin{equation}
\curl \MagF \cdot  \nu |_{\partial \Omega} = \mathbf{F} \cdot  \nu |_{\partial \Omega} = 0
\quad
\text{imply }
\quad  \mathbf{G} \in \mathbf{H}^1(\Omega) 
\,\, \mbox{and} \,\,
\mathbf{G}
\cdot \nu |_{\partial \Omega} = 0.
\end{equation}
Now taking the first component $\MagPotgen_1$, wee see that $f$ in \eqref{eq:mixed_bc} is given by the first component of $\mathbf{G}$,
and thus satisfies $f \in H^1_{0,1}(\Omega)$.  
Lemma~\ref{lem:MDN_H1} further ensures that $\fext \in H^1_{0,1}(\Omegaext_{3})$ holds.
For $\MagPotext = M_{D,N} \MagPotgen$ we argue as in the proof of Lemma~\ref{lemma-H3-regularity-u} and observe that choosing the correct case in the definition of $M_{D,N}$
	\begin{equation}
		-	\Delta \MagPotext \big|_{(x_1,x_2, x_3)} 
		=
		(-1)^m \Delta \MagPotgen \big|_{(\pm x_1,\pm x_2,\pm x_3)}
		=
		(-1)^{m} f  \big|_{(\pm x_1,\pm x_2,\pm x_3)}
		=
		\fext(x_1,x_2,x_3)
	\end{equation}
with $m=0$ for $x_1 > 0$ and $m=1$ for $x_1 < 0$. Hence, $\MagPotext \in  H^2_{0,1}(\Omegaextk{3})$ solves the mixed problem \eqref{eq:mixed_bc}
with right-hand side
$\fext \in H^1_{0,1}(\Omegaextk{3})$.
Again, interior regularity for elliptic problems (cf. 
\cite[Theorem~6.3.2]{Evans10}) gives  
$\MagPotgen = \MagPotext\vert_{\Omega} \in H^3(\Omega)$ with
\begin{align}
	\| \MagPotgen \|_{H^3(\Omega)} 
	\lesssim
	 \| \MagPotext \|_{L^2(\Omegaextk{3})}
	  +
	  \| \fext \|_{H^1(\Omegaextk{3})}
	   \lesssim 
	    \| \MagPotgen \|_{L^2(\Omega)} + \| f \|_{H^1(\Omega)},
\end{align}
which yields the claim.
\end{proof}

\section{Proofs of additional auxiliary results}
\label{sec:app_computations}

In this section, we present the proofs of Lemma~\ref{lem:epsilon_2_bound} and Lemma~\ref{lem:inverse-inequality-LOD}.
\begin{proof}[Proof of Lemma~\ref{lem:epsilon_2_bound}]
	Using $\duenergy(\sol,\MagPot) \varphi = 0 $
	and
	$\dAenergy(\sol,\MagPot) \mathbf{B} = 0$, we obtain the identity  
	\begin{eqnarray}
		\nonumber \lefteqn{ \dualp{\energystab''(\sol,\MagPot) (\sol ,\MagPot) }{(\psi_{\deltaH}, \mathbf{C}_{\deltah})}  }\\
		\nonumber &=&
		\Real  \int_\Omega \bigl( \kappainvci  \nabla \sol +  \MagPot \sol \bigr)  \cdot \bigl( \kappainvci  \nabla \psi_{\deltaH} +   \MagPot \psi_{\deltaH} \bigr)^*  
		+
		\bigl( |\sol|^2 -1 \bigr)  \sol \psi_{\deltaH}^* +2 |\sol|^2 \sol \psi_{\deltaH}^* 
		\dint{x} \\
		\nonumber &\enspace& +\,\,
		\int_\Omega  2   |\sol|^2  \MagPot \cdot \mathbf{C}_{\deltah}  + \kappainv  \Real \bigl( \ci  \sol^* \nabla \sol + \ci \sol^* \nabla \sol \bigr) \cdot \mathbf{C}_{\deltah} \dint{x}
		\\
		\nonumber &\enspace&+\,\,
		\int_\Omega 
		2     \Real (\sol \psi_{\deltaH}^* ) |\MagPot|^2
		+
		\kappainv  \Real \bigl( \ci  
		\sol^* \nabla \psi_{\deltaH}
		+
		\ci \psi_{\deltaH}^* \nabla \sol \bigr) 
		\cdot \MagPot  
		\dint{x}
		\\
		\nonumber &\enspace&+\,\,
		\int_\Omega 
		|\sol|^2 \mathbf{C}_{\deltah} \cdot \MagPot  
		+
		\curl \mathbf{C}_{\deltah} \cdot \curl \MagPot 
		+
		\div \mathbf{C}_{\deltah} \cdot \div \MagPot 
		\dint{x} 
		\\
		\nonumber &=&
		\Real  \int_\Omega 2 |\sol|^2 \sol \psi_{\deltaH}^* 
		\dint{x} \,\,+\,\,
		\int_\Omega 
		2   |\sol|^2  \MagPot \cdot \mathbf{C}_{\deltah}  
		+
		\kappainv  \Real \bigl( \ci  
		\sol^* \nabla \sol
		\bigr) 
		\cdot \mathbf{C}_{\deltah}  
		\dint{x}
		\\
		\label{appendix-eq-secE} &\enspace&+\,\,
		\int_\Omega 
		2     \Real (\sol \psi_{\deltaH}^* ) |\MagPot|^2   
		+
		\kappainv  \Real \bigl( \ci  
		\sol^* \nabla \psi_{\deltaH}
		+
		\ci \psi_{\deltaH}^* \nabla \sol \bigr) 
		\cdot \MagPot  
		+
		\mathbf{H} \cdot \curl \mathbf{C}_{\deltah}
		\dint{x}  .
	\end{eqnarray}
	Analogously, the same identity holds for $\dualp{\energystab''(\sol_{\deltaH},\MagPot_{\deltah}) (\sol_{\deltaH} ,\MagPot_{\deltah}) }{(\psi_{\deltaH}, \mathbf{C}_{\deltah})}$ if we replace $(\sol,\MagPot)$ by $(\sol_{\deltaH} , \MagPot_{\deltah})$ at all occurrences. Next, we use the decomposition
	\begin{eqnarray*}
		\varepsilon
		&=&
		\underbrace{\energystab''(\sol,\MagPot) (\sol ,\MagPot) - \energystab''(\sol_{\deltaH},\MagPot_{\deltah}) (\sol_{\deltaH} ,\MagPot_{\deltah})}_{=: \varepsilon_{1}}
		\,\,+\,\,
		\underbrace{\energystab''(\sol_{\deltaH},\MagPot_{\deltah}) (\sol_{\deltaH} ,\MagPot_{\deltah}) - \energystab''(\sol,\MagPot) (\sol_{\deltaH} , \MagPot_{\deltah})}_{=:  \varepsilon_{2}}
	\end{eqnarray*}
	to sort the terms and treat them together. For the first term we obtain with \eqref{appendix-eq-secE} that 	
	\begin{eqnarray*}
		\lefteqn{ \varepsilon_{1}(\psi_{\deltaH},\mathbf{C}_{\deltah}) \,\,\,=\,\,\,
			\underbrace{\Real  \int_\Omega 2 (|\sol|^2 \sol - |\sol_{\deltaH}|^2 \sol_{\deltaH}) \psi_{\deltaH}^* \dint{x}}_{=: \alpha_{1} }  \,\,+\,\,
			\underbrace{\int_\Omega 2 (  |\sol|^2  \MagPot -  |\sol_{\deltaH}|^2  \MagPot_{\deltah}) \cdot \mathbf{C}_{\deltah}   \dint{x}}_{=: \alpha_{2} } } \\
		&\enspace&+
		\underbrace{\int_\Omega \kappainv  \Real \bigl( \ci \sol^* \nabla \sol - \ci  \sol_{\deltaH}^* \nabla \sol_{\deltaH} \bigr) \cdot \mathbf{C}_{\deltah} \dint{x}}_{=: \alpha_{3} } \,\, + \,\,
		\underbrace{\int_\Omega 2 \Real (\sol \psi_{\deltaH}^* ) |\MagPot|^2 - 2 \Real (\sol_{\deltaH} \psi_{\deltaH}^* ) |\MagPot_{\deltah}|^2 \dint{x}}_{=: \alpha_{4} }
		\\
		&\enspace&+
		\underbrace{\int_\Omega 
			\kappainv  \Real \bigl( \ci  
			\sol^* \nabla \psi_{\deltaH} 
			+
			\ci \psi_{\deltaH}^* \nabla \sol \bigr) 
			\cdot \MagPot   
			-
			\kappainv  \Real \bigl( \ci  
			\sol_{\deltaH}^* \nabla \psi_{\deltaH}
			+
			\ci \psi_{\deltaH}^* \nabla \sol_{\deltaH} \bigr) 
			\cdot \MagPot_{\deltah}  
			\dint{x} }_{=: \alpha_{5} }.
	\end{eqnarray*}
	Using
	\begin{eqnarray}
		\nonumber\lefteqn{ \Real  \int_\Omega \bigl( \kappainvci  \nabla \sol_{\deltaH} +   \MagPot_{\deltah} \sol_{\deltaH} \bigr)  \cdot \bigl( \kappainvci  \nabla \psi_{\deltaH} +   \MagPot_{\deltah} \psi_{\deltaH} \bigr)^*  -  \bigl( \kappainvci  \nabla \sol_{\deltaH} +   \MagPot \sol_{\deltaH} \bigr)  \cdot \bigl( \kappainvci  \nabla \psi_{\deltaH} +   \MagPot \psi_{\deltaH} \bigr)^* \dint{x} }\\
		\label{term-a-identity} &=& \Real  \int_\Omega \ \tfrac{\ci}{\kappa} \nabla \sol_{\deltaH} \cdot (\MagPot_{\deltah} - \MagPot) \psi_{\deltaH}^* - \tfrac{\ci}{\kappa}  \sol_{\deltaH} \, (\MagPot_{\deltah} - \MagPot) \nabla \psi_{\deltaH}^* + (|\MagPot_{\deltah}|^2-|\MagPot|^2) \sol_{\deltaH} \psi_{\deltaH}^* \dint{x},
	\end{eqnarray}
	the second term satisfies
	\begin{eqnarray*}
		\lefteqn{ \varepsilon_{2}(\psi_{\deltaH},\mathbf{C}_{\deltah})
			\,\,=\,\,
			\dualp{\energystab''(\sol_{\deltaH},\MagPot_{\deltah}) (\sol_{\deltaH} ,\MagPot_{\deltah}) }{(\psi_{\deltaH},\mathbf{C}_{\deltah})} 
			-
			\dualp{\energystab''(\sol,\MagPot) (\sol_{\deltaH} , \MagPot_{\deltah}) }{(\psi_{\deltaH},\mathbf{C}_{\deltah})} }
		\\
		&=&
		\Real  \int_\Omega \bigl( \kappainvci  \nabla \sol_{\deltaH} +   \MagPot_{\deltah} \sol_{\deltaH} \bigr)  \cdot \bigl( \kappainvci  \nabla \psi_{\deltaH} +   \MagPot_{\deltah} \psi_{\deltaH} \bigr)^*  
		+
		\bigl( |\sol_{\deltaH}|^2 -1 \bigr)  \sol_{\deltaH} \psi_{\deltaH}^* + \sol_{\deltaH}^2 \sol_{\deltaH}^* \psi_{\deltaH}^* + |\sol_{\deltaH}|^2 \sol_{\deltaH} \psi_{\deltaH}^* 
		\dint{x} 
		\\
		&\enspace& +\,\,
		\int_\Omega 2 \Real (\sol_{\deltaH} \sol_{\deltaH}^* ) \MagPot_{\deltah} \cdot \mathbf{C}_{\deltah}  
		+ \kappainv  \Real \bigl( \ci \sol_{\deltaH}^* \nabla \sol_{\deltaH}
		+ \ci \sol_{\deltaH}^* \nabla \sol_{\deltaH} \bigr) \cdot \mathbf{C}_{\deltah} \dint{x}
		\\
		&\enspace& +\,\,
		\int_\Omega 2 \Real (\sol_{\deltaH} \psi_{\deltaH}^* ) \MagPot_{\deltah} \cdot \MagPot_{\deltah}  
		+ \kappainv  \Real \bigl( \ci \sol_{\deltaH}^* \nabla \psi_{\deltaH}
		+ \ci \psi_{\deltaH}^* \nabla \sol_{\deltaH} \bigr) \cdot \MagPot_{\deltah} \dint{x}
		\\
		&\enspace& +\,\, \int_\Omega |\sol_{\deltaH}|^2 \mathbf{C}_{\deltah} \cdot \MagPot_{\deltah}  
		+ \curl \mathbf{C}_{\deltah} \cdot \curl \MagPot_{\deltah} 
		+ \div \mathbf{C}_{\deltah} \cdot \div \MagPot_{\deltah} \dint{x} 
		\\
		&\enspace& -\,\,
		\Real  \int_\Omega \bigl( \kappainvci  \nabla \sol_{\deltaH} +   \MagPot \sol_{\deltaH} \bigr)  \cdot \bigl( \kappainvci  \nabla \psi_{\deltaH} +   \MagPot \psi_{\deltaH} \bigr)^*  
		+
		\bigl( |\sol|^2 -1 \bigr)  \sol_{\deltaH} \psi_{\deltaH}^* + \sol^2 \sol_{\deltaH}^* \psi_{\deltaH}^* + |\sol|^2 \sol_{\deltaH} \psi_{\deltaH}^* 
		\dint{x} 
		\\
		&\enspace& -\,\,
		\int_\Omega 
		2     \Real (\sol \sol_{\deltaH}^* ) \MagPot \cdot \mathbf{C}_{\deltah}  
		+
		\kappainv  \Real \bigl( \ci  
		\sol^* \nabla \sol_{\deltaH}
		+
		\ci \sol_{\deltaH}^* \nabla \sol \bigr) 
		\cdot \mathbf{C}_{\deltah}  
		\dint{x}
		\\
		&\enspace& -\,\,
		\int_\Omega 
		2     \Real (\sol \psi_{\deltaH}^* ) \MagPot \cdot \MagPot_{\deltah}  
		+
		\kappainv  \Real \bigl( \ci  
		\sol^* \nabla \psi_{\deltaH}
		+
		\ci \psi_{\deltaH}^* \nabla \sol \bigr) 
		\cdot \MagPot_{\deltah}  
		\dint{x}
		\\
		&\enspace& -\,\,
		\int_\Omega 
		|\sol|^2 \mathbf{C}_{\deltah} \cdot \MagPot_{\deltah}  
		+ \curl \mathbf{C}_{\deltah} \cdot \curl \MagPot_{\deltah} 
		+ \div \mathbf{C}_{\deltah} \cdot \div \MagPot_{\deltah} \dint{x} \\
		&\overset{\eqref{term-a-identity}}{=}&
		\underbrace{ \,\, \Real  \int_\Omega 2 \bigl( |\sol_{\deltaH}|^2 -|\sol|^2 \bigr)  \sol_{\deltaH} \psi_{\deltaH}^* + (\sol_{\deltaH}^2 -  \sol^2) \sol_{\deltaH}^* \psi_{\deltaH}^* \dint{x}  }_{=: \beta_{1} }
		\\
		&\enspace& +\,\,
		\underbrace{\int_\Omega 2 \Real (\sol_{\deltaH} \sol_{\deltaH}^* ) \MagPot_{\deltah} \cdot \mathbf{C}_{\deltah}  
			+  (|\sol_{\deltaH}|^2 - |\sol|^2 ) \MagPot_{\deltah} \cdot \mathbf{C}_{\deltah}
			- 2 \Real (\sol \sol_{\deltaH}^* ) \MagPot \cdot \mathbf{C}_{\deltah}  \dint{x} }_{=: \beta_{2} } \\
		&\enspace& +\,\,		
		\underbrace{\int_\Omega \kappainv  \Real \bigl( \ci (\sol_{\deltaH} -\sol)^* \nabla \sol_{\deltaH}
			+ \ci \sol_{\deltaH}^* (\nabla \sol_{\deltaH} - \nabla \sol) \bigr) \cdot \mathbf{C}_{\deltah} \dint{x}  }_{=: \beta_{3} }
		\\
		&\enspace& +\,\, \underbrace{\int_\Omega (3|\MagPot_{\deltah}|^2 - |\MagPot|^2  ) \Real( \sol_{\deltaH} \psi_{\deltaH}^* ) -
			2 \MagPot \cdot \MagPot_{\deltah}  \Real( \sol \psi_{\deltaH}^*)   \dint{x}}_{=: \beta_{4} }
		\\
		&\enspace& +\,\,
		\underbrace{ \Real \left( \frac{\ci}{\kappa} \int_\Omega  \, (\MagPot_{\deltah} \hspace{-2pt} - \hspace{-2pt} \MagPot) \cdot \left( \nabla \sol_{\deltaH}  \psi_{\deltaH}^* \hspace{-1pt}-\hspace{-1pt} \sol_{\deltaH}  \nabla \psi_{\deltaH}^* \right)  +   \bigl( (\sol_{\deltaH} \hspace{-1pt}-\hspace{-1pt} \sol)^* \nabla \psi_{\deltaH}
			+ \psi_{\deltaH}^* \nabla (\sol_{\deltaH} \hspace{-1pt}-\hspace{-1pt} \sol)  \bigr) \cdot \MagPot_{\deltah} \dint{x} \right)}_{=: \beta_{5} }.
	\end{eqnarray*}
	We investigate the various terms. For brevity, let us define $\errh \coloneqq \sol - \sol_{\deltaH}$ and $\Errh \coloneqq \MagPot -  \MagPot_{\deltah}$.
	\begin{itemize} [leftmargin=1.2em]
		\item[$\bullet$] $\alpha_{1}+\beta_{1}$: First, we note that with $\sol_{\deltaH} = \sol -  \errh$ we obtain 
		\begin{eqnarray*}
			2 \bigl( |\sol|^2 \sol - |\sol_{\deltaH}|^2 \sol_{\deltaH}  \bigr)
			&=& 2 \bigl( |\sol|^2 \sol  
			-
			( |\sol|^2  - \sol^* \errh - \errh^* \sol + |\errh|^2 ) (\sol - \errh)
			\bigr)
			\\
			&=& 4 |\sol|^2 \errh + 2 \sol^2 \errh^*  - |\errh|^2 + \sol |\errh|^2 - \sol^{\ast} \errh^2 + \errh |\errh|^2
		\end{eqnarray*}
		as well as
		\begin{eqnarray*}
			\lefteqn{ 2 \bigl( |\sol_{\deltaH}|^2 -  |\sol|^2 \bigr)  \sol_{\deltaH}  + ( \sol_{\deltaH}^2 - \sol^2) \sol_{\deltaH}^* }
			\\
			&=& 	2 \bigl( |\sol|^2 - \sol^* \errh - \errh^* \sol + |\errh|^2 -  |\sol|^2 \bigr)  (\sol - \errh)  + ( \sol^2 - 2 \sol \errh + \errh^2 - \sol^2) (\sol - \errh)^* 
			\\
			&=& - 4 |\sol|^2 \errh -2 \sol^2 \errh^* + 3 \sol^* \errh^2 + 6 |\errh|^2 \sol - 3 \errh |\errh|^2.
		\end{eqnarray*}
		Consequently,
		\begin{eqnarray*}
			|\alpha_{1}+\beta_{1}| &=& | \Real  \int_\Omega \left( 
			2 \bigl( |\sol|^2 \sol - |\sol_{\deltaH}|^2 \sol_{\deltaH}  \bigr) + 2 \bigl( |\sol_{\deltaH}|^2 -  |\sol|^2 \bigr)  \sol_{\deltaH}  + ( \sol_{\deltaH}^2 - \sol^2) \sol_{\deltaH}^*\right)  \psi_{\deltaH}^* \dint{x} |  \\ 
			&=&  | \Real  \int_\Omega ((7\sol- 1) |\errh|^2  + 2 \sol^* \errh^2 - 2 \errh |\errh|^2)  \, \psi_{\deltaH}^* \dint{x} | \\
			&\le & 10\,  \| \errh \|_{L^4}^2 \| \psi_{\deltaH}\|_{L^2} +  2\,  \| \errh \|_{L^6}^3 \| \psi_{\deltaH}\|_{L^2}.
		\end{eqnarray*}		
		\item[$\bullet$] $\alpha_{2}+\beta_{2}$: First, we note that
		\begin{eqnarray*}
			\lefteqn{ 2 (  |\sol|^2  \MagPot -  |\sol_{\deltaH}|^2  \MagPot_{\deltah}) \,+\, 2 \Real (\sol_{\deltaH} \sol_{\deltaH}^* ) \MagPot_{\deltah} + (|\sol_{\deltaH}|^2 - |\sol|^2 ) \MagPot_{\deltah}  - 2 \Real (\sol \sol_{\deltaH}^* ) \MagPot }\\
			&=& 2   |\sol|^2  \MagPot + (|\sol_{\deltaH}|^2 - |\sol|^2 ) \MagPot_{\deltah}  - 2 \Real (\sol \sol_{\deltaH}^* ) \MagPot 
			\,\,=\,\,  |\sol - \sol_{\deltaH}|^2 \MagPot   +  (|\sol_{\deltaH}|^2 - |\sol|^2 ) (\MagPot_{\deltah} - \MagPot).
		\end{eqnarray*}
		With this, we obtain again with $\errh = \sol - \sol_{\deltaH}$ and $\Errh = \MagPot -  \MagPot_{\deltah}$
		\begin{eqnarray*}
			\lefteqn{ |\alpha_{2}+\beta_{2}| 
				\,\,=\,\, | \int_\Omega |\errh|^2 \MagPot \cdot \mathbf{C}_{\deltah}  +  (|\sol_{\deltaH}|^2 - |\sol|^2 ) \Errh \cdot \mathbf{C}_{\deltah} \dint{x} | }\\
			&\le& \| \errh \|_{L^4}^2  \| \MagPot \|_{L^4} \| \mathbf{C}_{\deltah} \|_{L^4} + \| \errh \|_{L^4} \| \Errh \|_{L^4}  \| \mathbf{C}_{\deltah} \|_{L^4} \,\| |\sol_{\deltaH}|\hspace{-2pt}+\hspace{-2pt} |\sol| \|_{L^4} \\
			&\lesssim& \left(\| \errh \|_{L^4}^2 + \| \Errh \|_{L^4}^2 \right) \| \mathbf{C}_{\deltah} \|_{H^1}.
		\end{eqnarray*}	
		\item[$\bullet$] $\alpha_{3}+\beta_{3}$: We use
		\begin{eqnarray*}
			\sol^* \nabla \sol - \sol_{\deltaH}^* \nabla \sol_{\deltaH} 
			+ (\sol_{\deltaH} -\sol)^* \nabla \sol_{\deltaH} + \sol_{\deltaH}^* (\nabla \sol_{\deltaH} - \nabla \sol) 
			&=& (\sol - \sol_{\deltaH})^{\ast} (\nabla u - \nabla \sol_{\deltaH})
		\end{eqnarray*}
		to obtain
		\begin{eqnarray*}
			\lefteqn{ |\alpha_{3}+\beta_{3}| } \\
			&=& | \kappainv \int_\Omega  \Real \bigl( \ci \sol^* \nabla \sol - \ci  \sol_{\deltaH}^* \nabla \sol_{\deltaH} \bigr) \cdot \mathbf{C}_{\deltah}  - 
			\Real \bigl( \ci (\sol_{\deltaH} -\sol)^* \nabla \sol_{\deltaH}
			+ \ci \sol_{\deltaH}^* (\nabla \sol_{\deltaH} - \nabla \sol) \bigr) \cdot \mathbf{C}_{\deltah} \dint{x} | \\
			&=& | \kappainv \Real \int_\Omega \ci \, (\sol - \sol_{\deltaH})^{\ast} (\nabla u - \nabla \sol_{\deltaH}) \cdot \mathbf{C}_{\deltah} \dint{x} | \\
			&\le& \| \tfrac{1}{\kappa} \nabla \errh \|_{L^2}  \| \errh \|_{L^4} \| \mathbf{C}_{\deltah} \|_{L^4} \,\, \lesssim \,\, ( \| \errh \|_{\HonekappaSpace}^2 + \|  \errh \|_{L^4}^2 ) \,  \| \mathbf{C}_{\deltah} \|_{H^1}.
		\end{eqnarray*}	
		\item[$\bullet$] $\alpha_{4}+\beta_{4}$:  Noting that
		\begin{eqnarray*}
			\lefteqn{ 2 \Real (\sol \psi_{\deltaH}^* ) |\MagPot|^2 - 2 \Real (\sol_{\deltaH} \psi_{\deltaH}^* ) |\MagPot_{\deltah}|^2 
				+ (3|\MagPot_{\deltah}|^2 - |\MagPot|^2  ) \Real( \sol_{\deltaH} \psi_{\deltaH}^* ) - 2 \MagPot \cdot \MagPot_{\deltah}  \Real( \sol \psi_{\deltaH}^*) }\\
			&=&  \Real ( \sol  \psi_{\deltaH}^* ) |\MagPot - \MagPot_{\deltah}|^2
			+ \Real ((\sol - \sol_{\deltaH}) \psi_{\deltaH}^* ) \,( \MagPot - \MagPot_{\deltah}) \cdot ( \MagPot + \MagPot_{\deltah}),  \hspace{90pt}
		\end{eqnarray*}
		we obtain with the bounds from Lemma \ref{lem:sol_bounds_approx} 
		\begin{eqnarray*}
			|\alpha_{4}+\beta_{4}| 
			&=& 	| \int_\Omega  \Real ( \sol  \psi_{\deltaH}^* ) |\MagPot - \MagPot_{\deltah}|^2
			+ \Real ((\sol - \sol_{\deltaH}) \psi_{\deltaH}^* ) \,( \MagPot - \MagPot_{\deltah}) \cdot ( \MagPot + \MagPot_{\deltah}) \dint{x} | \\
			&\le& \| \MagPot - \MagPot_{\deltah} \|_{L^4}^2  \| \psi_{\deltaH} \|_{L^2}  + \| \sol - \sol_{\deltaH} \|_{L^6}  \| \psi_{\deltaH} \|_{L^2}   \| \MagPot - \MagPot_{\deltah} \|_{L^6}  \| \MagPot + \MagPot_{\deltah} \|_{L^6} \\
			&\lesssim&\left(  \| \errh \|_{L^6}^2  + \| \Errh \|_{L^4}^2 + \| \Errh \|_{L^6}^2 \right) \| \psi_{\deltaH} \|_{L^2}.
		\end{eqnarray*}	
		\item[$\bullet$] $\alpha_{5}+\beta_{5}$: It holds
		\begin{eqnarray*}
			\lefteqn{ \bigl( \sol^* \nabla \psi_{\deltaH} + \psi_{\deltaH}^* \nabla \sol \bigr) \cdot \MagPot 
				-  \bigl( \sol_{\deltaH}^* \nabla \psi_{\deltaH} + \psi_{\deltaH}^* \nabla \sol_{\deltaH} \bigr) \cdot \MagPot_{\deltah} }\\
			&\enspace& +\,\, (\MagPot_{\deltah} \hspace{-2pt} - \hspace{-2pt} \MagPot) \cdot \left( \nabla \sol_{\deltaH}  \psi_{\deltaH}^* \hspace{-1pt}-\hspace{-1pt} \sol_{\deltaH}  \nabla \psi_{\deltaH}^* \right)  +   \bigl( (\sol_{\deltaH} \hspace{-1pt}-\hspace{-1pt} \sol)^* \nabla \psi_{\deltaH}
			+ \psi_{\deltaH}^* \nabla (\sol_{\deltaH} \hspace{-1pt}-\hspace{-1pt} \sol)  \bigr) \cdot \MagPot_{\deltah} \\
			&=& ( \sol^* \nabla \psi_{\deltaH} + \sol_{\deltaH}  \nabla \psi_{\deltaH}^*) \cdot (\MagPot - \MagPot_{\deltah}) 
			\,\, +\,\,  \psi_{\deltaH}^*  (\nabla \sol - \nabla \sol_{\deltaH}) \cdot (\MagPot \hspace{-2pt} - \hspace{-2pt} \MagPot_{\deltah}).
		\end{eqnarray*}
		Together with $\Real \,( \ci \,\sol^* \nabla \psi_{\deltaH} ) = \Real \,( \sol (\ci \, \nabla \psi_{\deltaH})^{*} )= -  \Real \,( \ci \, \sol \, \nabla \psi_{\deltaH}^{*} )$, we hence obtain
		\begin{eqnarray*}
			\lefteqn{ |\alpha_{5}+\beta_{5}| } \\
			&=& \frac{1}{\kappa} | \int_\Omega \Real \bigl( \ci \, ( \sol^* \nabla \psi_{\deltaH} + \sol_{\deltaH}  \nabla \psi_{\deltaH}^*) \cdot (\MagPot - \MagPot_{\deltah}) 
			\,\, +\,\,  \ci \, \psi_{\deltaH}^*  (\nabla \sol - \nabla \sol_{\deltaH}) \cdot (\MagPot \hspace{-2pt} - \hspace{-2pt} \MagPot_{\deltah}) \bigr) \dint{x} | \\
			&=& \frac{1}{\kappa} | \int_\Omega \Real \bigl( \ci \, (\sol_{\deltaH} - \sol) \nabla \psi_{\deltaH}^* \cdot (\MagPot - \MagPot_{\deltah}) 
			\,\, +\,\,  \ci \, \psi_{\deltaH}^*  (\nabla \sol - \nabla \sol_{\deltaH}) \cdot (\MagPot \hspace{-2pt} - \hspace{-2pt} \MagPot_{\deltah}) \bigr) \dint{x} | \\
			&=& \frac{1}{\kappa} | \int_\Omega \Real \bigl( \, \ci \, 2 \, (\sol_{\deltaH} - \sol) \nabla \psi_{\deltaH}^* \cdot (\MagPot - \MagPot_{\deltah}) 
			\,\, +\,\,  \ci \, \psi_{\deltaH}^*  ( \sol_{\deltaH} - \sol ) \cdot \div (\MagPot \hspace{-2pt} - \hspace{-2pt} \MagPot_{\deltah}) \bigr) \dint{x} | \\
			&\le& 2 \| \tfrac{1}{\kappa} \nabla \psi_{\deltaH}\|_{L^2}  \|  \sol - \sol_{\deltaH}\|_{L^4}  \| \MagPot - \MagPot_{\deltah} \|_{L^4} + \| \tfrac{1}{\kappa} \psi_{\deltaH}\|_{L^4}  \|  \sol - \sol_{\deltaH} \|_{L^4}  \| \div(\MagPot - \MagPot_{\deltah}) \|_{L^2} \\
			&\lesssim& \left(  \|  \sol - \sol_{\deltaH}\|_{L^4}^2 + \| \MagPot - \MagPot_{\deltah} \|_{H^1}^2 \right)  \| \psi_{\deltaH}\|_{\HonekappaSpace}.
		\end{eqnarray*}
	\end{itemize}
	It remains to sum up the previous estimates. We obtain
	\begin{eqnarray*}
		\lefteqn{  |\sigma(\psi_{\deltaH},\mathbf{C}_{\deltah})| 
			\,\, \le \,\, |\alpha_{1}+\beta_{1}| + |\alpha_{2}+\beta_{2}|  + |\alpha_{3}+\beta_{3}|  + |\alpha_{4}+\beta_{4}|  + |\alpha_{5}+\beta_{5}|  } \\
		&\lesssim& (   \|  \sol - \sol_{\deltaH} \|_{L^6}^2  +  \|  \sol - \sol_{\deltaH} \|_{\HonekappaSpace}^2 + \| \MagPot - \MagPot_{\deltah} \|_{H^1}^2 )\,(  \| \psi_{\deltaH}\|_{\HonekappaSpace} + \| \mathbf{C}_{\deltah} \|_{H^1}),
	\end{eqnarray*}
and thus the assertion.
\end{proof}
%

\begin{proof}[Proof of Lemma~\ref{lem:inverse-inequality-LOD}]
	The proof follows  \cite[Lem.~10.8]{HeW22} with some modifications to account for the missing coercivity of $\abilmagLOD{\cdot}{\cdot}$ on $H^1(\Omega)$ and the influence of $\kappa$ on the estimates. Let $\varphi_{H}^{\LOD} \in \VShLOD$ be arbitrary, then we can write it as $\varphi_{H}^{\LOD} = (1 - \mathcal{C}) \varphi_{H}$ for some $\varphi_{H} \in \VSh$. By definition of the corrector $\mathcal{C}$ we have 
	for the $L^2$-projection $\LtwoprojFEM : H^1(\Omega) \rightarrow \VSh$ that $\LtwoprojFEM(\mathcal{C} \varphi_{H})=0$ hence with the approximation properties of $\LtwoprojFEM$ we conclude
	\begin{eqnarray}
		\label{est:CphiH-in-L2-est}
		\| \mathcal{C} \varphi_{H} \|_{L^2} = \| (1-\mathcal{C}) \varphi_{H} - \LtwoprojFEM((1-\mathcal{C}) \varphi_{H}) \|_{L^2} \lesssim 
		\deltaH \| \nabla (1-\mathcal{C}) \varphi_{H} \|_{L^2}.
	\end{eqnarray}
	Next, we obtain from $\abilmagLOD{(1 - \mathcal{C}) \varphi_{H}}{\mathcal{C} \varphi_{H}} = 0$ that
	\begin{eqnarray*}
		\tfrac{1}{\kappa^2} \| \nabla (1 - \mathcal{C}) \varphi_{H} \|_{L^2}^2 &\lesssim&  \abilmag{(1 - \mathcal{C}) \varphi_{H}}{(1 - \mathcal{C}) \varphi_{H}} \\
		&=& \abilmag{(1 - \mathcal{C}) \varphi_{H}}{ \varphi_{H}} +   \int_\Omega	( |\MagPot|^2 + 1) |(1 - \mathcal{C}) \varphi_{H} |^2 \dint{x} \\
		&\lesssim& \delta \| (1 - \mathcal{C}) \varphi_{H} \|_{\HonekappaSpace}^2 + \tfrac{1}{\delta}  \|  \varphi_{H} \|_{\HonekappaSpace}^2 + \| (1 - \mathcal{C}) \varphi_{H} \|_{L^2}^2,
	\end{eqnarray*}
	for any $\delta>0$ using Young's inequality. Hence, for sufficiently small $\delta$ (independent of $\deltaH$ or $\kappa$), we conclude $\| (1 - \mathcal{C}) \varphi_{H} \|_{\HonekappaSpace}^2 \lesssim \|  \varphi_{H} \|_{\HonekappaSpace}^2 + \| (1 - \mathcal{C}) \varphi_{H} \|_{L^2}^2$ which we can further estimate with the standard inverse inequality in Lagrange FE spaces as
	\begin{eqnarray*}
		\tfrac{1}{\kappa^2} \| \nabla (1 - \mathcal{C}) \varphi_{H} \|_{L^2}^2
		&\lesssim& \|  \varphi_{H} \|_{L^2}^2 + \tfrac{1}{\kappa^2 \deltaH^2} \|  \varphi_{H} \|_{L^2}^2 + \| (1 - \mathcal{C}) \varphi_{H} \|_{L^2}^2 \\
		&= & (1+\tfrac{1}{\kappa^2 \deltaH^2} ) (\varphi_{H} , (1 - \mathcal{C}) \varphi_{H} )_{L^2} +  \| (1 - \mathcal{C}) \varphi_{H} \|_{L^2}^2 \\
		&\lesssim&  (1+\tfrac{1}{\kappa^2 \deltaH^2} ) \| (1 - \mathcal{C}) \varphi_{H} \|_{L^2}^2  +  (1+\tfrac{1}{\kappa^2 \deltaH^2} )  (\mathcal{C} \varphi_{H} , (1 - \mathcal{C}) \varphi_{H} )_{L^2} \\
		&\lesssim&  (1+\tfrac{1}{\kappa^2 \deltaH^2} + \tfrac{1}{\delta}) \| (1 - \mathcal{C}) \varphi_{H} \|_{L^2}^2  +  \delta \, (1+\tfrac{1}{\kappa^2 \deltaH^2} ) \| \mathcal{C} \varphi_{H} \|_{L^2}^2 \\
		&\overset{\eqref{est:CphiH-in-L2-est}}{\lesssim}&  (1+\tfrac{1}{\kappa^2 \deltaH^2} + \tfrac{1}{\delta}) \| (1 - \mathcal{C}) \varphi_{H} \|_{L^2}^2  +  \delta \, (\deltaH^2+\tfrac{1}{\kappa^2} )  \| \nabla (1-\mathcal{C}) \varphi_{H} \|_{L^2}.
	\end{eqnarray*}
	Using $\deltaH \lesssim \kappa^{-1}$, we can absorb the right term for sufficiently small $\delta$ into the left-hand side and conclude
	\begin{eqnarray*}
		\tfrac{1}{\kappa^2} \| \nabla (1 - \mathcal{C}) \varphi_{H} \|_{L^2}^2
		&\lesssim&  %
		\tfrac{1}{\kappa^2 \deltaH^2} \| (1 - \mathcal{C}) \varphi_{H} \|_{L^2}^2.
		\qedhere
	\end{eqnarray*}
\end{proof}


\begin{bibdiv}
\begin{biblist}

\bib{Abr04}{article}{
      author={Abrikosov, A.~A.},
       title={Nobel lecture: Type-{II} superconductors and the vortex lattice},
        date={2004},
        ISSN={{0034-6861}},
     journal={Rev. Mod. Phys.},
      volume={76},
      number={3,1},
       pages={975\ndash 979},
        note={\url{https://doi.org/10.1103/RevModPhys.76.975}},
}

\bib{Aftalion99}{article}{
      author={Aftalion, A.},
       title={On the minimizers of the {G}inzburg-{L}andau energy for high
  kappa: the axially symmetric case},
        date={1999},
        ISSN={0294-1449},
     journal={Ann. Inst. H. Poincar\'{e} Anal. Non Lin\'{e}aire},
      volume={16},
      number={6},
       pages={747\ndash 772},
         url={https://doi.org/10.1016/S0294-1449(00)88186-X},
      review={\MR{1720515}},
}

\bib{AHP21acta}{article}{
      author={Altmann, R.},
      author={Henning, P.},
      author={Peterseim, D.},
       title={Numerical homogenization beyond scale separation},
        date={2021},
        ISSN={0962-4929},
     journal={Acta Numer.},
      volume={30},
       pages={1\ndash 86},
         url={https://doi.org/10.1017/S0962492921000015},
      review={\MR{4298217}},
}

\bib{BanY14}{article}{
      author={Bank, R.~E.},
      author={Yserentant, H.},
       title={On the {$H^1$}-stability of the {$L_2$}-projection onto finite
  element spaces},
        date={2014},
        ISSN={0029-599X},
     journal={Numer. Math.},
      volume={126},
      number={2},
       pages={361\ndash 381},
         url={https://doi.org/10.1007/s00211-013-0562-4},
      review={\MR{3150226}},
}

\bib{Bartels2005}{article}{
      author={Bartels, S.},
       title={A posteriori error analysis for time-dependent
  {G}inzburg-{L}andau type equations},
        date={2005},
        ISSN={0029-599X},
     journal={Numer. Math.},
      volume={99},
      number={4},
       pages={557\ndash 583},
         url={https://doi.org/10.1007/s00211-004-0560-7},
      review={\MR{2121069}},
}

\bib{Bar05}{article}{
      author={Bartels, S.},
       title={Robust a priori error analysis for the approximation of
  degree-one {G}inzburg-{L}andau vortices},
        date={2005},
        ISSN={0764-583X},
     journal={M2AN Math. Model. Numer. Anal.},
      volume={39},
      number={5},
       pages={863\ndash 882},
         url={https://doi.org/10.1051/m2an:2005038},
      review={\MR{2178565}},
}

\bib{BarMO11}{article}{
      author={Bartels, S.},
      author={M{\"u}ller, R.},
      author={Ortner, C.},
       title={Robust a priori and a posteriori error analysis for the
  approximation of {A}llen-{C}ahn and {G}inzburg-{L}andau equations past
  topological changes},
        date={2011},
        ISSN={0036-1429},
     journal={SIAM J. Numer. Anal.},
      volume={49},
      number={1},
       pages={110\ndash 134},
         url={https://doi.org/10.1137/090751530},
      review={\MR{2764423}},
}

\bib{BDH24}{article}{
      author={Blum, M.},
      author={D\"oding, C.},
      author={Henning, P.},
       title={Vortex-capturing multiscale spaces for the {G}inzburg-{L}andau
  equation},
        date={2025},
        ISSN={1540-3459,1540-3467},
     journal={Multiscale Model. Simul.},
      volume={23},
      number={1},
       pages={339\ndash 373},
         url={https://doi.org/10.1137/24M1664538},
      review={\MR{4857950}},
}

\bib{BrennerScott}{book}{
      author={Brenner, Susanne~C.},
      author={Scott, L.~Ridgway},
       title={The mathematical theory of finite element methods},
     edition={Third},
      series={Texts in Applied Mathematics},
   publisher={Springer, New York},
        date={2008},
      volume={15},
        ISBN={978-0-387-75933-3},
         url={https://doi.org/10.1007/978-0-387-75934-0},
      review={\MR{2373954}},
}

\bib{BrezisMironescu18}{article}{
      author={Brezis, H.},
      author={Mironescu, P.},
       title={Gagliardo-{N}irenberg inequalities and non-inequalities: the full
  story},
        date={2018},
        ISSN={0294-1449,1873-1430},
     journal={Ann. Inst. H. Poincar\'e{} C Anal. Non Lin\'eaire},
      volume={35},
      number={5},
       pages={1355\ndash 1376},
         url={https://doi.org/10.1016/j.anihpc.2017.11.007},
      review={\MR{3813967}},
}

\bib{Casas-trroeltsch-2015}{article}{
      author={Casas, E.},
      author={Tr\"{o}ltzsch, F.},
       title={Second order optimality conditions and their role in {PDE}
  control},
        date={2015},
        ISSN={0012-0456,1869-7135},
     journal={Jahresber. Dtsch. Math.-Ver.},
      volume={117},
      number={1},
       pages={3\ndash 44},
         url={https://doi.org/10.1365/s13291-014-0109-3},
      review={\MR{3311948}},
}

\bib{CheGZ10}{article}{
      author={Chen, H.},
      author={Gong, X.},
      author={Zhou, A.},
       title={Numerical approximations of a nonlinear eigenvalue problem and
  applications to a density functional model},
        date={2010},
        ISSN={0170-4214},
     journal={Math. Methods Appl. Sci.},
      volume={33},
      number={14},
       pages={1723\ndash 1742},
         url={https://doi.org/10.1002/mma.1292},
      review={\MR{2723492}},
}

\bib{ChenWuEllipticTheoryLp}{book}{
      author={Chen, Y.},
      author={Wu, L.},
       title={Second order elliptic equations and elliptic systems},
      series={Translations of Mathematical Monographs},
   publisher={American Mathematical Society, Providence, RI},
        date={1998},
      volume={174},
        ISBN={0-8218-0970-9},
         url={https://doi.org/10.1090/mmono/174},
        note={Translated from the 1991 Chinese original by Bei Hu},
      review={\MR{1616087}},
}

\bib{Chen97}{article}{
      author={Chen, Z.},
       title={Mixed finite element methods for a dynamical {G}inzburg-{L}andau
  model in superconductivity},
        date={1997},
        ISSN={0029-599X},
     journal={Numer. Math.},
      volume={76},
      number={3},
       pages={323\ndash 353},
         url={https://doi.org/10.1007/s002110050266},
      review={\MR{1452512}},
}

\bib{CheD01}{article}{
      author={Chen, Z.},
      author={Dai, S.},
       title={Adaptive {G}alerkin methods with error control for a dynamical
  {G}inzburg-{L}andau model in superconductivity},
        date={2001},
        ISSN={0036-1429},
     journal={SIAM J. Numer. Anal.},
      volume={38},
      number={6},
       pages={1961\ndash 1985},
         url={https://doi.org/10.1137/S0036142998349102},
      review={\MR{1856238}},
}

\bib{CosD00}{article}{
      author={Costabel, M.},
      author={Dauge, M.},
       title={Singularities of electromagnetic fields in polyhedral domains},
        date={2000},
        ISSN={0003-9527,1432-0673},
     journal={Arch. Ration. Mech. Anal.},
      volume={151},
      number={3},
       pages={221\ndash 276},
         url={https://doi.org/10.1007/s002050050197},
      review={\MR{1753704}},
}

\bib{DST21}{article}{
      author={Diening, Lars},
      author={Storn, Johannes},
      author={Tscherpel, Tabea},
       title={On the {S}obolev and {$L^p$}-stability of the
  {$L^2$}-projection},
        date={2021},
        ISSN={0036-1429,1095-7170},
     journal={SIAM J. Numer. Anal.},
      volume={59},
      number={5},
       pages={2571\ndash 2607},
         url={https://doi.org/10.1137/20M1358013},
      review={\MR{4320894}},
}

\bib{DHW24}{article}{
      author={D\"oding, C.},
      author={Henning, P.},
      author={W\"arneg{\aa}rd, J.},
       title={A two level approach for simulating {B}ose--{E}instein
  condensates by localized orthogonal decomposition},
        date={2024+},
     journal={M2AN Math. Model. Numer. Anal.},
      eprint={https://doi.org/10.1051/m2an/2024040},
         url={https://doi.org/10.1051/m2an/2024040},
}

\bib{DoHaMa23}{article}{
      author={Dong, Z.},
      author={Hauck, M.},
      author={Maier, R.},
       title={An improved high-order method for elliptic multiscale problems},
        date={2023},
        ISSN={0036-1429,1095-7170},
     journal={SIAM J. Numer. Anal.},
      volume={61},
      number={4},
       pages={1918\ndash 1937},
         url={https://doi.org/10.1137/22M153392X},
      review={\MR{4620425}},
}

\bib{DoeH24}{article}{
      author={D\"{o}rich, B.},
      author={Henning, P.},
       title={Error bounds for discrete minimizers of the {G}inzburg–{L}andau
  energy in the high-\(\boldsymbol{\kappa }\) regime},
        date={2024},
     journal={SIAM J. Numer. Anal.},
      volume={62},
      number={3},
       pages={1313\ndash 1343},
      eprint={https://doi.org/10.1137/23M1560938},
         url={https://doi.org/10.1137/23M1560938},
}

\bib{DouglasDupontWahlbin74}{article}{
      author={Douglas, Jim, Jr.},
      author={Dupont, Todd},
      author={Wahlbin, Lars},
       title={The stability in {$L\sp{q}$} of the {$L\sp{2}$}-projection into
  finite element function spaces},
        date={1974/75},
        ISSN={0029-599X,0945-3245},
     journal={Numer. Math.},
      volume={23},
       pages={193\ndash 197},
         url={https://doi.org/10.1007/BF01400302},
      review={\MR{383789}},
}

\bib{Du94b}{article}{
      author={Du, Q.},
       title={Finite element methods for the time-dependent {G}inzburg-{L}andau
  model of superconductivity},
        date={1994},
        ISSN={0898-1221},
     journal={Comput. Math. Appl.},
      volume={27},
      number={12},
       pages={119\ndash 133},
         url={https://doi.org/10.1016/0898-1221(94)90091-4},
      review={\MR{1284135}},
}

\bib{Du94}{article}{
      author={Du, Q.},
       title={Global existence and uniqueness of solutions of the
  time-dependent {G}inzburg-{L}andau model for superconductivity},
        date={1994},
        ISSN={0003-6811},
     journal={Appl. Anal.},
      volume={53},
      number={1-2},
       pages={1\ndash 17},
         url={https://doi.org/10.1080/00036819408840240},
      review={\MR{1379180}},
}

\bib{Du97}{article}{
      author={Du, Q.},
       title={Discrete gauge invariant approximations of a time dependent
  {G}inzburg-{L}andau model of superconductivity},
        date={1998},
        ISSN={0025-5718},
     journal={Math. Comp.},
      volume={67},
      number={223},
       pages={965\ndash 986},
         url={https://doi.org/10.1090/S0025-5718-98-00954-5},
      review={\MR{1464143}},
}

\bib{Du2003}{incollection}{
      author={Du, Q.},
       title={Diverse vortex dynamics in superfluids},
        date={2003},
   booktitle={Current trends in scientific computing ({X}i'an, 2002)},
      series={Contemp. Math.},
      volume={329},
   publisher={Amer. Math. Soc., Providence, RI},
       pages={105\ndash 117},
         url={https://doi.org/10.1090/conm/329/05847},
      review={\MR{2022637}},
}

\bib{DuGray96}{article}{
      author={Du, Q.},
      author={Gray, P.},
       title={High-kappa limits of the time-dependent {G}inzburg-{L}andau
  model},
        date={1996},
        ISSN={0036-1399},
     journal={SIAM J. Appl. Math.},
      volume={56},
      number={4},
       pages={1060\ndash 1093},
         url={https://doi.org/10.1137/S0036139995280506},
      review={\MR{1398408}},
}

\bib{DuGP92}{article}{
      author={Du, Q.},
      author={Gunzburger, M.~D.},
      author={Peterson, J.~S.},
       title={Analysis and approximation of the {G}inzburg-{L}andau model of
  superconductivity},
        date={1992},
        ISSN={0036-1445},
     journal={SIAM Rev.},
      volume={34},
      number={1},
       pages={54\ndash 81},
         url={https://doi.org/10.1137/1034003},
      review={\MR{1156289}},
}

\bib{DuGP93}{article}{
      author={Du, Q.},
      author={Gunzburger, M.~D.},
      author={Peterson, J.~S.},
       title={Modeling and analysis of a periodic {G}inzburg-{L}andau model for
  type-{${\rm II}$} superconductors},
        date={1993},
        ISSN={0036-1399},
     journal={SIAM J. Appl. Math.},
      volume={53},
      number={3},
       pages={689\ndash 717},
         url={https://doi.org/10.1137/0153035},
      review={\MR{1218380}},
}

\bib{QuJu05}{article}{
      author={Du, Q.},
      author={Ju, L.},
       title={Approximations of a {G}inzburg-{L}andau model for superconducting
  hollow spheres based on spherical centroidal {V}oronoi tessellations},
        date={2005},
        ISSN={0025-5718},
     journal={Math. Comp.},
      volume={74},
      number={251},
       pages={1257\ndash 1280},
         url={https://doi.org/10.1090/S0025-5718-04-01719-3},
      review={\MR{2137002}},
}

\bib{DuNicolaidesWu98}{article}{
      author={Du, Q.},
      author={Nicolaides, R.~A.},
      author={Wu, X.},
       title={Analysis and convergence of a covolume approximation of the
  {G}inzburg-{L}andau model of superconductivity},
        date={1998},
        ISSN={0036-1429},
     journal={SIAM J. Numer. Anal.},
      volume={35},
      number={3},
       pages={1049\ndash 1072},
         url={https://doi.org/10.1137/S0036142996302852},
      review={\MR{1619855}},
}

\bib{DuanZhang22}{article}{
      author={Duan, H.},
      author={Zhang, Q.},
       title={Residual-based a posteriori error estimates for the
  time-dependent {G}inzburg-{L}andau equations of superconductivity},
        date={2022},
        ISSN={0885-7474},
     journal={J. Sci. Comput.},
      volume={93},
      number={3},
       pages={Paper No. 79, 47},
         url={https://doi.org/10.1007/s10915-022-02041-0},
      review={\MR{4507136}},
}

\bib{EHMP19}{article}{
      author={Engwer, C.},
      author={Henning, P.},
      author={M{\aa}lqvist, A.},
      author={Peterseim, D.},
       title={Efficient implementation of the localized orthogonal
  decomposition method},
        date={2019},
        ISSN={0045-7825},
     journal={Comput. Methods Appl. Mech. Engrg.},
      volume={350},
       pages={123\ndash 153},
         url={https://doi.org/10.1016/j.cma.2019.02.040},
      review={\MR{3926249}},
}

\bib{Evans10}{book}{
      author={Evans, L.~C.},
       title={Partial differential equations},
     edition={Second},
      series={Graduate Studies in Mathematics},
   publisher={American Mathematical Society, Providence, RI},
        date={2010},
      volume={19},
        ISBN={978-0-8218-4974-3},
         url={https://doi.org/10.1090/gsm/019},
      review={\MR{2597943}},
}

\bib{Finnetal01}{article}{
      author={Finnemore, D.~K.},
      author={Ostenson, J.~E.},
      author={Bud'ko, S.~L.},
      author={Lapertot, G.},
      author={Canfield, P.~C.},
       title={Thermodynamic and transport properties of superconducting
  ${Mg}^{10}{B}_{2}$},
        date={2001},
     journal={Phys. Rev. Lett.},
      volume={86},
       pages={2420\ndash 2422},
      eprint={https://doi.org/10.1103/PhysRevLett.86.2420},
         url={https://link.aps.org/doi/10.1103/PhysRevLett.86.2420},
}

\bib{ForH10}{book}{
      author={Fournais, S.},
      author={Helffer, B.},
       title={Spectral methods in surface superconductivity},
      series={Progress in Nonlinear Differential Equations and their
  Applications},
   publisher={Birkh\"auser Boston, Inc., Boston, MA},
        date={2010},
      volume={77},
        ISBN={978-0-8176-4796-4},
      review={\MR{2662319}},
}

\bib{GHV17}{article}{
      author={Gallistl, D.},
      author={Henning, P.},
      author={Verf\"{u}rth, B.},
       title={Numerical homogenization of {${\bf{H}}(\rm curl)$}-problems},
        date={2018},
        ISSN={0036-1429},
     journal={SIAM J. Numer. Anal.},
      volume={56},
      number={3},
       pages={1570\ndash 1596},
         url={https://doi.org/10.1137/17M1133932},
      review={\MR{3810505}},
}

\bib{GJX19}{article}{
      author={Gao, H.},
      author={Ju, L.},
      author={Xie, W.},
       title={A stabilized semi-implicit {E}uler gauge-invariant method for the
  time-dependent {G}inzburg-{L}andau equations},
        date={2019},
        ISSN={0885-7474},
     journal={J. Sci. Comput.},
      volume={80},
      number={2},
       pages={1083\ndash 1115},
         url={https://doi.org/10.1007/s10915-019-00968-5},
      review={\MR{3977199}},
}

\bib{GaoS18}{article}{
      author={Gao, H.},
      author={Sun, W.},
       title={Analysis of linearized {G}alerkin-mixed {FEM}s for the
  time-dependent {G}inzburg-{L}andau equations of superconductivity},
        date={2018},
        ISSN={1019-7168},
     journal={Adv. Comput. Math.},
      volume={44},
      number={3},
       pages={923\ndash 949},
         url={https://doi.org/10.1007/s10444-017-9568-2},
      review={\MR{3814665}},
}

\bib{GiT01}{book}{
      author={Gilbarg, D.},
      author={Trudinger, N.~S.},
       title={Elliptic partial differential equations of second order},
      series={Classics in Mathematics},
   publisher={Springer-Verlag, Berlin},
        date={2001},
        ISBN={3-540-41160-7},
        note={Reprint of the 1998 edition},
      review={\MR{1814364}},
}

\bib{GirR86}{book}{
      author={Girault, V.},
      author={Raviart, P.-A.},
       title={Finite element methods for {N}avier-{S}tokes equations},
      series={Springer Series in Computational Mathematics},
   publisher={Springer-Verlag, Berlin},
        date={1986},
      volume={5},
        ISBN={3-540-15796-4},
         url={https://doi.org/10.1007/978-3-642-61623-5},
        note={Theory and algorithms},
      review={\MR{851383}},
}

\bib{Grisvard85}{book}{
      author={Grisvard, P.},
       title={Elliptic problems in nonsmooth domains},
      series={Monographs and Studies in Mathematics},
   publisher={Pitman (Advanced Publishing Program), Boston, MA},
        date={1985},
      volume={24},
        ISBN={0-273-08647-2},
      review={\MR{775683}},
}

\bib{HaP23}{article}{
      author={Hauck, M.},
      author={Peterseim, D.},
       title={Super-localization of elliptic multiscale problems},
        date={2023},
        ISSN={0025-5718},
     journal={Math. Comp.},
      volume={92},
      number={341},
       pages={981\ndash 1003},
         url={https://doi.org/10.1090/mcom/3798},
      review={\MR{4550317}},
}

\bib{HeP13}{article}{
      author={Henning, P.},
      author={Peterseim, D.},
       title={Oversampling for the {M}ultiscale {F}inite {E}lement {M}ethod},
        date={2013},
        ISSN={1540-3459},
     journal={SIAM Multiscale Model. Simul.},
      volume={11},
      number={4},
       pages={1149\ndash 1175},
         url={http://dx.doi.org/10.1137/120900332},
      review={\MR{3123820}},
}

\bib{HeW22}{article}{
      author={Henning, P.},
      author={W{\"a}rneg{\aa}rd, J.},
       title={Superconvergence of time invariants for the {G}ross-{P}itaevskii
  equation},
        date={2022},
        ISSN={0025-5718},
     journal={Math. Comp.},
      volume={91},
      number={334},
       pages={509\ndash 555},
         url={https://doi.org/10.1090/mcom/3693},
      review={\MR{4379968}},
}

\bib{HocJS15}{article}{
      author={Hochbruck, M.},
      author={Jahnke, T.},
      author={Schnaubelt, R.},
       title={Convergence of an {ADI} splitting for {M}axwell's equations},
        date={2015},
        ISSN={0029-599X,0945-3245},
     journal={Numer. Math.},
      volume={129},
      number={3},
       pages={535\ndash 561},
         url={https://doi.org/10.1007/s00211-014-0642-0},
      review={\MR{3311460}},
}

\bib{LaLjMa24}{article}{
      author={Lang, A.},
      author={Ljung, P.},
      author={M{\aa}lqvist, A.},
       title={Localized orthogonal decomposition for a multiscale parabolic
  stochastic partial differential equation},
        date={2024},
        ISSN={1540-3459,1540-3467},
     journal={Multiscale Model. Simul.},
      volume={22},
      number={1},
       pages={204\ndash 229},
         url={https://doi.org/10.1137/23M1569216},
      review={\MR{4695712}},
}

\bib{Li17}{article}{
      author={Li, B.},
       title={Convergence of a decoupled mixed {FEM} for the dynamic
  {G}inzburg-{L}andau equations in nonsmooth domains with incompatible initial
  data},
        date={2017},
        ISSN={0008-0624},
     journal={Calcolo},
      volume={54},
      number={4},
       pages={1441\ndash 1480},
         url={https://doi.org/10.1007/s10092-017-0237-0},
      review={\MR{3735822}},
}

\bib{LiZ15}{article}{
      author={Li, B.},
      author={Zhang, Z.},
       title={A new approach for numerical simulation of the time-dependent
  {G}inzburg-{L}andau equations},
        date={2015},
        ISSN={0021-9991},
     journal={J. Comput. Phys.},
      volume={303},
       pages={238\ndash 250},
         url={http://dx.doi.org/10.1016/j.jcp.2015.09.049},
      review={\MR{3422711}},
}

\bib{LiZ17}{article}{
      author={Li, B.},
      author={Zhang, Z.},
       title={Mathematical and numerical analysis of the time-dependent
  {G}inzburg-{L}andau equations in nonconvex polygons based on {H}odge
  decomposition},
        date={2017},
        ISSN={0025-5718},
     journal={Math. Comp.},
      volume={86},
      number={306},
       pages={1579\ndash 1608},
         url={http://dx.doi.org/10.1090/mcom/3177},
      review={\MR{3626529}},
}

\bib{LjMaMa22}{article}{
      author={Ljung, P.},
      author={Maier, R.},
      author={M{\aa}lqvist, A.},
       title={A space-time multiscale method for parabolic problems},
        date={2022},
        ISSN={1540-3459,1540-3467},
     journal={Multiscale Model. Simul.},
      volume={20},
      number={2},
       pages={714\ndash 740},
         url={https://doi.org/10.1137/21M1446605},
      review={\MR{4490296}},
}

\bib{MaQ23}{article}{
      author={Ma, L.},
      author={Qiao, Z.},
       title={An energy stable and maximum bound principle preserving scheme
  for the dynamic {G}inzburg-{L}andau equations under the temporal gauge},
        date={2023},
        ISSN={0036-1429,1095-7170},
     journal={SIAM J. Numer. Anal.},
      volume={61},
      number={6},
       pages={2695\ndash 2717},
         url={https://doi.org/10.1137/22M1539812},
      review={\MR{4667738}},
}

\bib{MaVer22}{article}{
      author={Maier, B.},
      author={Verf\"urth, B.},
       title={Numerical upscaling for wave equations with time-dependent
  multiscale coefficients},
        date={2022},
        ISSN={1540-3459,1540-3467},
     journal={Multiscale Model. Simul.},
      volume={20},
      number={4},
       pages={1169\ndash 1190},
         url={https://doi.org/10.1137/21M1438244},
      review={\MR{4500085}},
}

\bib{Mai21}{article}{
      author={Maier, R.},
       title={A high-order approach to elliptic multiscale problems with
  general unstructured coefficients},
        date={2021},
        ISSN={0036-1429},
     journal={SIAM J. Numer. Anal.},
      volume={59},
      number={2},
       pages={1067\ndash 1089},
         url={https://doi.org/10.1137/20M1364321},
      review={\MR{4246089}},
}

\bib{MaPer18}{article}{
      author={M{\aa}lqvist, A.},
      author={Persson, A.},
       title={Multiscale techniques for parabolic equations},
        date={2018},
        ISSN={0029-599X,0945-3245},
     journal={Numer. Math.},
      volume={138},
      number={1},
       pages={191\ndash 217},
         url={https://doi.org/10.1007/s00211-017-0905-7},
      review={\MR{3745014}},
}

\bib{MaP14}{article}{
      author={M{\aa}lqvist, A.},
      author={Peterseim, D.},
       title={Localization of elliptic multiscale problems},
        date={2014},
        ISSN={0025-5718},
     journal={Math. Comp.},
      volume={83},
      number={290},
       pages={2583\ndash 2603},
         url={http://dx.doi.org/10.1090/S0025-5718-2014-02868-8},
      review={\MR{3246801}},
}

\bib{MaP21}{book}{
      author={M{\aa}lqvist, A.},
      author={Peterseim, D.},
       title={Numerical homogenization by localized orthogonal decomposition},
      series={SIAM Spotlights},
   publisher={Society for Industrial and Applied Mathematics (SIAM),
  Philadelphia, PA},
        date={[2021] \copyright 2021},
      volume={5},
        ISBN={978-1-611976-44-1},
      review={\MR{4191211}},
}

\bib{McCSe65}{article}{
      author={McConville, T.},
      author={Serin, B.},
       title={{G}inzburg-{L}andau parameters of type-{II} superconductors},
        date={1965},
     journal={Phys. Rev.},
      volume={140},
       pages={A1169\ndash A1177},
      eprint={https://link.aps.org/doi/10.1103/PhysRev.140.A1169},
         url={https://link.aps.org/doi/10.1103/PhysRev.140.A1169},
}

\bib{Neu97}{book}{
      author={Neuberger, J.~W.},
       title={Sobolev gradients and differential equations},
      series={Lecture Notes in Mathematics},
   publisher={Springer-Verlag, Berlin},
        date={1997},
      volume={1670},
        ISBN={3-540-63537-8},
         url={https://doi.org/10.1007/BFb0092831},
      review={\MR{1624197}},
}

\bib{Pet17}{article}{
      author={Peterseim, D.},
       title={Eliminating the pollution effect in {H}elmholtz problems by local
  subscale correction},
        date={2017},
        ISSN={0025-5718},
     journal={Math. Comp.},
      volume={86},
      number={305},
       pages={1005\ndash 1036},
         url={http://dx.doi.org/10.1090/mcom/3156},
      review={\MR{3614010}},
}

\bib{SaS07}{book}{
      author={Sandier, E.},
      author={Serfaty, S.},
       title={Vortices in the magnetic {G}inzburg-{L}andau model},
      series={Progress in Nonlinear Differential Equations and their
  Applications},
   publisher={Birkh\"{a}user Boston, Inc., Boston, MA},
        date={2007},
      volume={70},
        ISBN={978-0-8176-4316-4; 0-8176-4316-8},
      review={\MR{2279839}},
}

\bib{SaS12}{article}{
      author={Sandier, E.},
      author={Serfaty, S.},
       title={From the {G}inzburg-{L}andau model to vortex lattice problems},
        date={2012},
        ISSN={0010-3616},
     journal={Comm. Math. Phys.},
      volume={313},
      number={3},
       pages={635\ndash 743},
         url={https://doi.org/10.1007/s00220-012-1508-x},
      review={\MR{2945619}},
}

\bib{Sch74}{article}{
      author={Schatz, Alfred~H.},
       title={An observation concerning {R}itz-{G}alerkin methods with
  indefinite bilinear forms},
        date={1974},
        ISSN={0025-5718},
     journal={Math. Comp.},
      volume={28},
       pages={959\ndash 962},
      review={\MR{0373326 (51 \#9526)}},
}

\bib{Ser99}{article}{
      author={Serfaty, S.},
       title={Stable configurations in superconductivity: uniqueness,
  multiplicity, and vortex-nucleation},
        date={1999},
        ISSN={0003-9527},
     journal={Arch. Ration. Mech. Anal.},
      volume={149},
      number={4},
       pages={329\ndash 365},
         url={https://doi.org/10.1007/s002050050177},
      review={\MR{1731999}},
}

\bib{SeS10}{incollection}{
      author={Serfaty, S.},
      author={Sandier, E.},
       title={Vortex patterns in {G}inzburg-{L}andau minimizers},
        date={2010},
   booktitle={X{VI}th {I}nternational {C}ongress on {M}athematical {P}hysics},
   publisher={World Sci. Publ., Hackensack, NJ},
       pages={246\ndash 264},
         url={https://doi.org/10.1142/9789814304634_0014},
      review={\MR{2730781}},
}

\bib{StZw97}{article}{
      author={Stintzing, S.},
      author={Zwerger, W.},
       title={{G}inzburg-{L}andau theory of superconductors with short
  coherence length},
        date={1997},
     journal={Phys. Rev. B},
      volume={56},
       pages={9004\ndash 9014},
      eprint={https://link.aps.org/doi/10.1103/PhysRevB.56.9004},
         url={https://link.aps.org/doi/10.1103/PhysRevB.56.9004},
}

\bib{Struwe08}{book}{
      author={Struwe, M.},
       title={Variational methods},
     edition={Fourth},
      series={Ergebnisse der Mathematik und ihrer Grenzgebiete. 3. Folge. A
  Series of Modern Surveys in Mathematics [Results in Mathematics and Related
  Areas. 3rd Series. A Series of Modern Surveys in Mathematics]},
   publisher={Springer-Verlag, Berlin},
        date={2008},
      volume={34},
        ISBN={978-3-540-74012-4},
        note={Applications to nonlinear partial differential equations and
  Hamiltonian systems},
      review={\MR{2431434}},
}

\bib{WuZh22}{article}{
      author={Wu, Z.},
      author={Zhang, Z.},
       title={Convergence analysis of the localized orthogonal decomposition
  method for the semiclassical {S}chr\"{o}dinger equations with multiscale
  potentials},
        date={2022},
        ISSN={0885-7474},
     journal={J. Sci. Comput.},
      volume={93},
      number={3},
       pages={Paper No. 73, 30},
         url={https://doi.org/10.1007/s10915-022-02038-9},
      review={\MR{4505120}},
}

\end{biblist}
\end{bibdiv}
\end{document}